\documentclass[12pt,openany,oneside]{book}
\usepackage[english,polish]{babel}

\usepackage[T1]{fontenc}
\usepackage[utf8]{inputenc}
\usepackage{indentfirst}
\usepackage{fancyhdr}
\usepackage{amsthm}
\usepackage{amsmath}
\usepackage{amssymb}
\usepackage{braket}
\usepackage{cite,enumitem,graphicx}
\usepackage[colorlinks=true,urlcolor=blue,
citecolor=red,linkcolor=blue,linktocpage,
bookmarksnumbered,bookmarksopen,pdfpagelabels,]{hyperref}

\newcommand*{\titleAT}{\begingroup 
\newlength{\drop} 
\drop=0.1\textheight 

\centering 
\rule{\textwidth}{1pt}\par 
\vspace{2pt}\vspace{-\baselineskip} 
\rule{\textwidth}{0.4pt}\par 

\vspace{\drop} 
{\Large \textbf{Ground state, bound state, and normalized solutions to semilinear Maxwell and Schr\"odinger equations}}

\let\cleardoublepage=\clearpage
\vspace{0.25\drop} 
\rule{0.3\textwidth}{0.4pt}\par 
\vspace{\drop} 

{\Large \textsc{Jacopo Schino}}\par 
\vfill
{Supervisor: \textit{dr hab. Jaros{\l}aw Mederski}}\\

\vfill 

{\large \textsc{Instytut Matematyczny Polskiej Akademii Nauk}}\par 

\vspace*{\drop} 

\rule{\textwidth}{0.4pt}\par 
\vspace{2pt}\vspace{-\baselineskip} 
\rule{\textwidth}{1pt}\par 

\endgroup}

\newtheorem{Th}{Theorem}[section]
\newtheorem{Prop}[Th]{Proposition}
\newtheorem{Lem}[Th]{Lemma}
\newtheorem{Cor}[Th]{Corollary}
\theoremstyle{definition}
\newtheorem{Def}[Th]{Definition}
\newtheorem{Ex}[Th]{Example}
\theoremstyle{remark}
\newtheorem{Rem}[Th]{Remark}

\newcommand{\R}{\mathbb{R}}
\newcommand{\Z}{\mathbb{Z}}
\newcommand{\N}{\mathbb{N}}
\newcommand{\rr}{\R^3}
\newcommand{\rn}{\R^N}
\newcommand{\rk}{\R^K}
\newcommand{\rnk}{\R^{N-K}}
\newcommand{\hrn}{H^1(\rn)}
\newcommand{\cA}{\mathcal{A}}
\newcommand{\cS}{\mathcal{S}}
\newcommand{\cD}{\mathcal{D}}
\newcommand{\cC}{\mathcal{C}}
\newcommand{\cH}{\mathcal{H}}
\newcommand{\cO}{\mathcal{O}}
\newcommand{\cJ}{\mathcal{J}}
\newcommand{\cG}{\mathcal{G}}
\newcommand{\cK}{\mathcal{K}}
\newcommand{\cE}{\mathcal{E}}
\newcommand{\cB}{\mathcal{B}}
\newcommand{\cP}{\mathcal{P}}
\newcommand{\cM}{\mathcal{M}}
\newcommand{\cN}{\mathcal{N}}
\newcommand{\cF}{\mathcal{F}}
\newcommand{\cL}{\mathcal{L}}
\newcommand{\cV}{\mathcal{V}}
\newcommand{\cW}{\mathcal{W}}
\newcommand{\cT}{\mathcal{T}}
\newcommand{\cX}{\mathcal{X}}
\newcommand{\cY}{\mathcal{Y}}
\newcommand{\cZ}{\mathcal{Z}}
\newcommand{\cQ}{\mathcal{Q}}
\newcommand{\DF}{\cD_{\cF}}
\newcommand{\UU}{\mathbf{U}}
\newcommand{\VV}{\mathbf{V}}
\newcommand{\BB}{\mathbf{B}}

\newcommand{\mbS}{\mathbb{S}}
\newcommand{\SO}{\cS\cO}
\newcommand{\fN}{\mathfrak{N}}
\newcommand{\fM}{\mathfrak{M}}
\newcommand{\fI}{\mathfrak{I}}
\newcommand{\fH}{\mathfrak{H}}
\newcommand{\fS}{\mathfrak{S}}
\newcommand{\fD}{\mathfrak{D}}

\DeclareMathOperator*{\essinf}{ess\,inf}
\DeclareMathOperator*{\esssup}{ess\,sup}

\numberwithin{equation}{section}

\begin{document}

\titleAT

\newpage
Oświadczam, że niniejsza rozprawa została napisana przeze mnie
samodzielnie.\\[2em]

\noindent \begin{minipage}{0.4\textwidth}
	\begin{flushright}
		Jacopo Schino

		\vskip 2em
		........................................\\
		{\small (data i podpis)}
	\end{flushright}
\end{minipage}

\hfill \break\\[7em]

Niniejsza rozprawa jest gotowa do oceny przez recenzentów.\\[2em]

\noindent \begin{minipage}{0.4\textwidth}
	\begin{flushright}
		\mbox{dr hab. Jaros{\l}aw Mederski}\vskip 2em
		........................................\\
		{\small (data i podpis)}
	\end{flushright}
\end{minipage}

\selectlanguage{english}
\tableofcontents

\chapter*{Acknowledgements}

\section*{Academic acknowledgements}

First and foremost, my academic acknowledgements are to my supervisor Jaros{\l}aw Mederski, who with patience and enthusiasm has introduced me to to world of professional mathematics and guided me therein.

Then, I am grateful to my other co-authors Andrzej Szulkin and Micha{\l}  Gaczkowski; in particular, I would like to thank Andrzej for his hospitality during my stay at the Stockholm University.

I thank Panos Smyrnelis for inviting me to his mini-course in Bilbao and his will to have me as a collaborator. I also warmly thank Iwona Chlebicka for all the times she helped me with this or that issue, especially with Orlicz spaces, at the beginning of my Ph.D., and with my application for Preludium.

Finally, I show my gratitude to the other organizers of the Young Researchers Colloquium I worked with since February 2019: Adam Abrams, Filip Rupniewski, Jacek Krajczok, Joachim Jelisiejew, and Mariusz Tobolski.

\section*{Personal acknowledgements}

Of course, my first personal acknowledgements are to my family, who have always supported me during my Ph.D. and showed they are here for me whenever I need.

Next, I thank again Jarek Mederski, for his constant kindness towards me, for always behaving as a friend and not just as a supervisor, and for helping me a lot even before the start of my Ph.D., with my application at IM PAN, the search for an accommodation in Warsaw, and many other things.

I warmly thank Gianni Calabrese, in whom my social life at the beginning of my second year basically consisted and who first made Warsaw a city worth living in. I would also like to thank all my friends at IM PAN (and nearby) for the moments spent together and the great memories we will always carry with us; in particular, I am grateful to Eleonora Romano and Reza Mohammadpour. I thank as well all my friends in Bari (broadly speaking) that through their deeds have somehow shortened the distance between me and them.

I show all my gratitude to the staff at IM PAN for the welcoming and lovely atmosphere they created and keep on. I thank Marta Szpak too, my Polish teacher, for the passion in her classes and to whom I owe the possibility to face everyday situations in Poland.

Many thanks to my boardgames friends, without whom I could have never enjoyed one of my biggest passions, and to Tomek Cie\'slak, for having me play football a few times. I am very grateful to my landlord Piotr Durka too for his constant understanding and helpfulness.

In the end, last but by no means least, I deeply thank Taisiia Borshchova: by far the sweetest person I have met in Poland, undoubtedly the only good thing of my first accommodation, she has always been nice to me, has always helped me -- never asking anything in return -- every time I needed, and proved to be a true friend.

\chapter*{Abstract}

This Ph.D. thesis is concerned with the existence of entire solutions to semilinear elliptic equations and is divided into two parts, dealing with unconstrained and constrained problems respectively. Chapter \ref{K:prel} contains a series of recalls about notions and properties used throughout this work and lies before the aforementioned division into parts.

In Part \ref{1}, we study existence and multiplicity results for equations of the form
\begin{equation*}
\nabla\times\nabla\times\UU=f(x,\UU), \quad \UU\colon\rn\to\rn,
\end{equation*}
where $N\ge3$ and $f=\nabla F\colon\rn\times\rn\to\rn$ is the gradient (with respect to $\UU$) of a given nonlinear function $ F\colon\rn\times\rn\to\R$. Here, when $N\ge4$, $\nabla\times\nabla\times\UU$ is defined using the identity, valid when $N=3$, $\nabla\times\nabla\times\UU=\nabla(\nabla\cdot\UU)-\Delta\UU$. Such problems are known as \textit{curl-curl} problems and arise, when $N=3$, from the nonlinear Maxwell equations in absence of electric charges, electric currents, and magnetization. The main issue is that the kernel of the differential operator $\nabla\times\nabla\times$ consists of the subspace of gradient fields and is therefore infinite-dimensional. Historically, two approaches have been used to tackle curl-curl problems by means of variational methods and both make use of divergence-free vector fields. The reason is that $\nabla\times\nabla\times\UU=-\Delta\UU$ for every divergence-free field $\UU$ and the vector Laplacian is a differential operator easier to handle.

In Chapter \ref{K:intro1}, we give an accurate physical derivation of curl-curl problems and then survey important results throughout the last decades, from the first works to the current days, including those illustrated in this Ph.D. thesis.

In Chapter \ref{K:nosym}, based on \cite{MeScSz}, we focus on the (physically relevant) case $N=3$. The nonlinearity $F$ is controlled, from above and from below, by a suitable nice Young function: in particular, Sobolev-supercritical at zero, Sobolev-subcritical but superquadratic at infinity, and satisfying the $\Delta_2$ and $\nabla_2$ conditions globally. Our approach makes use of a Helmholtz-type decomposition of the function space we work with into a divergence-free subspace and a curl-free subspace (the aforementioned kernel), i.e., $u=v+w$ with $\nabla\cdot v=\nabla\times w=0$ and $(v,w)$ uniquely determined; then we build a homeomorphism from the former subspace to a certain topological submanifold (of the whole space) that contains all the nontrivial solutions. This somehow allows us to work only with the divergence-free subspace, although the ``curl-free'' part must be taken care of; in fact, that is what causes the most difficulties in the methods we use. We prove the existence of a least-energy solution and, if $f$ is odd, of infinitely many distinct solutions. Unlike Chapter \ref{K:cylsym}, we do not use any symmetries; in particular, we provide the first multiplicity result about curl-curl problems in unbounded domains without any symmetry assumptions.

In Chapter \ref{K:cylsym}, based on \cite{Gacz}, we consider the general case $N\ge3$. Under certain symmetry assumptions about the nonlinearity, we exploit suitable group actions to reduce the curl-curl problem to the Schr\"odinger equation with singular potential
\[
-\Delta u+\frac{a}{|y|^2}u=\tilde{f}(x,u), \quad u\colon\rn\to\R,
\]
with $x=(y,z)\in\rk\times\rnk$, $K=2$, and $a=1$, studying as well the general case $2\le K<N$ and $a>-(K/2-1)^2$. More in detail, we require that $f(\cdot,\alpha w)=\tilde{f}(\cdot,\alpha)w$ for every $\alpha\in\R$ and every $w\in\mbS^{N-1}$ and that $\tilde{f}(gx,\cdot)=\tilde{f}(x,\cdot)$ for a.e. $x\in\rn$ and every $g\in\SO(2)\times\{I_{N-2}\}$. Extending to the case of weak solutions a well-known equivalence property that puts in a 1-to-1 correspondence the classical solutions to the two problems via the formula $\UU(x)=u(x)/|y|(-x_2,x_1,0)$, we prove new existence results (nontrivial solutions, least-energy solutions relatively to the functions with the same symmetry, infinitely many distinct solutions) about both, in the Sobolev-critical and -noncritical cases; in particular, we work with the curl-curl equation in the former case, using the same symmetry machinery to reduce $\nabla\times\nabla\times\UU$ to $-\Delta\UU$, and with the Schr\"odinger equation in the latter. The most prominent result is the existence, when $N=3$, of a divergent sequence of solutions in the critical case, obtained with the aid of another group action, which restores compactness; this is the first multiplicity result for curl-curl problems in unbounded domains in the Sobolev-critical case. Concerning the existence, in the noncritical case, of a least-energy solution and of infinitely many distinct solutions, we exploit an abstract critical point theory built in Chapter \ref{K:nosym}.

In Part \ref{2}, we look for least-energy solutions to autonomous Schr\"odinger systems of the form
\[
\begin{cases}
-\Delta u_j+\lambda_ju_j=\partial_jF(u)\\
\int_{\rn}u_j^2\,dx=\rho_j^2
\end{cases}
\forall j\in\{1,\dots,K\}, \quad u\colon\rn\to\rk,
\]
where $N,K\ge1$, $\rho=(\rho_1,\dots,\rho_K)\in]0,\infty[^K$ is given, and $\lambda=(\lambda_1,\dots,\lambda_K)\in\rk$ is part of the unknown. Solutions to such problems are called \textit{normalized} due to the $L^2$-constraints, which are what causes the quantity $\lambda$ to appear as a $K$-tuple of Lagrange multipliers. Equations of this type arise when seeking standing wave solutions to similar time-dependent problems and come from areas of Physics such as nonlinear optics and Bose--Einstein condensation. Their importance lies in the physical meaning of the masses (the $L^2$ norms squared) and the fact that such quantities are conserved in time in the corresponding evolution equations.

In Chapter \ref{K:intro2}, we introduce the problem, briefly comment some seminal papers and other results in the literature, and provide useful preliminary properties.

Depending on the assumptions about $F$ and, sometimes, on the value of $\rho$, the associated energy functional exhibits different behaviours: it can be bounded from below for all, some, or no values of $\rho$ and these cases are known, respectively, as mass-subcritical, -critical, and -supercritical. The first two are studied in Chapter \ref{K:sub}, based on \cite{Schino}, while the last is studied in Chapter \ref{K:super}, based on \cite{MeSc}. In both cases, we consider a minimizing sequence for the energy functional and work out proper assumptions so that such a sequence converges to a solution to the system. In the mass-supercritical case, since the functional is unbounded from below, we restrict it to a natural manifold, given by a suitable linear combination of the Nehari and the Poho\v{z}aev identities to get rid of the unknown quantity $\lambda$, in order to recover such boundedness. The outcome consists of a least-energy solution, relatively to the functions with the same symmetry or in the general sense depending on the structure of the nonlinearity.

The novelty of this approach consists of considering the $L^2(\rn)$ balls
\[
\Set{u\in\hrn|\left|u\right|_2\le\rho_j}, \quad j\in\{1,\dots,K\}
\]
instead of the $L^2(\rn)$ spheres
\[
\Set{u\in\hrn|\left|u\right|_2=\rho_j}, \quad j\in\{1,\dots,K\}
\]
in order to work with a weakly closed subset and have, a priori, additional information about the sign of the components of $\lambda$, which is due to the fact that the constraints are given by inequalities and that the critical points we obtain are minimizers.

When $K\ge2$, we need particular hypotheses about the nonlinearity in order to make use of the Schwarz symmetric rearrangements; nevertheless, we can still deal with rather generic functions, which is new about systems.

Finally, Chapter \ref{K:join} contains new results and deals with normalized solutions both to curl-curl problems and to nonautonomous Schr\"odinger equations with singular potential as in Chapter \ref{K:cylsym}, but always with autonomous nonlinearities. Such results are obtained combining the symmetry and the equivalence from Chapter \ref{K:cylsym} with the outcomes from Chapters \ref{K:sub} and \ref{K:super}. In  particular, the symmetry allows us to reduce the curl-curl problem to a vector-valued autonomous Schr\"odinger equation, with a single $L^2$-constraint and which we study directly, while the equivalence provides analogous results for the scalar-valued Schr\"odinger equation with singular potential. Again, we obtain least-energy solutions relatively to the functions with the same symmetry.

\selectlanguage{polish}

\chapter*{Streszczenie}

Niniejsza rozprawa doktorska dotyczy istnienia rozwiązań półliniowych równań eliptycznych i jest podzielona na dwie części, dotyczące odpowiednio problemów bez ograniczenia i z ograniczeniem. Rozdział \ref{K:prel} zawiera szereg odwołań do pojęć i własności używanych w tej pracy i  zaprezentowany jest przed wspomnianym podziałem na części.

W Części \ref{1} badamy istnienie i wielokrotność rozwiązań równania postaci 
\begin{equation*}
\nabla\times\nabla\times\UU=f(x,\UU), \quad \UU\colon\rn\to\rn,
\end{equation*}
gdzie $N\ge3$ oraz $f=\nabla F\colon\rn\times\rn\to\rn$ jest gradientem (ze względu na $\UU$) danej nieliniowej funkcji $ F\colon\rn\times\rn\to\R$. 
Tutaj, gdy $N\ge4$, $\nabla\times\nabla\times\UU$ jest zdefiniowane przy użyciu tożsamości $\nabla\times\nabla\times\UU=\nabla(\nabla\cdot\UU)-\Delta\UU$ zachodzącej, gdy $N=3$. Takie problemy są znane jako problemy \textit{curl-curl} i powstają, gdy $N=3$, z nieliniowych równań Maxwella przy braku ładunków elektrycznych, prądów elektrycznych i magnetyzacji. Główną trudnością jest fakt, że jądro operatora różniczkowego $\nabla\times\nabla\times$ składa się z podprzestrzeni pól gradientowych i dlatego jest nieskończenie wymiarowe. Historycznie rzecz biorąc, dwa podejścia były stosowane do rozwiązywania problemów curl-curl za pomocą metod wariacyjnych i oba wykorzystywały bezdywergencyjne pola wektorowe. Powodem był fakt, że $\nabla\times\nabla\times\UU=-\Delta\UU$ dla każdego pola  bezdywergencyjnego $\UU$, a wektorowy Laplacian jest łatwiejszym w stosowaniu operatorem różniczkowym.

W rozdziale \ref{K:intro1} podajemy dokładne fizyczne wyprowadzenie problemu curl-curl, a następnie przywołujemy ważne wyniki z ostatnich dziesięcioleci, od pierwszych prac do najnowszych wyników, w tym zilustrowanych w tej rozprawie doktorskiej. 

W rozdziale \ref{K:nosym} opartym na \cite{MeScSz}, skupiamy się na fizycznie istotnym przypadku $N=3$. Nieliniowość $F$ jest kontrolowana od góry i od dołu przez odpowiednią regularną (ang. {\em nice}) funkcję Younga: w szczególności nadkrytyczną  w zerze, podkrytyczną w nieskończoności w sensie wykładnika Sobolewa, ale superkwadratową w nieskończoności i spełniającą globalne warunki $\Delta_2$ i $\nabla_2$. Nasze podejście wykorzystuje dekompozycję typu Helmholtza przestrzeni funkcyjnej z którą pracujemy, na podprzestrzeń bezdywergencyjną  i podprzestrzeń bezwirową (wspomniane jądro), tj. $u=v+w$, gdzie $\nabla\cdot v =\nabla\times w = 0$ i para $(v,w)$ jest jednoznacznie określona; następnie budujemy homeomorfizm z poprzedniej podprzestrzeni do pewnej topologicznej podrozmaitości (całej przestrzeni) zawierającej wszystkie nietrywialne rozwiązania. To w pewien sposób pozwala nam pracować tylko z bezdywergencyjną podprzestrzenią, chociaż trzeba zadbać o tę część bezwirową. W rzeczywistości to właśnie powoduje najwięcej trudności w stosowanych przez nas metodach. Udowadniamy istnienie rozwiązania o najmniejszej energii oraz, jeśli $f$ jest nieparzyste, istnienie nieskończenie wielu różnych rozwiązań. W przeciwieństwie do Rozdziału \ref{K:cylsym} nie używamy żadnych symetrii; w szczególności podajemy pierwsze wyniki dotyczące wielokrotności rozwiązań  problemu curl-curl na nieograniczonej dziedzinie bez złożeń o symetrii.

W rozdziale \ref{K:cylsym} bazującym na \cite{Gacz}, rozważymy przypadek ogólny $N\ge3$. Przy pewnych założeniach dotyczących symetrii nieliniowości, wykorzystujemy odpowiednie działania grupowe, aby zredukować problem curl-curl do równania Schr\"odingera z potencjałem osobliwym
\[
-\Delta u+\frac{a}{|y|^2}u=\tilde{f}(x,u), \quad u\colon\rn\to\R,
\]
gdzie $x=(y,z)\in\rk\times\rnk$, $K=2$ i $a=1$,
badając również przypadek ogólny $2\le K<N$ oraz $a>-(K/2-1)^2$. Mówiąc bardziej szczegółowo, wymagamy, aby $f(\cdot,\alpha w)=\tilde{f}(\cdot,\alpha)w$ dla każdego $\alpha\in\R$ 
i każdego $w\in\mbS^{N-1}$ oraz $\tilde{f}(gx,\cdot) = \tilde{f}(x,\cdot)$ dla prawie wszystkich $x\in\rn$ i dla każdego $g\in\SO(2)\times\{I_ {N-2}\} $. 
Rozszerzając do przypadku słabych rozwiązań dobrze znaną równoważność klasycznych rozwiązań obu problemów za pomocą wzoru $\UU(x) = u(x)/|y|(- x_2, x_1,0)$, dowodzimy nowych wyników istnienia (rozwiązania nietrywialne, rozwiązania o najmniejszej energii wśród rozwiązań o tej samej symetrii, nieskończenie wiele różnych rozwiązań) w obu problemach, zarówno w przypadkach krytycznych jak i niekrytycznych w sensie wykładnika Sobolewa; w szczególności pracujemy z równaniem curl-curl w poprzednim przypadku, używając tej samej maszynerii symetrii, aby zredukować $\nabla\times\nabla\times\UU$ do $-\Delta\UU$ oraz z równaniem Schr\"odingera w tej ostatniej sytuacji. Najbardziej znaczącym rezultatem jest istnienie, gdy $N=3$, rozbieżnego ciągu rozwiązań w krytycznym przypadku. Wynik ten uzyskamy za pomocą innego działania grupowego, które daje nam zwartość problemu; to jest pierwszy wynik dotyczący wielokrotności rozwiązań problemu curl-curl na nieograniczonej dziedzinie w wykładnikiem krytycznym  Sobolewa. Odnośnie istnienia rozwiązań w niekrytycznym przypadku, rozwiązania o najmniejszej energii oraz nieskończenie wiele różnych rozwiązań uzyskujemy wykorzystując abstrakcyjną teorię punktów krytycznych zbudowaną w rozdziale \ref{K:nosym}.

W Części \ref{2} szukamy rozwiązań o najmniejszej energii dla autonomicznego układu równań Schr\"odingera w postaci
\[
\begin{cases}
-\Delta u_j+\lambda_ju_j=\partial_jF(u)\\
\int_{\rn}u_j^2\,dx=\rho_j^2
\end{cases}
\forall j\in\{1,\dots,K\}, \quad u\colon\rn\to\rk,
\]
gdzie $N,K\ge1$, $\rho=(\rho_1,\dots,\rho_K)\in]0,\infty[^K$ jest dane, oraz  $\lambda=(\lambda_1,\dots,\lambda_K)\in\rk$ jest wielkością nieznaną.
Rozwiązania takich problemów określane są jako \textit{unormowane} ze względu na ograniczenia $L^2$, które powodują, że  $\lambda$ pojawia się jako $K$-krotka mnożników Lagrange'a. Równania tego typu pojawiają się podczas poszukiwania rozwiązań fali stojącej dla podobnych problemów zależnych od czasu i pochodzą z takich dziedzin fizyki, jak nieliniowa optyka i kondensacja Bosego-Einsteina. Ich waga polega na fizycznym znaczeniu masy (normy $L^2$ podniesionej do kwadratu) oraz fakcie, że wielkości te są zachowywane w czasie w odpowiednich równaniach ewolucji.

W rozdziale \ref{K:intro2} wprowadzamy problem, krótko komentujemy niektóre nowatorskie artykuły i inne wyniki w literaturze oraz podajemy przydatne preliminaria.

W zależności od założeń dotyczących $F$, a czasami wartości $\rho$, powiązana funkcja energetyczna wykazuje różne zachowania: może być ograniczona od dołu dla wszystkich, niektórych lub żadnych wartości $\rho$ i te przypadki są znane odpowiednio jako masowo podkrytyczne, -krytyczne i -nadkrytyczne. Pierwsze dwa przypadki są omówione w rozdziale \ref{K:sub} bazującym na pracy \cite{Schino}, a ostatnie w rozdziale \ref{K:super} bazują na \cite{MeSc}. W obu przypadkach rozważamy ciąg minimalizujący dla funkcjonału energii i wypracowujemy odpowiednie założenia takie, aby ten ciąg był zbieżny do rozwiązania układu równań. W przypadku masowo nadkrytycznym funkcjonał jest nieograniczony od dołu i ograniczamy go do rozmaitości naturalnej, określonej przez odpowiednią liniową kombinację tożsamości Nehariego i Poho\v{z}aeva, aby pozbyć się nieznanej wielkości $\lambda$ -- wówczas uzyskujemy ograniczenie funkcjonału z dołu. Wynik składa się z rozwiązania o najmniejszej energii wśród funkcji o tej samej symetrii lub w ogólnym sensie  w zależności od struktury nieliniowości.

Nowość tego podejścia polega na rozważeniu kul przestrzeni $L^2(\rn)$
\[
\Set{u\in\hrn|\left|u\right|_2\le\rho_j}, \quad j\in\{1,\dots,K\}
\]
zamiast sfer przestrzeni $L^2(\rn)$
\[
\Set{u\in\hrn|\left|u\right|_2=\rho_j}, \quad j\in\{1,\dots,K\}
\]
i pozwala pracować ze słabo domkniętym podzbiorem i mieć, a priori, dodatkowe informacje o znaku składowych $\lambda$, które wynikają z faktu, że ograniczenia są wyznaczane przez nierówności, a punkty krytyczne, które otrzymujemy, są punktami minimalnymi.

Jeśli $K\ge2$, to potrzebujemy konkretnych założeń dotyczących nieliniowości, aby skorzystać z symetryzacji Schwarza; niemniej jednak nadal możemy zajmować się raczej ogólnymi funkcjami, co jest nowością w przypadku układów.

Ostatecznie, Rozdział \ref{K:join} zawiera nowe wyniki i dotyczy unormowanych rozwiązań zarówno problemów curl-curl jak i nieautonomicznych równań Schr\"odingera z potencjałem osobliwym, jak w rozdziale \ref{K:cylsym}, jednak zawsze z autonomicznymi nieliniowościami. Takie wyniki uzyskuje się łącząc symetrię oraz równoważność z rozdziału \ref{K:cylsym} z wynikami z rozdziałów \ref{K:sub} oraz \ref{K:super}. W szczególności symetria pozwala nam zredukować problem curl-curl do autonomicznego równania Schr\"odingera z niewiadomą o wartościach wektorowych z pojedynczym ograniczeniem $L^2$, które badamy bezpośrednio, podczas gdy równoważność zapewnia analogiczne wyniki dla równania Schr\"odingera o wartości skalarnej z potencjałem osobliwym. Ponownie otrzymujemy rozwiązania o najmniejszej energii wśród rozwiązań o tej samej symetrii.

\selectlanguage{english}

\chapter*{Notations}

The symbols $\cdot$ and $\times$ stand, respectively, for the inner product in $\rn$, $N\ge1$ integer, and the cross product in $\rr$. In particular, $\nabla\cdot\UU$ stands for the divergence of $\UU\colon\rn\to\rn$ and $\nabla\times\UU$ stands for the curl of $\UU\colon\rr\to\rr$. The elements of the standard basis are denoted by $e_i$, $i\in\{1,\dots,N\}$.

The space of matrices $N\times K$ is denoted by $\R^{N\times K}$. When $N=K$, $I_N\in\R^{N\times N}$ stands for the identity matrix.

The closed semiline of nonnegative numbers $[0,\infty[$ is sometimes denoted by $\R^+$. If $A\subset\rn$ is a measurable set, $|A|$ denotes its Lebesgue measure.

If $\alpha\in\N^N$, $N\ge1$ integer, then the length of $\alpha$ is denoted by $|\alpha|:=\sum_{i=1}^N\alpha_i$. If, moreover, $f$ is a function of class $\cC^k$, $k\ge|\alpha|$, then we denote $\displaystyle D^\alpha f:=\frac{\partial^{|\alpha|}f}{\partial x_1^{\alpha_1}\dots\partial x_N^{\alpha_N}}$.

$N$ will always stand for the dimension of the space $\rn$, the only exception being the term `$N$-function' (Chapter \ref{K:nosym}), where it is simply part of the name (short for `nice Young function'). $2^*$ denotes the Sobolev critical exponent, i.e., $2^*=\infty$ if $N\in\{1,2\}$ and $2^*=\frac{2N}{N-2}$ if $N\ge3$. In Part \ref{2} we deal with the value $2_\#:=2+\frac4N$ too.

If $X$ is a topological space and $A\subset X$, then $\mathring{A}$ and $\overline{A}$ stand, respectively, for the interior and the closure of $A$. When $X$ is a metric space, the open, resp. closed, ball with centre $x\in X$ and radius $\rho>0$ is denoted by $B(x,\rho)$, resp. $\overline{B}(x,\rho)$. When $X$ is also a vector space and $x=0$, we write $B(0,\rho)=B_\rho$ and $\overline{B}(0,\rho)=\overline{B}_\rho$; moreover, the sphere with centre $0$ and radius $\rho$ is denoted by $S_\rho$.

If $X$ is a normed space, then its dual space is denoted by $X'$. For $x,y\in X$ and $T\colon X\to\R$ Fr\'echet differentiable, we denote the Fr\'echet differential of $T$ at $x$ evaluated at $y$ by $T'(x)(y)$.

If $\Omega\subset\rn$ is an open subset, $N\ge1$ integer, then we denote by $\cC_c(\Omega)$ the space of continuous functions with compact support contained in $\Omega$. Likewise for $\cC_c^\infty(\Omega)$.

If $u$ is a real-valued function, then its positive and negative parts are denoted, respectively, by $u_+=\max\{u,0\}$ and $u_-=\max\{-u,0\}$. If $u\colon\rn\to\R$ is a measurable function, $|u|_p$ stands for the $L^p(\rn)$ norm of $u$, $1\le p\le\infty$.

Concerning sequences, we write $x_n$ for $(x_n)_{n\ge1}$ and $x_n\in X$ for $(x_n)_{n\ge1}\subset X$. Sometimes we use superscripts instead of subscripts and write $x^n$. Moreover, we will write $\lim_n$, $\liminf_n$, $\limsup_n$ for $\lim_{n\to\infty}$, $\liminf_{n\to\infty}$, $\limsup_{n\to\infty}$.

If $f$ is a function depending (also) on $x$, then the notation $f\in O(x)$ means that $f(x,\dots)/|x|$ is essentially bounded. Similarly, if $g$ is another, real-valued function depending on $x$, the notation $o\bigl(g(x)\bigr)$ stands for a quantity that tends to zero when divided by $g(x)$.

$C$ will always stand for a positive constant, whose value is allowed to change after an inequality symbol, e.g., `$\le$'.

Finally, concerning functions spaces, we will write, e.g., $H^1(\rn,\rn)$ or $L^p(\rn,\rn)$ in contexts involving single vector-valued functions (Part \ref{1}, Chapter \ref{K:join}) and $\hrn^K$ or $L^p(\rn)^K$ in contexts involving $K$-tuples of scalar-valued functions (Chapters \ref{K:intro2}--\ref{K:super}).

\chapter{Preliminaries}\label{K:prel}

\section{Sobolev spaces}

Let $N\ge1$ be an integer and $\Omega\subset\rn$ be open. If $f\in\cC^1(\Omega)$ and $\varphi\in\cC_c^\infty(\Omega)$, then, in view of the integration-by-parts formula \cite[Theorem C.2]{Evans}, there holds
\[
\int_{\Omega}\frac{\partial f}{\partial x_i}\varphi\,dx=-\int_{\Omega}f\frac{\partial\varphi}{\partial x_i}\,dx \quad \text{ for every } i\in\{1,\dots,N\}.
\]
Analogously, if $k\ge1$ is an integer, $\alpha\in\N^N$, $|\alpha|\le k$, and $f\in\cC^k(\Omega)$, then
\begin{equation}\label{e-weakD}
\int_{\Omega}D^\alpha f\varphi\,dx=(-1)^{|\alpha|}\int_{\Omega}fD^\alpha\varphi\,dx.
\end{equation}
Of course, the right-hand (resp. left-hand) side of \eqref{e-weakD} makes sense even if merely $f\in L_\textup{loc}^1(\Omega)$ (resp. $D^\alpha f\in L_\textup{loc}^1(\Omega)$). This motivates us to give the following definition.

\begin{Def}
Let $f,g\in L_\textup{loc}^1(\Omega)$. We say that $g$ is the $\alpha$-th \textit{weak} or \textit{distributional} derivative of $f$ if and only if
\begin{equation*}
\int_{\Omega}g\varphi\,dx=(-1)^{|\alpha|}\int_{\Omega}fD^\alpha\varphi\,dx \quad \text{ for every } \varphi\in\cC_c^\infty(\Omega).
\end{equation*}
In such a case we write $g=:D^\alpha f$.
\end{Def}

With the aid of the notion of weak (or distributional) derivative we can now introduce the \textit{Sobolev spaces}.

\begin{Def}\label{D:Sobo}
Let $k\ge1$ an integer and $1\le p\le\infty$. We define the Sobolev space
\begin{equation*}
W^{k,p}(\Omega):=\Set{f\in L^p(\Omega)|D^\alpha f\in L^p(\Omega) \text{ for all } \alpha\in\N^N \text{ with } |\alpha|\le k}.
\end{equation*}
\end{Def}

Sobolev spaces are normed (in fact, Banach) spaces once endowed with the following norm. If $1\le p<\infty$, then for $f\in W^{k,p}(\Omega)$ we define
\[
\|f\|_{W^{k,p}(\Omega)}:=\left(\sum_{|\alpha|\le k}\|D^\alpha f\|_{L^p(\Omega)}^p\right)^{1/p};
\]
if $p=\infty$, then for $f\in W^{k,\infty}(\Omega)$ we define
\[
\|f\|_{W^{k,\infty}(\Omega)}:=\max_{|\alpha|\le k}\|D^\alpha f\|_{L^\infty(\Omega)}.
\]

If $p=2$, then $W^{k,2}(\Omega)$ is a Hilbert space with scalar product
\[
(f|g)_{W^{k,2}(\Omega)}:=\sum_{|\alpha|\le k}\int_{\Omega}D^\alpha fD^\alpha g\,dx, \quad f,g\in W^{k,2}(\Omega).
\]
This explains the widely used notation $W^{k,2}(\Omega)=:H^k(\Omega)$. Part \ref{2} strongly deals with the space $H^1(\rn)$.

An important subspace of $W^{k,p}(\Omega)$ is the one denoted by $W_0^{k,p}(\Omega)$ and defined as the closure of $\cC_c^\infty(\Omega)$ with respect to the norm $\|\cdot\|_{W^{k,p}(\Omega)}$. If $\Omega=\rn$, then $W_0^{k,p}(\Omega)=W^{k,p}(\Omega)$.

Of course, there is no reason why the exponent $p$ in Definition \ref{D:Sobo} has to be the same for all the derivatives. For example, if $N\ge3$ and $k=1$, an important space, which plays a major role in Part \ref{1}, is
\[
\cD^{1,2}(\rn):=\Set{f\in L^{2^*}(\rn)|D^\alpha f\in L^2(\rn) \text{ for all } \alpha\in\N^N \text{ with } |\alpha|\le1},
\]
which can be equivalently defined as the completion of $\cC_c^\infty(\rn)$ with respect to the norm $u\mapsto|\nabla u|_2$.

Important results in functional analysis concern continuous and compact embeddings of Sobolev spaces. Concerning the formers, we have as follows.

\begin{Th}\label{T:cont}
Let $\Omega\subset\rn$ be open with boundary of class $\cC^1$\footnote{For the purpose of this work, it is enough to consider this case; however, the assumption about the $\cC^1$ regularity of the boundary can be weakened. Likewise for Theorem \ref{T:comp}.}, $k\ge1$ integer, and $1\le p<\infty$.
\begin{itemize}
	\item If $kp<N$, then
	\[
	W^{k,p}(\Omega)\hookrightarrow L^q(\Omega) \quad \text{ for all } q\in\left[p,\frac{np}{n-kp}\right].
	\]
	\item If $kp=N$, then
	\[
	W^{k,p}(\Omega)\hookrightarrow L^q(\Omega) \quad \text{ for all } q\in[p,\infty[;
	\]
	If, moreover, $p=1$, then
	\[
	W^{N,1}(\Omega)\hookrightarrow\cC(\Omega)\cap L^q(\Omega) \quad \text{ for all } q\in[1,\infty].
	\]
	\item If $kp>N$, then
	\[
	W^{k,p}(\Omega)\hookrightarrow\cC(\Omega)\cap L^q(\Omega) \quad \text{ for all } q\in[p,\infty].
	\]
\end{itemize}
\end{Th}

\begin{Rem}
If $|\Omega|<\infty$, then the embeddings in Theorem \ref{T:cont} hold also for $q\in[1,p[$.
\end{Rem}

As for compact embeddings (denoted by $\hookrightarrow\hookrightarrow$), we have the following.

\begin{Th}\label{T:comp}
Let $\Omega\subset\rn$ be open and bounded with boundary of class $\cC^1$, $k\ge1$ integer, and $1\le p<\infty$.
\begin{itemize}
	\item If $kp<N$, then
	\[
	W^{k,p}(\Omega)\hookrightarrow\hookrightarrow L^q(\Omega) \quad \text{ for all } q\in\biggl[1,\frac{np}{n-kp}\biggl[.
	\]
	\item If $kp=N$, then
	\[
	W^{k,p}(\Omega)\hookrightarrow\hookrightarrow L^q(\Omega) \quad \text{ for all } q\in[1,\infty[.
	\]
	\item If $kp>N$, then
	\[
	W^{k,p}(\Omega)\hookrightarrow\hookrightarrow\cC(\overline\Omega).
	\]
\end{itemize}
\end{Th}

\begin{Prop}
The embeddings in Theorems \ref{T:cont} and \ref{T:comp} hold without any assumptions about $\partial\Omega$ provided $W^{k,p}(\Omega)$ is replaced with $W_0^{k,p}(\Omega)$.
\end{Prop}

We conclude the first part of this section recalling that a much more detailed version of all the previous results can be found in \cite{Adams}. Finally, analogous properties hold for Sobolev spaces of the type $W^{k,p}(\Omega,\rk)$, $K\ge1$ integer, i.e., involving vector-valued functions.

\subsection{Compact embeddings in $\rn$}\label{SS:cpt}

As pointed out in Theorem \ref{T:comp}, compact embeddings involving $W^{k,p}(\Omega)$ holds only if $|\Omega|<\infty$. If we want similar results in the whole $\rn$, then we need to consider special subspaces. We limit our discussion to $\hrn$.

When $N\ge2$, let
\[
\fH_\textup{r}:=\Set{u\in\hrn|u=u(g\cdot) \text{ for all } g\in\cO(N)}
\]
be the subspace of $\hrn$ consisting of \textit{radial} functions. Next, following \cite{BartschWillem}, if $N=4$ or $N\ge6$, fix $2\le M\le N/2$ such that $N-2M\ne1$ and consider $\tau\in\cO(N)$ defined by $$\tau(x_1,x_2,x_3)=(x_2,x_1,x_3)$$ for every $x=(x_1,x_2,x_3)\in\R^M\times\R^M\times\R^{N-2M}=\rn$. Define
\[\begin{split}
\fH_\textup{d} & :=\Set{u\in\hrn|u=u(g\cdot) \text{ for all } g\in\cO(M)\times\cO(M)\times\cO(N-2M)}\\
H_\tau & :=\Set{u\in\hrn|u=-u(\tau\cdot)}
\end{split}\]
and note that $\fH_\textup{r}\cap H_\tau=\{0\}$. Finally let
\[
\fH_\textup{n}:=H_\tau\cap\fH_\textup{d},
\]
which a fortiori does \textit{not} contain any nontrivial radial functions. When $2M=N$, we agree that the component $x_3$ in the definition of $\tau$ and the group $\cO(N-2M)$ in the definition of $\fH_\textup{d}$ do not appear. Then it is well known that $\fH_\textup{r}$ and $\fH_\textup{d}$ are compactly embedded into $L^p(\rn)$ for every $2<p<2^*$, see, e.g., \cite[Theorem III.1]{Lions_F} or \cite[Corollary 1.25]{Willem} (concerning compact embeddings of radial functions, the first proof of this result is due to Strauss \cite{Strauss}, see also \cite[Corollary 1.26]{Willem}). Observe that, unlike Theorem \ref{T:comp}, the embedding is \textit{not} compact for $p=2$. In Chapter \ref{K:join}, when $N\ge4$, we will make use of a subspace similar to $\fH_\textup{d}$, i.e.,
\[
\fH_\textup{s} :=\Set{u\in\hrn|u=u(g\cdot) \text{ for all } g\in\cO(2)\times\cO(N-2)}.
\]
The same argument as for $\fH_\textup{d}$ proves that $\fH_\textup{s}$ embeds compactly into $L^p(\rn)$ for every $2<p<2^*$.

When $N=1$, $\fH_\textup{r}=\Set{u\in H^1(\R)|u(x)=u(-x) \text{ for a.e. } x\in\R}$ no longer embeds compactly into $L^p(\R)$, $2<p\le\infty$. Nevertheless, bounded sequences (in $H^1(\R)$) with additional assumptions are still precompact in $L^p(\R)$ in view of the following result (cf. \cite[Proposition 1.7.1]{Cazenave:book}).

\begin{Th}\label{T:Cazenave}
Let $u_n\in\fH_\textup{r}$ bounded. If $N\ge2$ or each $u_n$ is a nonincreasing function of $|x|$, then there exists $u\in\fH_\textup{r}$ such that, up to a subsequence, $u_n\to u$ in $L^p(\rn)$ for every $p\in]2,2^*[$ (for every $p\in]2,\infty]$ if $N=1$).
\end{Th}


\section{Nemytskii operators and differentiable functionals}

We begin this section, whose content is based on \cite[Appendix C]{Struwe} and \cite[Appendix A]{Willem} with the following definition.

\begin{Def}
Let $N,K,\nu\ge1$ integers and $\Omega\subset\rn$ open. A function $f\colon\Omega\times\rk\to\R^\nu$ is called a \textit{Carath\'eodory function} if and only if it is continuous in $u\in\rk$ for a.e. $x\in\Omega$ and measurable in $x\in\Omega$ for every $u\in\rk$.
\end{Def}

Given a Carath\'eodory function $f$, we can define the \textit{Nemytskii operator} associated with it as
\[
N_f(u):=f\bigl(\cdot,u(\cdot)\bigr)
\]
for $u\colon\Omega\to\rk$ in a suitable function space. 
A first important result on Nemytskii operators concerns their continuity, which depends on the growth conditions of the functions they are associated with.

For $p,q,r,s\in[1,\infty[$ define the Banach spaces $L^p(\rn,\rk)\cap L^q(\rn,\rk)$ and $L^r(\rn,\R^\nu)+L^s(\rn,\R^\nu)$ with norms, respectively,
\[\begin{split}
\|u\|_{p\land q} & :=|u|_p+|u|_q\\
\|u\|_{r\lor s} & :=\inf\Set{\left|v\right|_r+\left|w\right|_s | v\in L^r(\rn,\R^\nu),w\in L^s(\rn,\R^\nu),u=v+w}
\end{split}\]

\begin{Th}
If there exists $C>0$ such that for a.e. $x\in\Omega$ and every $u\in\rk$
\[
|f(x,u)|\le C(|u|^{p/r}+|u|^{q/s}),
\]
then $N_f\colon L^p(\rn,\rk)\cap L^q(\rn,\rk)\to L^r(\rn,\R^\nu)+L^s(\rn,\R^\nu)$ is continuous.
\end{Th}

The notion of Nemytskii operator can be used to construct functionals defined in function spaces. In general, the more regular $f$ is, the more regular the functional is as well. We recall that a functional $I\colon X\to\R$, where $(X,\|\cdot\|)$ is a normed space, is said to be \textit{Fr\'echet-differentiable} at some $u\in X$ if and only if there exists $I'(u)\in X'$ such that
\[
\lim_{v\to0}\frac{I(u+v)-I(v)-I'(u)(v)}{\|v\|}=0.
\]
If such a map $I'(u)$ exists, it is called the \textit{Frechet differential} of $I$ at $u$. Of course, $I$ is said to be Fr\'echet-differentiable if and only if it is Fr\'echet-differentiable at $u$ for every $u\in X$. Moreover, we say that $I$ is of class $\cC^1$ (and we write $I\in\cC^1(X)$) if and only if the map $u\in X\mapsto I'(u)\in X'$ is continuous.

\begin{Ex}
If $X$ is a Hilbert space with scalar product $(\cdot|\cdot)$, then the functional $\|\cdot\|^2\colon X\to\R$ is of class $\cC^1$ and
\[
\left(\|\cdot\|^2\right)'(u)(v)=2(u|v) \quad \text{ for every } u,v\in X.
\]
\end{Ex}

The next result is stated for $\Omega=\rn$ because this is the only domain treated in this Ph.D. thesis. Of course, analogous results for different (e.g., bounded) domains do exist.

\begin{Th}\label{T:diff}
Let $F\colon\rn\times\rk\to\R$ be differentiable in $u\in\rk$ and such that $\nabla_uF\colon\rn\times\rk\to\rk$ is a Carath\'eodory function. If there exists $C_1>0$ and, if $N\in\{1,2\}$, $p\ge2$ such that for a.e. $x\in\rn$ and every $u\in\rk$
\[
|\nabla_uF(x,u)|\le C_1(|u|+|u|^{2^*-1}) \quad \text{ if } N\ge3
\]
or
\[
|\nabla_uF(x,u)|\le C_1(|u|+|u|^{p-1}) \quad \text{ if } N\in\{1,2\},
\]
then $I\colon\hrn^K\to\R$ defined as
\[
I(u):=\int_{\rn}F\bigl(x,u(x)\bigr)\,dx
\]
is of class $\cC^1$ and
\[
I'(u)(v)=\int_{\rn}\nabla_uF\bigl(x,u(x)\bigr)\cdot v(x)\,dx \quad \text{ for every } u,v\in\hrn^K.
\]
If, moreover, $N\ge3$ and there exists $C_2>0$ such that for a.e. $x\in\rn$ and every $u\in\rk$
\[
|\nabla_uF(x,u)|\le C_2|u|^{2^*-1},
\]
then $I\colon\cD^{1,2}(\rn)^K\to\R$ is of class $\cC^1$ and
\[
I'(u)(v)=\int_{\rn}\nabla_uF\bigl(x,u(x)\bigr)\cdot v(x)\,dx \quad \text{ for every } u,v\in\cD^{1,2}(\rn)^K.
\]
\end{Th}

%

\section{Palais--Smale sequences, Cerami sequences, and the mountain pass geometry}

Let $(X,\|\cdot\|)$ be a normed space, $J\in\cC^1(X)$, and $c\in\R$. A \textit{Palais--Smale} sequence for $J$ is a sequence $x_n\in X$ such that $J(x_n)$ is bounded and $\lim_nJ'(x_n)=0$. A Palais--Smale sequence (for $J$) \textit{at level} $c$, $(PS)_c$ in short, is a Palais--Smale sequence with the additional property that $\lim_nJ(x_n)=c$.

Of course, a functional J need not have any such sequences. When $(X,\|\cdot\|)$ is a Banach space, a sufficient hypothesis about $J$ so that it does have a Palais--Smale sequence at a specific level $c$ is that it has the so-called \textit{mountain pass geometry} (cf. \cite{AmbRab,Rabin}), i.e., there exist $r>0$ and $e\in X\setminus\overline{B}_r$ such that
\[
\inf_{S_r}J>J(0)\ge J(e)
\]
(see, e.g., \cite[Theorem 1.15]{Willem}). In this case, the value $c$ is called the \textit{mountain pass level} and has the minimax characterization
\[
c=\inf_{\sigma\in\Sigma}\max_{t\in[0,1]}J\bigl(\sigma(t)\bigr),
\]
where
\[
\Sigma:=\Set{\sigma\in\cC([0,1],X)|\sigma(0)=0 \text{ and } \sigma(1)=e}.
\]
This characterization makes evident that
\[
\max_{t\in[0,1]}J(te)\ge c\ge\inf_{S_r}J.
\]
This is important because it implies that, if $x\in X$ is such that $J(x)=c$, then $x\ne0$. In particular, if $X$ is some function space and $x$ is a solution to a certain differential equation, then such a solution is nontrivial (i.e., not identically zero).

In most situations, it is enough to have a Palais--Smale sequence; however, one can prove that, in fact, the mountain pass geometry yields the existence of a \textit{Cerami} sequence at the mountain pass level (c.f. \cite{bbf,Cerami}). A Cerami sequence (at level $c\in\R$) is a Palais--Smale sequence $x_n\in X$ with the additional property that $\lim_n\|x_n\|J'(x_n)=0$.

Finally, we recall an important concept related to Palais--Smale or Cerami sequences, i.e., the Palais--Smale and Cerami conditions. In the same framework as before, we say that the functional $J$ satisfies the \textit{Palais--Smale condition}, resp. \textit{Cerami condition}, \textit{at level} $c$ if and only if every Palais--Smale, resp. Cerami, sequence for $J$ at level $c$ has a (strongly) convergent subsequence. Likewise, we say that the functional $J$ satisfies the \textit{Palais--Smale condition}, resp. \textit{Cerami condition}, if and only if every Palais--Smale, resp. Cerami, sequence for $J$ has a convergent subsequence (equivalently, if and only if it satisfies the Palais--Smale, resp. Cerami, condition at level $c$ for every $c\in\R$).

\section{Nehari and Poho\v{z}aev identities}\label{S:NP}

Let $N,K\ge1$ be integers and consider a solution $u\colon\rn\to\rk$ to the equation
\begin{equation}\label{e-eq}
	-\Delta u=\nabla_uF(x,u), \quad u\in E
\end{equation}
under suitable assumptions about $F$ (e.g., those of Theorem \ref{T:diff}), where $E=\hrn^K$ or $E=\cD^{1,2}(\rn)^K$ according to such assumptions. If we test \eqref{e-eq} with $u$ itself (i.e., multiply both sides of \eqref{e-eq} by $u$, integrate over $\rn$, and use the integration-by-parts formula), then we obtain the identity
\begin{equation*}
\int_{\rn}|\nabla u|^2-\nabla_uF(x,u)\cdot u\,dx=0,
\end{equation*}
known in the literature as the \textit{Nehari identity}. Consequently, every \textit{nontrivial} solution to \eqref{e-eq} belongs to the \textit{Nehari manifold}
\[
\Set{v\in E\setminus\{0\} | \int_{\rn}|\nabla v|^2-\nabla_uF(x,v)\cdot v\,dx=0}.
\]
Whether this set is a differentiable manifold depends on the regularity of $F$. In particular, if $\nabla_uF$ is merely a Carath\'eodory function, then it is only a topological manifold.

Now, let $J\colon E\to\R$ be the energy functional associated with \eqref{e-eq}, i.e., the functional whose critical points are the solutions to \eqref{e-eq} and vice versa. Explicitly,
\[
J(u)=\int_{\rn}\frac12|\nabla u|^2-F(x,u)\,dx.
\]
At least heuristically, if $u$ is a solution to \eqref{e-eq} (hence a critical point of $J$), then $1$ is a critical point of the functional
\[
t\in]0,\infty[\to J\bigl(u(t\cdot)\bigr)\in\R,
\]
i.e., $u$ satisfies the identity
\begin{equation*}
\int_{\rn}(N-2)|\nabla u|^2-2NF(u)\,dx=0,
\end{equation*}
known in the literature as the \textit{Poho\v{z}aev identity}. In order to make the argument rigorous, 
an intermediate step is to prove that any solution to \eqref{e-eq} lies in $W_\textup{loc}^{2,p}(\rn)$ for every $p<\infty$, see, e.g., \cite{BerLions}\footnote{This reference deals only with the case $N\ge3$, but the argument holds for $N\in\{1,2\}$ too, as observed, e.g., in \cite{Jeanjean97}.}. This is standard when $K=1$, while the generic case is treated, e.g., in \cite[Theorem 2.3]{BrezisLiebV}. As a consequence, we obtain that every nontrivial solultion to \eqref{e-eq} belongs to the \textit{Poho\v{z}aev manifold}
\[
\Set{v\in E\setminus\{0\} | \int_{\rn}(N-2)|\nabla v|^2-2NF(v)\,dx=0}.
\]

\section{Palais's principle of symmetric criticality}

There are situations where considering a particular subspace of some Sobolev space, consisting in the functions that enjoy a certain symmetry, can bring remarkable advantages. This is the case for compact embeddings, as seen in Subsection \ref{SS:cpt}, or, as we will see in Chapter \ref{K:cylsym}, to turn a differential operator into another that is easier to handle. Nevertheless, it is important to make sure that the supposed solution obtained this way is actually a solution to the problem investigated. In other words, we need to make sure that a critical point of the energy functional restricted to a particular subspace is a critical point of the free functional. This is what happens when, roughly speaking, the functional has the same symmetry as the one that defines the subspace. This is known as Palais's principle of symmetric criticality \cite{Palais} and reads as follows.

\begin{Th}\label{T:Palais}
Let $\cH$ be a Hilbert space, $\cG$ a topological group acting isometrically on $\cH$, and $J\in\cC^1(\cH)$ such that $J(gx)=J(x)$ for every $g\in\cG$ and $x\in\cH$. Define $\cH_\cG:=\Set{x\in\cH|gx=x \text{ for all } g\in\cG}$. If $x\in\cH_\cG$ is a critical point of $J|_{\cH_\cG}$, then it is a critical point of $J$.
\end{Th}

Concerning this Ph.D. thesis, Theorem \ref{T:Palais} is used first of all to recover compactness: in Chapter \ref{K:sub}, where we work with $\fH_\textup{r}$ or, if the nonlinearity $F$ is even, $\fH_\textup{n}$, and in Section \ref{S:SC}, where we exploit a different symmetry. It is used also in Sections \ref{S:equiv} and, again, \ref{S:SC} to reduce the \textit{curl-curl} operator $\nabla\times\nabla\times$ to the vector Laplacian $-\Delta$. 

\section{Schwarz rearrangements}

This section is based on \cite[Chapter 3]{LiebLoss}. Let $N\ge1$ be integer. If $A\subset\rn$ is measurable and $|A|<\infty$, its Schwarz rearrangement is denoted by $A^*$ and defined as the open ball\footnote{One could use closed balls instead. With the choice of open balls, the characteristic function $\chi_{A^*}$ is lower semicontinuous.} centred at $0$ having the same measure as $A$, i.e.,
\[
A^*:=B_r \quad \text{with} \quad r=\frac{N|A|}{|\mbS^{N-1}|}.
\]

We say that a measurable function $u\colon\rn\to\R$ vanishes at infinity if and only if $\left|\Set{x\in\rn|\left|u(x)\right|>t}\right|<\infty$ for every $t>0$. In particular, $u$ vanishes at infinity if it belongs to a Lebesgue space with finite exponent. For such a function, its Schwarz rearrangement is denoted by $u^*$ and defined as follows. If $u=\chi_A$ is a characteristic function (for some measurable $A\subset\rn$ with finite measure), then
\[
u^*=\chi_A^*:=\chi_{A^*}.
\]
For a generic function $u$, instead, we define
\[
u^*(x):=\int_0^\infty\chi_{\{|u|>t\}}^*(x)\,dt, \quad x\in\rn.
\]
It is clear that $u^*$ is nonnegative, radial, and radially nonincreasing. Moreover, the following properties hold true.

\begin{Th}
Let $u\colon\rn\to\R$ be a measurable function that vanishes at infinity.
\begin{itemize}
	\item For every $t>0$
	\[
	\Set{x\in\rn|u^*(x)>t}=\Set{x\in\rn|\left|u(x)\right|>t}^*.
	\]
	\item If $F\colon[0,\infty[\to\R$ is the difference of two monotone functions $F_1$ and $F_2$ such that $F_i\circ u\in L^1(\rn)$ for some $i\in\{1,2\}$ (in particular, if $F$ is of class $\cC^1$ and $F\circ u\in L^1(\rn)$), then
	\[
	\int_{\rn}F(|u|)\,dx=\int_{\rn}F(u^*)\,dx.
	\]
	\item If $u\in\hrn$, then
	\[
	\int_{\rn}|\nabla u^*|^2\,dx\le\int_{\rn}|\nabla u|^2\,dx.
	\]
\end{itemize}
\end{Th}

A simple proof of the last property can be found in \cite[Lemma 5]{Lieb}.

Finally, if $u=(u_1,\dots,u_K)\colon\rn\to\rk$ is measurable and vanishes at infinity, with $K\ge1$ integer, we denote $u^*:=(u_1^*,\dots,u_K^*)$.

\section{Krasnosel'skji genus}

Let $(X,\|\cdot\|)$ be a Banach space and denote
\[
\cA:=\Set{A\subset X | A=\overline{A}=-A}.
\]
For $A\in\cA$, $A\ne\emptyset$, define by $\gamma(A)$ the smallest positive integer $k$ such that there exists a continuous odd map $h\colon A\to\R^k\setminus\{0\}$. If no such $k$ exists (in particular, if $0\in A$), let $\gamma(A)=\infty$. Finally, let $\gamma(\emptyset)=0$. $\gamma(A)$ is called the \textit{Krasnosel'skji genus} of $A$.

The most important properties of the Krasnosel'skji genus are listed in the following Proposition (cf. \cite[Proposition II.5.4]{Struwe}).

\begin{Prop}
Let $A,B\in\cA$ and $\mathfrak{h}\in\cC(X,X)$ odd.
\begin{itemize}
	\item [(i)] $\gamma(A)\ge0$; $\gamma(A)=0$ if and only if $A=\emptyset$.
	\item [(ii)] If $A\subset B$, then $\gamma(A)\le\gamma(B)$.
	\item [(iii)] $\gamma(A\cup B)\le\gamma(A)+\gamma(B)$.
	\item [(iv)] $\gamma(A)\le\gamma\left(\overline{\mathfrak{h}(A)}\right)$.
	\item [(v)] If $A$ is compact and $0\not\in A$, then $\gamma(A)<\infty$ and there exists a neighbourhood $V$ of $A$ such that $\overline{V}\in\cA$ and $\gamma(\overline{V})=\gamma(A)$.
\end{itemize}
\end{Prop}

\begin{Rem}
If $A\subset X$ is a finite nonempty collection of antipodal points, then $A\in\cA$ and $\gamma(A)=1$.
\end{Rem}

\part{Unconstrained problems}\label{1}

\chapter{Introduction to Part \ref{1}}\label{K:intro1}

In Part \ref{1} of this Ph.D. thesis, we study existence and multiplicity results for \textit{curl-curl} problems of the form
\begin{equation}\label{e-curl}
\nabla\times\nabla\times\UU=f(x,\UU), \quad \UU\colon\rr\to\rr,
\end{equation}
where $f=\nabla F\colon\rr\times\rr\to\rr$ is the gradient (with respect to $\UU$) of a given nonlinear function $ F\colon\rr\times\rr\to\R$. Problems as in \eqref{e-curl} find their origins in Maxwell's equations (in $\rr$) in the differential form
\begin{equation}\label{e-Maxwell}
\begin{cases}
\nabla\times\cH=\cJ+\partial_t\cD \, & \text{(Amp\`ere's Law)}\\
\nabla\cdot\cD=\rho \, & \text{(Gauss's Electric Law)}\\
\nabla\times\cE=-\partial_t\cB \, & \text{(Faraday's Law)}\\
\nabla\cdot\cB=0 \, & \text{(Gauss's Magnetic Law)}
\end{cases}
\end{equation}
where $\cH,\cJ,\cD,\cE,\cB\colon\rr\times\R\to\rr$ are time-dependent vector fields and $\rho\colon\rr\times\R\to\R$ is the electric charge density. In particular, $\cH$ is the magnetic intensity field, $\cJ$ the electric current intensity, $\cD$ the electric displacement field, $\cE$ the electric field, and $\cB$ the magnetic induction. We consider as well the constitutive relations (still in $\rr$)
\begin{equation}\label{e-const}
\begin{cases}
\cD=\epsilon\cE+\cP\\
\cH=\frac1\mu\cB-\cM,
\end{cases}
\end{equation}
where $\cP,\cM\colon\rr\times\R\to\rr$ are, respectively, the polarization field (which depends on $\cE$, in general nonlinearly) and the magnetization field, while $\epsilon,\mu\colon\rr\to\R$ are, respectively, the permettivity and the permeability of the material.

In order to derive \eqref{e-curl} we make additional assumptions about the physical model. We begin by considering absence of electric charges ($\rho=0$), electric currents ($\cJ=0$), and magnetization ($\cM=0$). Then, plugging \eqref{e-const} into \eqref{e-Maxwell} and differentiating with respect to the time variable we obtain\footnote{We assume as well that the time and space derivatives can be switched.}
\[
\nabla\times\left(\frac1\mu\nabla\times\cE\right)+\epsilon\partial_t^2\cE=-\partial_t^2\cP.
\]
Moreover, we assume that $\cE$ and $\cP$ are \textit{monochromatic waves}, i.e., $\cE(x,t)=\cos(\omega t)E(x)$ and $\cP(x,t)=\cos(\omega t)P(x)$ for some $\omega\in\R$ and $E,P\colon\rr\to\rr$ (or, equivalently, $\cE(x,t)=\sin(\omega t)E(x)$ and $\cP(x,t)=\sin(\omega t)P(x)$), which leads to the \textit{time-harmonic} Maxwell equation
\[
\nabla\times\left(\frac1\mu\nabla\times E\right)-\epsilon\omega^2 E=\omega^2P.
\]
Finally, if $\mu\equiv1$, $\epsilon\equiv0$, and we set $\UU=E$ and $f=\omega^2P$, then we obtain \eqref{e-curl}. Note that, if $\epsilon\not\equiv0$, then we simply have an additional linear term of the form $V(x)\UU$ on the left-hand side of \eqref{e-curl}.

A different derivation, still based on Maxwell's equations and -- at the same time -- on the Born--Infeld theory, is provided in \cite{BenFor} in the magnetostatic case, i.e., when the magnetic field does not depend on time and the electric field is identically $0$. In this case, $\UU$ stands for the gauge potential of the magnetic field: $\nabla\times\UU=\cB$.

A major mathematical difficulty of \eqref{e-curl} and similar curl-curl problems is that the differential operator $\UU\mapsto\nabla\times\nabla\times\UU$ has an infinite-dimensional kernel, i.e., the space of gradient vector fields; this makes the associated energy functional
\begin{equation*}
\UU\mapsto\int_{\rr}\frac12|\nabla\times\UU|^2-F(x,\UU)\,dx
\end{equation*}
strongly indefinite, i.e., unbounded from above and below (when $F\ge0$) even on subspaces of finite codimension and such that its critical points have infinite Morse index. Another issue is that the Fr\'echet differential of the energy functional is not sequentially weak-to-weak* continuous, therefore the limit point of a weakly convergent sequence need not be a critical point of the functional. Moreover, one has to struggle with the lack of compactness because the problem is set in the whole space $\rr$.

We underline that the aforementioned difficulties in dealing with curl-curl problems -- even in bounded domains --  have given rise to several simplifications in the literature. The most widely used is the scalar or vector nonlinear Schr\"odinger equation, where, e.g., one assumes that the term $\nabla(\nabla\cdot\UU)$ in $\nabla\times\nabla\times\UU=\nabla(\nabla\cdot\UU)-\Delta\UU$ is negligible and can therefore be removed from the equation, or uses the so-called \textit{slowly varying envelope approximation}. Nevertheless, such approximations may produce non-physical solutions, which do not describe the \textit{exact} propagation of electromagnetic waves in Maxwell's equations, as remarked, e.g., in \cite{AAS-C,CCDY}, whence the importance of curl-curl problems from a physical point of view. As far as we know, the first papers dealing with exact solutions to Maxwell's equations are \cite{McStTr,Stuart90}, where the problem is turned in an ODE and treated with ad hoc techniques. The same approach is used in the series of papers \cite{Stuart93,Stuart04,StuartZhou96,StuartZhou01,StuartZhou03,StuartZhou05,StuartZhou10}.

To the best of our knowledge, the first work on \eqref{e-curl} using variational methods in $\rr$ is due to Benci and Fortunato \cite{BenFor}: they consider the autonomous case and a double-power type nonlinearity, i.e., $F(x,\UU)=F(\UU)\simeq\min\{|\UU|^q,|\UU|^p\}$ for some $2<p<6<q$. They introduce a series of brilliant ideas which will be exploited later on by other authors for other curl-curl problems, such as the splitting of the function space they work with into a divergence-free subspace and a curl-free subspace and the restriction of the energy functional to the aforementioned divergence-free subspace via a ``relative minimization'' trick that makes use of the strict convexity of $F$. Nonetheless, their argument contains a mistake: in order to recover compactness, the authors work with $\cO(3)$-equivariant (radial) vector fields, without realizing that the subspace of $\cO(3)$-equivariant and divergence-free vector fields is nothing but the trivial space $\{0\}$.

The second attempt to tackle \eqref{e-curl} is due to Azzollini, Benci, D'Aprile, and Fortunato \cite{AzzBenDApFor}, once again with an autonomous double-power type nonlinearity. Their strategy consists of the use of two group actions in order to reduce the curl-curl operator to the vector Laplace operator, which is easier to handle. First, they consider the group action of $\SO:=\SO(2)\times\{1\}$ on
\[
\cD^{1,2}(\rr,\rr)=\Set{\UU\in L^6(\rr,\rr)|\nabla\UU\in L^2(\rr,\R^{3\times3})}
\]
defined by
\begin{equation}\label{e-Fix}
(g\UU)(x):=g^{-1}\UU(gx)=g^T\UU(gx), \quad x\in\rr
\end{equation}
for every $g\in\SO$ and every $\UU\in\cD^{1,2}(\rr,\rr)$, and the subspace $\textup{Fix}(\SO)$ consisting of the vector fields $\UU\in\cD^{1,2}(\rr,\rr)$ which are invariant with respect to this action. They prove (cf. \cite[Lemma 1]{AzzBenDApFor}) that every $\UU\in\textup{Fix}(\SO)$ decomposes as $\UU=\UU_\rho+\UU_\tau+\UU_\zeta$, where $\UU_\rho(x)$, $\UU_\tau(x)$, and $\UU_\zeta(x)$ are the orthogonal projections of $\UU(x)$ onto $\textup{span}(x_1,x_2,0)$, $\textup{span}(-x_2,x_1,0)$, and $\textup{span}(0,0,1)$ respectively for a.e. $x=(x_1,x_2,x_3)\in\rr$. Second, they introduce the action
\begin{equation}\label{e-S}
\cS\UU=\cS(\UU_\rho+\UU_\tau+\UU_\zeta):=-\UU_\rho+\UU_\tau-\UU_\zeta
\end{equation}
on $\textup{Fix}(\SO)$ and consider the subspace
\begin{equation*}
\DF:=\Set{\UU\in\textup{Fix}(\SO)|\UU=\cS\UU},
\end{equation*}
i.e., $\UU(x)=(x_1^2+x_2^2)^{-1/2}u(x)(-x_2,x_1,0)$ for some $\SO$-invariant $u\colon\rr\to\R$. Since the divergence of every element of $\DF$ is identically $0$, there holds $\nabla\times\nabla\times\UU=\nabla(\nabla\cdot\UU)-\Delta\UU=-\Delta\UU$ for every $\UU\in\DF$ and so, using Palais's principle of symmetric criticality \cite{Palais} (Theorem \ref{T:Palais}), they reduce \eqref{e-curl} to
\begin{equation*}
-\Delta\UU=f(\UU), \quad \UU\in\DF
\end{equation*}
and then use a Lions-type lemma as in \cite{Lions84_1,Lions84_2} to obtain a nontrivial solution via a constrained minimization argument in the spirit of \cite{BerLions}.

The action $\cS$ defined in \eqref{e-S} is also used by D'Aprile and Siciliano \cite{DApSic} in a somewhat antipodal way, i.e., considering the subspace
\[
\Set{\UU\in\textup{Fix}(\SO)|\UU=-\cS\UU}
\]
and obtaining solutions of the form
\[
\UU(x)=\frac{u(x)}{\sqrt{x_1^2+x_2^2}}
\begin{pmatrix}
x_1\\
x_2\\
0
\end{pmatrix}
+v(x)
\begin{pmatrix}
0\\
0\\
1
\end{pmatrix}
\]
for some $\SO$-invariant $u,v\colon\rr\to\R$. Since in this case the curl-curl operator does not reduce to the vector Laplacian, they make use of the tools introduced in \cite{BenFor} and then find a nontrivial solution similarly to \cite{AzzBenDApFor}.

The first work on a nonautonomous curl-curl problem and without the use of any symmetry is by Bartsch and Mederski \cite{BartschMederski1}, where they investigate one similar to \eqref{e-curl} but on a bounded domain $\Omega\subset\rr$, pairing it with boundary conditions that model the case of a medium surrounded by a perfect conductor (i.e., the electric field on the boundary of the medium is tangential to it), obtaining the system
\begin{equation}\label{e-bddcurl}
\begin{cases}
\nabla\times\nabla\times\UU+\lambda\UU=f(x,\UU) & \, \text{in } \Omega\\
\nu\times\UU=0 & \, \text{on } \partial\Omega,
\end{cases}
\end{equation}
with $\lambda\le0$ and $\nu\colon\partial\Omega\to\mbS^2$ the outer normal unit vector. As in \cite{BenFor,DApSic}, the authors split the function space they work with into a divergence-free part and a curl-free part, but then, instead of using a constrained minimization method, they adopt techniques from \cite{SzWe,SzWeHandbook} that exploit a generalization of the Nehari manifold, which need not be of class $\cC^1$. First of all, they split the divergence-free subspace into two more subspaces, where the quadratic form induced by the left-hand side of the differential equation in \eqref{e-bddcurl} is positive definite and negative semidefinite respectively. Next, under some technical assumptions about $F$, the generalized Nehari manifold (also known in the literature as the Nehari--Pankov manifold) is proved to be homeomorphic to the unit sphere in the subspace of divergence-free vector fields where the aforementioned quadratic form is positive definite, and this homeomorphism is utilized as a counterpart of the one obtained with the ``relative minimization'' trick from \cite{BenFor,DApSic}. In addition, since the Nehari--Pankov manifold is a natural constraint, by a suitable minimization argument the authors find a ground state solution, i.e., a notrivial solution with minimal energy, as well as infinitely many solutions with divergent energy. The advantage of working in a bounded domain is that, despite the presence of the subspace of curl-free vector fields, which does not embed compactly in any ``good'' (e.g., Lebesgue) function space, a variant of the Palais--Smale condition is satisfied, which provides some compactness in the aforementioned minimization argument.

The same authors generalize in \cite{BartschMederski2} their results by allowing also the terms on the left-hand side of the differential equation in \eqref{e-bddcurl} to be nonautonomous, obtaining the more generic system
\[
\begin{cases}
\nabla\times\bigl(\mu(x)^{-1}\nabla\times\UU\bigr)-V(x)\UU=f(x,\UU) & \, \text{in } \Omega\\
\nu\times\UU=0 & \, \text{on } \partial\Omega,
\end{cases}
\]
where now $\mu,V\colon\Omega\to\R^{3\times3}$ and $\mu(x),V(x)$ are symmetric positive definite matrices. They relax other assumptions about $F$, including the ones that allow to build a homeomorphism between the Nehari--Pankov manifold and the unit sphere in a suitable subspace; as a consequence, they build another one between a larger manifold (i.e., containing the Nehari--Pankov one) and the whole subspace of divergence-free vector fields where a similar quadratic form is positive definite, exploiting once again the trick from \cite{BenFor,DApSic} and obtaining similar results to their previous work \cite{BartschMederski1}.

Returning to unbounded domains, the problem
\begin{equation}\label{e-curlna}
\nabla\times\nabla\times\UU+V(x)\UU=f(x,\UU) \quad \text{in } \rr
\end{equation}
is studied by Mederski in \cite{Mederski} with $f$ $1$-periodic in $x$ along every direction and $V\colon\rr\to\R$ in a suitable intersection of Lebesgue spaces, $V\le0$. Tools from \cite{BartschMederski1} and \cite{DApSic} are matched with the $\Z^3$-invariance of the problem with $V=0$ given by the periodicity of $f$. When $V\ne0$, this invariance is lost and a careful analysis of Palais--Smale sequences is needed, including a splitting result for bounded sequences introduced in \cite{DApSic}.

The version of \eqref{e-curlna} with the same symmetry as in \cite{AzzBenDApFor} and $F(x,\UU)=\Gamma(x)|\UU|^p/p$, $2<p<6$, is investigated by Bartsch, Dohnal, Plum, and Reichel in \cite{BDPR}, where they build a Nehari--Pankov manifold of class $\cC^1$ and find a nontrivial solution with minimal energy by minimizing the energy functional constrained to it. They also consider the case $F\le0$, finding a least energy solution at a negative level as a minimizer of the unconstrained energy functional, and the radially symmetric case for $V/\Gamma>0$, where the term $\nabla\times\UU$ vanishes and -- consequently -- \eqref{e-curlna} becomes an algebraic equation, with explicit solutions
\[
\UU(x)=s(|x|)\left(\frac{V(|x|)}{\Gamma(|x|)}\right)^\frac{1}{p-2}\frac{x}{|x|}
\]
for some measurable $s\colon]0,\infty[\,\to\{\pm1\}$.

The first multiplicity result in unbounded domains without any symmetry assumptions is due to Mederski, the author, and Szulkin \cite{MeScSz}, where they obtain infinitely many solutions and a least-energy solution to \eqref{e-curlna} for $V=0$. Some of the techniques are borrowed from \cite{BartschMederski2}, others are introduced therein. They also generalize the double-power type nonlinearities of \cite{AzzBenDApFor,BenFor,DApSic,Mederski} by means of $N$-functions and Orlicz spaces. For more details, see Chapter \ref{K:nosym} in this thesis.

In all the papers mentioned so far, except for \cite{BDPR} in the particular case $F\le0$, the Sobolev-critical exponent $6$ (in dimension $N=3$) is not dealt with. The first results in this direction are due to Mederski \cite{MederskiJFA}, who considers the curl-curl equivalent of the Brezis-Nirenberg problem on bounded domains \cite{BrNi}
\begin{equation}\label{e-curlBN}
\begin{cases}
\nabla\times\nabla\times\UU+\lambda\UU=|\UU|^4\UU & \, \text{in } \Omega\\
\nu\times\UU=0 & \, \text{on } \partial\Omega,
\end{cases}
\end{equation}
$\lambda\le0$. His approach makes use of the same symmetry as in \cite{AzzBenDApFor} together with a compact perturbation relative to the case $\lambda=0$, which leads to a series of considerations on the difference between the energy functional associated with \eqref{e-curlBN} and the one when $\lambda=0$. He also obtains additional results in the Sobolev subcritical case about the continuity and the (strict) monotonicity of the ground state energy map, i.e., the function that maps $\lambda$ to the least energy achieved by a nontrivial solution to \eqref{e-bddcurl}.

The counterpart of \eqref{e-curlBN} set in the whole space $\rr$
\[
\nabla\times\nabla\times\UU=|\UU|^4\UU
\]
is investigated by Mederski and Szulkin in \cite{MeSzu}, where they find a ground state solution again with the aid of the Nehari--Pankov manifold. Additional results in that paper concern optimal constants in Sobolev-type inequalities involving the curl operator $\nabla\times$, in $\rr$ or in bounded domains. In particular, they improve the results in \cite{MederskiJFA} as they do not require any symmetry assumptions.

Multiple entire solutions in the Sobolev-critical case are obtained, for the first time, by Gaczkowski, Mederski, and the author in \cite{Gacz} combining the symmetry introduced in \cite{AzzBenDApFor} with another introduced in \cite{Ding}, which restores compactness. They also extend rigorously an equivalence result, known for the classical formulations, that relates the weak solutions to \eqref{e-curl} with the weak solutions to the Schr\"odinger equation with singular potential
\[
-\Delta u+\frac{u}{x_1^2+x_2^2}=\tilde{f}(x,u) \quad \text{in } \rr
\]
under some assumptions relating $f$ and $\tilde{f}$. The noncritical case is dealt with too. For more details, see Chapter \ref{K:cylsym} in this thesis. Similar equations appear also in context that are not related to Maxwell's equations and curl-curl problems. For instance, in \cite[Theorem 3.9]{EstLions} they are the result of the limiting problem for a nonlinear Schr\"odinger equation of critical growth with an external magnetic field, while in \cite{BadBenRol} they are derived from Schr\"odinger equations of the form
\[
-\Delta v=g(|v|)\frac{v}{|v|}
\]
assuming that $v(x)=u(x)e^{i\theta(x_1,x_2)}$, $u\ge0$, where $\theta(x_1,x_2)$ gives the angle of the point $(x_1,x_2)$ in the plane $\Set{y=(y_1,y_2,y_3)\in\rr|y_3=0}$.

In the end, we would like to mention the surveys about curl-curl problems \cite{BarMed,MedSURV} and the following two papers, which study cases that are not taken into account in this thesis: \cite{HirRei}, where different symmetries are considered, and \cite{QinTangLIN}, with asymptotically linear nonlinearities.

\chapter{Maxwell's equations and absence of symmetry}\label{K:nosym}

\section{Statement of the results}

In this chapter, based on \cite{MeScSz}, we study the curl-curl problem \eqref{e-curl}, which for the reader's convenience we rename
\begin{equation}\label{e-Curl}
\nabla\times\nabla\times\UU=\nabla_\UU F(x,\UU)=:f(x,\UU) \quad \text{in } \rr,
\end{equation}
without any hypotheses about the symmetry of the solutions to \eqref{e-Curl}.

The nonlinearity is controlled by an $N$-function $\Phi\colon\R\to[0,\infty[$ (cf. (F3) below) which satisfies the following assumptions:
\begin{itemize}
	\item [(N1)] $\Phi$ satisfies the $\Delta_2$ and $\nabla_2$ conditions globally;
	\item [(N2)] $\displaystyle\lim_{t\to0}\frac{\Phi(t)}{t^6}=\lim_{t\to\pm\infty}\frac{\Phi(t)}{t^6}=0$;
	\item [(N3)] $\displaystyle\lim_{t\to\pm\infty}\frac{\Phi(t)}{t^2}=\infty$.
\end{itemize}
As a reference about $N$-functions, which will be rigorously introduced in Section \ref{S:Orlicz}, we mention \cite{RaoRen}. Note that in \cite{MeScSz} it was additionally required that $\Phi$ be strictly convex and of class $\cC^1$.

Now we list our assumptions about the nonlinearity.
\begin{itemize}
	\item [(F1)] $F\colon\rr\times\rr\to\R$ is differentiable with respect to $\UU\in\rr$ for a.e. $x\in\rr$ and $f\colon\rr\times\rr\to\rr$ is a Carath\'eodory function. Moreover, $f$ is $\Z^3$-periodic in $x$, i.e., $f(x,\UU)=f(x+y,\UU)$ for every $\UU\in \rr$ and a.e. $x\in\R^3$ and $y\in\Z^3$.
	\item [(F2)] $F$ is uniformly strictly convex in $\UU$, i.e., for every compact $K\subset(\rr\times\rr)\setminus\Set{(\UU,\UU)|\UU\in\rr}$
	\[
	\inf_{\substack{x\in\rr \\ (\UU,\VV)\in K}}\frac{F(x,\UU)+F(x,\VV)}{2}-F\left(x,\frac{\UU+\VV}{2}\right)>0.
	\]
	\item [(F3)] There exist $c_1,c_2>0$ such that
	\[
	|f(x,\UU)|\le c_1\Phi'(|\UU|) \text{ and } F(x,\UU)\ge c_2\Phi(|\UU|)
	\]
	for every $\UU\in\rr$ and a.e. $x\in\rr$.
	\item [(F4)] $f(x,\UU)\cdot\UU\ge2F(x,\UU)$ for every $\UU\in\rr$ and a.e. $x\in\rr$.
	\item [(F5)] For every $\UU,\VV\in\rr$ and a.e. $x\in\rr$ such that $f(x,\UU)\cdot\VV=f(x,\VV)\cdot\UU>0$
	\[
	F(x,\UU)-F(x,\VV)\le\frac{\bigl(f(x,\UU)\cdot\UU\bigr)^2-\bigl(f(x,\UU)\cdot\VV\bigr)^2}{2f(x,\UU)\cdot\UU}.
	\]
\end{itemize}

Under these assumptions, the energy functional defined as
\begin{equation}\label{e-J}
E(\UU)=\int_{\rr}\frac12|\nabla\times\UU|^2-F(x,\UU)\,dx
\end{equation}
is well defined and of class $\cC^1$ (see Proposition \ref{P:classC1}) in the space
\[
\cD(\textup{curl},\Phi):=\Set{\UU\in L^\Phi(\rr,\rr)|\nabla\times\UU\in L^2(\rr,\rr)},
\]
where
\[
L^\Phi(\rr,\rr)=\Set{\UU\colon\rr\to\rr \text{ measurable }|\int_{\rr}\Phi(|\UU|)\,dx<\infty}.
\]

Note that, as observed in \cite[Remark 3.3 (b)]{MederskiCPDE}, if $F,\bar{F}\colon\rr\times\rr\to\R$ satisfy (F1)--(F5), then so does $F+\bar{F}$. This is not trivial concerning (F5) as it is not an additive assumption.

We provide some examples for $F$. Let $G\colon\rr\times\R\to\R$ be differentiable in the second variable $t$, with $g:=\partial_tG$ a Carath\'eodory function, fix $M\in GL(3,\R)$, and define $F(x,\UU):=G(x,|M\UU|)$. If $G(x,0)=0$ and $t\mapsto g(x,t)/t$ is nondecreasing for $t>0$, then $F$ satisfies (F4) (cf. \cite{SzWe}) and (F5).

Now let $\Gamma\in L^\infty(\rr)$ be $\Z^3$-periodic, positive, and bounded away from $0$ and let $W\in\cC^1(\R)$ such that $W(0)=W'(0)=0$ and $t\mapsto W'(t)/t$ is nondecreasing for $t>0$. If we define $F(x,\UU):=\Gamma(x)W(|M\UU|^2)$, where $M$ is as before, then (F1), (F2), (F4), and (F5) hold. If we take $W(t^2)=\bigl((1+|t|^q)^{p/q}-1\bigr)/p$ or $W(t^2)=\min\{|t|^p/p-1/p+1/q,|t|^q/q\}$ with $2<p<6<q$, then (F3) holds too with $\Phi(t)=W(t^2)$. Note that such $\Phi$'s are models for the double-power type nonlinearities considered in \cite{AzzBenDApFor,BenFor,DApSic,Mederski}.

We observe the following: if $W'$ is constant on $[a,b]$ for some $0<a<b<\infty$, then (take for simplicity $M$ as the identity matrix)
\[
0<F(x,\UU)-F(x,\VV)=\frac{\bigl(f(x,\UU)\cdot\UU\bigr)^2-\bigl(f(x,\UU)\cdot\VV\bigr)^2}{2f(x,\UU)\cdot\UU}
\]
for every $\UU\in\rr$ and $\VV\in\textup{span}(\UU)$ such that $a<|\VV|<|\UU|<b$, therefore the stronger variant of (F5) \cite[(F5)]{Mederski} (see also \cite[(F7)]{BartschMederski1}) is no longer satisfied and we cannot make use of techniques relying on a minimization over the Nehari--Pankov manifold
\begin{equation*}
\fN:=\Set{\UU\in\cD(\textup{curl},\Phi)\setminus\{0\}|E'(\UU)(\UU)=E'(\UU)(\nabla\varphi)=0 \, \forall\varphi\in\cC_c^\infty(\rr)}
\end{equation*}
as in \cite{BartschMederski1,Mederski} (see also \cite{SzWe,SzWeHandbook}). Note that $\fN$ need not be a differentiable manifold because in general $E$ is only of class $\cC^1$.

In addition, we allow nonlinearities that do not fit in the double-power case. If
\[
W(t^2)=
\begin{cases}
\frac12(t^2-1)\ln(1+|t|)-\frac14t^2+\frac12|t| & \, \text{if } |t|>1\\
\frac{\ln2}{q}(|t|^q-1)+\frac14 & \, \text{if } |t|\le1
\end{cases}
\]
for some $q>6$ and
\begin{equation}\label{e-ex1}
f(x,\UU)=
\begin{cases}
\Gamma(x)\ln(1+|\UU|)\UU & \, \text{if } |t|>1\\
\ln2\,\Gamma(x)|\UU|^{q-2}\UU & \, \text{if } |t|\le1,
\end{cases}
\end{equation}
then (F1)--(F5) hold, but $F$ cannot be controlled by any $N$-function related to the sum of two suitable Lebesgue spaces, as it is the case with a double-power behaviour. Moreover, there exists no $\eta>2$ such that $F$ satisfies the classical Ambrosetti--Rabinowitz condition \cite{AmbRab}
\[
f(x,\UU)\cdot\UU\ge\eta F(x,\UU)>0
\]
if $|\UU|>1$.

Another example is given by $F(x,\UU)=\Gamma(x)\Phi(|\UU|)$, where $\Phi(0)=0$ and
\begin{equation}\label{e-ex2}
\Phi'(t)=
\begin{cases}
\frac{\ln2\,t^5}{16\ln t} & \, \text{if } |t|>2\\
t & \, \text{if } 1\le|t|\le2\\
\frac{t^5}{1-\ln|t|} & \, \text{if } 0<|t|<1.
\end{cases}
\end{equation}
Once again, such $F$ does not satisfy the Ambrosetti--Rabinowitz condition for $1<|\UU|<2$. Furthermore, $\Phi$ satisfies (N2), but
\[
\lim_{t\to\pm\infty}\frac{\Phi(t)}{|t|^p}=\lim_{t\to0}\frac{\Phi(t)}{|t|^q}=\infty
\]
for every $2<p<6<q$, which is not the case (concerning the limit at zero) for a similar example given in \cite[page 256]{MeScSz}, where $\lim_{t\to0}\Phi(t)/|t|^{6+\varepsilon}=0$ for sufficiently small $\varepsilon>0$. We point out that, in the last two examples, it is convenient to use Lemma \ref{L:all} \textit{(i)} to check (N1).

The main result of this chapter reads as follows.

\begin{Th}\label{T:mainSZ}
Assume that (F1)--(F5) hold.
\begin{itemize}
	\item [(a)] There exists a ground state solution to \eqref{e-Curl}, i.e., $\UU\in\fN$ such that $E'(\UU)=0$ and
	\[
	E(\UU)=\inf_\fN E.
	\]
	\item [(b)] If, moreover, $F$ is even in $\UU$, then there exists an infinite sequence $\UU_n\in\fN$ of geometrically distinct solutions to \eqref{e-Curl}, i.e., $(\Z^3*\UU_n)\cap(\Z^3*\UU_m)=\emptyset$ for every $n\ne m$, where
	\[
	\Z^3*\UU=\Set{\UU(\cdot+z)|z\in\Z^3}.
	\]
\end{itemize}
\end{Th}

Instead of working directly with $\fN$, we make use of a second subset of $\cD(\textup{curl},\Phi)$, i.e.,
\[
\fM:=\Set{\UU\in\cD(\textup{curl},\Phi)|E(\UU)(\nabla\varphi)=0 \, \forall\varphi\in\cC_c^\infty(\rr)},
\]
which clearly contains $\fN$. A very important property of $\fM$ is that $E'$ is weak-to-weak* continuous when restricted to it, because it allows to find a solution to \eqref{e-Curl} as a weak limit point of a Cerami sequence, which is later proved to be a ground state solution.

\section{Functional and Orlicz setting}\label{S:Orlicz}

In the first part of this section, we recall a series of basic facts about $N$-functions and Orlicz spaces. Our discussion is based on \cite{RaoRen}.

A function $\Phi\colon\R\to[0,\infty[$ is called an $N$-function (or nice Young function) if and only if it is even, convex, and
\[
\Phi(t)=0\Leftrightarrow t=0, \quad \lim_{t\to0}\frac{\Phi(t)}{t}=0, \quad \lim_{t\to\infty}\frac{\Phi(t)}{t}=\infty.
\]

Given an $N$-function $\Phi$, we can define a second $N$-function as
\[
\Psi(t)=\sup\Set{s\left|t\right|-\Phi(s)|s\ge0}.
\]
$\Psi$ is called the \textit{complementary function} to $\Phi$ and $(\Phi,\Psi)$ is called a \textit{complementary pair} of $N$-functions. Recall from \cite[Section 1.3]{RaoRen} that $\Phi'$ and $\Psi'$ exist a.e.\footnote{In fact, $\Phi'$ and $\Psi'$ exist everywhere if we define them as the left (or right) derivatives.}, $\Psi'(t)=\inf\{s\ge 0:\Phi'(s)>t\}$ for $t\ge 0$, and $\Psi$ can be expressed as
\[
\Psi(t)=\int_0^{|t|}\Psi'(s)\,ds.
\]

We recall from \cite[Section 2.3]{RaoRen} that $\Phi$ satisfies the $\Delta_2$ \textit{condition globally} (denoted $\Phi\in\Delta_2$) if and only if there exists $K>1$ such that for every $t\in\R$
\[
\Phi(2t)\le K\Phi(t)
\]
(here $2$ can be replaced by any constant $a>1$), while $\Phi$ satisfies the $\nabla_2$ \textit{condition globally} (denoted $\Phi\in\nabla_2$) if and only if there exists $K'>1$ such that for every $t\in\R$
\[
\Phi(K't)\ge 2K'\Phi(t).
\]
When the inequalities above hold for every sufficiently large $|t|$, we say that $\Phi$ satisfies the $\Delta_2$ or $\nabla_2$ conditions \textit{locally}.

The set
\[
\cL^\Phi:=\cL^\Phi(\rr,\rr):=\Set{\UU\colon\rr\to\rr \text{ measurable }|\int_{\rr}\Phi(|\UU|)\,dx<\infty}
\]
is called the \textit{Orlicz class} relative to $\Phi$ and need not be a vector space, because it need not be closed under multiplication by a scalar. In order to have a vector space, we need to consider the set
\[
L^\Phi:=L^\Phi(\rr,\rr):=\Set{\UU\colon\rr\to\rr \text{ measurable }|\exists \, a>0:a\UU\in\cL^\Phi},
\]
called the \textit{Orlicz space} relative to $\Phi$. The property of $\cL^\Phi$ being a vector space is related to $\Phi$ satisfying the $\Delta_2$ condition. More precisely, $\cL^\Phi(\rr,\rr)=L^\Phi(\rr,\rr)$ if and only if $\Phi\in\Delta_2$ (cf. \cite[Theorem III.I.2]{RaoRen}).

Identifying functions equal a.e., $L^\Phi$ becomes a Banach space (cf. \cite[Theorems III.II.3 and III.III.10]{RaoRen}) if endowed with the norm
\[
|\UU|_\Phi:=\inf\Set{\alpha>0|\int_{\rr}\Phi\left(\frac{|\UU|}{\alpha}\right)\,dx\le1}.
\]
One can define an equivalent norm (cf. \cite[Proposition III.III.4]{RaoRen}) as
\[
|\UU|_{\Phi,1}:=\sup\Set{\int_{\rr}\left|\UU\right|\left|\VV\right|\,dx|\VV\in\cL^\Psi \text{ and } \int_{\rr}\Psi(|\VV|)\,dx\le1}.
\]
Here and in the sequel, $\Psi$ stands for the complementary function to $\Phi$.

Finally, if $\Phi\in\Delta_2$, then $L^\Phi$ is separable (cf. \cite[Theorem III.V.1]{RaoRen}) and its dual space is $L^\Psi$ (cf. \cite[Corollary IV.I.9]{RaoRen}); if, moreover, $\Psi\in\Delta_2$, then $L^\Phi$ and $L^\Psi$ are reflexive (cf. \cite[Theorem IV.I.10]{RaoRen}). As usual, the equality $(L^\Phi)'=L^\Psi$ is meant as an isometric isomorphism, where the norm on $(L^\Phi)'$ is induced in the standard way from the one in $L^\Phi$. In particular, if we consider the norm $|\cdot|_\Phi$ in $L^\Phi$, then such an isometry holds if we consider the norm $|\cdot|_{\Psi,1}$ in $L^\Psi$, i.e., $\bigl((L^\Phi)',\|\cdot\|_{(L^\Phi)'}\bigr) \simeq (L^\Psi,|\cdot|_{\Psi,1})$; similarly if we consider the norms $|\cdot|_{\Phi,1}$ and $|\cdot|_\Psi$.

We point out that the results in \cite{RaoRen} are given for scalar vector fields. Nonetheless, this is irrelevant in view of the following lemma. For $\Omega\subset\rr$ measurable define
\[
L^\Phi(\Omega):=\Set{u\colon\Omega\to\R \text{ measurable }|\int_{\Omega}\Phi(|au|)\,dx<\infty \text{ for some } a>0}
\]
and, with a small abuse of notation, endow it with the norm $|\cdot|_\Phi$ defined as before.

\begin{Lem}
$L^\Phi=L^\Phi(\rr)^3$ and their norms are equivalent.
\end{Lem}
\begin{proof}
For $\UU\in L^\Phi(\rr)^3$ we use the norm $|\UU|_{\Phi,3}:=\max_{i\in\{1,2,3\}}|\UU_i|_\Phi$, with $\UU=(\UU_1,\UU_2,\UU_3)$. Since $\Phi$ is increasing on positive numbers, for every $k>0$
\[
\int_{\R^3}\Phi\left(\frac{|\UU_i|}k\right)dx \le \int_{\R^3}\Phi\left(\frac{|\UU|}k\right)dx,
\]
hence, if the second integral is less than $1$, so is the first one. Taking the infimum over $k>0$ we obtain $|\UU_i|_\Phi\le|\UU|_\Phi$ for every $i\in\{1,2,3\}$ and $|\UU|_{\Phi,3}\le|\UU|_\Phi$. In particular, $L^\Phi(\rr)^3\subset L^\Phi$.

On the other hand, since $\Phi$ is convex,
\[
\int_{\R^3}\Phi\left(\frac{|\UU|}{k}\right)dx \le \frac13\sum_{i=1}^3 \int_{\R^3}\Phi\left(\frac{3|\UU_i|}{k}\right)dx\le\int_{\rr}\Phi\left(\frac{3\max_{i=1,2,3}|\UU_i|}{k}\right)\,dx,
\]
so $|\UU|_\Phi \le 3\max_{i=1,2,3}|\UU_i|_\Phi = 3|\UU|_{\Phi,3}$ and, in particular, $L^\Phi\subset L^\Phi(\rr)^3$.
\end{proof}

We conclude this series of recalls with the following properties.

\begin{Lem}\label{L:all}
\begin{itemize}
	\item [(i)] The following are equivalent:
	\begin{itemize}
		\item [-] $\Phi\in\Delta_2$,
		\item [-] there exists $K>1$ such that $t\Phi'(t)\le K\Phi(t)$ for every $t\in\R$,
		\item [-] there exists $K'>1$ such that $t\Psi'(t)\ge K'\Psi(t)$ for every $t\in\R$,
		\item [-] $\Psi\in\nabla_2$.
	\end{itemize}
	\item [(ii)] For every $\UU\in L^\Phi$, $\VV\in L^\Psi$ there holds
	$$\int_{\R^3}|\UU||\VV|\,dx\le\min\{|\UU|_{\Phi,1}|\VV|_{\Psi},|\UU|_{\Phi}|\VV|_{\Psi,1}\}.$$
	\item [(iii)] Let $\UU_n$, $\UU\in L^\Phi$. Then $|\UU_n-\UU|_\Phi\to 0$ implies that $\int_{\R^3}\Phi(|\UU_n-\UU|)\,dx\to 0$. If $\Phi\in\Delta_2$, then $\int_{\R^3}\Phi(|\UU_n-\UU|)\,dx\to 0$ implies $|\UU_n-\UU|_\Phi\to 0$.
	\item [(iv)] Let $X\subset L^\Phi$. Then $X$ is bounded if $\Set{\int_{\R^3}\Phi(\left|\UU\right|)\,dx|\UU\in X}$ is bounded. If $\Phi\in\Delta_2$, then the opposite implication holds.
\end{itemize}
\end{Lem}
\begin{proof}
\textit{(i)} follows from \cite[Theorem II.III.3]{RaoRen}, \textit{(ii)} follows from \cite[Formula (III.III.17)]{RaoRen}, \textit{(iii)} follows from \cite[Theorem III.IV.12]{RaoRen}, \textit{(iv)} follows from \cite[Corollary III.IV.15]{RaoRen}.
\end{proof}

From now on, we assume (F1)--(F5) and (N1)--(N3) hold and $\Phi$ denotes the $N$-function mentioned in the assumption (F3).

We recall that $\cD(\textup{curl},\Phi)$ is the completion of $\cC_c^\infty(\rr,\rr)$ with respect to the norm
\[
\|\UU\|_{\textup{curl},\Phi}:=\sqrt{|\nabla\times\UU|_2^2+|\UU|_\Phi^2}.
\]

The subspace of divergence-free vector fields is defined as
\[\begin{split}
\cV & := \Set{\UU\in\cD(\textup{curl},\Phi)|\int_{\rr}\UU\cdot\nabla\varphi\,dx=0 \text{ for every } \varphi\in\cC_c^\infty(\rr)}\\
& = \Set{\UU\in\cD(\textup{curl},\Phi)|\nabla\cdot\UU=0},
\end{split}\]
where the divergence of $\UU$ is understood in the distributional sense. As usual, let $\cD:=\cD^{1,2}(\rr,\rr)$ be the closure of $\cC_c^\infty(\rr,\rr)$ with respect to the norm $\|\UU\|_\cD=|\nabla\UU|_2$. Finally, let $\cW$ be the closure of $\Set{\nabla\varphi|\varphi\in\cC_c^\infty(\rr)}$ in $L^\Phi$.

\begin{Lem}\label{L:emb}
	$L^6(\R^3,\R^3)$ is continuously embedded in $L^{\Phi}$.
\end{Lem}
\begin{proof}
	In view of (N2), there exists $C>0$ such that
	$\Phi(t)\leq C|t|^6$ for every $t\in\R$, therefore we can conclude by Lemma \ref{L:all} \textit{(iii)}.
\end{proof}

The following Helmholtz decomposition holds.

\begin{Lem}\label{L:Hel}
$\cV$ and $\cW$ are closed subspaces of $\cD(\textup{curl},\Phi)$  and
\begin{equation}\label{e-Hel}
\cD(\textup{curl},\Phi)=\cV\oplus \cW.
\end{equation}
Moreover, $\cV\subset\cD$ and the norms $\|\cdot\|_{\cD}$ and $\|\cdot\|_{\textup{curl},\Phi}$ are equivalent in $\cV$.
\end{Lem}
\begin{proof}
Let $w\in\cW$ and $\varphi_n\in\cC_c^{\infty}(\R^3)$ such that $|w -\nabla\varphi_n|_{\Phi}\to 0$.
Then for every $\psi\in\cC_c^{\infty}(\R^3,\R^3)$
\[
\int_{\R^3}w\cdot\nabla\times\psi\,dx
=\lim_n\int_{\R^3}\nabla \varphi_n\cdot\nabla\times\psi\,dx=
\lim_n\int_{\R^3}\nabla\times(\nabla \varphi_n)\cdot\psi\,dx=0
\]
where we have used Lemma \ref{L:all} \textit{(ii)} and the fact that $\nabla\times\psi \in L^\Psi$. Hence $\nabla\times w=0$ in the sense of distributions and
$\|w\|_{\textup{curl},\Phi}=|w|_{\Phi}$. Therefore $\cW$ is closed in $\cD(\textup{curl},\Phi)$; moreover, it is easily checked that $\cV$ is closed in $\cD(\textup{curl},\Phi)$.

Now, take any $\UU\in\cD(\textup{curl},\Phi)$ and $\varphi_n\in\cC_c^{\infty}(\R^3,\R^3)$ such that $\varphi_n\to\UU$ in $\cD(\textup{curl},\Phi)$. Let $\varphi_n^2\in\cC^{\infty}(\R^3)$ be the Newtonian potential of $\nabla\cdot\varphi_n$, i.e., $\varphi_n^2$ solves $\Delta \varphi_n^2 = \nabla\cdot\varphi_n$. Note that the derivative $\partial_i\varphi_n^2$ is the Newtonian potential of $\nabla\cdot\partial_i\varphi_n$. Since $\varphi_n\in\cC_c^{\infty}(\R^3)$, then by \cite[Proposition 1]{Iwaniec}, $\nabla\varphi_n^2$ and $ \nabla(\partial_i\varphi_n^2) \in L^r(\R^3,\R^3)$ for every $r\in\,]1,\infty[$. Hence from Lemma \ref{L:emb}
\[
\nabla \varphi_n^2\in L^6(\R^3,\R^3)\subset L^{\Phi}
\] 
and $\varphi_n^1:=\varphi_n-\nabla \varphi_n^2\in L^{\Phi}$. Moreover, $\varphi_n^1,\partial_i\varphi_n^1\in L^r(\R^3,\R^3)$.
We also have $\nabla\times\varphi_n^1 = \nabla\times\varphi_n$  and $\nabla\cdot\varphi_n^1=0$ pointwise. Using these two equalities and integrating by parts we obtain
$|\nabla\varphi^1_n|_2 = |\nabla\times\varphi^1_n|_2 = |\nabla\times\varphi_n|_2$. It follows that for every $m,n$
\[
|\nabla(\varphi_n^1-\varphi_m^1)|_{2}=|\nabla\times(\varphi_n^1-\varphi_m^1)|_{2}=|\nabla\times(\varphi_n-\varphi_m)|_{2}\leq \|\varphi_n-\varphi_m\|_{\textup{curl},\Phi},
\]
thus $\varphi_n^1$ is a Cauchy sequence in $\cD$. Let $v:=\lim_n\varphi_n^1$ in $\cD$. Then
\[
\int_{\R^3}v\cdot\nabla \varphi\,dx=\lim_n\int_{\R^3}\varphi^1_n\cdot\nabla\varphi\,dx=0
\]
for every $\varphi\in\cC_c^{\infty}(\R^3)$, hence $\nabla\cdot v=0$ and $v\in\cV$.
Moreover,
\[
|\nabla\times(\varphi_n^1-v)|_{2}=|\nabla(\varphi_n^1-v)|_{2}\to 0,
\]
so $\varphi_n^1\to v$ in  $\cD(\textup{curl},\Phi)$ and
$\nabla \varphi^2_n=\varphi_n-\varphi^1_n\to\UU-v$ in $\cD(\textup{curl},\Phi)$. Since $\cW$ is closed in $\cD(\textup{curl},\Phi)$, then $\UU-v\in \cW$ and we get the decomposition
$$\UU=v+(\UU-v)\in\cV+\cW.$$

Now take $\UU\in\cV\cap\cW$. Then $\nabla\times\UU=0$, so by \cite[Lemma 1.1 \textit{(i)}]{Leinfelder}, $\UU=\nabla\xi$ for some $\xi\in W^{1,6}_\textup{loc}(\R^3)$. Since $\nabla\cdot\UU=0$, $\xi$ is harmonic and therefore so is $\UU$. Hence 
\[
0 = -\int_{\R^3}\UU\cdot\Delta\UU\,dx = \int_{\R^3}|\nabla\UU|^2\,dx
\] 
and so $\UU=0$, thus  $\cV\cap\cW=\{0\}$ and we obtain \eqref{e-Hel}.

The equivalence of norms follows from Lemma \ref{L:emb}.
\end{proof}

In view of Lemmas \ref{L:emb} and \ref{L:Hel}, $\cV$ is continuously embedded in $L^\Phi$. We introduce a norm in $\cV\times\cW$ by the formula
\[
\|(v,w)\|:=\sqrt{\|v\|_{\cD}^2+|w|_{\Phi}^2}.
\]
We also define a functional on $\cV\times\cW$ that is the counterpart of $E$, defined in $\cV\oplus\cW=\cD(\textup{curl},\Phi)$, as
\[
J(v,w):=\int_{\rr}\frac12|\nabla v|^2-F(x,v+w)\,dx.
\]
In order to prove that $J\in\cC^1(\cV\times\cW)$, we need the following result.

\begin{Lem}\label{L:forC1}
There exists $C>0$ such that for every $t\in\R$
\[
\Psi\bigl(\Phi'(t)\bigr)\le C\Phi(t).
\]
In particular, if $\UU\in L^\Phi$, then $\Phi'(|\UU|)\in L^\Psi$.
\end{Lem}
\begin{proof}
Since $\Phi\in\Delta_2$, from Lemma \ref{L:all} \textit{(i)} and \cite[Theorem I.III.3]{RaoRen} there holds
\[
\Psi\bigl(\Phi'(t)\bigr)=t\Phi'(t)-\Phi(t)\le(K-1)\Phi(t).\qedhere
\]
\end{proof}

\begin{Prop}\label{P:classC1}
$J\in\cC^1(\cV\times\cW)$ and
\[
J'(v,w)(v',w')=\int_{\rr}\nabla v\cdot\nabla v'-f(x,v+w)\cdot(v'+w')\,dx
\]
for every $v,v'\in\cV$ and every $w,w'\in\cW$.
\end{Prop}
\begin{proof}
From Lemma \ref{L:emb}, the proof is complete if we prove that $I\in\cC^1(L^\Phi)$, where $I(\UU) := \int_{\R^3}F(x,\UU)\,dx$, and $I'(\UU)(\VV)=\int_{\rr}f(x,\UU)\cdot\VV\,dx$ for every $\UU,\VV\in L^\Phi$. Let $\UU,\VV\in L^\Phi$ and $t\ne0$ (we can assume $|t|\le1$). Then
\[\begin{split}
\frac{I(\UU)-I(\UU+t\VV)}{t} & =\frac1t\int_{\rr}F(x,\UU)-F(x,\UU+t\VV)\,dx\\
& =\int_{\rr}f(x,\UU+\theta t\VV)\cdot\VV\,dx
\end{split}\]
for some $\theta=\theta(t)\in[0,1]$. Since $f$ is a Charat\'eodory function, for a.e. $x\in\rr$
\[
\lim_{t\to0}f\bigl(x,\UU(x)+\theta t\VV(x)\bigr)\cdot\VV(x)=f\bigl(x,\UU(x)\bigr)\cdot\VV(x).
\]
Moreover, from (F3) and the monotonicity of $\Phi'$,
\[
|f(x,\UU+\theta t\VV)\cdot\VV| \le c_1\Phi'(|\UU+\theta t\VV|)|\VV| \le c_1\Phi'(|\UU|+|\VV|)|\VV| \in L^1(\rr)
\]
owing to Lemmas \ref{L:all} \textit{(ii)} and \ref{L:forC1}. In view of the dominated convergence theorem, $I'(\UU)(\VV)=\int_{\rr}f(x,\UU)\cdot\VV\,dx$. In addition, the same argument proves that, for a fixed $\UU\in L^\Phi$, the linear map
\[
\VV\in L^\Phi\mapsto I'(\UU)(\VV)\in\R
\]
is continuous. Now we prove that $I'\in \cC(L^\Phi,L^\Psi)$ (recall that $L^\Psi\simeq(L^\Phi)'$). Let $\UU_n\to\UU$ in $L^\Phi$. We want to prove that $\lim_n|f(\cdot,\UU_n)-f(\cdot,\UU)|_\Psi=0$, or equivalently
\[
\lim_n\left|\frac{f(\cdot,\UU_n)-f(\cdot,\UU)}{2c_1}\right|_\Psi=0,
\]
where $c_1$ is the same as (F3). In view of Lemma \ref{L:all} \textit{(iii)}, this is the same as proving that
\[
\lim_n\int_{\rr}\Psi\left(\left|\frac{f(\cdot,\UU_n)-f(\cdot,\UU)}{2c_1}\right|\right)\,dx=0.
\]
By the same classical argument used for Lebesgue spaces (see, e.g., \cite[Proof of Theorem 3.11]{Rudin}) together with Lemma \ref{L:all} \textit{(iv)}, $\UU_n\to\UU$ a.e. in $\rr$ up to a subsequence, hence $\Psi\bigl(|f(\cdot,\UU_n)-f(x,\UU)|/(2c_1)\bigr)\to0$ a.e. in $\rr$. With the intent to use again the dominated convergence theorem, from (F3), Lemma \ref{L:forC1}, and the convexity of $\Psi$ we have
\[\begin{split}
\Psi\left(\left|\frac{f(x,\UU_n)-f(x,\UU)}{2c_1}\right|\right) & \le \Psi\left(\frac{\Phi'(|\UU_n|)+\Phi'(|\UU|)}2\right)\\
& \le \frac{\Psi\bigl(\Phi'(|\UU_n|)\bigr)+\Psi\bigl(\Phi'(|\UU|)\bigr)}2\\
& \le \frac{\Phi(|\UU_n|)+\Phi(|\UU|)}2
\end{split}\]
and
\[
\Phi(|\UU_n|)\le\frac{\Phi(2|\UU_n-\UU|)}2+\frac{\Phi(2|\UU|)}2.
\]
Finally, since $\Phi(2|\UU_n-\UU|)\to0$ in $L^1(\rr)$, from \cite[Lemma A.1]{Willem} there exists $g\in L^1(\rr)$ such that $\Phi(2|\UU_n-\UU|)\le g$ for every $n$.
\end{proof}

We conclude this section with the variational formulation of \eqref{e-Curl} and showing the ``equivalence'' of the functionals $E$ and $J$.

\begin{Prop}\label{P:var}
If $\UU=v+w\in\cV\oplus\cW$, then the following are equivalent:
\begin{itemize}
	\item [(i)] $(v,w)$ is a critical point of $J$;
	\item [(ii)] $\UU$ is a critical point of $E$;
	\item [(iii)] $\UU$ is a (weak) solution to \eqref{e-Curl}.
\end{itemize}
\end{Prop}
\begin{proof}
For the first equivalence, let $\UU'=v'+w'\in\cV\oplus\cW$. Then we have
\[
\int_{\R^3}f(x,v+w)\cdot(v'+w')\,dx=\int_{\R^3}f(x,\UU)\cdot\UU'\,dx
\]
and, since $\nabla\times w=\nabla\times w'=0$,
\[
\int_{\R^3}\nabla\times v\cdot\nabla\times v'\,dx=\int_{\R^3}\nabla\times\UU\cdot\nabla\times\UU'\,dx
\]
so that
\[\begin{split}
\int_{\R^3}\nabla\times v\cdot\nabla\times v'\,dx=\int_{\R^3}f(x,v+w)\cdot(v'+w')\,dx\\
\Leftrightarrow\int_{\R^3}\nabla\times\UU\cdot\nabla\times\UU'\,dx=\int_{\R^3}f(x,\UU)\cdot\UU'\,dx
\end{split}\]
and the conclusion follows from Lemma \ref{L:Hel}.
For the second equivalence we just need to observe that for every $\varphi\in\cC_c^\infty(\R^3,\R^3)$
\[
\int_{\R^3}\nabla\times\UU\cdot\nabla\times\varphi\,dx=\int_{\R^3}\UU\cdot\nabla\times\nabla\times\varphi\,dx.\qedhere
\]
\end{proof}

\section{An abstract critical point theory}\label{S:abstract}

We recall the abstract setting from \cite{BartschMederski1,BartschMederski2}.
Let $X$ be a reflexive Banach space with norm $\|\cdot\|$ and a topological direct sum decomposition $X=X^+\oplus X^-$, where $X^+$ is a Hilbert space with scalar product $(\cdot|\cdot)$. For $u\in X$ we denote by $u^+\in X^+$ and $u^-\in X^-$ the corresponding summands such that $u = u^++u^-$. We assume $\|u\|^2 = \|u^+\|^2+\|u^-\|^2 = (u^+|u^+) + \|u^-\|^2$ for every $u\in X$. The topology $\cT$ on $X$ is defined as the product of the norm topology in $X^+$ and the weak topology in $X^-$. Thus $u_n\overset{\cT}{\to}u$ is equivalent to $u_n^+\to u^+$ and $u_n^-\rightharpoonup u^-$.

Let $J\colon X\to\R$ be a functional of the form 
\begin{equation*}
J(u) = \frac12\|u^+\|^2-I(u)
\end{equation*}
for some $I\colon X\to\R$. The set
\begin{equation*}
\cM := \Set{u\in X|\, J'(u)|_{X^-}=0}=\Set{u\in X|\, I'(u)|_{X^-}=0}
\end{equation*}
obviously contains all the critical points of $J$. Suppose the following assumptions hold.
\begin{itemize}
	\item[(I1)] $I\in\cC^1(X)$ and $I(u)\ge I(0)=0$ for every $u\in X$.
	\item[(I2)] If
	$u_n\overset{\cT}{\to} u$, then $\liminf_n I(u_n)\ge I(u)$.
	\item[(I3)] If $u_n\overset{\cT}{\to}u$ and $I(u_n)\to I(u)$, then $u_n\to u$.
	\item[(I4)] $\|u^+\|+I(u)\to\infty$ as $\|u\|\to\infty$.
	\item[(I5)] If $u\in\cM$, then $I(u)<I(u+v)$ for every $v\in X^-\setminus\{0\}$.
\end{itemize}

Clearly, if a strictly convex functional $I$ satisfies (I4), then (I2) and (I5) hold. Observe that for every $u\in X^+$ we define $m(u)\in\cM$ as the unique global maximizer of $J|_{u\oplus X^-}$. Note that $m$ need not be $\cC^1$ and $\cM$ need not be a differentiable manifold because $I'$ is only required to be continuous. Recall from \cite{BartschMederski2} that 
$J$ satisfies the $(PS)_c^\cT$-condition on $\cM$ if and only if each $(PS)_c$-sequence $u_n\in\cM$ has a subsequence converging in the $\cT$-topology. In order to apply classical critical point theory to $J\circ m: X^+\to \R$ like the mountain pass geometry we need some additional assumptions.
\begin{itemize}
	\item[(I6)] There exists $r>0$ such that $a:=\inf\limits_{u\in X^+,\|u\|=r} J(u)>0$.
	\item[(I7)] $\displaystyle\frac{I(t_nu_n)}{t_n^2}\to\infty$ if $t_n\to\infty$ and $u_n^+\to u^+\ne 0$.
\end{itemize}

According to \cite[Theorem 4.4]{BartschMederski2}, if (I1)--(I7) hold and 
\[
c_\cM := \inf_{\gamma\in\Gamma}\max_{t\in [0,1]} J\bigl(\gamma(t)\bigr),
\]
where
\[
\Gamma := \Set{\gamma\in\cC([0,1],\cM)|\gamma(0)=0,\ \|\gamma(1)^+\|>r, \text{ and } J\bigl(\gamma(1)\bigr)<0},
\]
then $c_{\cM}\ge a>0$ and $J$ has a $(PS)_{c_\cM}$-sequence $u_n\in\cM$. If, in addition,
$J$ satisfies the $(PS)_{c_\cM}^\cT$-condition in $\cM$,  then $c_\cM$ is achieved by a critical point of $J$. Since we look for solutions to \eqref{e-Curl} in $\R^3$ and not in a bounded domain as in \cite{BartschMederski2}, the $(PS)_{c_\cM}^\cT$-condition is no longer satisfied. We consider the set
\begin{equation*}
\begin{split}
\cN & := \Set{u\in X\setminus X^-| J'(u)|_{\R u\oplus X^-}=0}\\
& =\Set{u\in\cM\setminus X^-| J'(u)(u)=0} \subset\cM,
\end{split}
\end{equation*}
which clearly contains all the nontrivial critical points of $J$, and require the following condition on $I$:
\begin{itemize}
	\item[(I8)] $\displaystyle\frac{t^2-1}{2}I'(u)(u)+I(u) - I(tu+v)\le0$ for every $u\in \cN$, $t\ge 0$, $v\in X^-$.
\end{itemize}
In \cite{BartschMederski1,BartschMederski2} it was additionally assumed that the \emph{strict inequality holds} if $u\neq tu+v$. This stronger variant of (I8) implies that for every $u^+\in X^+\setminus\{0\}$ the functional $J$ has a unique critical point $n(u^+)$ on the half-space $\R^+u^+ +X^-$. Moreover, $n(u^+)$ is a global maximizer of $J$ on the space $\R u^++X^-$, the map
$$n\colon\Set{u^+\in X^+|\|u^+\|=1} \to \cN$$
is a homeomorphism, the set $\cN$ is a topological manifold, and it is enough to look for critical points of $J\circ n$. This is the approach of \cite{SzWe,SzWeHandbook}. However, if the weaker condition (I8) holds, this procedure cannot be repeated. In particular, $\cN$ need not be a manifold. However, the following holds.

\begin{Lem}\label{L:max}
If $u\in\cN$, then $u$ is a (not necessarily unique) maximizer of $J$ on $\R^+u +X^-$.
\end{Lem}

\begin{proof}
	Let $u\in\cN$. In view of (I8), we get by explicit computations
	\begin{equation*}
	J(tu+v)=J(tu+v)-J'(u)\left(\frac{t^2-1}{2}u+tv\right)\le J(u)
	\end{equation*}
	for every $t\geq 0$ and every $v\in X^-$.
\end{proof}

Let 
\[
\cJ := J\circ m\colon X^+\to\R. 
\]
Before proving the main results of this section, we recall the following properties from \cite[Proof of Theorem 4.4]{BartschMederski2}. Note that (I8) was not used there.
\begin{itemize}
	\item[(i)] For every $u^+\in X^+$ there exists a unique $u^-\in X^-$ such that $m(u^+):=u^++u^-\in\cM$. This $m(u^+)$ is the minimizer of $I$ on $u^++X^-$.
	\item[(ii)] $m:X^+\to \cM$  is a homeomorphism, its inverse being $u\in\cM\mapsto u^+\in X^+$.
	\item[(iii)] $\cJ=J\circ m \in\cC^1(X^+,\R)$.
	\item[(iv)]$\cJ'(u^+) = J'\bigl(m(u^+)\bigr)|_{X^+}\colon X^+\to\R$ for every $u^+\in X^+$.
\end{itemize}
Property (i) was already discussed earlier. We will also need the following property. 

\begin{Lem}\label{L:infty}
Let $X_k$ be a $k$-dimensional subspace of $X^+$. Then $\cJ(u)\to-\infty$ whenever $\|u\|\to\infty$ and $u\in X_k$.
\end{Lem}

\begin{proof}
	It suffices to show that each sequence $u_n^+\in X_k$ such that $\|u_n^+\|\to\infty$ contains a subsequence along which $\cJ$ tends to $-\infty$. Let $u_n^+ = t_nv_n$ with $\|v_n\|=1$ and set $m(u_n^+)=u_n^++u_n^-\in\cM$. Then, passing to a subsequence and using (I7), we obtain
	\[
	\frac{\cJ(t_nv_n)}{t_n^2} = \frac12 - \frac{I\bigl(t_n(v_n+u_n^-/t_n)\bigr)}{t_n^2} \to - \infty.\qedhere
	\]
\end{proof}

In view of (I4), it is clear that if $u_n$ is a bounded Cerami sequence for $\cJ$, then $m(u_n)\in\cM$ is a bounded Cerami sequence for $J$. We introduce the set $\cN_0 := \Set{u\in X^+\setminus\{0\}|\cJ'(u)(u)=0}$, i.e., the Nehari manifold for $\cJ$. Denote $c_{\cN_0} := \inf_{\cN_0}\cJ$.

\begin{Th}\label{T:Link1}
	Suppose $J \in \cC^1(X)$ satisfies (I1)--(I8). Then:
	\begin{itemize}
		\item[(a)] $c_{\cM}\ge a>0$ and $\cJ$ has a Cerami sequence $u_n\in X^+$ at the level $c_\cM$. 
		\item[(b)] $c_{\cM}=c_{\cN}:= \inf_\cN J$.
	\end{itemize}
\end{Th}
\begin{proof}
	\textit{(a)} Set
	\begin{equation}\label{e-Gamma_Phi}
	\Sigma := \Set{\sigma\in\cC([0,1],X^+)|\sigma(0)=0, \|\sigma(1)\|>r,
	\cJ(\sigma(1))<0}.
	\end{equation}
	Observe that $\cJ$ has the mountain pass geometry and $\Gamma$ and $\Sigma$ are related as follows: if $\gamma\in\Gamma$, then $\gamma^+\in\Sigma$ and $J\bigl(\gamma(t)\bigr)=\cJ\bigl(\gamma^+(t)\bigr)$; if $\sigma\in\Sigma$, then $m\circ\sigma\in\Gamma$ and $\cJ\bigl(\sigma(t)\bigr) = J\bigl(m\bigl(\sigma(t)\bigr)\bigr)$. Hence the mountain pass value for $\cJ$ is given by
	\begin{equation}\label{e-cm}
	c_{\cM}=\inf_{\sigma\in\Sigma}\max_{t\in [0,1]}J\bigl(m\bigl(\sigma(t)\bigr)\bigr)=\inf_{\sigma\in\Sigma}\max_{t\in [0,1]}\cJ\bigl(\sigma(t)\bigr) \ge a > 0.
	\end{equation} 
	The existence of a Cerami sequence $u_n\in X^+$ for $\cJ$ at the level $c_\cM$ then follows.
			
	The map $u\mapsto m(u)$ is a homeomorphism between $\cN_0$ and $\cN$ and, since $\cJ(u)=J\bigl(m(u)\bigr)$, $c_{\cN_0}=c_{\cN}$. For $u\in X^+\setminus\{0\}$, consider $\cJ(tu)$, $t>0$. From Lemma \ref{L:infty}, $\cJ(tu)\to-\infty$ as $t\to\infty$. Hence $\max_{t > 0}\cJ(tu)\geq a$ exists. If $t_1u,t_2u\in \cN_0$, then $m(t_1u),m(t_2u)\in\cN$, so from Lemma \ref{L:max}, $\cJ(t_1u) =\cJ(t_2u)$. Consequently, there exist $0 < t_\textup{min} \le t_\textup{max}$ such that $\cJ(tu)\in\cN_0$ if and only if $t\in[t_\textup{min},t_\textup{max}]$ and $\cJ(tu)$ has the same value for those $t$. Hence $\cJ'(tu)(u)>0$ for $0<t<t_\textup{min}$ and $\cJ'(tu)(u)<0$ for $t>t_\textup{max}$.  It follows that $X^+\setminus \cN_0$ consists of two connected components, hence each path in $\Sigma$ must intersect $\cN_0$. Therefore $c_{\cM}\ge c_{\cN_0}$. Since $c_{\cN_0} = \inf_{u\in X^+\setminus\{0\}}\max_{t>0}\cJ(tu)$, \eqref{e-cm} implies $c_{\cM} = c_{\cN_0} = c_{\cN}$. Note in particular that $\cJ\ge 0$ on $B_r$, where $r$ is given in (I6), so the condition $\|\sigma(1)\|>r$ in the definition of $\Sigma$ is redundant because it must necessarily hold if $\cJ\bigl(\sigma(1)\bigr)<0$.
\end{proof}

Since $c_{\cN_0}=c_{\cN}=c_{\cM}>0$, $\cN_0$ is bounded away from 0 and hence closed in $X^+$, while $\cN$ is bounded away from $X^-$ and hence closed in $X$.

For a topological group acting on $X$, denote the \emph{orbit} of $u\in X$ by $G\ast u$, i.e., 
$$G\ast u:=\Set{gu|g\in G}.$$
A set $A\subset X$ is called $G$\emph{-invariant} if and only if $gA\subset A$ for all $g\in G$. $J\colon X\to\R$ is called \emph{$G$-invariant} and $T\colon X\to X'$ (or $T\colon X\to X$) $G$\emph{-equivariant} if and only if $J(gu)=J(u)$ and $T(gu)=gT(u)$ for all $g\in G$, $u\in X$.
 
In order to deal with  multiplicity of critical points, assume that $G$ is a topological group such that
\begin{itemize}
	\item[(G)] $G$ acts on $X$ by isometries and $(G*u)\setminus\{u\}$ is bounded away from $u$ for every $u\ne 0$. Moreover, $J$ is $G$-invariant and $X^+,X^-$ are $G$-invariant.
\end{itemize}
Observe that $\cM$ is $G$-invariant and $m\colon X^+\to\cM$ is $G$-equivariant if (G) holds. In our application to \eqref{e-Curl} we have $G=\mathbb{Z}^3$ acting by translations.

\begin{Lem} \label{L:discrete}
	For all $u,v\in X$ there exists $d=d_{u,v}>0$ such that $\|gu-hv\|>d$ for every $g,h\in G$ satisfying $gu\ne hv$. Moreover, $d_{u,v}$ only depends on the orbits of $u$ and $v$.
\end{Lem}
\begin{proof}
	Since $gu\ne hv$, $u$ and $v$ are not both $0$ and we can assume $u\ne0$. If $v\in G*u$, then $g^{-1}hv\in G*u$ and the claim follows from (G). If $v\not\in G*u$, then we can assume that $v$ minimizes the distance from $u$ to $G*v$, thus it suffices to take $d := \frac12\|u-v\|$. As for the last part, let $\bar{u}=e_1u$ and $\bar{v}=e_2v$ for some $e_1,e_2\in G$. If $g,h\in G$ are such that $g\bar{u}\ne h\bar{v}$, then setting $\bar{g}=ge_1$ and $\bar{h}=he_2$ we have $\bar{g}u\ne\bar{h}v$ and $\|g\bar{u}-h\bar{v}\|=\|\bar{g}u-\bar{h}v\|>d_{u,v}$.
\end{proof}

We will use the notations
\begin{gather*}
\cJ^\beta := \Set{u\in X^+|\cJ(u)\le\beta}, \quad  \cJ_\alpha := \Set{u\in X^+|\cJ(u)\ge\alpha}, \\
\cJ_\alpha^\beta := \cJ_\alpha\cap\cJ^\beta, \quad \cK:=\Set{u\in X^+|\cJ'(u)=0}.
\end{gather*}
Since all the nontrivial critical points of $J$ are in $\cN$, it follows from Theorem \ref{T:Link1} that $\cJ(u)\ge a$ for all $u\in\cK\setminus\{0\}$.

We introduce the following variant of the Cerami condition between the levels $\alpha, \beta\in\R$, $\alpha\le\beta$.
\begin{itemize}
	\item[$(M)_\alpha^\beta$]
	\begin{itemize}
		\item[(a)] There exists $M_\alpha^\beta>0$ such that $\limsup_n\|u_n\|\le M_\alpha^{\beta}$ for every $u_n\in X^+$ satisfying $(1+\|u_n\|)\cJ'(u_n)\to 0$ and $$\alpha\le\liminf_n\cJ(u_n)\le\limsup_n\cJ(u_n)\leq\beta.$$
		\item[(b)] Suppose, in addition, that the number of critical orbits\footnote{A critical orbit is the orbit of a critical point.} in $\cJ_\alpha^\beta$ is finite. Then there exists $m_\alpha^\beta>0$ such that if $u_n,v_n$ are two sequences as above and there exists $n_0\ge1$ such that $\|u_n-v_n\|<m_\alpha^\beta$ for every $n\ge n_0$, then  $\liminf_n\|u_n-v_n\|=0$.
	\end{itemize}
\end{itemize}

Note that if $J$ is even, then $m$ is odd (hence $\cJ$ is even) and $\cM$ is symmetric,  i.e., $\cM=-\cM$. Note also that $(M)_\alpha^\beta$ is a condition on $\cJ$ and \emph{not} on $J$. Our main multiplicity result reads as follows.

\begin{Th}\label{T:CrticMulti}
	Suppose $J \in \cC^1(X)$ satisfies (I1)--(I8) and $\dim(X^+)=\infty$.
	\begin{itemize}
		\item [(a)] If  $(M)_0^{c_\cM+\varepsilon}$ holds for some $\varepsilon>0$, then either $c_{\cM}$ is attained by a critical point or there exists a sequence of critical values $c_n$ such that $c_n>c_{\cM}$ and $c_n\to c_{\cM}$ as $n\to\infty$.
		\item [(b)] If $(M)_0^{\beta}$ holds for every $\beta>0$ and $J$ is even, 
		then $J$ has infinitely many distinct critical orbits.
	\end{itemize}
\end{Th}

By a standard argument (cf. \cite[Lemma II.3.2]{Struwe}, \cite[Lemma 2.2]{Willem}) we can find
a locally Lipschitz continuous pseudo-gradient vector field $V\colon X^+\setminus \cK\to X^+$ associated with $\cJ$, i.e.,
\begin{equation*}\begin{split}
\|V(u)\|&<1\\
\cJ'(u)\bigl(V(u)\bigr) &> \frac12\|\cJ'(u)\|
\end{split}\end{equation*}
for every $u\in X^+\setminus \cK$. Moreover, if $J$ is even, then we can assume $V$ is odd.  
Let
$\eta\colon\cG\to  X^+\setminus \cK$ be the flow defined by
\begin{equation*}
\left\{
\begin{aligned}
&\partial_t \eta(t,u)=-V(\eta(t,u))\\
&\eta(0,u)=u
\end{aligned}
\right.
\end{equation*}
where $\cG:=\Set{(t,u)\in [0,\infty[\times (X^+\setminus \cK)|t<T(u)}$ and $T(u)$ is the maximal time of existence of $\eta(\cdot,u)$. Recall that $\cJ$ is decreasing along the trajectories of $\eta$, i.e.
\[
u\in X^+ \text{ and } 0\le s<t<T(u) \Rightarrow \cJ\bigl(\eta(s,u)\bigr)>\cJ\bigl(\eta(t,u)\bigr).
\]
We prove Theorem \ref{T:CrticMulti} by contradiction. From now on we assume:  
\[
\text{There is a finite number of distinct critical orbits $\Set{G\ast u|u\in\cK}$.}
\]

\begin{Lem}\label{L:flow}
	Suppose $(M)_0^\beta$ holds for some $\beta>0$ and let $u \in \cJ_0^\beta\setminus \cK$. Then either $\lim_{t \to T(u)^-}\eta(t,u)$ exists and belongs to $\cK$ or $\lim_{t\to T(u)^-}\cJ(\eta(t,u)) = -\infty$. In the latter case, $T(u)=\infty$.
\end{Lem}
\begin{proof}
	Suppose $T(u)<\infty$ and let $0\le s<t<T(u)$. Then
	\[
	\|\eta(t,u)-\eta(s,u)\| \le \int_s^t \left\|V\bigl(\eta(\tau,u)\bigr)\right\|\,d\tau \le t-s.
	\] 
	Hence the limit exists and, if it were not a critical point, then $\eta(\cdot,u)$ could be extended for $t>T(u)$.
	
	Suppose now $T(u)=\infty$ and $\cJ\bigl(\eta(\cdot,u)\bigr)$ is bounded from below. We distinguish three cases: 
	\begin{itemize}
		\item[(i)] $t\mapsto \eta(t,u)$ is bounded,
		\item[(ii)] $t\mapsto \eta(t,u)$ is unbounded but $\|\eta(t,u)\|\not\to\infty$,
		\item[(iii)] $\|\eta(t,u)\|\to\infty$.
	\end{itemize} 
	(i) We follow an argument in \cite{SzWe}. We will show that for every $\varepsilon>0$ there exists $t_\varepsilon>0$ such that
	$\|\eta(t_\varepsilon,u)-\eta(t,u)\|<\varepsilon$ for all $t \ge t_\varepsilon$ (this implies $\lim_{t \to \infty}\eta(t,u)$ exists and so, as before, it is a critical point). Arguing by contradiction, we can find $\varepsilon\in\,]0,m_0^\beta/2[$, $R>0$ and $t_n\to\infty$ such that $\|\eta(t_n,u)\|\le R$ and $\|\eta(t_n,u)-\eta(t_{n+1},u)\| = \varepsilon$ for all $n$.  Let $t_n^1$ be the smallest $t\in\,]t_n,t_{n+1}[$ such that $\|\eta(t_n,u)-\eta(t_n^1,u)\|= \varepsilon/3$ and $t_n^2$ the largest $t\in\,]t_n^1,t_{n+1}[$ such that $\|\eta(t_{n+1},u)-\eta(t_n^2,u)\|= \varepsilon/3$. Set $\kappa_n := \min \Set{\lVert\cJ'\bigl(\eta(t,u)\bigr)\rVert|t\in[t_n,t_n^1]}>0$. Then
	\begin{equation*}\begin{split}
		\frac{\varepsilon}3 & = \|\eta(t_n^1,u)-\eta(t_n,u)\| \le \int_{t_{n}}^{t_n^1}  \big\|V\bigl(\eta(t,u)\bigr)\big\| \,dt \le t_n^1-t_n \\ 
		& \le \frac{2}{\kappa_n} \int_{t_n}^{t_n^1} \cJ'\bigl(\eta(t,u)\bigr)\bigl(V\bigl(\eta(t,u)\bigr)\bigr)\,dt = \frac2{\kappa_n}\Bigl(\cJ\bigl(\eta(t_n,u)\bigr)-\cJ\bigl(\eta(t_n^1,u)\bigr)\Bigr).
	\end{split}\end{equation*}
	Since $\cJ(\eta(t_n,u))-\cJ(\eta(t_n^1,u))\to 0$, also $\kappa_n\to 0$, hence we can choose $s_n^1\in[t_n,t_n^1]$ such that, if $u_n:= \eta(s_n^1,u)$, then $\cJ'(u_n)\to 0$. Since  $\|\eta(s_n^1,u)\|$ is bounded, $u_n$ is a Cerami sequence. A similar argument shows the existence of $v_n:= \eta(s_n^2,u)$, $s_n^2\in[t_n^2,t_{n+1}]$, such that $\cJ'(v_n)\to 0$. Hence
	\[
	\frac\varepsilon 3 \le \liminf_n\|u_n-v_n\|\le \limsup_n\|u_n-v_n\|\le\varepsilon +\frac{2}{3}\varepsilon<m_0^\beta,
	\]
	a contradiction to $(M)_0^\beta (b)$.
	
	(ii) Observe that there are no Cerami sequences in $X^+\setminus\overline B_{M_0^\beta}$ at any level $\alpha\in[0, \beta]$ according to $(M)_0^\beta (a)$. Since $\eta(t,u)$ is unbounded but $\|\eta(t,u)\|\not\to\infty$, there exists $R>M_0^\beta$ such that $\eta(t,u)\in B_R$ for arbitrarily large $t>0$. Then we can find $t_n, t_n^1$ so that $t_n\to\infty$, $\|\eta(t_n,u)\| = R+1$ and $t_n^1$ is the smallest $t>t_n$ with $\|\eta(t,u)\| = R$. We can also assume that $\|\eta(s,u)\|\leq R+1$ for $s\in [t_n,t_n^1]$. Let $\kappa_n$ be as above. Then
	\[
	1\le \|\eta(t_n^1,u)-\eta(t_n,u)\| \le \frac2{\kappa_n}\Bigl(\cJ\bigl(\eta(t_n,u)\bigr)-\cJ\bigl(\eta(t_n^1,u)\bigr)\Bigr)
	\]
	and hence $\kappa_n\to 0$. So we see that there exists $u_n:= \eta(s_n^1,u)$, $s_n^1\in[t_n,t_n^1]$, such that  $R\le \|u_n\|\le R+1$ and $\cJ'(u_n)\to 0$. Thus we have found a Cerami sequence in $X^+\setminus\overline{B}(0,M_0^\beta)$ which is impossible. This shows that case (ii) can never occur.
	
	(iii) There exist $R_0>0$ and $\delta>0$ such that $\|\cJ'(v)\|\ge \delta/\|v\|$ whenever $\|v\|\ge R_0$ and $v\in \cJ_0^\beta$, because otherwise there exists an unbounded Cerami sequence. Choose $t_0>0$ so that $\|\eta(t,u)\|\ge R_0$ and $\cJ\bigl(\eta(t_0,u)\bigr)-\cJ\bigl(\eta(t,u)\bigr) \le \delta/8$ for $t\ge t_0$. For $n\gg1$ let $t_n\ge t_0$ be the smallest $t$ such that $\|\eta(t,u)\|=n$ and let $\kappa_n := \min\Set{\lVert\cJ'\bigl(\eta(t,u)\bigr)\rVert|t\in[t_0,t_n]}$. From the choice of $t_n$,
	\[
	\kappa_n \ge \min_{t\in [t_0,t_n]} \frac{\delta}{\|\eta(t,u)\|} = \frac{\delta}{\|\eta(t_n,u)\|}.
	\]
	It follows from the same argument as above that for $n\gg1$
	\begin{equation*}\begin{split}
		\frac12 \|\eta(t_n,u)\| & \le \|\eta(t_n,u)-\eta(t_0,u)\| \le \frac2{\kappa_n}\Bigl(\cJ\bigl(\eta(t_0,u)\bigr)-\cJ\bigl(\eta(t_n,u)\bigr)\Bigr) \\
		& \le \frac{2}{\delta}\|\eta(t_n,u)\|\Bigl(\cJ\bigl(\eta(t_0,u)\bigr)-\cJ\bigl(\eta(t_n,u)\bigr)\Bigr)\le\frac14\|\eta(t_n,u)\|.
	\end{split}\end{equation*}
	This is a contradiction and hence also case (iii) can be ruled out. 
\end{proof}

Let $\cA := \Set{A\subset X^+|A=-A \text{ and } A \text{ is compact} }$,
\[
\cH := \Set{h\colon X^+\to X^+ \text{ odd homeomorphism } |\cJ\bigl(h(u)\bigr)\le \cJ(u) \, \forall u\in X^+},
\]
and for $A\in\cA$ set
\[
i^*(A) := \min_{h\in\cH} \gamma\bigl(h(A)\cap S_r\bigr)
\]
where $r$ is as in (I6) and $\gamma$ is the Krasnosel'skii genus. This is a variant of Benci's  pseudoindex \cite{bbf,be} and the following properties are adapted from \cite[Lemma 2.16]{SquassinaSzulkin}.

\begin{Lem}\label{L:index}
	Let $A,B\in\cA$.
	\begin{itemize}
	\item [(i)] If $A\subset B$, then $i^*(A)\le i^*(B)$.
	\item [(ii)] $i^*(A\cup B)\le i^*(A)+\gamma(B)$.
	\item [(iii)] If $g\in \cH$, then $i^*(A)\le i^*\bigl(g(A)\bigr)$.
	\item [(iv)] For every $k$-dimensional subspace $X_k\subset X^+$ there exists $R_0>0$ such that $i^*(X_k\cap\overline B_R)\ge k$ if $R\ge R_0$.
\end{itemize}
\end{Lem}

\begin{proof}
	\textit{(i)} It follows immediately from the properties of the Krasnosel'skii genus.
	
	\textit{(ii)} For every $h\in\cH$
	\[\begin{split}
	i^*(A\cup B) & \le \gamma\bigl(h(A\cup B)\cap S_r\bigr) = \gamma\Bigl(\bigl(h(A)\cup h(B)\bigr)\cap S_r\Bigr)\\
	& \le \gamma\bigl(h(A)\cap S_r\bigr) + \gamma\bigl(h(B)\bigr) = \gamma\bigl(h(A)\cap S_r\bigr) + \gamma(B)
	\end{split}\]
	where in the last equality we used that $h$ is a homeomorphism. Taking the minimum over all $h\in\cH$ on the right-hand side we obtain the conclusion.
	
	\textit{(iii)} Since $\cJ\bigl(g(u)\bigr)\le\cJ(u)$ for all $u\in X^+$, $h\circ g\in\cH$ if $h\in\cH$. Hence $\Set{h\circ g|h\in\cH}\subset \cH$ and therefore
	\[
	\min_{h\in\cH}\gamma\bigl(h(A)\cap S_r\bigr) \le \min_{h\in\cH}\gamma\bigl((h\circ g)(A)\cap S_r\bigr).
	\]
	
	\textit{(iv)} Since the statement is obviously true for $k=0$, we can assume that $k\ge1$. From Lemma \ref{L:infty}, $\cJ(u)<0$ on $X_k\setminus B_R$ if $R$ is sufficiently large. Let $D := X_k\cap\overline B_R$. Suppose by contradiction that $i^*(D)<k$ and choose $h\in\cH$ such that $\gamma\bigl(h(D)\cap S_r\bigr)<k$.
	
	If $k=1$, then $\gamma\bigl(h(D)\cap S_r\bigr)=0$, i.e., $h(D)\cap S_r=\emptyset$. Therefore either $\|h(u)\|<r$ for every $u\in D$ or $\|h(u)\|>r$ for every $u\in D$. Since $h(0)=0$, the latter is ruled out. If $u\in\partial D=X_k\cap S_R$, recalling that $\cJ(v)\ge 0$ for every $v\in B_r$, we then have $0\le\cJ\bigl(h(u)\bigr)\le\cJ(u)<0$, a contradiction.
	
	If $k\ge2$, fix an odd mapping 
	$$f\colon h(D)\cap S_r\to \R^{k-1}\setminus\{0\}.$$ 
	Let $U:= h^{-1}(B_r\cap X^+)\cap X_k$. There follows that $U\subset D\setminus\partial D$ and hence $U$ is an open and bounded neighbourhood of 0 in $X_k$. If $u\in \partial U$, then $h(u)\in S_r$ and therefore $f\circ h|_{\partial U}\colon \partial U \to \R^{k-1}\setminus\{0\}$, contradicting the Borsuk-Ulam theorem \cite[Proposition II.5.2]{Struwe}, \cite[Theorem D.17]{Willem}.  
\end{proof}

\begin{proof}[Proof of Theorem \ref{T:CrticMulti}]
	\textit{(a)} Suppose that $\cJ$ does not have any critical values in $[c_{\cM},c_{\cM}+\varepsilon_0]$ for some $\varepsilon_0\in\,]0,\varepsilon]$. Thus $\cJ$ has only the trivial critical point 0 in $\cJ^{c_{\cM}+\varepsilon_0}$.
	Take $u\in\cJ^{c_{\cM}+\varepsilon_0}$ and
	observe that, from Lemma \ref{L:flow}, either $\lim_{t \to T(u)^-}\eta(t,u)=0$ or $\lim_{t\to T(u)^-}\cJ\bigl(\eta(t,u)\bigr) = -\infty$. 
	Hence we may define the entrance time map $e:\cJ^{c_{\cM}+\varepsilon_0}\to [0,\infty[$ by the formula
	\[
	e(u):=\inf\Set{t\in [0,T(u)[|  \cJ\bigl(\eta(t,u)\bigr)\leq c_{\cM}/2}.
	\]
	Let us prove that $e$ is continuous. Take $u_n,\bar u\in\cJ^{c_{\cM}+\varepsilon_0}$ such that $u_n\to\bar u$: if by contradiction $e(u_n)\to\infty$ up to a subsequence, then, recalling that $\cJ$ is decreasing along the trajectories of $\eta$ and fixing $t\in]e(\bar{u}),T(\bar{u})[$,
	\[\begin{split}
	\frac{c_{\cM}}{2} & = \lim_n\cJ\bigl(\eta\bigl(e(u_n),u_n\bigr)\bigr) \le \lim_n\cJ\bigl(\eta(t,u_n)\bigr)\\
	& = \cJ\bigl(\eta(t,\bar{u})\bigr) < \cJ\bigl(\eta\bigl(e(u),u\bigr)\bigr) = \frac{c_{\cM}}{2};
	\end{split}\]
	therefore $e(u_n)\to t$ up to a subsequence for some $t\ge0$ and so
	\[
	\cJ\bigl(\eta(t,\bar u)\bigr) = \lim_n\cJ\bigl(\eta\bigl(e(u_n),u_n\bigr)\bigr) = \frac{c_{\cM}}{2} = \cJ\bigl(\eta\bigl(e(\bar u),\bar u\bigr)\bigr),
	\]
	i.e., $t=e(\bar u)$. Now take any $\sigma\in\Sigma$ such that 
	\[
	\cJ\bigl(\sigma(t)\bigr)<c_{\cM}+\varepsilon_0 \quad \text{for all } t\in[0,1],
	\]
	where $\Sigma$ is given by \eqref{e-Gamma_Phi}. Since $e$ is continuous,
	$\tilde{\sigma}(t):=\eta\big(e\bigl(\sigma(t)\bigr),\sigma(t)\big)$ is a continuous path in $X^+$ such that $\cJ\bigl(\tilde{\sigma}(1)\bigr)\leq \cJ\bigl(\sigma(1)\bigr)<0$. Hence $\tilde{\sigma}\in\Sigma$ and
	\[
	c_{\cM}=\inf_{\sigma'\in\Sigma}\sup_{t\in [0,1]}\cJ\bigl(\sigma'(t)\bigr)\leq \sup_{t\in [0,1]}\cJ\bigl(\tilde{\sigma}(t)\bigr)\leq c_{\cM}/2.
	\]
	The obtained contradiction proves that either $c_{\cM}$ is a critical value or for every $\varepsilon_0\in\,]0,\varepsilon]$ there exists a critical value in $]c_{\cM},c_{\cM}+\varepsilon_0]$.
	
	\textit{(b)} Take $\beta\geq a$ and let
	\begin{equation*}
		\cK^\beta:=\Set{u\in\cK|\cJ(u)=\beta}.
	\end{equation*}
	Since there are finitely many critical orbits, there exists $\varepsilon_0>0$ such that
	\begin{equation*}
	\cK\cap \cJ_{\beta-\varepsilon_0}^{\beta+\varepsilon_0}=\cK^\beta.
	\end{equation*} 
	In view of Lemma \ref{L:discrete}, there exists $\delta\in \,]0,m_0^{\beta+\varepsilon_0}[$ such that $\overline B(u,\delta)\cap \overline B(v,\delta) = \emptyset$ for every $u,v\in\cK^\beta$, $u\ne v$. We show that there exists $\varepsilon\in\,]0,\varepsilon_0[$ such that
	\begin{equation}\label{e-entrancetime1}
	\lim_{t\to T(u)^-}\cJ\bigl(\eta(t,u)\bigr) < \beta -\varepsilon \quad\hbox{for every } u\in \cJ^{\beta+\varepsilon}_{\beta-\varepsilon}\setminus B(\cK^\beta,\delta).
	\end{equation}
	We assume $\cK^\beta\ne\emptyset$, the other case being trivial.
	Consider the set
	\[
	A_0:=\Set{u\in \cJ^{\beta+\varepsilon_0}_{\beta-\varepsilon_0}\setminus B(\cK^\beta,\delta)| \lim_{t \to T(u)^-}\eta(t,u)\in \cK^\beta}.
	\]
	For every $u\in A_0$ we define
	\begin{equation*}\begin{split}
		t_0(u)&:=\inf\Set{t\in [0,T(u)[|  \eta(s,u)\in B(\cK^\beta,\delta)\hbox{ for all }s> t},\\
		t(u)&:=\inf\Set{t\in [t_0(u),T(u)[|  \eta(t,u)\in B(\cK^\beta,\delta/2)}
	\end{split}\end{equation*}
	and note that $0\leq t_0(u)< t(u)<T(u)$. There holds
	\begin{equation}\label{e-t}
	\begin{split}
	\frac{\delta}{2} & \leq
	\big\|\eta\bigl(t_0(u),u\bigr)-\eta\bigl(t(u),u\bigr)\big\|  \leq \int_{t_0(u)}^{t(u)}\big\|V\bigl(\eta(s,u)\bigr)\big\|\, ds\\
	& \leq t(u)-t_0(u). 
	\end{split}\end{equation}
	Let
	$$\rho:=\inf\Set{\big\lVert\cJ'\bigl(\eta(t,u)\bigr)\big\rVert | u\in A_0,\ t\in[t_0(u),t(u)]}.$$
	If $\rho=0$ then we find $u_n\in A_0 $ and $t_n\in\,]t_0(u_n),t(u_n)[$ such that $$\cJ'\bigl(\eta(t_n,u_n)\bigr)\to 0 \hbox{ as } n\to\infty.$$
	Since $t_n> t_0(u_n)$, we have $\eta(t_n,u_n)\in B(\cK^\beta,\delta)$ and passing to a subsequence we can find $u_0\in\cK^\beta$ and $g_n\in G$ such that
	$$g_n\eta(t_n,u_n)\in B(u_0,\delta).$$ Since $t_n<t(u_n)$, we see that
	$$g_{n}\eta(t_n,u_n)\notin B(\cK^\beta,\delta/2).$$
	Let $w_n:=u_0$, $v_n:=g_n\eta(t_n,u_n)$. Then $w_n$ and $v_n$ are two Cerami sequences such that $\delta/2 \le \|v_n-w_n\| \le \delta<m_0^{\beta+\varepsilon_0}$, a contradiction.
	Therefore $\rho>0$ and we take 
	$$\varepsilon < \min\Big\{\varepsilon_0,\frac{\delta\rho}8\Big\}.$$ 
	Now let $u\in \cJ_{\beta-\varepsilon}^{\beta+\varepsilon}\setminus B(\cK^\beta,\delta)$ and suppose by contradiction that
	\begin{equation}\label{e-contrA0}
	\lim_{t\to T(u)^-}\cJ\bigl(\eta(t,u)\bigr)\ge\beta-\varepsilon,
	\end{equation}
	which owing to Lemma \ref{L:flow} yields $\lim_{t \to T(u)^-}\eta(t,u)\in\cK^\beta$, i.e., $u\in A_0$. Since
	\begin{equation*}\begin{split}
		\cJ\Bigl(\eta\bigl(t(u),u\bigr)\Bigr)-\cJ\Bigl(\eta\bigl(t_0(u),u\bigr)\Bigr) & =-\int_{t_0(u)}^{t(u)}\cJ'\bigl(\eta(s,u)\bigr)\Bigl(V\bigl(\eta(s,u)\bigl)\Bigr)\,ds\\
		& \leq -\frac12\int_{t_0(u)}^{t(u)}\big\|\cJ'\bigl(\eta(s,u)\bigr)\big\|\, ds,
	\end{split}\end{equation*}
	we obtain using \eqref{e-t}
	\begin{equation*}\begin{split}
		\lim_{t\to T(u)^-}\cJ\bigl(\eta(t,u)\bigr)&\leq \cJ\Bigl(\eta\bigl(t(u),u\bigr)\Bigr)
		\leq\beta +\varepsilon -\frac12\int_{t_0(u)}^{t(u)}\big\|\cJ'\bigl(\eta(s,u)\bigr)\big\|\, ds\\
		&\leq \beta+\varepsilon -\frac{\delta\rho}4 < \beta-\varepsilon.
	\end{split}\end{equation*}
	This contradicts \eqref{e-contrA0}, hence \eqref{e-entrancetime1} holds.
	
	Define 
	\[
	\beta_k :=  \inf_{A\in\cA,\,i^*(A)\ge k} \max_{u\in A}\cJ(u), \quad k=1,2,\dots
	\]
	and note that from Lemma \ref{L:index} all the $\beta_k$ are well defined, finite, and $a\le\beta_1\le\beta_2\le\dots$.
	Let $\beta=\beta_k$  for some $k\geq 1$. If the set $\cK^\beta$ is nonempty, then it is discrete and we can order its elements in pairs $\pm u_j$ and let the map $\phi\colon\cK^\beta\to\R\setminus\{0\}$ be given by $\phi(\pm u_j)=\pm 1$. From the choice of $\delta$ we can extend $\phi$ to $\overline B(\cK^\beta,\delta)$, whence
	$$\gamma\bigl(\overline B(\cK^\beta,\delta)\bigr)=\gamma (\cK^\beta)\le1.$$
	Choose $\varepsilon>0$ such that \eqref{e-entrancetime1} holds. 
	Take Lipschitz continuous cutoff functions $\chi,\xi$ such that $\chi=0$ in $B(\cK^\beta,\delta/4)$, $\chi=1$ in $X^+\setminus B(\cK^\beta,\delta/2)$, $\xi=1$ in $\cJ_{\beta-\varepsilon}^{\beta+\varepsilon}$, and $\xi=0$ in $X^+\setminus U$ , where $U$ is an open neighbourhood of $\cJ_{\beta-\varepsilon}^{\beta+\varepsilon}$ with $\cK\cap U=\cK^\beta$.
	Let
	$\widetilde\eta:\R\times X^+\to  X^+$ be the flow given by
	\begin{equation*}
	\left\{
	\begin{aligned}
	&\partial_t \widetilde\eta(t,u)=-\chi\bigl(\widetilde\eta(t,u)\bigr)\xi\bigl(\widetilde\eta(t,u)\bigr) V\bigl(\widetilde\eta(t,u)\bigr)\\
	&\widetilde\eta(0,u)=u.
	\end{aligned}
	\right.
	\end{equation*}
	Then $\widetilde\eta(t,u)=\eta(t,u)$ as long as $t\ge 0$ and $\widetilde\eta(t,u)\in \cJ^{\beta+\varepsilon}_{\beta-\varepsilon}\setminus B(\cK^\beta,\delta/2)$. Using  \eqref{e-entrancetime1} we can define the entrance time map $e\colon\cJ^{\beta+\varepsilon}\setminus B(\cK^\beta,\delta)\to [0,\infty[$ as
	\[
	e(u):=\inf\Set{t\in [0,\infty[\,|  \cJ\bigl(\widetilde\eta(t,u)\bigr)\leq \beta -\varepsilon}.
	\]
	Arguing as before we can prove that $e$ is continuous; moreover, $e$ is clearly even. Take any $A\in\cA$ such that $i^*(A)\geq k$ and $\cJ(u)\leq \beta+\varepsilon$ for $u\in A$. Let $T:=\max_{u\in A} e(u)<\infty$ ($A$ is compact). Set $h:=\widetilde\eta(T,\cdot)$ and note that $h\in\cH$ and
	\[
	h\bigl(A\setminus B(\cK^\beta,\delta)\bigr) \subset \cJ^{\beta-\varepsilon}.
	\]
	Therefore
	\[
	i^*\bigl(A\setminus B(\cK^\beta,\delta)\bigr)\leq 
	i^*\Bigl(h\bigl(A\setminus B(\cK^\beta,\delta)\bigr)\Bigr)\leq  k-1
	\]
	and
	\begin{equation}\label{e-LSvaluse}
	k\leq i^*(A)\leq \gamma \bigl(\overline B(\cK^\beta,\delta)\cap A\bigr)+
	i^*\bigl(A\setminus B(\cK^\beta,\delta)\bigr)
	\leq \gamma(\cK^\beta)+k-1.
	\end{equation}
	Thus $\cK^\beta\neq \emptyset$ and, since $\gamma(\cK^\beta)\le1$, we conclude $\gamma(\cK^\beta)=1$.	If $\beta_k=\beta_{k+1}$ for some
	$k\geq 1$, then \eqref{e-LSvaluse} implies $\gamma(\cK^{\beta_k})\geq 2$, a contradiction. Hence we get an infinite sequence
	$\beta_1<\beta_2<\cdots$
	of critical values which contradicts our assumption that $\cK$ consists of a finite number of distinct orbits.
\end{proof}

\section{Properties of the energy functional}

We recall that (N1)--(N3) and (F1)--(F5) hold, that $\Phi$ is the same $N$-function as in (F3) and that $\Psi$ is its complementary function. We will check that assumptions (I1)--(I8) are satisfied in order to apply Theorems \ref{T:Link1} and \ref{T:CrticMulti}.

Define
\begin{equation*}
\mathcal{M} := \Set{(v,w)\in\cV\times\cW|
J'(v,w)(0,\psi)=0\,\hbox{ for every }\psi\in \cW}
\end{equation*}
and the Nehari--Pankov manifold for $J$
\begin{equation*}\begin{split}
\mathcal{N} := \Set{(v,w)\in\cM|u\neq 0 \text{ and }J'(v,w)(v,w)=0}.
\end{split}\end{equation*}
Observe that $\UU=v+w\in\fN$ if and only if $(v,w)\in\cN$. Moreover, $\mathcal{N}$ contains all the nontrivial critical points of $J$.

\begin{Prop}\label{P:uv_N}
	If $(\bar v,\bar w)\in \cV\times\cW$ then
	$$J(t\bar  v,t\bar w+\psi)-J'(\bar v,\bar w)\left(\frac{t^2-1}{2}\bar v,\frac{t^2-1}{2}\bar w+t\psi\right)\leq J(\bar v,\bar w)$$
	for every $\psi\in\cW$ and $t\geq 0$. 
\end{Prop}
\begin{proof}
	Let $(\bar v,\bar w)\in\cV\times\cW$, $\psi\in\cW$, $t\geq 0$. We define
	$$D(t,\psi):=J(t \bar v,t\bar w+\psi)-J(\bar v,\bar w)-J'(\bar v,\bar w)\left(\frac{t^2-1}{2}\bar v,\frac{t^2-1}{2}\bar w+t\psi\right)$$
	and observe that
	\begin{equation*}\begin{split}
		D(t,\psi)=\int_{\R^3} f(x,\bar v+\bar w)\cdot\left(\frac{t^2-1}{2}(\bar v+\bar w)+t\psi\right)\\
		+F(x,\bar v+\bar w)-F\bigl(x,t(\bar v+\bar w)+\psi\bigr)\,dx.
	\end{split}\end{equation*}
	For fixed $x,v,w\in\R^3$, define the map $\phi\colon[0,\infty[\,\times\,\R^3\to\R$ as follows:
	\[
	\phi(t,\psi):=
	f(x,v+w)\cdot\left(\frac{t^2-1}{2}(v+w)+t\psi\right)+F(x,v+w)-F\bigl(x,t(v+w)+\psi\bigr).
	\]
	We prove that $\phi(t,\psi)\leq 0$ for all $t\geq 0$ and all $\psi\in\R^3$. This is clear if $v+w=0$, thus let $v+w\ne 0$ and define $\zeta := t(v+w)+\psi$. 
	From (F4) we have
	\begin{equation*}\begin{split}
		\phi(t,\psi) \le \, & f(x,v+w)\cdot\left(\frac{t^2-1}{2}(v+w)+t\bigl(\zeta-t(v+w)\bigr)\right)\\
		& +\frac12 f(x,v+w)\cdot(v+w)-F(x,\zeta) \\
		= \, & -\frac12t^2 f(x,v+w)\cdot(v+w) + t f(x,v+w)\cdot\zeta - A|\zeta|^2\\
		& + A|\zeta|^2-F(x,\zeta).
	\end{split}\end{equation*}
	If $A>0$ is large enough, then the quadratic form (in $t$ and $\zeta$) above is negative definite. Moreover, $A|\zeta|^2-F(x,\zeta)$ is bounded from above owing to the superquadraticity of $F$ implied by (F3) and (N3). Hence $\phi(t,\psi)\to -\infty$ as $t+|\psi|\to\infty$ and $\phi$ attains its maximum at some $(t,\psi)$ with $t\geq 0$.  If $t=0$, then $\phi(t,\psi)=\phi(0,\psi)\leq 0$. If $t>0$, then
	\begin{equation*}\begin{split}
	\partial_t\phi(t,\psi) & = f(x,v+w)\cdot\bigl(t(v+w)+\psi\bigr) - f\bigl(x,t(v+w)+\psi\bigr)\cdot(v+w) = 0, \\
	\nabla_\psi\phi(t,\psi) & = tf(x,v+w) - f(x,t(v+w)+\psi) = 0.
	\end{split}\end{equation*} 
	Using the latter equation in the former, we see that both terms in the former are positive (because $f(x,v+w)\cdot(v+w)>0$ from (F3 and (F4)) and $f(x,v+w)\cdot\psi=0$. This and (F5) imply
	\begin{equation*}\begin{aligned}
	\phi(t,\psi) & = \frac{t^2-1}{2}f(x,v+w)\cdot(v+w)+F(x,v+w)-F\bigl(x,t(v+w)+\psi\bigr)\\
	& \le 0.\qedhere
	\end{aligned}\end{equation*}
\end{proof}

Consider $\fI\colon L^{\Phi}\to \R$ and  $I\colon L^{\Phi}\times\cW\to\R$ given by
\begin{equation}\label{e-DefOfXi}
I(v,w):=\fI(v+w):=\int_{\R^3}F(x,v+w)\, dx\quad\hbox{ for }(v,w)\in L^{\Phi}\times\cW.
\end{equation}
$\fI$ and $I$ are of class $\cC^1$ and strictly convex in view of Proposition \ref{P:classC1} and (F2) respectively. We need a variant for sequences in $L^{\Phi}$ of the Brezis-Lieb lemma \cite{BrezisLieb}.

\begin{Lem}\label{L:BrezLieb}
	Let $\UU_n\in L^{\Phi}$ bounded such that
	$\UU_n\to \UU$ a.e. in $\R^3$.
	Then
	\[
	\lim_n\int_{\R^3}F(x,\UU_n)-F(x,\UU_n-\UU)\, dx=\int_{\R^3}F(x,\UU)\, dx.
	\]
\end{Lem}
\begin{proof}
	Note that
	\begin{equation*}\begin{split}
		\int_{\R^3}F(x,\UU_n)-F(x,\UU_n-\UU)\, dx
		&=\int_{\R^3}\int_0^1\frac{d}{dt}F(x,\UU_n-\UU+t\UU)\, dtdx\\
		&=\int_0^1\int_{\R^3}f(x,\UU_n-\UU+t\UU)\cdot\UU\, dxdt
	\end{split}\end{equation*}
	and $f(x,\UU_n-\UU+t\UU)$ is bounded
	in $L^{\Psi}$ owing to (F3) and Lemmas \ref{L:all} \textit{(iv)} and \ref{L:forC1}.
	Thus for every $\Omega\subset\R^3$
	\begin{eqnarray}\label{e-*}
	\int_{\Omega}|f(x,\UU_n-\UU+t\UU)\cdot\UU|\,dx
	\leq |f(\cdot,\UU_n-\UU+t\UU)|_{\Psi}
	|\UU\chi_{\Omega}|_{\Phi}.
	\end{eqnarray}
	From \cite[Definition III.IV.2, Corollary III.IV.5 and Theorem III.IV.14]{RaoRen}, the space $L^\Phi$ has an absolutely continuous norm (cf. \cite[Definition III.I.13]{RaoRen}), so \eqref{e-*} yields that for every $\varepsilon>0$ there exists $\delta>0$
	such that, if $|\Omega|<\delta$, then
	\[
	\int_{\Omega}|f(x,\UU_n-\UU+t\UU)\cdot\UU|\,dx<\varepsilon
	\]
	for every $n$, i.e., $f(x,\UU_n-\UU+t\UU)\cdot\UU$ is uniformly integrable.
	Using \eqref{e-*} once more we see that for every $\varepsilon>0$ there exists $\Omega\subset\R^3$ with 
	$|\Omega|<\infty$ such that
	\begin{equation}\label{e-tight}
	\int_{\rr\setminus\Omega}f(x,\UU_n-\UU+t\UU)\cdot\UU\,dx<\varepsilon.
	\end{equation}
	As a matter of fact, since $\Phi\circ|\UU|\in L^1(\rr)$, there exists a sequence $\Omega_n\subset\rr$ such that each $\Omega_n$ has finite measure and
	\[
	\lim_n\int_{\rr}\Phi(|\UU|\chi_{\rr\setminus\Omega_n})\,dx = \lim_n\int_{\rr}\Phi(|\UU|)\chi_{\rr\setminus\Omega_n}\,dx = 0,
	\]
	where in the first equality we used that $\Phi(t)=0\Leftrightarrow t=0$. In view of (N1) and Lemma \ref{L:all} \textit{(iii)}, this yelds $\lim_n|\UU\chi_{\rr\setminus\Omega_n}|_\Phi=0$. If we take $\Omega = \Omega_{\bar{n}}$ for some sufficiently large $\bar{n}$ such that $|\UU\chi_{\rr\setminus\Omega}|_{\Phi} < \varepsilon(\sup_k|f(\cdot,\UU_k-\UU+t\UU)|_{\Psi})^{-1}$, then \eqref{e-tight} holds.
	Since $\UU_n\to\UU$ a.e. on $\R^3$, it follows from Vitali's convergence theorem that
	\[\begin{aligned}
	\int_{\R^3}F(x,\UU_n)-F(x,\UU_n-\UU)\,dx & \to
	\int_0^1\int_{\R^3}f(x,t\UU)\cdot\UU\,dxdt\\
	& =\int_{\R^3}F(x,\UU)\,dx.\qedhere
	\end{aligned}\]
\end{proof}

Now we can prove the following result.

\begin{Lem}\label{L:WeakStrong}
	If $\UU_n\rightharpoonup\UU$ in $L^{\Phi}$ and
	$\fI(\UU_n)\to \fI(\UU)$,
	then $\UU_n\to\UU$ in $L^{\Phi}$.
\end{Lem}
\begin{proof}
	We show that $\UU_n\to\UU$ a.e. in $\R^3$ up to a subsequence.
	Since $\fI(\UU_n)\to \fI(\UU)$, we have
	\begin{equation} \label{e-i3}
	\lim_n\int_{\R^3}F(x,\UU_n)\,dx=\int_{\R^3}F(x,\UU)\,dx.
	\end{equation}
	Then from (F2) we infer that for every $0<r\leq R$
	\begin{equation} \label{e-mrR}
	m_{r,R}:=\inf_{\substack{x,u_1,u_2\in\R^3\\ r\leq|u_1-u_2|\\|u_1|,|u_2|\leq R} }	\frac{1}{2}\bigl(F(x,u_1)+F(x,u_2)\bigr)-F\left(x,\frac{u_1+u_2}{2}\right)>0.
	\end{equation} 
	Observe that from \eqref{e-i3}  and (F2) there holds
	$$
	0
	\leq \limsup_n\int_{\R^3}\frac12\bigl(F(x,\UU_n) + F(x,\UU)\bigr)
	- F\left(x,\frac{\UU_n+\UU}{2}\right)\,dx
	\leq 0.
	$$
	Therefore, setting
	$$
	\Omega_n:=\Set{x\in\R^3|\left|\UU_n-\UU\right|\geq r \text{ and } |\UU_n|,|\UU|\leq R},
	$$
	from \eqref{e-mrR} we have
	$$
	|\Omega_n| m_{r,R}
	\leq \int_{\R^3}\frac12\bigl(F(x,\UU_n) + F(x,\UU)\bigr)
	- F\left(x,\frac{\UU_n+\UU}{2}\right)\, dx,
	$$
	thus $|\Omega_n|\to 0$ as $n\to\infty$. Since $0<r\leq R$ are arbitrarily chosen, we deduce
	$$
	\UU_n\to\UU\hbox{ a.e. in }\R^3.
	$$
	In view of Lemma \ref{L:BrezLieb}, we obtain
	\[
	\int_{\R^3}F(x,\UU_n)\, dx-\int_{\R^3}F(x,\UU_n-\UU)\, dx\to\int_{\R^3}F(x,\UU)\, dx
	\]
	and hence
	$$\int_{\R^3}F(x,\UU_n-\UU)\, dx\to0.$$
	From (F3) and Lemma \ref{L:all} \textit{(iii)}, we get $|\UU_n-\UU|_\Phi\to 0$.
\end{proof}

\begin{Prop}\label{P:DefOfm(u)} Conditions  (I1)--(I8) are satisfied and
	there is a Cerami sequence $(v_n,w_n)\in\cM$ at level $c_{\cN}$, 
	where
	$$c_{\cN}:=\inf_{(v,w)\in\mathcal{N}}J(v,w)>0.$$
\end{Prop}
\begin{proof}
	Setting $X:=\cV\times\cW$,
	$X^+:=\cV\times\{0\}$, and $X^-:=\{0\}\times\cV$ we check the assumptions (I1)--(I8) for the functional $J\colon X\to\R$ 
	given by
	\[
	J(v,w)=\frac{1}{2}\|v\|^2_\cD-I(v,w)
	\]
	(cf. \eqref{e-J} and \eqref{e-DefOfXi}). Recall that $\|(v,w)\|^2=\|v\|_{\cD}^2+|w|_{\Phi}^2$. The convexity and differentiability of $I$, (F3), and Lemma \ref{L:WeakStrong} yield the following.
	\begin{itemize}
		\item[(I1)] $I|_{\cV\times\cW}\in \cC^1(\cV\times\cW,\R)$ and $I(v,w)\geq I(0,0)=0$ for every $(v,w)\in \cV\times\cW$.
		\item[(I2)] If $v_n\to v$ in $\cV$ and $w_n\rightharpoonup w$ in $\cW$, then $\displaystyle\liminf_nI(v_n,w_n)\geq I(v,w)$.
		\item[(I3)] If $v_n\to v$ in $\cV$, $w_n\rightharpoonup w$ in $\cW$, and $I(v_n,w_n)\to I(v,w)$, then $(v_n,w_n)\to (u,w)$.
	\end{itemize}

	Moreover,
	\begin{itemize}
		\item[(I6)] There exists $r>0$ such that $\inf_{\|v\|_\cD=r}J(v,0)>0$.
	\end{itemize}
	As a matter of fact, from (F3) and (N2) there exist $C>0$ (cf. the proof of Lemma \ref{L:emb}) such that for any $v\in\cV$
	$$J(v,0)=\|v\|_\cD^2 -\int_{\R^3}F(x,v)\,dx
	\geq\|v\|_\cD^2-C\int_{\R^3}|v|^6\,dx
	\geq\|v\|_\cD^2-C\|v\|_\cD^6,$$
	thus (I6) is satisfied taking $r$ sufficiently small. Note that
	\begin{itemize}
		\item[(I4)] $\|v\|_\cD+I(v,w)\to\infty$ as $\|(v,w)\|\to\infty$.
	\end{itemize}
	To see this, let $\|(v_n,w_n)\|\to\infty$. If $\|v_n\|_\cD\to\infty$, then we are done; otherwise, up to a subsequence $\|v_n\|_\cD$ is bounded and, a fortiori, so is $|v_n|_\Phi$. This yields $|v_n+w_n|_\Phi\ge|w_n|_\Phi-|v_n|_\Phi\to\infty$. Using (F3) and Lemma \ref{L:all} \textit{(iv)} we have
	\[
	I(v_n,w_n)\ge c_2\int_{\rr}\Phi(|v_n+w_n|)\,dx\to\infty.
	\]
	
	This, together with the uniform strict convexity of $F$, implies
	\begin{itemize}
		\item[(I5)] If $(v,w)\in\cM$, then $I(v,w)<I(v,w+\psi)$ for every $\psi\in \cW\setminus\{0\}$.
	\end{itemize}

	Next, we prove 
	\begin{itemize}
		\item[(I7)] $I\bigl(t_n(v_n,w_n)\bigr)/t_n^2\to\infty$ if $t_n\to\infty$ and $v_n\to v$ for some $v\neq 0$.
	\end{itemize}
	Observe that from (F3)
	\begin{equation} \label{e-I7check}
	\begin{split}
	\frac{I\bigl(t_n(v_n,w_n)\bigr)}{t_n^2}&=
	\int_{\R^3}\frac{F\bigl(x,t_n(v_n+w_n)\bigr)}{t_n^2}\,dx\\
	&\geq c_2
	\int_{\R^3}\frac{\Phi(t_n|v_n+w_n|)}{t_n^2}\,dx\\ 
	&= c_2
	\int_{\R^3}\frac{\Phi(t_n|v_n+w_n|)}{t_n^2|v_n+w_n|^2}|v_n+w_n|^2\,dx.
	\end{split}
	\end{equation} 
	Take $R_0>0$ such that $v\neq 0$ in $L^2(B_{R_0})$.
	In view of (N3), there exists $C>0$ such that
	$$C\Phi(t)\geq t^2\quad\hbox{for }t\geq 1.$$
	Then, writing
	\[
	B_R=\bigl(B_R\cap\{|v_n+w_n|\ge1\}\bigr)\cup\bigl(B_R\cap\{|v_n+w_n|<1\}\bigr),
	\]
	we have
	\begin{equation}\label{e-estim}
	\int_{B_R}|v_n+w_n|^2\,dx\le C\int_{\R^3}\frac{\Phi(t_n|v_n+w_n|)}{t_n^2}\,dx+|B_R|
	\end{equation}
	and $I\bigl(t_n(v_n,w_n)\bigr)/t_n^2\to\infty$ (up to a subsequence) if $v_n+w_n$ is unbounded in $L^2(B_R,\R^3)$ for some $R\geq R_0$. Now, suppose that $v_n+w_n$ is bounded in $L^2(B_R,\R^3)$ for every $R\geq R_0$. We can assume passing to a subsequence that $v_n\to v$ a.e. and $w_n\rightharpoonup w$ in $L^2_\textup{loc}(\R^3,\R^3)$ for some $w$. Given $\varepsilon>0$, let 
	\begin{equation}\label{e-aetoinfI7}
	\Omega_n := \Set{x\in\R^3|\left|v_n(x)+w_n(x)\right|\ge\varepsilon}.
	\end{equation}
	We claim that there exists $\varepsilon>0$ such that $\limsup_n|\Omega_n|>0$. Arguing by contradiction, suppose $|\Omega_n|\to0$ for every $\varepsilon$. Then $v_n+w_n\to 0$ in measure, so up to a subsequence $v_n+w_n\to 0$ a.e., hence $w_n\to -v$ a.e. and  $w_n\rightharpoonup -v$ in $L^2_\textup{loc}(\R^3,\R^3)$. Since $\nabla\times w_n=0$ in the distributional sense, the same is true of $v$, thus there exists $\xi\in H^1_\textup{loc}(\R^3)$ such that $v=\nabla \xi$ due to \cite[Lemma 1.1 (i)]{Leinfelder}. As $\nabla\cdot(\nabla \xi)=\nabla\cdot v =0$, it follows that $\xi$ is harmonic and so is $v$. Recalling that $v\in\cD$, we obtain $v=0$ as in the proof of Lemma \ref{L:Hel}. This is a contradiction.
	Taking $\varepsilon$ in \eqref{e-aetoinfI7} such that $\limsup_n|\Omega_n|>0$, we obtain
	\[
	\int_{\R^3}\frac{\Phi(t_n|v_n+w_n|)}{t_n^2|v_n+w_n|^2}|v_n+w_n|^2\,dx \ge \int_{\Omega_n}\frac{\Phi(t_n|v_n+w_n|)}{t_n^2|v_n+w_n|^2}|v_n+w_n|^2\,dx \to \infty.
	\]
	
	Finally, Proposition \ref{P:uv_N} shows that
	\begin{itemize}
		\item[(I8)']
		$\displaystyle\frac{t^2-1}{2}I'(v,w)(v,w)+tI'(v,w)(0,\psi)+I(v,w)-I(tv,tw+\psi)\leq 0$
		for every $t\geq 0$, $v\in\cV$, and $w,\psi\in\cW$,
	\end{itemize}
	which is a stronger version of (I8).
	
	Applying Theorem \ref{T:Link1} we conclude.
\end{proof}

Since there is no compact embedding of $\cV$ into $L^{\Phi}$, we cannot expect that the Palais--Smale or Cerami conditions are satisfied. 
We need the following variant of Lions's lemma.

\begin{Lem}\label{L:Conv}
	Suppose that $v_n\in\cD$ is bounded and for some $R>\sqrt{3}$
	\begin{equation}\label{e-LionsCond1}
	\lim_n\sup_{y\in \Z^3}\int_{B(y,R)}|v_n|^2\,dx=0.
	\end{equation}
	Then  
	$$\lim_n\int_{\R^3}\Phi(|v_n|)\,dx=0.$$
\end{Lem}
\begin{proof}
	This follows from \cite[Lemma 1.5]{MederskiZeroMass} since $\Phi$ satisfies (N2). Note that we can take the supremum over $\Z^3$ in \eqref{e-LionsCond1} because the radius $R$ is greater than the length of the diagonal of the unitary cube in $\rr$.
\end{proof}

We collect further properties of $I$.

\begin{Lem}\label{L:DefofW}
	(a) For every $v\in L^{\Phi}$ there exists a unique $w(v)\in \cW$ such that
	\begin{equation}\label{e-DefofW(u)}
	I\bigl(v,w(v)\bigr)=\min_{w\in\cW}I(v,w).
	\end{equation}
	Moreover, $w\colon L^{\Phi}\to\cW$ is continuous.\\
	(b) $w$ maps bounded sets into bounded sets and $w(0)=0$.
\end{Lem}
\begin{proof}
	\textit{(a)} Let $v\in L^{\Phi}$. 
	Since  $w\in\cW\mapsto I(v,w)\in\R$ is continuous, strictly convex, and coercive,  there exists a unique $w(v)\in\cW$ such that (\ref{e-DefofW(u)}) holds.
	We show that the map $w\colon L^{\Phi}\to\cW$ is continuous. Let $v_n\to v$ in $L^{\Phi}$. Since 
	\begin{equation}\label{e-mI}
	0\leq I\bigl(v_n,w(v_n)\bigr)\leq I(v_n,0),
	\end{equation}
	$w(v_n)$ is bounded and we may assume  $w(v_n)\rightharpoonup w_0$ for some $w_0\in\cW$.
	Observe that by the (sequential) lower semicontinuity of $I$ we get
	\[\begin{split}
	I\bigl(v,w(v)\bigr) & \le I(v,w_0)\leq \liminf_nI\bigl(v_n,w(v_n)\bigr)\leq\limsup_n I\bigl(v_n,w(v_n)\bigr)\\
	& \leq \lim_nI\bigl(v_n,w(v)\bigr)=I\bigl(v,w(v)\bigr),
	\end{split}\]
	hence $w(v)=w_0$ from the uniqueness of $w(v)$, and from Lemma \ref{L:WeakStrong} we have $v_n+w(v_n)\to v+w(v)$ in $L^{\Phi}$, which yields $w(v_n)\to w(v)$ in $\cW$.
	
	\textit{(b)} This follows from (\ref{e-mI}), (F3), and Lemma \ref{L:all} \textit{(iv)}.
\end{proof}

Let $m(v):=\bigl(v,w(v)\bigr)\in\cM$ for $v\in\cV$. Then, in view of Lemma \ref{L:DefofW} \textit{(a)}, $m\colon\cV\to\cM$ is continuous. The following lemma implies that every Cerami sequence of $J$ in $\cM$ and every Cerami sequence of $J\circ m$ are bounded.

\begin{Lem}\label{L:Coercive} If $\beta>0$ and
	$v_n\in\cV$ are such that $J\bigl( m(v_n)\bigr)\leq\beta$ and $\lim_n(1+\|v_n\|)(J\circ m)'(v_n)=0$, then $v_n$ is bounded.
\end{Lem}
\begin{proof}
	Let us write $m(v_n)=(v_n,w_n)\in\mathcal{M}$ and suppose by contradiction $\lim_n\|(v_n,w_n)\|=\infty$. Since $w_n=w(v_n)$, $\|(v_n,w_n)\|\to\infty$ if and only if $\|v_n\|_\cD\to\infty$.
	Let $\bar{v}_n:=v_n/\|v_n\|_\cD$ and $\bar{w}_n:=w_n/\|v_n\|_\cD$.
	Assume 
	\[
	\lim_n\sup_{y\in\Z^3}\int_{B(y,R)}|\bar{v}_n|^2\,dx=0
	\]
	for some fixed $R>\sqrt{3}$. From Lemma \ref{L:Conv}, $\lim_n\int_{\R^3}\Phi(|\bar{v}_n|)\,dx=0$ and  we obtain a contradiction: recalling that $J'(v_n,w_n)(0,w_n)=0$,
	Proposition \ref{P:uv_N} with $t_n=s/\|v_n\|_\cD$ and $\psi_n=-t_nw_n$ implies that for every $s>0$
	\begin{equation*}\begin{split}
		\beta & \ge\limsup_nJ(v_n,w_n) \\
		& \ge \limsup_nJ(s\bar{v}_n,0)-\lim_nJ'(v_n,w_n)\left(\frac{t_n^2-1}{2}v_n,-\frac{t_n^2+1}{2}w_n\right)\\
		& =\limsup_nJ(s\bar{v}_n,0)\overset{\text{(F3)}}{\ge}\frac{s^2}{2}-\lim_nc_1\int_{\R^3}\Phi(s|\bar v_n|)\,dx
		=\frac{s^2}{2},
	\end{split}\end{equation*}
	which is impossible. Hence $\liminf_n\int_{B(y_n,R)}|\bar{v}_n|^2\,dx>0$, where $y_n\in\Z^3$ maximizes $y\in\Z^3\mapsto\int_{B(y,R)}|\bar{v}_n|^2\,dx\in\R$. Since $\cM$ and $J$ are invariant with respect to $\mathbb{Z}^3$-translations, we may assume that
	$$\int_{B_R}|\bar{v}_n|^2\,dx\geq c>0$$
	for some constant $c$ and every $n\gg1$. This implies that, up to a subsequence, $\bar{v}_n\rightharpoonup\bar v\ne 0$ in $\cD$, $\bar{v}_n\to\bar{v}$ in $L^2_\textup{loc}(\R^3,\R^3)$ and $\bar{v}_n\to\bar v$ a.e. in $\R^3$ for some $\bar v\in\cD$.	From (F4),
	\[\begin{split}
	2J(v_n,w_n) & - J'(v_n,w_n)(v_n,w_n) = \int_{\R^3}f(x,v_n+w_n)\cdot(v_n+w_n)\,dx\\
	& - 2\int_{\rr}F(x,v_n+w_n)\,dx \ge 0,
	\end{split}\]
	so $J(v_n,w_n)$ is bounded from below and, using (F3),
	$$\alpha \le \frac{J(v_n,w_n)}{\|v_n\|_\cD^2}\le \frac{1}{2}\|\bar{v}_n\|_\cD^2-c_2\int_{\R^3}\frac{\Phi(|v_n+w_n|)}{|v_n+w_n|^2}|\bar{v}_n+\bar{w}_n|^2\,dx$$
	for some $\alpha\in\R$. Hence it suffices to show that the integral on the right-hand side above tends to $\infty$. We can argue as in the proof of (I7) in Proposition \ref{P:DefOfm(u)}. In particular, \eqref{e-estim} holds with $\bar v_n+\bar w_n$ instead of $v_n+w_n$ and $\|v_n\|_\cD$ instead of $t_n$ and, if $\Omega_n$ is as in \eqref{e-aetoinfI7} (again, with $v_n+w_n$ replaced with $\bar v_n+\bar w_n$), then $\lim_n|\Omega_n|>0$ along a subsequence.
\end{proof}

\begin{Cor}\label{C:Ceramibdd}
	Let $\beta>0$. There exists $M_\beta>0$ such that for every $v_n\in\cV$ satisfying $0\le\liminf_nJ\bigl(m(v_n)\bigr)\le\limsup_nJ\bigl(m(v_n)\bigr)\le\beta$ and $\lim_n(1+\|v_n\|)J'\bigl(m(v_n)\bigr)=0$ there holds $\limsup_n\|v_n\|\le M_\beta$.
\end{Cor}
\begin{proof}
	If no finite bound $M_\beta$ exists, for every $k$ there exists a sequence $v_n^k\in X$ satisfying the assumptions above and such that $\limsup_n\|v_n^k\|\ge k$. In particular, there exists $n(k)$ such that $-1/k<J\bigl(m(v_{n(k)}^k)\bigr)<\beta+1/k$, which yields $0\le\liminf_kJ\bigl(m(v_{n(k)}^k)\bigr)\le\limsup_kJ\bigl(m(v_{n(k)}^k)\bigr)\le\beta$. In the same way we prove $\lim_k\bigl(1+\|v_{n(k)}^k\|\bigr)(J\circ m)'\bigl(v_{n(k)}^k\bigr)=0$ and $\limsup_k\|v_{n(k)}^k\|=\infty$, a contradiction with Lemma \ref{L:Coercive}.
\end{proof}

\subsection{Weak-to-weak* continuity}

\begin{Lem}\label{L:CoEmb} If $\Omega\subset\rr$ is a Lipschitz domain with finite measure, then $H^1(\Omega)$ is compactly embedded in $L^\Phi(\Omega)$. 
\end{Lem}
\begin{proof}
	Suppose $u_n\rightharpoonup 0$ in $H^1(\Omega)$. Then $u_n\rightharpoonup 0$ in $L^6(\Omega)$ and $u_n\to 0$ in $L^2(\Omega)$ and, up to a subsequence, a.e. in $\Omega$. From (N2), for every $\varepsilon>0$ there exists $C_\varepsilon$ such that $\Phi(t)\le \varepsilon t^6$ for $t>C_\varepsilon$, whence
	\[\begin{split}
	\int_{\Omega} \Phi(|u_n|)\,dx & = \int_{\Omega\cap\{|u_n|\le C_\varepsilon\}} \Phi(|u_n|)\,dx + \int_{\Omega\cap\{|u_n| > C_\varepsilon\}} \Phi(|u_n|)\,dx\\
	& \le \int_{\Omega\cap\{|u_n|\le C_\varepsilon\}} \Phi(|u_n|)\,dx +\varepsilon\sup_k|u_k|_6^6.
	\end{split}\]
	From the dominated convergence theorem and since $\varepsilon$ is arbitrary, we have $\int_{\Omega} \Phi(|u_n|)\,dx\to 0$ and, due to Lemma \ref{L:all} (iii), $|u_n|_\Phi\to 0$.
\end{proof}

\begin{Prop} \label{P:ae}
	If $v_n\rightharpoonup v$ in $\cD$, then $w(v_n)\rightharpoonup w(v)$ in $\cW$ and, after passing to a subsequence, $w(v_n)\to w(v)$ a.e. in $\R^3$.
\end{Prop}
\begin{proof}
	It follows from the definition \eqref{e-DefofW(u)} of $w(v)$ that
	\begin{equation} \label{e-1}
	\int_{\R^3}f\bigl(x,v_n+w(v_n)\bigr)\cdot z\,dx = 0 = \int_{\R^3}f\bigl(x,v+w(v)\bigr)\cdot z\,dx \quad \forall z \in \cW.
	\end{equation}
	Since $v_n$ is bounded, so is $w(v_n)$ from Lemma \ref{L:DefofW} \textit{(b)}, hence we can assume $w(v_n)\rightharpoonup w_0$ for some $w_0$. In addition, since $v_n\to v$ in $L^2_\textup{loc}(\R^3,\R^3)$, $v_n\to v$ a.e. after passing to a subsequence.
	
	Let $\Omega\subset\R^3$ be bounded and let $\zeta\in \cC_c^\infty(\R^3)$, $0\le\zeta\le1$, such that $\zeta=1$ in $\Omega$. From (F3) and Lemmas \ref{L:all} \textit{(ii)}, \ref{L:forC1}, and \ref{L:CoEmb}, there exist $C>0$ such that
	\begin{equation} \label{e-10}
	\begin{split}0 & \le\int_{\R^3}\left|f\bigl(x,v_n+w(v_n)\bigr)\right|\left|v_n-v\right|\zeta\,dx\\
	& \le C\left|\Phi'\bigl(|v_n+w(v_n)|\bigr)\right|_{\Psi} |(v_n-v)\zeta|_\Phi\to 0.
	\end{split}\end{equation}
	Choose $R>0$ such that $\textup{supp}\zeta\subset B_R$. In view of (N3), $w(v_n)$ is bounded in $L^2(B_R,\R^3)$. As a matter of fact, $w(v_n)\chi_{\{|w(v_n)|<1\}}$ is of course bounded in $L^2(B_R,\R^3)$, while
	\[
	\int_{B_R\cap\{|w(v_n)|\ge 1\}}|w(v_n)|^2\,dx \le C \int_{B_R\cap\{|w(v_n)|\ge 1\}}\Phi(|w(v_n)|)\,dx \le C
	\]
	for some $C>0$. From \cite[Lemma 1.1 \textit{(i)}]{Leinfelder}, there exists $\xi_n\in H^1(B_R)$ such that $w(v_n)=\nabla\xi_n$. We can assume $\int_{B_R}\xi_n\,dx = 0$, so that from the Poincar\'e inequality
	\[
	\|\xi_n\|_{H^1(B_R)}\le C|\nabla\xi_n|_{L^2(B_R)}\le C
	\]
	for some $C>0$. Hence in view of Lemma \ref{L:CoEmb}, up to a subsequence, $\xi_n\to \xi$ in  $L^\Phi(B_R)$  for some $\xi\in H^1(B_R)$. Similarly as in \eqref{e-10}, we have
	\begin{equation} \label{e-11}
	\lim_n\int_{\R^3}\left|f\bigl(x,v_n+w(v_n)\bigr)\right|\left|\nabla\zeta\right|\left|\xi_n-\xi\right|\,dx=0.
	\end{equation}
	The limits in \eqref{e-10} and \eqref{e-11} are 0 also if $f\bigl(x,v_n+w(v_n)\bigr)$ is replaced with $f(x,v+\nabla\xi)$. Combining \eqref{e-1}--\eqref{e-11} we obtain
	\begin{equation} \label{e-2}
	\int_{\R^3}\Bigl(f\bigl(x,v_n+w(v_n)\bigr)-f(x,v+\nabla\xi)\Bigr)\cdot\bigl(v_n-v+w(v_n)-\nabla\xi\bigr)\zeta\,dx\to0,
	\end{equation}
	where we have taken $z = \nabla\bigl((\xi_n-\xi)\zeta\bigr)$ in \eqref{e-1}. We prove that  $v_n+w(v_n) \to v+ \nabla\xi$ a.e. in $\Omega$. The convexity of $F$ implies that 
	\[
	F\left(x,\frac{\UU_1+\UU_2}{2}\right)\geq F(x,\UU_1)+f(x,\UU_1)\cdot\frac{\UU_2-\UU_1}{2}
	\]
	and
	\[
	F\left(x,\frac{\UU_1+\UU_2}{2}\right)\geq F(x,\UU_2)+f(x,\UU_2)\cdot\frac{\UU_1-\UU_2}{2}.
	\]
	Summing these inequalities and using (F2), we obtain that for every $0<r\leq R$ and every $|\UU_1-\UU_2|\ge r$, $|\UU_1|,|\UU_2| \le R$
	\[\begin{split}
	m_{r,R} & \le \frac12\bigl(F(x,\UU_1)+F(x,\UU_2)\bigr) - F\left(x,\frac{\UU_1+\UU_2}2\right)\\
	& \le \frac14\bigl(f(x,\UU_1)-f(x,\UU_2)\bigr)\cdot(\UU_1-\UU_2)
	\end{split}\]
	where $m_{r,R}$ is defined in \eqref{e-mrR}. Since $\zeta=1$ in $\Omega$, it follows from \eqref{e-2} that $v_n+w(v_n) \to v+ \nabla\xi$ a.e. in $\Omega$ as claimed. Since $w(v_n)\rightharpoonup w_0$, we have $w_0=\nabla\xi$ and by the usual diagonal procedure we obtain a.e. convergence to $v+w_0$ in $\R^3$. 
	Take any $w\in\cW$ and observe that by Vitali's convergence theorem 
	$$0=\int_{\R^3}\langle f(x,v_n+w(v_n)), w\rangle\,dx\to \int_{\R^3}\langle f(x,v+w_0), w\rangle\,dx.$$
	The uniqueness of $w(v)$ (see Lemma \ref{L:DefofW}) implies that $w_0=w(v)$.
	%
\end{proof}

Although in general $J'$ is not (sequentially) weak-to-weak$^*$ continuous, it is so for sequences on the topological manifold $\mathcal{M}$. Obviously, the same regularity holds for $E'$ and $\fM$.

\begin{Cor}\label{C:Jweaklycont}
	If $(v_n,w_n)\in\mathcal{M}$ and $(v_n,w_n)\rightharpoonup (v_0,w_0)$ in $\cV\times\cW$, then  $J'(v_n,w_n)\rightharpoonup J'(v_0,w_0)$, i.e.,
	$$J'(v_n,w_n)(\phi,\psi)\to J'(v_0,w_0)(\phi,\psi)$$
	for every $(\phi,\psi)\in\cV\times\cW$.
\end{Cor}
\begin{proof}
	From Lemma \ref{L:DefofW} \textit{(a)} we get $w_n=w(v_n)$. In view of Proposition \ref{P:ae}, we may assume  $v_n+w_n\to v_0+w_0$ a.e. in $\R^3$ (where $w_0=w(v_0)$).
	For $(\phi,\psi)\in\cV\times\cW$ we have
	\[\begin{split}
	J'(v_n,w_n)(\phi,\psi) & -J'(v_0,w_0)(\phi,\psi)=\int_{\R^3}(\nabla v_n-\nabla v_0)\cdot\nabla\phi\,dx\\
	& -\int_{\R^3}\bigl(f(x,v_n+w_n)-f(x,u_0+w_0)\bigr)\cdot(\phi+\psi)\,dx. 
	\end{split}\]
	Arguing as in the proof of Lemma \ref{L:BrezLieb}, we prove that
	\[
	\bigl(f(x,v_n+w_n)-f(x,u_0+w_0)\bigr)\cdot(\phi+\psi)
	\]
	is uniformly integrable and tight, hence from Vitali's convergence theorem
	\[
	J'(v_n,w_n)(\phi,\psi)\to J'(v_0,w_0)(\phi,\psi).\qedhere
	\]
\end{proof}

\section{Proof of Theorem \ref{T:mainSZ}}

Recall that the group $G:=\Z^3$ acts isometrically by translations on $X=\cV\times\cW$ and $J$ is $\Z^3$-invariant.
Let
\[
\cK:=\Set{v\in \cV | (J\circ m)'(u)=0}
\]
and suppose that $\cK$ consists of a finite number of distinct orbits. It is clear that $\Z^3$ satisfies the condition (G) in Section \ref{S:abstract}.
Then, in view of Lemma \ref{L:discrete},
\[
\kappa:= \inf\Set{\lVert v-v'\rVert_{\cD}|J'\bigl(m(v)\bigr) = J'\bigl(m(v')\bigr) = 0 \text{ and } v\ne v'}>0.
\]

\begin{Lem}\label{L:Discreteness}
	Let $\beta\ge c_{\cN}$ and suppose that $\cK$ has a finite number of distinct orbits.  
	If $u_n,v_n\in\cV$ are two Cerami sequences for $J\circ m$ such that $0\le\liminf_nJ\bigl(m(u_n)\bigr)\le \limsup_nJ\bigl(m(u_n)\bigr)\le\beta$, $0\le\liminf_nJ\bigl(m(v_n)\bigr)\le \limsup_nJ\bigl(m(v_n)\bigr)\le\beta$ and $\liminf_n\|u_n-v_n\|_{\cD}< \kappa$, then $\lim_n\|u_n-v_n\|_{\cD}=0$.
\end{Lem}
\begin{proof}
	Let $m(u_n)=(u_n,w^1_n)$, $m(v_n)=(v_n,w^2_n)$. From Lemma \ref{L:Coercive}, $m(u_n)$ and $m(v_n)$ are bounded.
	We first consider the case 
	\begin{equation}
	\label{e-pqconvergence}
	\lim_n|u_n-v_n|_{\Phi}=0
	\end{equation}
	and  prove that
	\begin{equation}
	\label{e-convergence}
	\lim_n\|u_n-v_n\|_\cD=0.
	\end{equation}
	From (F3) and Lemmas \ref{L:all} \textit{(ii)} and \ref{L:forC1} we have
	\[
	\begin{split}
	\|u_n-v_n\|_{\mathcal{D}}^2 = \, & J'\bigl(m(u_n)\bigr)(u_n-v_n,0)-J'\bigl(m(v_n)\bigr)(u_n-v_n,0)\\
	& +\int_{\rr} \Bigl(f\bigl(x,m(u_n)\bigr)-f\bigl(x,m(v_n)\bigr)\Bigr)\cdot(u_n-v_n)\,dx\le\\
	\le \, & o(1)+\int_{\rr}\Bigl(\big|f\bigl(x,m(u_n)\bigr)\big|+\big|f\bigl(x,m(v_n)\bigr)\big|\Bigr)|u_n-v_n|\,dx\\
	\le \, & o(1)+c_1\int_{\rr}\Bigl(\Phi'\bigl(|m(u_n)|\bigr)+\Phi'\bigl(|m(v_n)|\bigr)\Bigr)|u_n-v_n|\,dx \\
	\le \, & o(1)+ C\Bigl(\big|\Phi'\bigl(|m(u_n)|\bigr)\big|_\Psi+\big|\Phi'\bigl(|m(v_n)|\bigr)\big|_\Psi\Bigr)|u_n-v_n|_{\Phi}\to0
	\end{split}
	\]
	which gives \eqref{e-convergence}.
	
	Suppose now \eqref{e-pqconvergence} does not hold. From Lemmas \ref{L:all} \textit{(iii)} and \ref{L:Conv}, for a fixed $R>\sqrt{3}$ there exist $\varepsilon>0$ and $y_n\in\Z^3$ such that, up to a subsequence,
	\begin{equation}
	\label{e-boundedaway}
	\int_{B(y_n,R)}|u_n-v_n|^2\,dx\ge\varepsilon
	\end{equation}
	(cf. the proof of Lemma \ref{L:Coercive}). As $J$ is $\Z^3$-invariant, we can assume $y_n=0$. Since $m(u_n)$ and $m(v_n)$ are bounded, up to a subsequence
	\begin{equation}
	\label{e-weakconvergence}
	(u_n,w^1_n)\rightharpoonup(u,w^1)\hbox{ and }(v_n,w^2_n)\rightharpoonup(v,w^2)\quad\hbox{in }\cV\times\cW 
	\end{equation}
	for some $(u,w^1),(v,w^2)\in \cV\times\cW$. Since $u_n\to u$ and $v_n\to v$ in $L^2_\textup{loc}(\R^3,\R^3)$, $u\ne v$ according to \eqref{e-boundedaway}.
	From Corollary \ref{C:Jweaklycont} and \eqref{e-weakconvergence} we infer that
	\[
	J'(u,w^1)=J'(v,w^2)=0.
	\]
	Thus
	\[
	\liminf_n\|u_n - v_n\|_{\cD} \ge\|u - v\|_{\cD}\ge \kappa
	\]
	which is a contradiction.
\end{proof}

\begin{proof}[Proof of Theorem \ref{T:mainSZ}]
	\textit{(a)} The existence of a Cerami sequence $(v_n,w_n)\in\cM$ at the level $c_{\cN}$ follows from Proposition \ref{P:DefOfm(u)}; this sequence is bounded due to Lemma \ref{L:Coercive}. Similarly as in the proof of Lemma \ref{L:Discreteness} we find $v\in\cV\setminus\{0\}$ such that
	$(v_n, w_n)\rightharpoonup(v,w)$, $(v_n, w_n)\to(v,w)$ a.e. in $\R^3$ (both along a subsequence), and $\cJ'(v,w)=0$ (with $w=w(v)$). More precisely, if $|v_n|_\Phi\to 0$, then \eqref{e-convergence} with $u_n=0$ holds by the same argument. This is impossible because $J\bigl(m(v_n)\bigr)\to c_\cN>0$. Hence \eqref{e-boundedaway} with $u_n=0$ is satisfied and we can assume, up to translating by $y_n$, that $\int_{B_R}|v_n|^2\,dx\ge\varepsilon$, whence $v\ne 0$.
	From Fatou's lemma and (F4),
	\begin{equation*}\begin{split}
		c_{\cN}&=\lim_nJ(v_n,w_n)=\lim_n\left(J(v_n,w_n)-\frac12J'(v_n,w_n)(v_n,w_n)\right)\\
		&\geq J(v,w)-\frac12J'(v,w)(v,w) =\cJ(v,w).
	\end{split}\end{equation*}
	Since $(v,w)\in\cN$, $J(v,w)=c_{\cN}$ and $\UU=v+w$ solves \eqref{e-Curl} due to Proposition \ref{P:var}.
	
	Note that here we have \emph{not} assumed $\cK$ has finitely many distinct orbits.
	
	\textit{(b)} In order to complete the proof we use directly Theorem \ref{T:CrticMulti} \textit{(b)}. That (I1)--(I8) are satisfied and $(M)_0^\beta$ holds for all $\beta>0$ follows from Proposition \ref{P:DefOfm(u)}, Corollary \ref{C:Ceramibdd}, and Lemma \ref{L:Discreteness}.
\end{proof}

\chapter{Maxwell's and nonautonomous Schr\"odinger equations with cylindrical symmetry}\label{K:cylsym}

\section{Statement of the results}\label{S:statecs}

In this chapter, based on \cite{Gacz}, we study the curl-curl equation
\begin{equation}\label{e-curlsym}
\nabla\times\nabla\times\UU=h(x,\UU) \quad \text{in } \rn
\end{equation}
and the (scalar) Schr\"odinger equation
\begin{equation}\label{e-SchSym}
-\Delta u+\frac{a}{r^2}u=f(x,u) \quad \text{in } \rn
\end{equation}
under cylindrical symmetry, see respectively the definitions of $\DF$ and $X_{\SO}$ below. Here $N\ge3$ (whereas \cite{Gacz} deals with \eqref{e-curlsym} exclusively in the case $N=3$), $a$ is greater than a certain value $a_0\in\,]-\infty,0]$, and $r$ is the Euclidean norm of the first $K$ components of the point $x=(y,z)\in\rk\times\rnk=\rn$, $2\le K<N$, i.e.
\[
r^2=r_y^2=r_x^2=\sum_{i=1}^Ky_i^2=\sum_{i=1}^Kx_i^2.
\]
Of course, the curl operator $\nabla\times$ is naturally defined only in dimension $N=3$\footnote{The curl operator is defined as well in dimension $N=2$ considering $\R^2$ as a subspace of $\rr$, but we do not take this case into account because this chapter deals only with problems in dimension $N\ge3$.}, therefore, in order to find a suitable counterpart in higher dimensions, we use the identity
\begin{equation}\label{e-curlLap}
\nabla\times\nabla\times\UU=\nabla(\nabla\cdot\UU)-\Delta\UU, \quad \UU\in\cC^2(\rr,\rr).
\end{equation}
Since its right-hand side is defined for every dimension, we define the curl-curl operator in dimension $N\ge3$ using \eqref{e-curlLap}. As a consequence, if $\UU,\VV\colon\rn\to\rn$ are two functions with square-integrable gradients, we define as well
\[
\int_{\rn}\nabla\times\UU\cdot\nabla\times\VV\,dx:=\int_{\rn}\nabla\UU\cdot\nabla\VV-\nabla\cdot\UU\nabla\cdot\VV\,dx,
\]
where, with a small abuse of notation, we used $\cdot$ for both the scalar product in $\R^{N\times N}\simeq\R^{N^2}$ and in $\rn$ (in the latter context also to indicate the divergence operator). Note that we only defined the objects $\nabla\times\nabla\times\UU$ and $\int_{\rn}\nabla\times\UU\cdot\nabla\times\VV\,dx$, \emph{not} $\nabla\times\UU$, but this is enough to study \eqref{e-curlsym}. A generalization of $\nabla\times\UU$, on the other hand, is given in \cite[Section 3]{MeSzu}.

The main result in Section \ref{S:equiv} consists of the extension of the following equivalence result between solutions to \eqref{e-curlsym} and \eqref{e-SchSym} to the case of weak solutions, i.e., critical points of the corresponding energy functionals, defined in suitable Sobolev spaces; this is done in Theorem \ref{T:ScalVec}. Suppose $u\in\cC^2(\rn)$ is such that $u(\widetilde{g}x)=u(x)$ for every $x\in\rn$ and every $g\in\SO(2)$, where
\[
\widetilde{g}:=
\begin{pmatrix}
g & 0\\
0 & I_{N-2}
\end{pmatrix}
\in\SO(N),
\]
and define $\UU(x):=u(x)/|(x_1,x_2)|(-x_2,x_1,0)$. Suppose additionally that $h(x,\alpha w)=f(x,\alpha)w$ for $w=|(\xi_1,\xi_2)|^{-1}(-\xi_2,\xi_1,0)$, $\xi_1^2+\xi_2^2>0$, and $\alpha\in\R$. Then by explicit computations one can prove that $\UU$ is divergence-free and, moreover, $u$ solves (pointwise) \eqref{e-SchSym} with $a=1$ in $\rn_*:=\rn\setminus\{x_1^2+x_2^2=0\}$ if and only if $\UU$ solves (pointwise) \eqref{e-curlsym} in $\rn_*$.

Recall that $\cD^{1,2}(\rn)$ is the completion of $\cC_c^\infty(\rn)$ with respect to the usual norm $|\nabla u|_2$ and define analogously $\cD^{1,2}(\rn,\rn)$. Moreover, let
\[
X:=\Set{u\in\cD^{1,2}(\rn)|\int_{\rn}\frac{u^2}{r^2}\,dx<\infty}
\]
and define $X_{\SO}$ as the subspace of $X$ consisting of the functions invariant under the usual action of $\SO:=\SO(K)\times\{I_{N-K}\}\subset\SO(N)$. Note that this is equivalent to requiring that such functions be invariant under the action of $\cO(K)\times\{I_{N-K}\}$ because for every $\xi_1,\xi_2\in\mbS^{\nu-1}$, $\nu\ge2$, there exists $g\in\SO(\nu)$ such that $\xi_2=g\xi_1$. $X$ is a Hilbert space once endowed with the scalar product
\[
(u,v)\in X\times X\mapsto\int_{\rn}\nabla u\cdot\nabla v+\frac{uv}{r^2}\,dx\in\R.
\]
Notice that, if $K>2$, then
\[
\int_{\rn}\frac{|u|^2}{r^2}\,dx\le\left(\frac{2}{K-2}\right)^2\int_{\rn}|\nabla u|^2\,dx
\]
for every $u\in\cD^{1,2}(\rn)$, see \cite{BadTar}, so $X$ and $\cD^{1,2}(\R^N)$ coincide. If $K=2$, then $\cC_c^\infty(\rn)\not\subset X$ because the quantity $\varphi^2/r^2$ need not be integrable for $\varphi\in\cC_c^\infty(\rn)$. If $a>-(\frac{K-2}2)^2$, then we can define an equivalent norm in $X$ as
\[
\|u\|:=\sqrt{\int_{\rn}|\nabla u|^2+\frac{a}{r^2}u^2\,dx}.
\]

Define the functionals $E\colon\cD^{1,2}(\rn,\rn)\to\R$ and $J\colon X\to\R$ as
\[\begin{split}
J(u) & =\int_{\rn}\frac12\left(|\nabla u|^2+\frac{a}{r^2}u^2\right)-F(x,u)\,dx\\
E(\UU) & =\int_{\rn}\frac12|\nabla\times\UU|^2-H(x,\UU)\,dx,
\end{split}\]
where $F(x,u):=\int_0^uf(x,t)\,dt$ and $H(x,\UU):=\int_0^1h(x,t\UU)\cdot\UU\,dx$. The first set of assumptions about $f$ are as follows.
\begin{itemize}
	\item [(F1)] $f\colon\R^N\times\R\to\R$ is a Carath\'eodory function such that $f(gx,u)=f(x,u)$ for every $g\in\SO$, a.e. $x\in\R^N$, and every $u\in\R$.
	Moreover $f$ is $\Z^{N-K}$-periodic in the last $N-K$ components of $x$, i.e., $f(x,u)=f\bigl(x+(0,\xi),u\bigr)$ for every $u\in\R$, a.e. $x\in\R^N$, and a.e. $\xi\in\Z^{N-K}$.
	\item [(F2)] $\displaystyle\lim_{u\to 0}\frac{f(x,u)}{|u|^{2^*-1}}=\lim_{|u|\to\infty}\frac{f(x,u)}{|u|^{2^*-1}}=0$ uniformly with respect to $x\in\R^{N}$.
	\item [(F3)] $\displaystyle\lim_{|u|\to\infty}\frac{F(x,u)}{|u|^2}=\infty$ uniformly with respect to $x\in\R^{N}$.
	\item [(F4)] $\displaystyle u\mapsto\frac{f(x,u)}{|u|}$ is nondecreasing on $]-\infty,0[$ and $]0,\infty[$ for a.e. $x\in\R^{N}$.
\end{itemize}

It is straightforward to check that $J$ is invariant under the action of $\SO$ and (cf. \cite[Section 2]{AzzBenDApFor}) $E$ is invariant under the action defined in \eqref{e-Fix}. The first result in this chapter concerns the existence of solutions to \eqref{e-SchSym} and reads as follows.

\begin{Th}\label{T:ExMul}
	Suppose that $a>-\bigl(\frac{K-2}{2}\bigr)^2$ and (F1)--(F4) hold. Then
	there exists a ground state solution $u$ to \eqref{e-SchSym} in $X_{\SO}$. If, in addition, $f$ is odd in $u$, then $u$ is nonnegative and	
	\eqref{e-SchSym} has infinitely many geometrically distinct solutions in $X_{\SO}$.
\end{Th}

By \textit{ground state solution in} $X_{\SO}$ we mean a nontrivial solution that minimizes $J$ over all the nontrivial solutions in $X_{\SO}$. It need not be a ground state solution in the general sense, i.e., minimizing $J$ over all the nontrivial solutions in $X$. Two solutions are called \textit{geometrically distinct} if and only if one cannot be obtained via a translation  of the other in the last $N-K$ variables by a vector in $\Z^{N-K}$.

Observe that (F4) implies $f(x,u)u\ge2F(x,u)\ge0$. When $f$ does not depend on $y$, however, we can consider sign-changing nonlinearities under the following weaker variant of the Ambrosetti-Rabinowitz condition \cite{AmbRab}:
\begin{itemize}
	\item [(F5)] There exists $\gamma>2$ and $u_0\in\R$ such that $f(z,u)u\ge\gamma F(z,u)$ for a.e. $z\in\rnk$ and every $u\in\R$ and $\essinf_{z\in\rnk}F(z,u_0)>0$.
\end{itemize}

In this different setting we can prove the existence of a nontrivial solution to \eqref{e-SchSym}.

\begin{Th}\label{T:Ex}
	Suppose that $a>-\bigl(\frac{K-2}{2}\bigr)^2$, (F1)--(F2) and (F5) hold, and $f$ does not depend on $y$. Then there exists a nontrivial solution $u$ to \eqref{e-SchSym} in $X_{\SO}$. 
\end{Th}

Observe that, in Theorems \ref{T:ExMul} or \ref{T:Ex}, every solution $\bar{u}$ can be supposed to be nonnegative if $f(x,u)\ge0$ for $x\in\rn$ and $u\le0$ because
\[\begin{split}
0\ge-\|\bar{u}_-\|^2 & =\int_{\R^N}\nabla\bar{u}\cdot\nabla\bar{u}_-+\frac{a}{r^2}\bar{u}\bar{u}_-\,dx=\int_{\R^N}f(z,\bar{u})\bar{u}_-\,dx\\
& =\int_{\R^N}f(z,-\bar{u}_-)\bar{u}_-\,dx\ge 0,
\end{split}\]
i.e., $\bar{u}_-=0$ and $\bar{u}=\bar{u}_+\ge0$.

We recall that, if $f$ does not depend on $x$ and $a=1$, then  Badiale, Benci, and Rolando \cite{BadBenRol} found a nontrivial and nonnegative solution to \eqref{e-SchSym} under more restrictive assumptions than in Theorem \ref{T:Ex}; in particular (cf. assumption (f$_1$) therein), they assumed the double-power behaviour $|f(u)|\leq C\min\{|u|^{p-1},|u|^{q-1}\}$ for $u\in\R$, some constant $C>0$, and $2<p<2^*<q$. For example, they cannot deal with nonlinearities such as \eqref{e-ex1} or \eqref{e-ex2} (where $\UU\in\rr$ is replaced with $u\in\R$), not even in the autonomous case $\Gamma\equiv1$ or replacing $f$ with $f\chi_{[0,\infty[}$. On the other hand, we admit such nonlinearities in Theorem \ref{T:ExMul} if $\Gamma\in L^\infty(\rn)$ is positive, bounded away from $0$, $\Z^{N-K}$-periodic in the last $N-K$ variables, and $\SO$-invariant. If, moreover, $\Gamma$ does not depend on $y$, then Theorem \ref{T:Ex} admits nonlinearities as in \eqref{e-ex2}, also if $f(x,u)$ is replaced with $f(x,u)\chi_{\{u\ge0\}}$ (which is not allowed in Theorem \ref{T:ExMul} due to (F3)). Finally, as previously mentioned, Theorem \ref{T:Ex} admits sign-changing nonlinearities. 

When $K=2$, combining the equivalence result provided by Theorem \ref{T:ScalVec} with Theorems \ref{T:ExMul} and \ref{T:Ex}, we can prove the following, as long as $h\colon\rn\times\rn\to\rn$ satisfies
\begin{equation}\label{e-fgh}
h(\cdot,0)=0 \text{ and } h(\cdot,\alpha w)=f(\cdot,\alpha)w \text{ for all } \alpha\in\R \text{ and } w\in\mbS^{N-1}.
\end{equation}
Note that \eqref{e-fgh} implies $h$ and $f$ are odd the second variable:
\[\begin{split}
h(x,\UU)=f(x,|\UU|)\UU/|\UU|=-f(x,|\UU|)(-\UU)/|\UU|=-h(x,-\UU), \quad \UU\ne0,\\
f(x,-\alpha)=\mathfrak{p}\bigl(f(x,-\alpha)e_1\bigr)=\mathfrak{p}\bigl(h(x,-\alpha e_1)\bigr)=\mathfrak{p}\bigl(-f(x,\alpha)e_1\bigr)=-f(x,\alpha),
\end{split}\]
where $\mathfrak{p}$ is the orthogonal projection onto the first axis.

Let $\cF$ be the space of the vector fields $\UU\colon\rn\to\rn$ such that
\begin{equation}\label{e-uU}
\UU(x)=\frac{u(x)}{r}
\begin{pmatrix}
-x_2\\
x_1\\
0
\end{pmatrix}
\end{equation}
for some $\SO$-invariant $u\colon\rn\to\R$, where $0\in\R^{N-2}$, and define $\DF:=\cD^{1,2}(\rn,\rn)\cap\cF$.

\begin{Cor}\label{C:V1}
	Suppose that $K=2$ and $h$ satisfies \eqref{e-fgh}.
	\begin{itemize}
		\item [(a)] If (F1)--(F4) hold, then \eqref{e-curlsym} has infinitely many geometrically distinct solutions in $\DF$, one of which is a ground state solution in $\DF$.
		\item [(b)] If (F1)--(F2), (F5) hold and $h$ does not depend on $y$, then \eqref{e-curlsym} has a nontrivial solution $\UU\in\DF$.
	\end{itemize}
\end{Cor}

Simiarly as before, by \textit{ground state solution in} $\DF$ we mean a nontrivial solution that minimizes $E$ over all the nontrivial solutions in $\DF$; also the definition of \textit{geometrically distinct solutions} easily adapts to \eqref{e-curlsym}. 

Now we consider the problems \eqref{e-curlsym} and \eqref{e-SchSym} in the Sobolev-critical case, so from now till the end of this section we assume $K=2$ and
\[
h(x,\UU)=|\UU|^{2^*-2}\UU \quad \text{and} \quad f(x,u)=|u|^{2^*-2}u.
\]
Of course Theorems \ref{T:ExMul} or \ref{T:Ex} no longer apply because (F2) does not hold. We recall that Badiale, Guida, and Rolando \cite{BadGuiRol} found a ground state solution in $X_{\SO}$ to \eqref{e-curlsym} for $a>0$ (but without further restrictions on $K\ge2$); an immediate consequence of this and Theorem \ref{T:ScalVec} is the following.

\begin{Cor}\label{C:V2}
There exists a ground state solution in $\DF$ to
\begin{equation}\label{e-crit}
\nabla\times\nabla\times\UU=|\UU|^{2^*-2}\UU \quad \text{in } \rn.
\end{equation}
\end{Cor}

The main results of this chapter in the Sobolev-critical case concern the existence of unbounded sequences of solutions. In order to achieve them, we need to recover compactness and the first step in this direction is to introduce this definition. From now on, we additionally assume $N=3$.

\begin{Def}\label{D:sym}
For $g_1,g_2\in\SO(2)$ we denote $g=\bigl(\begin{smallmatrix}
g_1 & 0\\
0 & g_2
\end{smallmatrix}\bigr)\equiv(g_1,g_2)\in\SO(2)\times\SO(2)$. We say that $\UU\in\cD^{1,2}(\rr,\rr)$ is $\SO(2\times2)$-symmetric if and only if for every $g_1,g_2\in\SO(2)$ and a.e. $x\in\rr$
\[
\frac{\UU\Bigl(\pi\bigl(g\pi^{-1}(x)\bigr)\Bigr)}{\psi\Bigl(\pi\bigl(g\pi^{-1}(x)\bigr)\Bigr)}=\frac{\widetilde{g_1}\UU(x)}{\psi(x)},
\]
where $\psi(x)=\sqrt{\frac{2}{1+|x|^2}}$, $\pi\colon\mbS^3\setminus\{Q\}\to\rr$ is the stereographic projection, and $Q=(1,0,0,0)$ is the north pole.
\end{Def}

For further remarks on Definition \ref{D:sym} and the symmetry introduced therein see Subsection \ref{SS:sym}. The subspace of $\cD^{1,2}(\rr,\rr)$ consisting of $\SO(2\times2)$-symmetric vector fields is denoted $\cD_{\SO(2\times2)}$. Now we can state our main results in the Sobolev-critical case.

\begin{Th}\label{T:crit}
There exists a sequence $\UU_n\in\cD_{\SO(2\times2)}$ of solutions to \eqref{e-crit} such that $\lim_nE(\UU_n)=\infty$ and each $\UU_n$ is of the form \eqref{e-uU}.
\end{Th}

Clearly the property $\lim_nE(\UU_n)=\infty$ implies that the solutions $\UU_n$ are geometrically distinct (up to a subsequence), as $E$ is invariant under translations. Once again, we exploit Theorem \ref{T:ScalVec} to obtain the following.

\begin{Cor}\label{C:S}
There exists a sequence $u_n\in X_{\SO}$ of solutions to
\[
-\Delta u+\frac{u}{r^2}=u^5 \quad\text{in }\R^3
\]
such that $\lim_nJ(u_n)=\infty$ and each $|u_n|$ is satisfies
\[
\frac{\Big|u_n\Bigl(\pi\bigl(g\pi^{-1}(x)\bigr)\Bigr)\Big|}{\psi\Bigl(\pi\bigl(g\pi^{-1}(x)\bigr)\Bigr)}=\frac{|u_n(x)|}{\psi(x)}
\]
for every $g_1,g_2\in\SO(2)$ and a.e. $x\in\rr$.
\end{Cor}

Observe that Theorem \ref{T:ScalVec} is used in Corollaries \ref{C:V1} and \ref{C:V2} to build solutions to \eqref{e-curlsym} from solutions to \eqref{e-SchSym}, while it is used in Corollary \ref{C:S} to build solutions to \eqref{e-SchSym} from solutions to \eqref{e-curlsym}.

\subsection{Remarks on Definition \ref{D:sym}}\label{SS:sym}

As previously said, the symmetry defined in Definition \ref{D:sym} is used to restore compactness in the whole $\rr$ in the Sobolev-critical case. To the best of our knowledge, the first who faced this issue was Ding \cite{Ding}, who proved the existence of infinitely many sign-changing solutions to
\begin{equation}\label{e-DingR}
-\Delta u=|u|^{2^*-2}u, \quad u\in\cD^{1,2}(\rn).
\end{equation}
His approach is as follows: first, he proves that the solutions to \eqref{e-DingR} are in 1-to-1 correspondence with the solutions to
\begin{equation}\label{e-DingS}
-\Delta_\mathfrak{g}v+\frac{N(N-2)}{4}v=|v|^{2^*-2}v, \quad v\in H^1(\mbS^N)
\end{equation}
via the stereographic projection and the conformal map $\psi$, i.e., $v\in H^1(\mbS^N)$ solves \eqref{e-DingS} if and only if $u=(v\circ\pi^{-1})\psi\in\cD^{1,2}(\rn)$ solves \eqref{e-DingR}; the symbol $\Delta_\mathfrak{g}$ in \eqref{e-DingS} stands for the Laplace-Beltrami operator \cite{Aubin,ONeill}. Then, he introduces a group action in $H^1(\mbS^N)$ such that the subspace consisting of the functions invariant with respect to this action is compactly embedded in $L^{2^*}(\mbS^N)$. More in details, for $k\ge m\ge2$ integers such that $N+1=k+m$, he considers the action of $\cO(k)\times\cO(m)\subset\cO(N+1)$ defined as
\[
(g_1,g_2)v(\xi):=v(g_1\xi_1,g_2\xi_2)
\]
for $g_1\in\cO(k)$, $g_2\in\cO(m)$, $v\in H^1(\mbS^N)$, and $\mbS^N\ni\xi=(\xi_1,\xi_2)\in\R^k\times\R^m$.

Ding, therefore, ``leaves'' the Euclidean space $\rn$ and works on the sphere $\mbS^N$, where the definition of the group action is more straightforward. If one decides to work directly in $\rn$, then the action of $\cO(k)\times\cO(m)$ on $H^1(\mbS^N)$ must be adapted to $\cD^{1,2}(\rn)$. This was done by Clapp and Pistoia \cite{ClPis}: if $v\in H^1(\mbS^N)$ is such that
\[
v(g_1\xi_1,g_2\xi_2)=v(\xi)
\]
for every $g_1\in\cO(k)$, every $g_2\in\cO(m)$, and a.e. $\xi\in\mbS^N$, and if we define $u(x)=\psi(x)v\bigl(\pi^{-1}(x)\bigr)$, then $u\in\cD^{1,2}(\rn)$ satisfies
\[
\frac{\psi(x)}{\psi\Bigl(\pi\bigl(g\pi^{-1}(x)\bigr)\Bigr)}u\bigl(\pi\Bigl(g\pi^{-1}(x)\bigr)\Bigr)=u(x)
\]
for every $g=
\bigl(\begin{smallmatrix}
g_1 & 0\\
0 & g_2
\end{smallmatrix}\bigr)
\equiv(g_1,g_2)\in\cO(k)\times\cO(m)$ and a.e. $x\in\rn$, see \cite[Section 3]{ClPis}.

Both \cite{Ding} and \cite{ClPis} deal with scalar-valued functions. Although the arguments seem to be valid also for vector-valued functions, the case studied in Section \ref{S:SC} is slightly more delicate because we need to combine this action, developed for the Laplace operator, with those defined in \eqref{e-Fix} and \eqref{e-S}, which allow to reduce the curl-curl operator to the vector Laplacian. This is the reason why we introduce the symmetry from Definition \ref{D:sym}, which matches well with that induced from \eqref{e-Fix} and \eqref{e-S}, cf. Lemma \ref{L:equisym}.

\section{An equivalence result}\label{S:equiv}

Throughout this section we have $K=2$. The reason is that it is straightforward to generalize the decomposition of $\textup{Fix}(\SO)$ given in \cite[Lemma 1]{AzzBenDApFor} and recalled in Chapter \ref{K:intro1} to the case $N\ge3$ (cf. Lemma \ref{L:dec}), but not to the case $2\le K<N$. Another difficulty is that $\SO(\nu)$ is \emph{not} abelian if $\nu\ge3$.

We recall that for $g\in\SO(2)$ we denote
\[
\widetilde{g}=
\begin{pmatrix}
g & 0\\
0 & I_{N-2}
\end{pmatrix}
\in\SO(N)
\]
and $\SO=\Set{\widetilde{g}|g\in\SO(2)}$; moreover, $\textup{Fix}(\SO)\subset\cD^{1,2}(\rn,\rn)$ is the subspace of the vector fields invariant under the action defined in \eqref{e-Fix}. We recall also that $\DF=\cD^{1,2}(\rn,\rn)\cap\cF$, where $\cF$ is the space of the vector fields $\UU\colon\rn\to\rn$ that satisfy \eqref{e-uU} for some $\SO$-invariant $u\colon\rn\to\R$. Note that $\DF$ is a closed subspace of $\cD^{1,2}(\rn,\rn)$ and that $\DF\subset\textup{Fix}(\SO)$. Finally, we recall the notation $x=(y,z)\in\R^2\times\R^{N-2}=\rn$.

The main result in this section reads as follows.

\begin{Th}\label{T:ScalVec}
	Assume $f$ satisfy (F1) and there exists $C>0$ such that $|f(x,u)|\leq C|u|^{2^*-1}$ for a.e. $x\in\rn$ and every $u\in\R$; assume also that $h$ satisfy \eqref{e-fgh}. Suppose that 
	$\UU$ and $u$ satisfy \eqref{e-uU} for a.e. $x\in\rn$. Then
	$\UU\in\DF$ if and only if $u\in X_{\SO}$ and, in such a case, $\nabla\cdot\UU=0$ and
	$J(u)=E(\UU)$. Moreover,
	$u\in X_{\SO}$ is a solution to \eqref{e-SchSym} with $a=1$ 
	if and only if $\UU\in\DF$ is a solution to \eqref{e-curlsym}.
\end{Th}

\begin{Lem}\label{L:Oapprox}
	If $\UU\in\textup{Fix}(\SO)$, then there exists $\UU_n\in\cC_c^\infty(\rn,\rn)\cap\textup{Fix}(\SO)$ such that $\lim_n|\nabla\UU_n-\nabla \UU|_{2}=0$.
\end{Lem}
\begin{proof}
	Since $\UU\in\cD^{1,2}(\rn,\rn)$, there exists $\VV_n\in\cC_c^\infty(\rn,\rn)$ such that $\lim_n|\nabla \VV_n-\nabla\UU|_{2}=0$. Let $$\UU_n(x):=\int_{\SO}g^{-1}\VV_n(gx)\,d\mu(g)=\int_{\SO}g^T\VV_n(gx)\,d\mu(g),$$ where $\mu$ is the Haar measure of $\SO$ (note that $\SO$ is compact).
	
	For every $e\in\SO$ we have
	\[
	\UU_n(ex)=\int_{\SO}g^T\VV_n(gex)\,d\mu(g)=e\int_{\SO}g'^T\VV_n(g'x)\,d\mu(g')=e\UU_n(x),
	\]
	i.e., $\UU_n\in\textup{Fix}(\SO)$. Moreover, $\UU_n\in\cC_c^\infty(\rn,\rn)$ because so does $\VV_n$.
	
	First we prove that $|\UU_n-\UU|_{2^*}\to 0$. From Jensen's inequality there holds
	\[\begin{split}
	|\UU_n-\UU|_{2^*}^{2^*} & =\int_{\rn}\left|\int_{\SO}g^T\VV_n(gx)-\UU(x)\,d\mu(g)\right|^{2^*}dx\\
	& \leq\int_{\SO}\int_{\rn}\left|g^T\VV_n(gx)-\UU(x)\right|^{2^*}dx\,d\mu(g)\\ & =\int_{\SO}\int_{\rn}\left|g^T\VV_n(gx)-g^T\UU(gx)\right|^{2^*}dx\,d\mu(g)\\
	& =\int_{\SO}\int_{\rn}\left|\VV_n(gx)-\UU(gx)\right|^{2^*}dx\,d\mu(g)\\ &
	=\int_{\SO}|\VV_n-\UU|_{2^*}^{2^*}\,d\mu(g)=|\VV_n-\UU|_{2^*}^{2^*}\to 0.
	\end{split}\]
	Finally, $\UU_n$ is a Cauchy sequence in $\cD^{1,2}(\rn,\rn)$ because, similarly as before,
	\[
	|\nabla \UU_n-\nabla \UU_m|_2\le|\nabla \VV_n-\nabla\VV_m|_2\to 0,
	\]
	hence $\UU_n\to\UU$ in $\cD^{1,2}(\rn.\rn)$.
\end{proof}

The following result was proved in \cite[Lemma 1]{AzzBenDApFor} for $N=3$. Its generalization to the case $N\ge3$ is easy, but we include such a proof for the reader's convenience. Recall that $\rn_*=\rn\setminus\{x_1^2+x_2^2=0\}$.

\begin{Lem}\label{L:dec}
	For every $\UU\in\textup{Fix}(\SO)$ there exist $\UU_\rho,\UU_\tau,\UU_{\zeta,i}\in\textup{Fix}(\SO)$, $i\in\{3,\dots,N\}$, such that for every $x\in\rn_*$ at which $\UU(x)$ is defined
	\begin{itemize}
		\item $\UU_\rho(x)$ is the orthogonal projection of $\UU(x)$ onto $\textup{span}\{(x_1,x_2,0)\}$,
		\item $\UU_\tau(x)$ is the orthogonal projection of $\UU(x)$ onto $\textup{span}\{(-x_2,x_1,0)\}$,
		\item $\UU_{\zeta,i}(x)$ is the orthogonal projection of $\UU(x)$ onto $\textup{span}\{e_i\}$,
	\end{itemize}
	and $\nabla\UU_\rho(x),\nabla\UU_\tau(x),\nabla\UU_{\zeta,i}(x)$ are pairwise orthogonal in $\R^{N\times N}\simeq\R^{N^2}$, $i\in\{3,\dots,N\}$. In particular, $\UU=\UU_\rho+\UU_\tau+\sum_{i=3}^N\UU_{\zeta,i}$.
\end{Lem}
\begin{proof}
	For $\UU\in\textup{Fix}(\SO)$ we write $\UU=(\UU_1,\dots,\UU_N)$. For $x\in\rn_*$ let $\UU_\rho(x)$, $\UU_\tau(x)$, and $\UU_{\zeta,i}(x)$ be the orthogonal projections as in the statement. Then
	\begin{equation*}\begin{split}
	\UU_\rho(x) = U_\rho(x)
	\begin{pmatrix}
	x_1 \\ x_2 \\ 0
	\end{pmatrix}
	= \frac{x_1\UU_1(x)+x_2\UU_2(x)}{r^2}
	\begin{pmatrix}
	x_1 \\ x_2 \\ 0
	\end{pmatrix},
	\\
	\UU_\tau(x) = U_\tau(x)
	\begin{pmatrix}
	-x_2 \\ x_1 \\ 0
	\end{pmatrix}
	= \frac{-x_2\UU_1(x)+x_1\UU_2(x)}{r^2}
	\begin{pmatrix}
	-x_2 \\ x_1 \\ 0
	\end{pmatrix},
	\\
	\UU_{\zeta,i}(x) = U_{\zeta,i}(x)e_i = \UU_i(x)e_i.
	\end{split}\end{equation*}
	
	We prove that $U_\rho$, $U_\tau$, and $U_{\zeta,i}$ are $\SO$-invariant. This is trivial for $U_{\zeta,i}$ because $i\ge3$; moreover, note that
	\[\begin{split}
	U_\rho(x) & = \frac1{r^2}
	\begin{pmatrix}
	\UU_1(x) & \UU_2(x)
	\end{pmatrix}
	\begin{pmatrix}
	x_1 \\ x_2
	\end{pmatrix},
	\\
	U_\tau(x) & = \frac1{r^2}
	\begin{pmatrix}
	\UU_1(x) & \UU_2(x)
	\end{pmatrix}
	\begin{pmatrix}
	0 & -1 \\ 1 & 0
	\end{pmatrix}
	\begin{pmatrix}
	x_1 \\ x_2
	\end{pmatrix},
	\end{split}\]
	therefore, exploiting that $\SO(2)$ is abelian,
	\[\begin{split}
	U_\rho(\widetilde{g}x) & = \frac1{r^2}
	\begin{pmatrix}
	\UU_1(x) & \UU_2(x)
	\end{pmatrix}
	g^Tg
	\begin{pmatrix}
	x_1 \\ x_2
	\end{pmatrix}
	\\
	& = \frac1{r^2}
	\begin{pmatrix}
	\UU_1(x) & \UU_2(x)
	\end{pmatrix}
	\begin{pmatrix}
	x_1 \\ x_2
	\end{pmatrix}
	=U_\rho(x),
	\\
	U_\tau(\widetilde{g}x) & = \frac1{r^2}
	\begin{pmatrix}
	\UU_1(x) & \UU_2(x)
	\end{pmatrix}
	g^Tg
	\begin{pmatrix}
	0 & -1 \\ 1 & 0
	\end{pmatrix}
	\begin{pmatrix}
	x_1 \\ x_2
	\end{pmatrix}
	\\
	& = \frac1{r^2}
	\begin{pmatrix}
	\UU_1(x) & \UU_2(x)
	\end{pmatrix}
	\begin{pmatrix}
	0 & -1 \\ 1 & 0
	\end{pmatrix}
	\begin{pmatrix}
	x_1 \\ x_2
	\end{pmatrix}
	=U_\tau(x)
	\end{split}\]
	for every $g\in\SO(2)$ and a.e. $x\in\rn$.
	
	Now we prove that $\UU_\rho,\UU_\tau,\UU_{\zeta,i}\in\cD^{1,2}(\rn,\rn)$. From their very definition, $\UU_\rho,\UU_\tau,\UU_{\zeta,i}\in L^6(\rn,\rn)\cap H^1_\textup{loc}(\rn_*,\rn)$. In what follows we denote by $\partial_j\UU_\rho|_{\rn_*}$ the partial derivative along $x_j$ of $\UU_\rho$ in $\rn_*$ in the sense of distributions and by $\partial_j\UU_\rho$ the function defined a.e. in $\rn$ that represents it, i.e.
	\begin{equation}\label{e-weakder}
	\int_{\rn}\partial_j\UU_\rho|_{\rn_*}\cdot\varphi\,dx=-\int_{\rn}\UU_\rho\cdot\partial_j\varphi\,dx
	\end{equation}
	for every $\varphi\in\cC_c^\infty(\rn_*,\rn)$; similar notations are used for $\UU_\tau$ and $\UU_{\zeta,i}$. With a small abuse of notation, the symbol $\cdot$ will stand for the scalar product both in $\rn$ and in $\R^{N\times N}\simeq\R^{N^2}$.
	
	It is obvious that for every $i,j\in\{3,\dots,N\}$ and a.e. $x\in\rn$
	\[
	\nabla\UU_{\zeta,i}(x)\cdot\nabla\UU_{\zeta,j}(x)=\nabla\UU_{\zeta,i}(x)\cdot\nabla\UU_\rho(x)=\nabla\UU_{\zeta,i}(x)\cdot\nabla\UU_\tau(x)=0.
	\]
	Moreover, by explicit computations,
	\[\begin{split}
	\nabla\UU_\rho(x)\cdot\nabla\UU_\tau(x) = \, & \nabla[U_\rho(x)x_2]\cdot\nabla[U_\tau(x)x_1]-\nabla[U_\rho(x)x_1]\cdot\nabla[U_\tau(x)x_2]\\
	= \, & x_2U_\tau(x)\partial_1U_\rho(x)+x_1U_\rho(x)\partial_2U_\tau(x)\\
	& -x_1U_\tau(x)\partial_2U_\rho(x)-x_2U_\rho(x)\partial_1U_\tau(x)=0,
	\end{split}\]
	where the last equality follows from the fact that the $\SO$-invariance of $U_\rho$ implies $x_1\partial_2U_\rho=x_2\partial_1U_\rho$ (and likewise for $U_\tau$).
	
	This yields
	\begin{equation*}
	|\nabla\UU|^2=|\nabla\UU_\rho|^2+|\nabla\UU_\tau|^2+\sum_{i=3}^N|\nabla\UU_{\zeta,i}|^2 \quad \text{ a.e. in } \rn,
	\end{equation*}
	whence for every $i\in\{3,\dots,N\}$ and every $j\in\{1,\dots,N\}$
	\[
	\partial_j\UU_\rho,\partial_j\UU_\tau,\partial_j\UU_{\zeta,i}\in L^2(\rn,\rn),
	\]
	hence the proof will be complete once we show that $\partial_j\UU_\rho|_{\rn_*}$ coincide with the distributional derivative (along $x_j$) of $\UU_\rho$ in $\rn$, i.e., \eqref{e-weakder} holds for every $\varphi\in\cC_c^\infty(\rn,\rn)$, and likewise for $\UU_\tau$ and $\UU_{\zeta,i}$. For $\varepsilon>0$ consider $\eta_\varepsilon\in\cC^\infty(\rn,\R)$ such that
	\[
	\eta_\varepsilon(x)=0 \text{ for } |r|\le\frac\varepsilon2, \quad \eta_\varepsilon(x)=1 \text{ for } r\ge\varepsilon, \quad 0\le\eta_\varepsilon\le1, \quad |\nabla\eta_\varepsilon|\le\frac4\varepsilon.
	\]
	Let $\varphi\in\cC_c^\infty(\rn,\rn)$ and set $\varphi_\varepsilon:=\eta_\varepsilon\varphi\in\cC_c^\infty(\rn_*,\rn)$. There holds
	\begin{equation}\label{e-eps1}
	\begin{split}
	\int_{\rn}\partial_j\UU_\rho\cdot\varphi_\varepsilon\,dx & = -\int_{\rn}\UU_\rho\cdot\partial_j\varphi_\varepsilon\,dx\\
	& = -\int_{\rn}\eta_\varepsilon\UU_\rho\cdot\partial_j\varphi\,dx-\int_{\rn}\partial_j\eta_\varepsilon\UU_\rho\cdot\varphi.
	\end{split}\end{equation}
	From the dominated convergence theorem,
	\begin{equation}\label{e-eps2}
	\begin{split}
	\lim_{\varepsilon\to0^+}\int_{\rn}\partial_j\UU_\rho\cdot\varphi_\varepsilon\,dx=\int_{\rn}\partial_j\UU_\rho|_{\rn_*}\cdot\varphi\,dx\\
	\lim_{\varepsilon\to0^+}\int_{\rn}\eta_\varepsilon\UU_\rho\cdot\partial_j\varphi\,dx=\int_{\rn}\UU_\rho\cdot\partial_j\varphi\,dx.
	\end{split}\end{equation}
	Let $R>0$ such that $\varphi(x)=0$ for $|x|\ge R$ and set $\Omega_\varepsilon:=B_R\cap\{r<\varepsilon\}$ and observe that 
	$|\Omega_\varepsilon|\le C\varepsilon^2$ for some $C>0$ depending only on $N$ and $R$, thus
	\begin{equation}\label{e-eps3}
	\begin{split}
	\left|\int_{\rn}\partial_j\eta_\varepsilon\UU_\rho\cdot\varphi\,dx\right| & \le\|\varphi\|_\infty\frac4\varepsilon\int_{\Omega_\varepsilon}|\UU_\rho|\,dx\\
	& \le\|\varphi\|_\infty\frac4\varepsilon|\Omega_\varepsilon|^{\frac{N+2}{2N}}|\UU_\rho|_{2^*}\to0
	\end{split}\end{equation}
	as $\varepsilon\to0^+$. Letting $\varepsilon\to0^+$ in \eqref{e-eps1} and using \eqref{e-eps2} and \eqref{e-eps3} we have the desired result. Finally, similar computations hold for $\UU_\tau$ and $\UU_{\zeta,i}$.
\end{proof}

Let
\[
\cH:=\Set{\UU\colon\rn\to\rn|\UU(x)=O(y) \text{ uniformly in } z \text{ as } y\to0}.
\]

\begin{Prop}\label{P:Density}
There holds
\[
\DF=\overline{\cC_c(\rn,\rn)\cap\cC^\infty(\rn_*,\rn)\cap\cH\cap\DF},
\]
where the closure is taken in $\cD^{1,2}(\rn,\rn)$.
\end{Prop}
\begin{proof}
The inclusion `$\supset$' is obvious since $\DF$ is closed. Now let $\UU\in\DF$. Since $\UU\in\textup{Fix}(\SO)$, in view of Lemma \ref{L:Oapprox} there exists $\UU^n\in\cC^\infty_0(\rn,\rn)\cap\textup{Fix}(\SO)$ such that $\UU^n=(\UU_1^n,\dots,\UU_N^n)\to\UU$ in $\cD^{1,2}(\rn,\rn)$.

For every $n$, let $\UU_\rho^n,\UU_\tau^n,\UU_{\zeta,i}^n\in\cD^{1,2}(\rn,\rn)\cap\textup{Fix}(\SO)$ defined in Lemma \ref{L:dec} and associated with $\UU^n$, i.e.
\begin{itemize}
	\item $\UU_\rho^n(x)$ is the projection of $\UU^n(x)$ onto $\textup{span}\{(x_1,x_2,0)\}$,
	\item $\UU_\tau^n(x)$ is the projection of $\UU^n(x)$ onto $\textup{span}\{(-x_2,x_1,0)\}$,
	\item $\UU_{\zeta,i}^n(x)$ is the projection of $\UU^n(x)$ onto $\textup{span}\{e_i\}$.
\end{itemize}
In particular, $\UU_\rho^n,\UU_\tau^n,\UU_{\zeta,i}^n\in\cC^\infty(\rn_*,\rn)$, they vanish outside a sufficiently large ball in $\rn$ (in fact, $\UU_{\zeta,i}^n\in\cC_c^\infty(\rn,\rn)$) and $\UU_n(x)=\UU_\rho^n(x)+\UU_\tau^n(x)+\UU_{\zeta,i}^n(x)$ for every $x\in\rn_*$. Moreover, $\nabla\UU_\rho^n(x)$, $\nabla\UU_\tau^n(x)$, $\nabla\UU_{\zeta,i}^n(x)$ are pairwise orthogonal in $\R^{N\times N}\simeq\R^{N^2}$ for every $x\in\rn_*$.

This implies that $\UU_\tau^n\to\UU$ in $\cD^{1,2}(\rn,\rn)$, hence we are only left to prove that $\UU_\tau^n\in\cC_c(\rn,\rn)\cap\cH$.

Since $\UU^n\in\textup{Fix}(\SO)$, for every $g\in\SO(2)$ and every $z\in\R^{N-2}$
\[
\widetilde{g}\UU^n(0,0,z)=\UU^n\bigl(\widetilde{g}(0,0,z)\bigr)=\UU^n(0,0,z),
\]
which yields $\UU_1^n(0,0,z)=\UU_2^n(0,0,z)=0$ (just take $g=-I_2$). Moreover,
\[
\UU_\rho^n(x)=\frac{\UU_n\cdot(x_1,x_2,0)}{|(x_1,x_2)|^2}
\begin{pmatrix}
x_1\\
x_2\\
0
\end{pmatrix}
\, \text{ and } \, \UU_\tau^n(x)=\frac{\UU_n\cdot(-x_2,x_1,0)}{|(x_1,x_2)|^2}
\begin{pmatrix}
-x_2\\
x_1\\
0
\end{pmatrix},
\]
therefore, from the uniform continuity of $\UU^n$,
\[
\lim_{y\to0}\UU_\rho^n(x)=\lim_{y\to0}\UU_\tau^n(x)=0
\]
uniformly with respect to $z\in\R^{N-2}$. Hence we can extend $\UU_\rho^n$ and $\UU_\tau^n$ to $\rn$ by setting them equal to $0$ on $\{0\}\times\{0\}\times\R^{N-2}$ and obtain that $\UU_\rho^n,\UU_\tau^n\in\cC_c(\rn,\rn)$ and $\UU^n(x)=\UU_\rho^n(x)+\UU_\tau^n(x)+\sum_{i=3}^N\UU_{\zeta,i}^n(x)$ for every $x\in\rn$.

To prove that $\UU_\rho^n+\UU_\tau^n\in\cH$, first we notice that $\UU_\rho^n+\UU_\tau^n=\UU^n-\sum_{i=3}^N\UU_{\zeta,i}^n\in\cC_c^\infty(\R^3,\R^3)$ and, using Taylor's expansion,
\begin{equation*}\begin{split}
	\left(\UU^n-\sum_{i=3}^N\UU_{\zeta,i}^n\right)(x) = \, & \left(\UU^n-\sum_{i=3}^N\UU_{\zeta,i}^n\right)(0,z)\\
	& +\left[\nabla\left(\UU^n-\sum_{i=3}^N\UU_{\zeta,i}^n\right)(0,z)\right]
	\begin{pmatrix}
	y & 0
	\end{pmatrix}
	+o(y)\\
	= \, & \left[\nabla\left(\UU^n-\sum_{i=3}^N\UU_{\zeta,i}^n\right)(0,z)\right]
	\begin{pmatrix}
	y & 0
	\end{pmatrix}
	+o(y)
\end{split}\end{equation*} 
as $y\to 0$, thus $\UU_\rho^n+\UU_\tau^n\in\cH$. Finally, note that $|\UU_\tau^n|\le|\UU_\rho^n+\UU_\tau^n|$, whence $\UU_\tau^n\in\cH$.
\end{proof}

From now on, $u\colon\rn\to\R$ and $\UU\colon\rn\to\rn$ are two functions satisfying \eqref{e-uU}. An obvious consequence is that $|u|=|\UU|$.

\begin{Lem}\label{L:DivFree}
$\UU\in\DF$ if and only if $u\in X_{\SO}$; in such a case, $\nabla\cdot\UU=0$. If, moreover, $f$ satisfies the assumptions of Theorem \ref{T:ScalVec} and $h(x,\alpha w)=f(x,\alpha)w$ for $\alpha\in\R$ and $w=|(\xi_1,\xi_2)|^{-1}(-\xi_2,\xi_1,0)$, $\xi_1^2+\xi_2^2>0$, then $J(u)=E(\UU)$.
\end{Lem}
\begin{proof}
Suppose that $u\in X_{\SO}$. We show that the pointwise gradient a.e. of $\UU=(\UU_1,\UU_2,0)$ in $\rn$ is also the distributional gradient of $\UU$ in $\rn$. As a matter of fact, for $\partial_1\UU_1$ we have
\[\begin{split}
\int_{\rn}u(x)\frac{-x_2}{\sqrt{x_1^2+x_2^2}}\partial_1\varphi(x)\,dx\\
=\int_{\rn}\left(\partial_1u(x)\frac{x_2}{\sqrt{x_1^2+x_2^2}}-u(x)\frac{x_1x_2}{(x_1^2+x_2^2)^{3/2}}\right)\varphi(x)\,dx\\
=-\int_{\rn}\partial_1\left(u(x)\frac{-x_2}{\sqrt{x_1^2+x_2^2}}\right)\varphi(x)\,dx<\infty
\end{split}\]
for every $\varphi\in\cC_c^\infty(\rn)$ because $\displaystyle\int_{\rn}u(x)\frac{x_1x_2}{(x_1^2+x_2^2)^{3/2}}\varphi(x)\,dx<\infty$ for $u\in X$.	
For $\partial_1\UU_2$ similarly we get
\begin{equation*}\begin{split}
	\int_{\rn}u(x)\frac{x_1}{\sqrt{x_1^2+x_2^2}}\partial_1\varphi(x)\,dx\\
	=-\int_{\rn}\left(\partial_1u(x)\frac{x_1}{\sqrt{x_1^2+x_2^2}}+u(x)\left(\frac{1}{\sqrt{x_1^2+x_2^2}}-\frac{x_1^2}{(x_1^2+x_2^2)^{3/2}}\right)\right)\varphi(x)\,dx\\
	=-\int_{\rn}\partial_1\left(u(x)\frac{x_1}{\sqrt{x_1^2+x_2^2}}\right)\varphi(x)\,dx<\infty
\end{split}\end{equation*}
for every $\varphi\in\cC_c^\infty(\rn)$.
The remaining cases are similar. 

Now observe that $\UU\in L^{2^*}(\rn,\rn)\cap\cF$. Moreover,
\[
\partial_1\UU_1=\partial_1u\frac{x_2}{\sqrt{x_1^2+x_2^2}}-u\frac{x_1x_2}{(x_1^2+x_2^2)^{3/2}}\in L^2(\rn)
\]
and
\[
\partial_1\UU_2=-\partial_1u\frac{x_1}{\sqrt{x_1^2+x_2^2}}-u\left(\frac{1}{\sqrt{x_1^2+x_2^2}}-\frac{x_1^2}{(x_1^2+x_2^2)^{3/2}}\right)\in L^2(\rn)
\]
because $u\in X$.
Again, the remaining cases are similar and we infer $\UU\in\cD_\cF$.

Now suppose that $\UU\in\DF$ and, due to Proposition \ref{P:Density}, let $\BB_n\in\cC_c(\rn,\rn)\cap\cC^\infty(\rn_*,\rn)\cap\cH\cap\DF$ such that $\lim_n|\nabla\BB_n-\nabla\UU|_2=0$ and let $b_n\colon\rn\to\R$ be $\SO$-invariant such that $\BB_n$ and $b_n$ satisfy \eqref{e-uU}.

We prove that $b_n\in X_{\SO}$. Of course $b_n\in\cC^\infty(\rn_*)$ and, since $|\BB_n|=|b_n|$, $b_n\in\cC_c(\rn)\subset L^{2^*}(\rn)$ and $b_n(x)=O(y)$ uniformly with respect to $z$ as $y\to0$, therefore
\[
\int_{\rn}\frac{b_n^2}{r^2}\,dx<\infty.
\]
Moreover, $\nabla\BB_n\in L^2(\rn,\R^{N\times N})$, where
\[
\nabla\BB_n(x)=\frac{1}{\sqrt{x_1^2+x_2^2}}
\begin{pmatrix}
-x_2\\
x_1\\
0
\end{pmatrix}
\nabla b_n(x)^T+\frac{b_n(x)}{(x_1^2+x_2^2)^{3/2}}
\begin{pmatrix}
x_1x_2 & -x_1^2 & 0\\
x_2^2 & -x_1x_2 & 0\\
0 & 0 & 0
\end{pmatrix}
\]
and the second summand above is square-integrable because
\[\begin{split}
\left|\frac{1}{(x_1^2+x_2^2)^{3/2}}
\begin{pmatrix}
x_1x_2 & -x_1^2 & 0\\
x_2^2 & -x_1x_2 & 0\\
0 & 0 & 0
\end{pmatrix}
\right|_{\R^{N\times N}}\\
=\frac{1}{(x_1^2+x_2^2)^{3/2}}\Bigg|
\begin{pmatrix}
x_1\\
x_2\\
0
\end{pmatrix}
\begin{pmatrix}
x_2 & -x_1 & 0
\end{pmatrix}
\Bigg|_{\R^{N\times N}}=\frac{1}{\sqrt{x_1^2+x_2^2}},
\end{split}\]
where $|\cdot|_{\R^{N\times N}}$ stands for the matrix norm in $\R^{N\times N}$. It follows that $\nabla b_n\in L^2(\rn,\rn)$, thus $b_n\in X_{\SO}$.

Since $\lim_n|b_n-u|_{2^*}=\lim_n|\BB_n-\UU|_{2^*}=0$, it is enough to prove that $b_n$ is a Cauchy sequence in $X$, therefore we compute
\[\begin{split}
\|b_n-b_m\|^2 & =\int_{\rn}\nabla(b_n-b_m)\cdot\nabla(b_n-b_m)+\frac{(b_n-b_m)(b_n-b_m)}{r^2}\,dx\\
& =\int_{\rn}\nabla(\BB_n-\BB_m)\cdot\nabla(\BB_n-\BB_m)\,dx=|\nabla(\BB_n-\BB_m)|_2^2\to0
\end{split}\]
as $n,m\to\infty$.

Next, since $u\in X_{\SO}$, as in the first part we have that the pointwise divergence a.e. of $\UU$ is also the distributional divergence of $\UU$, hence $\nabla\cdot\UU=0$ follows from explicit computations


Finally, observe that if $u\in X_{\SO}$ and $\UU\in\DF$ satisfy \eqref{e-uU}, then $\|u\|^2=|\nabla\UU|_2^2=|\nabla\times\UU|_2^2$ and $F\bigl(x,u(x)\bigr)=H\bigl(x,\UU(x)\bigr)$ for a.e. $x\in\rn$.
\end{proof}

Note that $J$ is invariant under the action of $\SO$ and $E$ is invariant under the action defined in \eqref{e-Fix}: the former is trivial, while a proof of the latter can be found in \cite[Section 2]{AzzBenDApFor}.

\begin{proof}[Proof of Theorem \ref{T:ScalVec}]
	The first part follows directly from Lemma \ref{L:DivFree}. As in \cite[Proposition 1]{AzzBenDApFor}, we have that $\DF=\Set{\UU\in\textup{Fix}(\SO)|\cS\UU=\UU}$ and that $E(\cS\UU)=E(\UU)$ for every $\UU\in\textup{Fix}(\SO)$, where $\cS$ is defined in \eqref{e-S}.
	
	Finally, if $\VV\in\DF$ and $v\in X_{\SO}$ satisfy \eqref{e-uU}, then arguing as in Lemma \ref{L:DivFree} there holds
	\[
	\int_{\rn}\nabla\times\UU\cdot\nabla\times\VV\,dx=\int_{\rn}\nabla\UU\cdot\nabla\VV\,dx=\int_{\rn}\nabla u\cdot\nabla v+\frac{uv}{r^2}\,dx
	\]
	and
	\[\begin{split}
	\int_{\R^3}h\bigl(x,\UU(x)\bigr)\cdot\VV(x)\,dx & =\int_{\rn}h\biggl(x,\frac{u}{r}
	\left(\begin{smallmatrix}
	-x_2\\
	x_1\\
	0
	\end{smallmatrix}\right)
	\biggr)\cdot\frac{v(x)}{r}
	\left(\begin{smallmatrix}
	-x_2\\
	x_1\\
	0
	\end{smallmatrix}\right)
	\,dx\\
	& =\int_{\rn}f\left(x,u(x)\right)\frac{1}{r}
	\left(\begin{smallmatrix}
	-x_2\\
	x_1\\
	0
	\end{smallmatrix}\right)
	\cdot\frac{v(x)}{r}
	\left(\begin{smallmatrix}
	-x_2\\
	x_1\\
	0
	\end{smallmatrix}\right)
	\,dx\\
	& =\int_{\rn}f\bigl(x,u(x)\bigr)v(x)\,dx,
	\end{split}\]
	thus the conclusion follows from Theorem \ref{T:Palais}.
\end{proof}

\section{The Sobolev noncritical case}

In this section, we prove Theorems \ref{T:ExMul} and \ref{T:Ex}. Throughout this section we assume $f$ satisfies (F1) and (F2). The following lemma is proved in \cite[Proposition A.2]{MederskiZeroMass}.

\begin{Lem}\label{L:Lions}
	Suppose that $u_n\in\cD^{1,2}(\R^N)$ is bounded and $\SO$-invariant and for every $R>0$
	\begin{equation}\label{e-LionsCond}
	\lim_n\sup_{z\in \R^{N-K}} \int_{B((0,z),R)} u_n^2\,dx=0.
	\end{equation}
	Then
	\[
	\lim_n\int_{\R^N}\Phi(u_n)\,dx=0
	\]
	for every continuous function $\Phi\colon\R\to [0,\infty[$ such that
	\begin{equation}\label{e-LionsPhi}
	\displaystyle\lim_{s\to 0}\frac{\Phi(s)}{|s|^{2^*}}=\lim_{|s|\to\infty}\frac{\Phi(s)}{|s|^{2^*}}=0.
	\end{equation}
\end{Lem}

We need the following results as well.

\begin{Lem}\label{L:pqDec}
	Let $1\le p\le 2^*\le q<\infty$. If $u\in L^{2^*}(\R^N)$, then
	\[
	|u\chi_{\{|u|\le 1\}}|_q^q,\;|u\chi_{\{|u|>1\}}|_p^p\le|u|_{2^*}^{2^*}.
	\]
\end{Lem}
\begin{proof}
	There holds
	\[
	\int_{\R^N}|u|^q\chi_{\{|u|\le 1\}}\,dx\le\int_{\R^N}|u|^{2^*}\chi_{\{|u|\le1\}}\,dx\le|u|_{2^*}^{2^*}
	\]
	and
	\[
	\int_{\R^N}|u|^p\chi_{\{|u|>1\}}\,dx\le\int_{\R^N}|u|^{2^*}\chi_{\{|u|>1\}}\,dx\le|u|_{2^*}^{2^*}.\qedhere
	\]
\end{proof}

\begin{Lem}\label{L:Lions2}
	Suppose that $u_n,v_n\in\cD^{1,2}(\R^N)$ are bounded and $\SO$-invariant and that $u_n$ satisfies \eqref{e-LionsCond} for every $R>0$. Then
	\[
	\lim_n\int_{\R^N}|f(x,v_n)u_n|\,dx=0.
	\]
\end{Lem}
\begin{proof}
	Let $1<p<2^*<q<\infty$, define $\Phi(t):=\int_0^{|t|}\min\{s^{p-1},s^{q-1}\}\,ds$, and note that $\Phi$ satisfies \eqref{e-LionsPhi}. (F2) implies that for every $\varepsilon>0$ there exists $C_\varepsilon>0$ such that for every $t\in\R$ and a.e. $x\in\R^N$ we have $|f(x,t)|\le\varepsilon|t|^{2^*-1}+C_\varepsilon|\Phi'(t)|$. Moreover
	\[\begin{split}
	\int_{\R^N}|\Phi'(v_n)u_n|\,dx & = \int_{\R^N}|\Phi'(v_n)u_n|\chi_{\{|u_n|>1\}}\,dx+\int_{\R^N}|\Phi'(v_n)u_n|\chi_{\{|u_n|\le 1\}}\,dx\\
	& =:A_n+B_n.
	\end{split}\]
	Concerning the first integral $A_n$, Lemmas \ref{L:Lions} and \ref{L:pqDec} imply that, for some $C>0$,
	\[\begin{split}
	A_n&= \int_{\R^N}|v_n|^{p-1}\chi_{\{|v_n|>1\}}|u_n|\chi_{\{|u_n|>1\}}\,dx+\int_{\R^N}|v_n|^{q-1}\chi_{\{|v_n|\le 1\}}|u_n|\chi_{\{|u_n|>1\}}\,dx\\
	& \le\left(|v_n\chi_{\{|v_n|>1\}}|_p^{p-1}+\big||v_n|^{q-1}\chi_{\{|v_n|\le 1\}}\big|_{{(p-1)/p}}\right)|u_n\chi_{\{|u_n|>1\}}|_p\\
	& \le C\left(|v_n\chi_{\{|v_n|>1\}}|_p^{p-1}+|v_n\chi_{\{|v_n|\le 1\}}|_q^{q(p-1)/p}\right)\left(\int_{\rn}\Phi(u_n)\,dx\right)^\frac{1}{p}\\
	& \le C\sup_k\|v_k\|^\frac{2^*(p-1)}{p}\left(\int_{\rn}\Phi(u_n)\,dx\right)^\frac{1}{p}\to0
	\end{split}\]
	because $\bigl(|v_n|^{q-1}\bigr)^{\frac{p}{p-1}}\chi_{\{|v_n|\le 1\}}\le|v_n|^q\chi_{\{|v_n|\le 1\}}$.
	
	Similar computations hold for the second integral $B_n$, therefore we have
	\[
	\limsup_n\int_{\rn}|f(x,v_n)u_n|\,dx\le\varepsilon\sup_k|v_k|_{2^*}^{2^*-1}\sup_k|u_k|_{2^*}^{2^*}
	\]
	and conclude letting $\varepsilon\to0^+$.
\end{proof}

In order to prove Theorem \ref{T:ExMul} we aim to use the abstract critical point theory from Section \ref{S:abstract}, in particular Theorems \ref{T:Link1} and \ref{T:CrticMulti} \textit{(b)}. We need to prove that assumptions (I1)--(I8), (G), and $(M)_0^\beta$ for every $\beta>0$ are satisfied. For simplicity, we set
\[
I(u):=\int_{\rn}F(x,u)\,dx \quad \text{for }u\in X_{\SO},
\]
while
\[
\cN:=\Set{u\in X_{\SO}\setminus\{0\}|J'(u)u=0}
\]
is the Nehari manifold. Such conditions in our setting read as follows.
\begin{itemize}
	\item [(I1)] $I\in\cC^1(X_{\SO})$ and $I(u)\ge I(0)=0$ for every $u\in X_{\SO}$.
	\item [(I2)] If $u_n\to u$, then $\liminf_nI(u_n)\ge I(u)$.
	\item [(I3)] If $u_n\to u$ and $I(u_n)\to I(u)$, then $u_n\to u$.
	\item [(I4)] $\|u\|+I(u)\to\infty$ as $\|u\|\to\infty$.
	\item [(I6)] There exists $\delta>0$ such that $\displaystyle\inf_{u\in X_{\SO},\|u\|=\delta}J(u)>0$.
	\item [(I7)] If $t_n\to\infty$ and $u_n\to u_0\ne0$, then $\displaystyle\frac{I(u_n)}{t_n^2}\to\infty$.
	\item [(I8)] $\displaystyle\frac{t^2-1}{2}I'(u)u+I(u)-I(tu)\le0$ for every $u\in\cN$ and every $t\ge0$.
	\item [(G)] $\Z^{N-K}$ acts on $X_{\SO}$ by isometries and $(\Z^{N-K}*u)\setminus\{u\}$ is bounded away from $u$ for every $u\in X_{\cO}$. Moreover $J$ is $\Z^{N-K}$-invariant and $X_{\SO}$ is $\Z^{N-K}$-invariant.
	\item [$(M)_0^\beta$]
	\begin{itemize}
		\item [(a)] There exists $M_\beta>0$ such that $\limsup_n\|u_n\|\le M_\beta$ for every $u_n\in X_{\SO}$ such that $0\le\liminf_nJ(u_n)\le\limsup_nJ(u_n)\le\beta$ and $\lim_n(1+\|u_n\|)J'(u_n)=0$.
		\item [(b)] If $J$ has finitely many critical orbits, then there exists $m_\beta>0$ such that, if $u_n,v_n\in X_{\SO}$ are as above and there exists $n_0\ge1$ such that $\|u_n-v_n\|<m_\beta$ for $n\ge n_0$, then $\liminf_n\|u_n-v_n\|=0$.
	\end{itemize}
\end{itemize}

We omitted (I5) because it is an empty condition. The action of $\Z^{N-K}$ on $X_{\SO}$ is given as follows: $z\ast u(x):=u\bigl(x+(0,z)\bigr)$ for $z\in \Z^{N-K}$ and $u\in X_{\SO}$. $\Z^{N-K}\ast u$ is called the {\em orbit} of $u$ and if, in addition, $u$ is a critical point of $J$, then $\Z^{N-K}\ast u$ is a {\em critical orbit}.

Note that (I1)--(I4) and (G) are obviously satisfied (recall that (F4) implies $F\ge0$), while (I6)--(I8) and $(M)_0^\beta$ will be verified in the next lemmas. We begin with (I6)--(I8).

\begin{Lem}\label{L:I678}
	(a) There exists $\delta_0>0$ such that for every $0<\delta<\delta_0$
	\[
	\inf\Set{J(u)|u\in X_{\SO} \text{ and } \|u\|=\delta}>0
	\]
	(b) Suppose $f$ satisfies (F3). If $t_n\to\infty$ and $u_n\to u_0\in X_{\SO}\setminus\{0\}$, then
	\[
	\lim_n\frac{1}{t_n^2}\int_{\rn}F(x,t_nu_n)\,dx=\infty.
	\]
	(c) Suppose $f$ satisfies (F4). For every $u\in X_{\SO}$ and every $t\ge0$ $$\frac{t^2-1}{2}\int_{\rn}f(x,u)u\,dx+\int_{\rn}F(x,u)\,dx-\int_{\rn}F(x,tu)\,dx\le0.$$
\end{Lem}
\begin{proof}
	\textit{(a)} From (F2) and the embeddings $X\hookrightarrow\cD^{1,2}(\rn)\hookrightarrow L^{2^*}(\rn)$, there exists $C>0$ such that for every $u\in X_{\SO}$
	\[
	\int_{\rn}F(x,u)\,dx\le C|u|_{2^*}^{2^*}\le C\|u\|^{2^*},
	\]
	whence
	\[
	J(u)\ge\frac12\|u\|^2-C\|u\|^{2^*}
	\]
	and the statement holds true for $\delta_0\ll1$.
	
	\textit{(b)} Since $X_{\SO}$ is locally compactly embedded into $L^2(\R^N)$, up to a subsequence $u_n\to u_0$ a.e. in $\rn$. Moreover, there exists $\varepsilon>0$ such that $\limsup_n|\Omega_n|>0$, where $\Omega_n:=\Set{x\in\rn|\lVert u_n(x)\rVert\ge\varepsilon}$, for otherwise $u_n\to0$ in measure and consequently, up to a subsequence, a.e. in $\rn$. From (F3)
	\[
	\frac{1}{t_n^2}\int_{\rn}F(x,t_nu_n)\,dx=\int_{\rn}\frac{F(x,t_nu_n)}{t_n^2|u_n|^2}|u_n|^2\,dx\ge\varepsilon^2\int_{\Omega_n}\frac{F(x,t_nu_n)}{t_n^2|u_n|^2}\,dx\to\infty
	\]
	as $n\to\infty$.
	
	\textit{(c)} For fixed $x\in\rn$ and $u\in\R$ we prove that $\phi(t)\le0$ for every $t\ge0$, where $$\phi(t):=\frac{t^2-1}{2}f(x,u)u+F(x,u)-F(x,tu).$$ This is trivial for $u=0$, so suppose $u\ne0$.
	Note that $\phi(1)=0$, so it is enough to prove that $\phi$ is nondecreasing on $[0,1]$ and nonincreasing on $[1,\infty[$. This is the case in view of (F4) and because
	\[
	\phi'(t)=tf(x,u)u-f(x,tu)u=t|u|u\biggl(\frac{f(x,u)}{|u|}-\frac{f(x,tu)}{|tu|}\biggr)
	\]
	for $t>0$, therefore $\phi(t)\le0$ for every $t\ge0$ as $\phi\in\cC^1([0,\infty[)$.
\end{proof}

The following lemma shows that $(M)_0^\beta$ holds for every $\beta>0$.

\begin{Lem}\label{L:Mbeta}
	Suppose $f$ satisfies (F3) and (F4).
	\begin{itemize}
		\item [(a)] For every $\beta>0$ there exists $M_\beta>0$ such that $\limsup_n\|u_n\|\le M_\beta$ for every $u_n\subset X_{\SO}$ such that $J(u_n)\le\beta$ for $n\gg1$ and $\lim_n(1+\|u_n\|)J'(u_n)=0$.
		\item [(b)] If the number of critical orbits of $J$ is finite, then there exists $\kappa>0$ such that, if $u_n,v_n\in X_{\SO}$ are as above for some $\beta>0$ and there exists $n_0\ge1$ such that $\|u_n-v_n\|<\kappa$ for $n\ge n_0$, then $\lim_n\|u_n-v_n\|=0$.
	\end{itemize}
\end{Lem}
\begin{proof}
	\textit{(a)} Let $u_n\in X_{\SO}$ as in the assumptions. Suppose that $u_n$ is unbounded and define $\bar{u}_n:=u_n/\|u_n\|$. Passing to a subsequence we can assume that $\lim_n\|u_n\|=\infty$. Similarly to the proof of Lemma \ref{L:Lions2}, for every $\varepsilon>0$ there exists $C_\varepsilon>0$ such that
	\[
	\int_{\rn}F(x,\bar{u}_n)\,dx\leq\varepsilon|\bar{u}_n|_{2^*}^{2^*}+ C_\varepsilon\Phi(\bar{u}_n),
	\]
	where $\Phi$ is defined therein. If $\bar{u}_n$ satisfies \eqref{e-LionsCond} for every $R>0$ (hence the same holds for $s\bar{u}_n$, $s\ge0$), then in view of Lemma \ref{L:Lions}
	\[
	\limsup_n\int_{\rn}F(x,s\bar{u}_n)\,dx\leq \varepsilon s^2\limsup_n|\bar{u}_n|_{2^*}^{2^*}
	\]
	for every $\varepsilon>0$, hence $\lim_n\int_{\rn}F(x,s\bar{u}_n)\,dx=0$.
	Then applying Lemma \ref{L:I678} \textit{(c)} with $u=u_n$ and $t=s/\|u_n\| =: t_n$ we obtain, up to a subsequence, that for every $s\ge0$
	\[\begin{split}
	\beta&\ge\limsup_nJ(u_n)\ge\limsup_nJ(s\bar{u}_n)-\lim_n\frac{t_n^2-1}{2}J'(u_n)u_n\\
	&=\limsup_nJ(s\bar{u}_n)\ge\frac{s^2}{2}-\lim_n\int_{\rn}F(x,s\bar{u}_n)\,dx=\frac{s^2}{2},
	\end{split}\]
	a contradiction. Hence $\lim_n\int_{B((0,z_n),R)}|\bar{u}_n|^2\,dx>0$ up to a subsequence for some $R>\sqrt{N-K}$ and $z_n\subset\Z^{N-K}$, where $z_n$ maximizes $z\mapsto\int_{B((0,z),R)}|u_n|^2\,dx$. Exploiting the $\Z^{N-K}$-invariance, we can assume that
	\[
	\int_{B_R}|\bar{u}_n|^2\,dx\ge c
	\]
	for $n\gg1$ and some $c>0$.
	
	There follows that there exists $\bar{u}\in X_{\SO}\setminus\{0\}$ such that, up to a subsequence, $\bar{u}_n\rightharpoonup\bar{u}$ in $X$ and $\bar{u}_n\to\bar{u}$ in $L^2_\textup{loc}(\rn)$ and a.e. in $\rn$.
	
	From (F4), $2J(u_n)-J'(u_n)u_n=\int_{\rn}f(x,u_n)u_n-2F(x,u_n)\,dx\ge0$, thus $J(u_n)$ is bounded and due to (F3) we obtain
	\[
	o(1)=\frac{J(u_n)}{\|u_n\|^2}\le \frac12-\int_{\rn}\frac{F(x,u_n)}{|u_n|^2}|\bar{u}_n|^2\,dx\to-\infty,
	\]
	which is a contradiction.	
	This shows that $u_n$ is indeed bounded. If by contradiction there exists no upper bound $M_\beta$, then for every $k\in\mathbb{N}$ there exists $u_n^k\in X_{\SO}$ as in the statement such that $\limsup_n\|u_n^k\|>k$. Then, as in Corollary \ref{C:Ceramibdd}, it is easy to build a subsequence $u_{n(k)}^k\in X_{\SO}$ that is unbounded, again a contradiction.
	
	\textit{(b)} Assume that there are finitely many critical orbits of $J$. From (G) we easily see that $$\kappa:=\inf\big\{\|u-v\|:u\ne v\text{ and }J'(u)=J'(v)=0\big\}>0.$$
	
	Let $u_n,v_n$ be as in the statement. In view of \textit{(a)}, they are bounded. If $\int_{\rn}f(x,u_n)(u_n-v_n)\,dx$ or $\int_{\rn}f(x,v_n)(u_n-v_n)\,dx$ do not converge to $0$, then in view of Lemma \ref{L:Lions2} and the $\Z^{N-K}$-invariance there exist $R>\sqrt{N-K}$ and $\varepsilon>0$ such that $$\int_{B_R}|u_n-v_n|^2\,dx\ge\varepsilon.$$
	We can assume that $u_n\rightharpoonup u$ and $v_n\rightharpoonup v$ in $X$, so $u\ne v$. Hence $J'(u)=J'(v)=0$ and consequently $$\liminf_n\|u_n-v_n\|\ge\|u-v\|\ge\kappa,$$ in contrast with the assumptions.
	
	Therefore it follows that
	\[
	\lim_n\int_{\rn}f(x,u_n)(u_n-v_n)\,dx=\lim_n\int_{\rn}f(x,v_n)(u_n-v_n)\,dx=0
	\]
	and, finally,
	\begin{equation*}\begin{aligned}
	\|u_n-v_n\|^2= J'(u_n)(u_n-v_n)-J'(v_n)(u_n-v_n)\\
	+\int_{\rn}\bigl(f(x,u_n)-f(x,v_n)\bigr)(u_n-v_n)\,dx\to0.\qedhere
	\end{aligned}\end{equation*}
\end{proof}

\begin{proof}[Proof of Theorem \ref{T:ExMul}]
	Note that $\cN$ contains all the nontrivial critical points of $J$. Applying Theorem \ref{T:Link1} we obtain a Cerami sequence $u_n\in X_{\SO}$ at the level $c:=\inf_\cN J>0$. Lemma \ref{L:Mbeta}\textit{(a)} implies that there exists $u\in X_{\SO}$ such that $u_n\rightharpoonup u$ up to a subsequence, thus $J'(u)=0$.
	
	If by contradiction $\int_{\rn}f(x,u_n)u_n\,dx\to0$, then similarly to the proof of Lemma \ref{L:Mbeta} \textit{(b)} we infer that $u_n\to0$, in contrast with $J(u_n)\to c$. Hence, again similarly to the proof of Lemma \ref{L:Mbeta} \textit{(b)}, $u\ne0$.
	
	Fatou's Lemma and (F4) imply
	\[\begin{split}
	c & = \lim_nJ(u_n) = \lim_n\left(J(u_n) - \frac12J'(u_n)u_n\right)\\
	& = \lim_n\int_{\R^N}\frac12 f(x,u_n)u_n - F(x,u_n)\,dx \ge\int_{\R^N}\frac12 f(x,u)u-F(x,u)\,dx\\
	& = J(u) - \frac12J'(u)u = J(u) \ge c.
	\end{split}\]
	
	Now assume $f$ is odd in $u$, which implies that $J$ is even. The existence of infinitely many geometrically distinct critical points of $J$ follows directly from Theorem \ref{T:CrticMulti} \textit{(b)}. As for the fact that the ground state solution is nonnegative, since $J(v)=J(|v|)$ and $J'(v)v=J'(|v|)|v|$ for every $v\in X_{\SO}$, we obtain that $|u|$ is again a ground state solution.
\end{proof}

\begin{Lem}\label{L:positive}
	Suppose that $f$ does not depend on $y$ and satisfies (F5). Then there exists $w\in X_{\SO}$ such that $\int_{\rn}F(z,w)\,dx>\frac{1}{2}\int_{\R^N}|\nabla_z w|^2\,dx$.
\end{Lem}
\begin{proof}
	For every $R\ge3$ we define a continuous function $\phi_R\colon\R\to \R$ such that $\phi_R(t)=0$ for $|t|<1$ and for $|t|>R+1$, $\phi_R(t)=1$ for $2\leq |t|\leq R$ and $\phi_R$ is affine for $1\le|t|\le2$ and for $R\le|t|\le R+1$. Then let $w_R(x):=u_0\phi_R(|y|)\phi_R(|z|)$. Observe that $w_R\in X_{\SO}$ and there exist constants $C_1,C_2>0$ such that
	\begin{equation*}\begin{split}
		\int_{\R^N} F(z,w_R)\,dx\geq \, & C_1R^N\essinf_{z\in\R^{N-K}}F(z,u_0)\\
		& -C_2R^{N-1}\sup_{R\leq |u|\leq R+1}\esssup_{z\in\R^{N-K}}F(z,u)\\
		& -C_2\sup_{1\leq |u|\leq 2}\esssup_{z\in\R^{N-K}}F(z,u)
	\end{split}\end{equation*}
	and
	$$\int_{\R^N}|\nabla_z w|^2\leq C_2R^{N-1}.$$
	Then the statement holds true for $R\gg1$. 
\end{proof}

\begin{proof}[Proof of Theorem \ref{T:Ex}]
	First we prove that $J$ has the mountain pass geometry. Let $w\in X_{\SO}$ as in Lemma \ref{L:positive}. From
	Lemma \ref{L:I678} \textit{(a)},
	there exists $0<\delta<\|w\|$ such that $\inf\Set{J(u)|u\in X_{\cO} \text{ and } \|u\|=\delta}>0$.
	Moreover, for every $\lambda>0$ we have
	\[\begin{split}
	J\bigl(w(\lambda\cdot,\cdot)\bigr)= \, & \frac{1}{2\lambda^{K-2}}\int_{\R^N}|\nabla_y w|^2+\frac{a}{r^2}|w|^2\,dx\\
	& +\frac{1}{\lambda^K}\int_{\rn}\frac{1}{2}|\nabla_z w|^2-F(z,w)\,dx\to-\infty
	\end{split}\]
	as $\lambda\to 0^+$. Consequently, the existence of a Palais--Smale sequence $u_n\in X$ for $J|_{X_{\SO}}$ at the mountain pass level $c>0$ follows. Such a sequence is bounded because (F5) holds: for every $n\gg1$
	\[\begin{split}
	c+1 & +\|u_n\|\ge J(u_n)-\frac1\gamma J'(u_n)(u_n)\\
	& =\left(\frac12-\frac1\gamma\right)\|u_n\|^2+\int_{\rn}\frac1\gamma f(u_n)u_n-F(u_n)\,dx\ge\left(\frac12-\frac1\gamma\right)\|u_n\|^2.
	\end{split}\]
	
	Now, suppose by contradiction that \eqref{e-LionsCond} holds for every $R>0$. Fix $R>\sqrt{N-K}$ such that \eqref{e-LionsCond} holds with the supremum being taken over $\Z^{N-K}$. Since $F$ and $(z,u)\mapsto f(z,u)u$ satisfy \eqref{e-LionsPhi} uniformly with respect to $z\in\R^{N-K}$, from Lemma \ref{L:Lions} and arguing as in Lemma \ref{L:Lions2} we obtain
	\[
	c=\lim_nJ(u_n)-\frac12J'(u_n)u_n=\lim_n\int_{\R^N}\frac12 f(z,u_n)u_n-F(z,u_n)\,dx=0,
	\]
	which is a contradiction. Then there exist $R>\sqrt{N-K}$ and $\varepsilon>0$ such that, up to a subsequence,
	\begin{equation}\label{e-bddaway}
	\int_{B((0,z_n),R)}u_n^2\,dx\ge\varepsilon
	\end{equation}
	where $z_n\in\Z^{N-K}$ maximizes $z\mapsto\int_{B((0,z),R)}|u_n|^2\,dx$.
	Since $J$ is invariant with respect to $\Z^{N-K}$ translations, up to replacing $u_n$ with $u_n(\cdot-z_n)$ we can suppose that $z_n=0$. Since $u_n$ is bounded, there exists $u\in X_{\SO}$ such that $u_n\rightharpoonup u$ in $X$, which in turn implies that $J'(u)=0$ and that $u_n\to u$ in $L^2\bigl(B_R\bigr)$ and a.e. in $\rn$; in particular, $u\ne 0$ because \eqref{e-bddaway} holds.
\end{proof}

\begin{proof}[Proof of Corollary \ref{C:V1}]
	The proof follows from Theorems \ref{T:ExMul}, \ref{T:Ex} and \ref{T:ScalVec}.
\end{proof}

\section{The Sobolev critical case}\label{S:SC}

In this section, we prove Theorem \ref{T:crit}; moreover, not only do we assume $K=2$ (for the same reasons as in Section \ref{S:equiv}) but we also assume $N=3$ (which a fortiori forces us to have $K=2$). The reason is that in Lemma \ref{L:decom} below we make use of the explicit structure of $\SO(2)$, i.e.
\[
\SO(2)=\Set{
\begin{pmatrix}
\cos\alpha & -\sin\alpha\\
\sin\alpha & \cos\alpha
\end{pmatrix}
|\alpha\in\R}.
\]

We recall that throughout this section $2^*=6$,
\begin{equation*}\begin{split}
	E(\UU)&=\frac12\int_{\R^3}|\nabla\times\UU|^2\,dx-\frac{1}{6}\int_{\R^3}|\UU|^6\,dx,\\
	J(u)&=\frac12\int_{\R^3}|\nabla u|^2+\frac{u^2}{r^2}\,dx-\frac{1}{6}\int_{\R^3}u^6\,dx.
\end{split}\end{equation*}

The only exception to the restriction $N=3$ is Corollary \ref{C:V2}, which we prove here.

\begin{proof}[Proof of Corollary \ref{C:V2}]
The proof follows from \cite[Theorem 1]{BadGuiRol} and Theorem \ref{T:ScalVec}.
\end{proof}

Let $\pi\colon\mbS^3\setminus\{Q\}\to\R^3$ be the stereographic projection, where $Q=(0,0,0,1)$ is the north pole, and let
\[
\psi\colon x\in\R^3\mapsto\sqrt{\frac{2}{|x|^2+1}}\in\R.
\]
Explicitly,
\[\begin{split}
\pi(\xi)&=\frac{1}{1-\xi_4}(\xi_1,\xi_2,\xi_3),\quad\xi=(\xi_1,\xi_2,\xi_3,\xi_4),\\
\pi^{-1}(x)&=\frac{1}{|x|^2+1}(2x_1,2x_2,2x_3,|x|^2-1),\quad x= (x_1,x_2,x_3).
\end{split}\]

Recall that $\widetilde{g}=\left(\begin{smallmatrix}
g & 0\\
0 & 1
\end{smallmatrix}\right)$ for $g\in\SO(2)$, $\SO=\Set{\widetilde{g}|g\in\SO(2)}$, and $\cD_{\SO(2\times2)}$ is the subspace of $\cD^{1,2}(\R^3,\R^3)$ of $\SO(2\times2)$-symmetric vector fields according to Definition \ref{D:sym}.

\begin{Lem}\label{L:equisym}
	$\cD_{\SO(2\times2)}\subset\textup{Fix}(\SO)$.
\end{Lem}
\begin{proof}
	Let $g_1\in\SO(2)$ and define $g:=(g_1,I_2)\in\SO(2)\times\SO(2)$. Note that
	$$g\pi^{-1}(x)=\pi^{-1}(\widetilde{g_1}x)$$
	for every $x\in\R^3$, therefore
	\[
	\widetilde{g_1}\UU(x)=\frac{\psi(x)}{\psi\Bigl(\pi\bigl(g\pi^{-1}(x)\bigr)\Bigr)}\UU\Bigl(\pi\bigl(g\pi^{-1}(x)\bigr)\Bigr)=\frac{\psi(x)}{\psi(\widetilde{g_1}x)}\UU(\widetilde{g_1}x)=\UU	(\widetilde{g_1}x).\qedhere
	\]
\end{proof}

\begin{Lem}\label{L:compact}
	The embedding
	$\cD_{\SO(2\times2)}\hookrightarrow L^6(\R^3,\R^3)$ is compact.
\end{Lem}
\begin{proof}
	For every $\UU\in\cD_{\SO(2\times2)}$ define $\VV(\xi):=\frac{\UU(\pi(\xi))}{\psi(\pi(\xi))}$ for $\xi\in\mathbb{S}^3\setminus\{Q\}$. Note that $\VV\in H^1(\mathbb{S}^3,\R^3)$ and, similarly as in \cite[Lemma 3.1]{ClPis}, $|\nabla \UU|_2=\|\VV\|_{H^1(\mathbb{S}^3,\R^3)}$ and $|\UU|_{6}=\|\VV\|_{L^6(\rr,\rr)}$,
	where 
	$$\|\VV\|^2_{H^1(\mathbb{S}^3,\R^3)}=\int_{\mathbb{S}^3}|\nabla_\mathfrak{g}\VV|^2+\frac34|\VV|^2\,dV_\mathfrak{g}$$ is the norm in 
	$H^1(\mathbb{S}^3,\R^3)$ and $\nabla_\mathfrak{g}$ is the gradient on $\mathbb{S}^3$ \cite{Aubin,ONeill}.
	Therefore $\UU\mapsto\VV$ is an isometric isomorphism between $\cD^{1,2}(\R^3,\R^3)$ and $H^1(\mathbb{S}^3,\R^3)$ and between $L^6(\R^3,\R^3)$ and $L^6(\mbS^3,\R^3)$. Note that, since $\UU$ is $\SO(2\times2)$-symmetric, then $\VV(g\xi)=\widetilde{g_1}\VV(\xi)$ for every $g=(g_1,g_2)\in\SO(2)\times\SO(2)$ and, consequently, $|\VV|$ is $\SO(2)\times\SO(2)$-invariant, or equivalently $\cO(2)\times\cO(2)$-invariant. 
	
	Let $\UU_n\in\cD_{\SO(2\times2)}$ such that $\UU_n\rightharpoonup0$ in $\cD_{\SO(2\times2)}$. Then $\VV_n\rightharpoonup0$ in $H^1(\mathbb{S}^3,\R^3)$ and, up to a subsequence, $\VV\to0$ a.e. in $\mathbb{S}^3$; this implies that $|\VV_n|\rightharpoonup0$ in $H^1(\mathbb{S}^3)$ and so, in view of \cite[Lemma 5]{Ding}, $|\VV_n|\to0$ in $L^6(\mathbb{S}^3)$. Hence $\VV_n\to0$ in $L^6(\mathbb{S}^3,\R^3)$ and so $\UU_n\to0$ in $L^6(\R^3,\R^3)$.
\end{proof}

For $\UU\in\textup{Fix}(\SO)$ recall from Lemma \ref{L:dec} the definition of $\UU_\rho$, $\UU_\tau$, and $\UU_{\zeta,i}$, $i\in\{3,\dots,N\}$.

\begin{Lem}\label{L:decom}
	If $\UU\in\cD_{\SO(2\times2)}$, then $\UU_\rho,\UU_\tau,\UU_\zeta\in\cD_{\SO(2\times2)}$.
\end{Lem}
\begin{proof}
	We first prove that $\UU_\tau\in\cD_{\SO(2\times2)}$. Let $\alpha_i\in\R$, $g_i=
	\left(\begin{smallmatrix}
	\cos\alpha_i & -\sin\alpha_i\\
	\sin\alpha_i & \cos\alpha_i
	\end{smallmatrix}\right)
	\in\SO(2)$, $i\in\{1,2\}$, and set $g=
	\left(\begin{smallmatrix}
	g_1 & 0\\
	0 & g_2
	\end{smallmatrix}\right)
	$. We want to prove that
	\begin{equation}\label{e-tau}
	\frac{\psi(x)}{\psi\Bigl(\pi\bigl(g\pi^{-1}(x)\bigr)\Bigr)}\UU_\tau\Bigl(\pi\bigl(g\pi^{-1}(x)\bigr)\Bigr)=\widetilde{g_1}\UU_\tau(x)
	\end{equation}
	provided
	\[
	\frac{\psi(x)}{\psi\Bigl(\pi\bigl(g\pi^{-1}(x)\bigr)\Bigr)}\UU\Bigl(\pi\bigl(g\pi^{-1}(x)\bigr)\Bigr)=\widetilde{g_1}\UU(x).
	\]
	
	We compute the two sides of \eqref{e-tau} separately. We use the convention that $\R^3=\R^{3\times1}$ and treat the scalar product in $\R^3$ as matrix multiplication.
	
	As for the right-hand side we have
	\[\begin{split}
	\widetilde{g_1}\UU_\tau(x) & =\frac{\widetilde{g_1}
	\left(\begin{smallmatrix}
	-x_2\\
	x_1\\
	0
	\end{smallmatrix}\right)
	\UU^T(x)
	\left(\begin{smallmatrix}
	-x_2\\
	x_1\\
	0
	\end{smallmatrix}\right)
	}{x_1^2+x_2^2}=\frac{
	\left(\begin{smallmatrix}
	-x_2\cos\alpha_1-x_1\sin\alpha_1\\
	-x_2\sin\alpha_1+x_1\cos\alpha_1\\
	0
	\end{smallmatrix}\right)
	\UU^T(x)
	\left(\begin{smallmatrix}
	-x_2\\
	x_1\\
	0
	\end{smallmatrix}\right)
	}{x_1^2+x_2^2}\\
	& =\frac{-x_2\UU_1(x)+x_1\UU_2(x)}{x_1^2+x_2^2}
	\left(\begin{smallmatrix}
	-x_2\cos\alpha_1-x_1\sin\alpha_1\\
	-x_2\sin\alpha_1+x_1\cos\alpha_1\\
	0
	\end{smallmatrix}\right).
	\end{split}\]
	
	Let us write $\pi=
	\left(\begin{smallmatrix}
	\pi_1\\
	\pi_2\\
	\pi_3
	\end{smallmatrix}\right)
	$. As for the left-hand side we have
	\[\begin{split}
	\frac{\psi(x)}{\psi\Bigl(\pi\bigl(g\pi^{-1}(x)\bigr)\Bigr)}\UU_\tau\Bigl(\pi\bigl(g\pi^{-1}(x)\bigr)\Bigr)\\
	=\frac{\psi(x)}{\psi\Bigl(\pi\bigl(g\pi^{-1}(x)\bigr)\Bigr)}\frac{
	\left(\begin{smallmatrix}
	-\pi_2(g\pi^{-1}(x))\\
	\pi_1(g\pi^{-1}(x))\\
	0
	\end{smallmatrix}\right)
	\UU^T\Bigl(\pi\bigl(g\pi^{-1}(x)\bigr)\Bigr)
	\left(\begin{smallmatrix}
	-\pi_2(g\pi^{-1}(x))\\
	\pi_1(g\pi^{-1}(x))\\
	0
	\end{smallmatrix}\right)
	}{\pi_1^2\bigl(g\pi^{-1}(x)\bigr)+\pi_2^2\bigl(g\pi^{-1}(x)\bigr)}\\
	=\frac{
	\left(\begin{smallmatrix}
	-\pi_2(g\pi^{-1}(x))\\
	\pi_1(g\pi^{-1}(x))\\
	0
	\end{smallmatrix}\right)
	\UU^T(x)\widetilde{g_1}^T
	\left(\begin{smallmatrix}
	-\pi_2(g\pi^{-1}(x))\\
	\pi_1(g\pi^{-1}(x))\\
	0
	\end{smallmatrix}\right)
	}{\pi_1^2\bigl(g\pi^{-1}(x)\bigr)+\pi_2^2\bigl(g\pi^{-1}(x)\bigr)}.
	\end{split}\]
	
	Let us compute
	\[
	g\pi^{-1}(x)=\frac{1}{|x|^2+1}
	\left(\begin{smallmatrix}
	2x_1\cos\alpha_1-2x_2\sin\alpha_1\\
	2x_1\sin\alpha_1+2x_2\cos\alpha_1\\
	2x_3\cos\alpha_2-(|x|^2-1)\sin\alpha_2\\
	2x_3\sin\alpha_2+(|x|^2-1)\cos\alpha_2
	\end{smallmatrix}\right),
	\]
	\[
	\pi\bigl(g\pi^{-1}(x)\bigr)=\frac{1}{|x|^2+1-2x_3\sin\alpha_2\mp(|x|^2-1)\cos\alpha_2}
	\left(\begin{smallmatrix}
	2x_1\cos\alpha_1-2x_2\sin\alpha_1\\
	2x_1\sin\alpha_1+2x_2\cos\alpha_1\\
	2x_3\cos\alpha_2-(|x|^2-1)\sin\alpha_2
	\end{smallmatrix}\right),
	\]
	\[
	\UU^T(x)\widetilde{g_1}^T=
	\left(\begin{smallmatrix}
	\UU_1(x)\cos\alpha_1-\UU_2(x)\sin\alpha_1\\
	\UU_1(x)\sin\alpha_1+\UU_2(x)\cos\alpha_1\\
	\UU_3(x)
	\end{smallmatrix}\right)^T,
	\]
	\[
	\UU^T(x)\widetilde{g_1}^T
	\left(\begin{smallmatrix}
	-\pi_2(g\pi^{-1}(x))\\
	\pi_1(g\pi^{-1}(x))
	\\
	0
	\end{smallmatrix}\right)
	=\frac{2\bigl(-x_2\UU_1(x)+x_1\UU_2(x)\bigr)}{|x|^2+1-2x_3\sin\alpha_2-(|x|^2-1)\cos\alpha_2},
	\]
	and
	\[
	\pi_1^2\bigl(g\pi^{-1}(x)\bigr)+\pi_2^2\bigl(g\pi^{-1}(x)\bigr)=\frac{4x_1^2+4x_2^2}{\bigl(|x|^2+1-2x_3\sin\alpha_2-(|x|^2+1)\cos\alpha_2\bigr)^2},
	\]
	so for the left-hand side we have
	\[\begin{split}
	\frac{
	\left(\begin{smallmatrix}
	-\pi_2(g\pi^{-1}(x))\\
	\pi_1(g\pi^{-1}(x))\\
	0
	\end{smallmatrix}\right)
	\UU^T(x)\widetilde{g_1}^T
	\left(\begin{smallmatrix}
	-\pi_2(g\pi^{-1}(x))\\
	\pi_1(g\pi^{-1}(x))\\
	0
	\end{smallmatrix}\right)
	}{\pi_1^2\bigl(g\pi^{-1}(x)\bigr)+\pi_2^2\bigl(g\pi^{-1}(x)\bigr)}\\
	=\frac{-x_2\UU_1(x)+x_1\UU_2(x)}{x_1^2+x_2^2}
	\left(\begin{smallmatrix}
	-x_2\cos\alpha_1-x_1\sin\alpha_1\\
	-x_2\sin\alpha_1+x_1\cos\alpha_1\\
	0
	\end{smallmatrix}\right)
	\end{split}\]
	and \eqref{e-tau} holds.
	
	Similar computations hold for $\UU_\rho$, so $\UU_\rho\in\cD_{\SO(2\times2)}$. Finally, $\UU_\zeta=\UU-\UU_\rho-\UU_\tau\in\cD_{\SO(2\times2)}$.
\end{proof}

Note that $E|_{\DF}=L|_{\DF}$ in view of Lemma \ref{L:DivFree}, where $L\colon\cD^{1,2}(\R^3,\R^3)\to\R$ is defined as $$L(\UU):=\frac12\int_{\R^3}|\nabla\UU|^2\,dx-\frac{1}{6}\int_{\R^3}|\UU|^6\,dx.$$ We set $\cY:=\cD_{\SO(2\times2)}\cap\cF$.

\begin{Lem}\label{L:inf}
	$\cY$ is infinite dimensional.
\end{Lem}
\begin{proof}
	Let $e=
	\bigl(\begin{smallmatrix}
	0 & -1\\
	1 & 0
	\end{smallmatrix}\bigr)
	\in\SO(2)$ and let $\cX$, resp. $\cZ$, be the subspace of $\cD_{\SO(2\times2)}$ consisting in the vector fields $\UU$ such that
	\[
	\UU(x)=\frac{u(x)}{r}
	\begin{pmatrix}
	x_1\\
	x_2\\
	0
	\end{pmatrix}
	, \text{ resp. }\UU(x)=u(x)
	\begin{pmatrix}
	0\\
	0\\
	1
	\end{pmatrix},
	\]
	for some $\SO$-invariant $u\colon\rr\to\R$. In order to prove that $\cY$ is infinite dimensional, we build an isomorphism between $\cX$ and $\cY$ and an isomorphism between $\cX$ and $\cZ$. The conclusion will follow from the fact that $\cD_{\SO(2\times2)}$ is infinite dimensional and that, in view of Lemmas \ref{L:dec} and \ref{L:decom}, we get the decomposition $\cD_{\SO(2\times2)}=\cX\oplus\cY\oplus\cZ$.
	
	For every $\UU\in\cX$ define $\widetilde{\UU}(x):=\UU(\widetilde{e}x)$. It is clear that $\widetilde{\UU}\in\cY$ and that $\UU\mapsto\widetilde{\UU}$ is an isomorphism.
	
	Now consider $\UU\in\cX$ and let $u\colon\R^3\to\R$ be $\SO$-invariant such that $\UU(x)=\frac{u(x)}{r}
	\left(\begin{smallmatrix}
	x_1\\
	x_2\\
	0
	\end{smallmatrix}\right)
	$. Define $\overline{\UU}(x):=u(x)
	\left(\begin{smallmatrix}
	0\\
	0\\
	1
	\end{smallmatrix}\right)
	$. By similar arguments to those used in the proof of Lemma \ref{L:DivFree} it is easy to check that $\overline{\UU}\in\cD^{1,2}(\R^3,\R^3)$. Finally, explicit computations show that $\overline{\UU}$ is $\SO(2\times2)$-symmetric (hence $\overline{\UU}\in\cZ$) and trivially $\UU\mapsto\overline{\UU}$ is an isomorphism.
\end{proof}

\begin{proof}[Proof of Theorem \ref{T:crit}]
	Lemma \ref{L:equisym} implies that $\cY\subset\cD_\cF$; moreover, $\cY$ is closed in $\cD^{1,2}(\R^3,\R^3)$ and infinite dimensional by Lemma \ref{L:inf}. Since $\UU\mapsto\frac{\UU}{\psi}\circ\pi$ is a linear isometry between $\cD^{1,2}(\R^3,\R^3)$ and $H^1(\mathbb{S}^3,\R^3)$ and between $L^6(\R^3,\R^3)$ and $L^6(\mbS^3,\R^3)$, one easily checks that $E|_{\DF}$ is invariant under the action of $\SO(2\times2)$. Hence every $\UU\in\cY$ is a solution to \eqref{e-crit} if and only if it is a critical point of $E|_\cY$ owing to Theorem \ref{T:Palais}.
	
	We want to make use of \cite[Theorem 9.12]{Rabin}. From the embeddings $\cY\subset\cD^{1,2}(\rr,\rr)\subset L^6(\rr,\rr)$, there exists $\delta>0$ such that
	\[
	\inf\Set{E(\UU)|\UU\in\cY\text{ and }|\nabla\UU|_2=\delta}>0
	\]
	(cf. the proof of Lemma \ref{L:I678} \textit{(a)}). In addition, if $Y\subset\cY$ is a finite dimensional subspace, then the norms $|\nabla(\cdot)|_2$ and $|\cdot|_6$ are equivalent in $Y$, thus there exists $R=R(Y)>0$ such that $E(\UU)\le0$ for every $\UU\in Y$ with $|\UU|_6\ge R$. We are only left to prove that $E|_\cY$ satisfies the Palais--Smale condition at every positive level, therefore let $\UU_n\in\cY$ be a Palais--Smale sequence for $E|_\cY$ at some $c>0$. 
	Similarly to the proof of Theorem \ref{T:Ex} we prove that $\UU_n$ is bounded, hence there exists $\UU\in\cY$ such that, up to a subsequence,
	\begin{equation}\label{e-wc}
	\UU_n\rightharpoonup\UU \quad \text{in } \cD^{1,2}(\rr,\rr)
	\end{equation}
	and, in view of Lemma \ref{L:compact},
	\begin{equation}\label{e-sc}
	\UU_n\to\UU \quad \text{in } L^6(\rr,\rr).
	\end{equation}
	Since $\lim_nE|_\cY'(\UU_n)=0$, from \eqref{e-wc} and \eqref{e-sc} we obtain
	\begin{equation}\label{e-U}
	0=\lim_nE'(\UU_n)(\UU)=|\nabla\UU|_2^2-|\UU|_6^6
	\end{equation}
	and, since $\UU_n$ is bounded,
	\begin{equation}\label{e-Un}
	0=\lim_nE'(\UU_n)(\UU_n)=\lim_n|\nabla\UU_n|_2^2-|\UU|_6^6.
	\end{equation}
	From \eqref{e-U} and \eqref{e-Un} we have $\lim_n|\nabla\UU_n|_2=|\nabla\UU|_2$, which, together with \eqref{e-wc}, yields $\lim_n\UU_n=\UU$ in $\cD^{1,2}(\rr,\rr)$.
\end{proof}

\begin{proof}[Proof of Corollary \ref{C:S}]
	It follows from Theorems \ref{T:crit} and  \ref{T:ScalVec}.
\end{proof}

\part{Constrained problems}\label{2}

\chapter{Introduction to Part \ref{2}}\label{K:intro2}

In Part \ref{2} of this Ph.D. thesis, we study the existence of least energy solutions to the Schr\"odinger system
\begin{equation}\label{e-Sch}
\begin{cases}
-\Delta u_1+\lambda_1u_1=\partial_1F(u)\\
\dots\\
-\Delta u_K+\lambda_Ku_K=\partial_KF(u)
\end{cases}
\text{in } \rn
\end{equation}
paired with the constraint
\begin{equation}\label{e-constraint}
\int_{\rn}u_j^2\,dx=\rho_j^2 \quad \forall j\in\{1,\dots,K\},
\end{equation}
where $N,K\ge1$, $\rho_1,\dots,\rho_K>0$ are given, $(\lambda_1,\dots,\lambda_K)\in\rk$, and $F\colon\rn\to\R$ is a nonlinear function. Problems as in \eqref{e-Sch} appear when one seeks \textit{standing wave} solutions $\Phi(x,t)=\left(e^{-i\lambda_1t}u_1(x),\dots,e^{-i\lambda_Kt}u_K(x)\right)$ to the \textit{time dependent} Schr\"odinger system
\begin{equation}\label{e-tSch}
\mathrm{i}\frac{\partial\Phi}{\partial t}-\Delta\Phi=g(|\Phi|)\Phi
\end{equation}
(in which case one has $F(u)=F(|u|)$) for some $u=(u_1,\dots,u_K)\colon\rn\to\rk$ and $(\lambda_1,\dots,\lambda_K)\in\rk$.

An important feature of \eqref{e-Sch} is that the values $\lambda_j$ are part of the unknown, being the Lagrange multipliers associated with the $L^2$-constraint \eqref{e-constraint}. Solutions to \eqref{e-Sch}--\eqref{e-constraint} or similar problems are known in the literature as \textit{normalized solutions} and have raised much interest in the last decades. The importance of such constraints is due to quantities related to solutions to \eqref{e-tSch} that are conserved in time (see, e.g., \cite{Cazenave:book,CazeLions}), i.e., the energy
\begin{equation}\label{e-energy}
J(u):=\int_{\rn}\frac12|\nabla u|^2-F(u)\,dx=\int_{\rn}\frac12\sum_{j=1}^K|\nabla u_j|^2-F(u)\,dx
\end{equation}
and the masses
\[
\int_{\rn}u_j^2\,dx, \quad j\in\{1,\dots,K\}.
\]
Moreover, the masses have a precise physical meaning: they represent the power supplies in nonlinear optics \cite{Akozbek,SluEgg} and the total number of atoms in Bose--Einstein condensation \cite{LSSY,PiSt}.


When $K=1$ and
\[
F(u)=\frac1p|u|^p, \quad 2<p<2^*, \, p\ne2_\#:=2+\frac4N,
\]
one can solve \eqref{e-Sch} with $\lambda=\lambda_1$ fixed (e.g., $\lambda=1$) and then, denoting by $u$ such a solution, consider proper $\alpha,\mu>0$ such that $\alpha u(\mu\cdot)$ solves \eqref{e-Sch}--\eqref{e-constraint} for some $\lambda'\in\R$.

Concerning the problem when $\lambda$ is fixed, the situation appears to be rather understood, at least in the case of $K=2$, positive solutions, and $F$ of power type. In this direction, we refer to \cite{Bartsch,BarDanWan,ChenZou,LinWei1,LinWei2,MaiMonPel,Mandel,SatWan,Sirakov,Soave0,SoaTav,TerVer,WeiWeth}, see also the references therein.

As for the case when $\lambda$ is part of the unknown and $F$ is not homogeneous or $K\ge2$, much depends on whether the energy functional $J\colon\hrn^K\to\R$ restricted to
\begin{equation}\label{e-Sset}
\cS:=\Set{u\in\hrn^K|\int_{\rn}u_j^2\,dx=\rho_j^2 \quad \forall j\in\{1,\dots,K\}},
\end{equation}
where $J(u)$ is defined in \eqref{e-energy}, is bounded from below or not. Such a property depends on $F$, which is not surprising, and on $\rho=(\rho_1,\dots,\rho_k)$, which is possibly less expected. In the literature, the case when $J$ is bounded from below for every $\rho$ is known as \textit{mass-subcritical} or $L^2$\textit{-subcritical}, while the case when $J$ is unbounded from below for every $\rho$ is know as \textit{mass-supercritical} or $L^2$\textit{-supercritical}. The case when the behaviour of $J$ from below actually depends on the parameter $\rho$ is referred to as \textit{mass-critical} or $L^2$\textit{-critical}. If $F$ is of power type, then these three cases correspond to the exponent $2<p<2^*$ being, respectively, less than, greater than, or equal to the threshold value $2_\#$.

The simplest case, hence the first that was investigated, is the mass-subcritical one, because solutions can be looked for as global minimizers of $J|_\cS$. This is the strategy adopted by Stuart \cite{Stuart82}, for the scalar case $K=1$, and Lions \cite{Lions84_2}, for the general case $K\ge1$: the former exploits tools from bifurcation theory (hence giving results also about this topic) from previous work of his \cite{Stuart79,Stuart80,Stuart81}, the latter provides necessary and sufficient conditions for minimizing sequences of $J|_\cS$ to converge (in particular, sufficient conditions for $\inf_\cS J$ to be achieved) and then discusses when such equivalent conditions hold. In fact, Lions deals with the mass-critical case too, because one of the sufficient conditions for $\inf_\cS J$ to be achieved is that $\rho$ is large (cf. \cite[p. 299]{Lions84_2}). More precisely, he always requires a mass-subcritical behaviour at infinity $\lim_{|u|\to\infty}F(u)/|u|^{2_\#}=0$, but the assumption of a mass-subcritical behaviour at zero $\lim_{u\to0}F(u)/|u|^{2_\#}=\infty$ can be traded with the mass being large, together with other hypotheses.

Concerning the mass-supercritical (nonhomogeneous) case, the first work is, to the best of our knowledge, by Jeanjean \cite{Jeanjean97}. Unable to apply a direct minimization method, he finds a solution exploiting the mountain pass geometry of the energy functional together with a characterization of the mountain pass level and a splitting result (or, when $N\ge2$, the compact embedding of radial functions $H^1_\textup{rad}(\rn)\subset L^p(\rn)$, $2<p<2^*$).

After these seminal results, the field of normalized solutions raised more and more interest in the PDE community and much work on it has been done, studying single equations and systems, in bounded domains and in all of $\rn$, using the most various techniques. It is impossible, or at the very least extremely audacious and demanding, to keep track of all the developments had in four decades of Mathematics, therefore we will limit ourselves to some relevant papers. As for single equations in $\rn$ we mention \cite{AckWeth,BarSo0,BarSo1,BartschVale,BellJean,BellJeanLuo,BiegMed,BCGJ,BuffEstSere,CazeLions,CinJean,JeanLu1,JeanLu2,Schino,Soave,SoaveC}, while concerning systems of equations in $\rn$ we mention \cite{BarJean,BarJeanSo,BarSo0,BarSo1,BarSoa-m,GouJean1,GouJean2,NguyenWang,MeSc,NoTaVe2,Schino,Shibata}. In particular, \cite{BartschVale,BarSo0,BarSo1,BarSoa-m,GouJean2,JeanLu2} deal with the issue of multiple solutions. Problems as \eqref{e-Sch} but in bounded domains are investigated in \cite{FibichMerle,NoTaVe1,PierVer} (single equations) and \cite{NoTaVe2,NoTaVe3,TavTer} (systems).

A common difficulty when studying \eqref{e-Sch}--\eqref{e-constraint} is that the embedding $\hrn\hookrightarrow L^2(\rn)$ is not compact, as remarked in Subsection \ref{SS:cpt}, not even when considering particular subspaces, which implies that a weak limit point of a sequence in $\cS$ need not belong to $\cS$. This issue, which is of course closely related to the strong convergence of such sequences, is also connected with the sign of the components $u_j$ of the limit point and/or of the corresponding Lagrange multipliers $\lambda_j$. Such issues have usually been investigated with more or less involved tools such as the strict subadditivity and monotonicity of the ground state energy map \cite{JeanLu1,JeanLu2,LiZou,Soave} or additional properties of Palais--Smale sequences \cite{BarJeanSo,Jeanjean97,JeanLu1,Soave} (roughly speaking, the \textit{ground state energy map} associates $\rho$ with the least possible energy that a solution to \eqref{e-Sch}--\eqref{e-constraint} can have and is defined rigorously in Sections \ref{S:gsemsub} and \ref{S:gsemsuper}). In addition, in the mass-supercritical setting, several of the papers mentioned so far make use of a complex topological argument by Ghoussoub \cite{Ghoussoub} based on the $\sigma$-homotopy stable family of compact subsets of $\cM$ (such a set is defined in \eqref{e-Mset}).

Here we propose a new, simpler approach, which overcomes both the issue of the strong convergence in $L^2(\rn)$ and that of the elaborate tools used to deal with it. This approach was introduced by Bieganowski and Mederski \cite{BiegMed} for $K=1$ in the mass-supercritical regime, then extended to the case $K\ge1$ by Mederski and the author \cite{MeSc}; it was also adapted to the mass-(sub)critical setting by the author \cite{Schino}.

Instead of considering the set $\cS$, we take into account
\begin{equation}\label{e-Dset}
\cD:=\Set{u\in\hrn^K|\int_{\rn}u_j^2\,dx\le \rho_j^2 \quad \forall j\in\{1,\dots,K\}}.
\end{equation}
Clearly, $\cS\subset\cD$; moreover, $\cS=\partial\cD$ if $K=1$ and $\cS\subsetneq\partial\cD$ if $K\ge2$ because in this case
\[
\partial\cD=\bigcup_{j=1}^K\cD_1\times\dots\times\cD_{j-1}\times\cS_j\times\cD_{j+1}\times\dots\times\cD_K,
\]
with the obvious definitions
\[\begin{split}
\cD_j & := \Set{u\in\hrn|\int_{\rn}u^2\,dx\le \rho_j^2},\\
\cS_j & := \Set{u\in\hrn|\int_{\rn}u^2\,dx=\rho_j^2},
\end{split}\]
which, in our approach, is what makes the difference between the cases $K=1$ and $K\ge2$.

From the lower semicontinuity of the norm, it is clear that every weak limit point of a sequence in $\cD$ stays in $\cD$. In the mass-subcritical case, as well as in the mass-critical one under additional assumptions, the functional $J|_\cD$ is still bounded from below and we easily obtain a minimizer; this is the case studied in Chapter \ref{K:sub}. The price to pay, of course, is that we need to prove that such a minimizer actually belongs to $\cS$, which requires further assumptions when $K\ge2$; nonetheless, this is still more direct than the arguments used in most of the papers previously mentioned.

In the mass-superctitical setting, dealt with in Chapter \ref{K:super}, the functional $J|_\cD$ is a fortiori unbounded from below, hence additional tools are called for in order to still obtain least energy solutions. It is possible to prove (see Section \ref{S:prel} for more details) that, if $(\lambda,u)$ is a nontrivial solution to \eqref{e-Sch}, then $u$ belongs to the set
\begin{equation}\label{e-Mset}
\cM:=\Set{u\in\hrn^K\setminus\{0\}|\int_{\rn}|\nabla u|^2\,dx=\frac{N}2\int_{\rn}H(u)\,dx},
\end{equation}
where $H(u):=\nabla F(u)\cdot u-2F(u)$, therefore it makes sense to consider the functional $J|_{\cM\cap\cD}$. Fortunately, this functional is bounded from below and we can apply a constrained minimization argument to obtain a minimizer. Under mild assumptions, $\cM$ is shown to be a natural constraint in the sense that every minimizer of $J|_{\cM\cap\cD}$ is in fact a critical point of $J|_\cD$. As in the mass-subcritical and -critical cases, we need to prove that the minimizer we have obtained lies in $\cS$ and this is done under further hypotheses. The difference is that $\cM$ is \textit{not} weakly closed, therefore, when we consider a (weakly convergent) minimizing sequence of $J|_{\cM\cap\cD}$, we need to project the limit point onto $\cM$ again and make sure such a projection belongs to $\cD$.

A second advantage of working with $\cD$ instead of $\cS$ is that the Lagrange multipliers $\lambda_j$ are nonnegative. This result is based on Clarke \cite{Clarke} and basically relies on two facts: the critical points we consider are \textit{minimizers} (of $J|_\cD$ in the mass-subcritical and -critical cases, of $J|_{\cM\cap\cD}$ in the mass-supercritical case) and the set $\cD$ is defined via \textit{inequalities}. There are at least two reasons why this is an advantage: the first is that the sign of the Lagrange multipliers has often to do with important aspects, as previously remarked, and indeed we do make use of such information also in our proofs; the second is that, in case one wants to prove the strict positivity, the case $\lambda_j=0$ for every $j\in\{1,\dots,K\}$ is easier to rule out than the case $\lambda_j\le0$. In addition, the strict positivity is important from non-mathematical points of view too, as there are situations in physics, e.g., concerning the eigenvalues of equations describing the behaviour of ideal gases, where the chemical potentials $\lambda_i$ of the standing waves have to be positive, see, e.g., \cite{LSSY,PiSt}.

Since the nonnegativity of $\lambda_j$ is used both in Chapter \ref{K:sub} and \ref{K:super}, we state and prove the aforementioned result in this introductory chapter.

\begin{Prop}\label{P:Clarke}
Let $\cH$ be a real Hilbert space and $f,\phi_j,\psi_k\in\cC^1(\cH)$, $j\in\{1,\dots,m\}$, $k\in\{1,\dots,n\}$. Suppose that for every
\[
x\in\bigcap_{j=1}^m\phi_j^{-1}(0)\cap\bigcap_{k=1}^n\psi_k^{-1}(0)
\]
the differential
\[
\bigl(\phi_j'(x),\psi_k'(x)\bigr)_{1\le j\le m,1\le k\le n}\colon\cH\to\R^{m+n}
\]
is surjective. If $\bar{x}\in\cH$ minimizes $f$ over
\[
\Set{x\in E|\phi_j(x)\le0 \, \forall j=1,\dots,m\text{ and }\psi_k(x)=0 \, \forall k=1,\dots,n},
\]
then there exist $(\lambda_j)_{j=1}^m\in[0,\infty[^m$ and $(\sigma_k)_{k=1}^n\in\R^n$ such that
\[
f'(\bar{x})+\sum_{j=1}^m \lambda_j\phi_j'(\bar{x})+\sum_{k=1}^n \sigma_k\psi_k'(\bar{x})=0.
\]
\end{Prop}
\begin{proof}
From \cite[Theorem 1, Corollary 1]{Clarke} there exist $\tau\ge0$, $(\lambda_j)_{j=1}^m\in[0,\infty[^m$, and $(\sigma_k)_{k=1}^n\in\R^n$, not all zero, such that
\begin{equation}\label{e-Clarke1}
\tau f'(\bar{x})+\sum_{j=1}^m \lambda_j\phi_j'(\bar{x})+\sum_{k=1}^n \sigma_k\psi_k'(\bar{x})=0.
\end{equation}
If $\tau>0$, then we can divide both sides of \eqref{e-Clarke1} by $\tau$ and, up to relabelling $\lambda_j$ and $\sigma_k$, conclude the proof, hence assume by contradiction that $\tau=0$, i.e.
\begin{equation}\label{e-Clarke2}
\sum_{j=1}^m \lambda_j\phi_j'(\bar{x})+\sum_{k=1}^n \sigma_k\psi_k'(\bar{x})=0.
\end{equation}
If $\phi_j(\bar{x})<0$ for some $j\in\{1,\dots,m\}$, then of course $\lambda_j=0$, hence, up to considering a (possibly empty) subset of $\{1,\dots,m\}$ in \eqref{e-Clarke2}, we can assume that $\phi_1(\bar{x})=\dots=\phi_{m_0}(\bar{x})=0$ and $\lambda_{m_0+1}=\dots=\lambda_m=0$ for some $0\leq m_0\leq m$, where $m_0=0$ denotes that $\lambda_j=0$ for all $j\in \{1,\dots,m\}$ and $m_0=m$ denotes $\phi_j(\bar{x})=0$ for all $j\in \{1,\dots,m\}$. Then the differential
\[
\bigl(\phi_1'(\bar{x}),\dots,\phi_{m_0}'(\bar{x}),\psi_1'(\bar{x}),\dots,\psi_n'(\bar{x})\bigr)\colon E\to\R^{m_0+n}
\]
is surjective and so, for every $j\in\{1,\dots,m_0\}$ (resp. $k\in\{1,\dots,n\}$), there exists $y\in\cH$ such that $\phi_j'(\bar{x})(y)\ne0$, $\phi_i'(\bar{x})(y)=0$ for every $i\in\{1,\dots,m_0\}\setminus\{j\}$ and $\psi_k'(\bar{x})(y)=0$ for every $k\in\{1,\dots,n\}$ (resp. $\psi_k'(\bar{x})(y)\ne0$, $\psi_i'(\bar{x})(y)=0$ for every $i\in\{1,\dots,n\}\setminus\{k\}$ and $\phi_j'(\bar{x})(y)=0$ for every $j\in\{1,\dots,m_0\}$). This and \eqref{e-Clarke2} imply $\lambda_j=0$ for every $j\in\{1,\dots,m_0\}$ and $\sigma_k=0$ for every $k\in\{1,\dots,n\}$, a contradiction.
\end{proof}

Note that both the statement of Proposition \ref{P:Clarke} and its proof remain valid if the functions $\psi_k$'s are removed from them.

Finally, we conclude this chapter recalling the \textit{Gagliardo--Nirenberg interpolation inequality} (in a less generic form, which however is enough for our purposes), which plays a basic role in Chapters \ref{K:sub} and \ref{K:super}.

\begin{Th}\label{T:GN}
Let $p\in]2,\infty[$ if $N\in\{1,2\}$ or $p\in]2,2^*]$ if $N\ge3$. Then there exists $C_{N,p}>0$ such that
\[
|u|_p\le C_{N,p}|\nabla u|_2^{\delta_p}|u|_2^{1-\delta_p} \quad \text{ for all } u\in\hrn,
\]
where $\displaystyle\delta_p:=N\left(\frac12-\frac1p\right)$.
\end{Th}

Observe that, in Theorem \ref{T:GN},
\[
\delta_pp
\begin{cases}
<2 \quad \text{if } p<2_\#,\\
=2 \quad \text{if } p=2_\#,\\
>2 \quad \text{if } p>2_\#.
\end{cases}
\]

\chapter{Autonomous Schr\"odinger equations in the mass-critical and -subcritical cases}\label{K:sub}

\section{Statement of the results}

In this chapter, based on \cite{Schino}, we study the problem
\begin{equation}\label{e-main}
	\begin{cases}
		-\Delta u_j+\lambda_ju_j=\partial_jF(u)\\
		\int_{\rn}u_j^2\,dx=\rho_j^2\\
		(\lambda_j,u_j)\in\R\times\hrn
	\end{cases}
	j\in\{1,\dots,K\}
\end{equation}
with $N,K\ge1$. Here $\rho = (\rho_1,\dots,\rho_K)\in]0,\infty[^K$ is prescribed, while $(\lambda,u)=(\lambda_1,\dots,\lambda_K,u_1,\dots u_K)$ is the unknown. The nonlinearity $F$ satisfies the following assumptions, which correspond to the mass-(sub)critical case.
\begin{itemize}
	\item [(F0)] $F\in\cC^1(\R^K)$, $F\ge0$, and
	\begin{itemize}
		\item if $N=1$, there exists $S>0$ such that $|\nabla F(u)| \le S|u|$ for every $u\in[-1,1]^K$;
		\item if $N=2$, for every $b>0$ there exists $S_b>0$ such that $|\nabla F(u)| \le S_b\bigl(|u| + e^{b|u|^2}-1\bigr)$ for every $u\in\rk$;
		\item if $N\ge3$, there exists $S>0$ such that $|\nabla F(u)| \le S(|u|+|u|^{2^*-1})$ for every $u\in\rk$.
	\end{itemize}
	\item [(F1)] $\displaystyle\eta_\infty:=\limsup_{|u|\to\infty}\frac{F(u)}{|u|^{2_\#}}<\infty$.
	\item [(F2)] $\displaystyle\lim_{u\to0}\frac{F(u)}{|u|^2}=0$.
	\item [(F3)] $\displaystyle\eta_0:=\liminf_{u\to0}\frac{F(u)}{|u|^{2_\#}}>0$.
\end{itemize}

In (F3) the case $\eta_0=\infty$ is allowed. Let us recall that the functional $J\colon\hrn^K\to\R$ defined in \eqref{e-energy} is well defined and of class $\cC^1$ in view of (F0) (in particular, when $N=2$, this is due to the Moser--Trudinger inequality, see, e.g., \cite[Theorem 1.1]{Ruf}). When $N\ge5$ and $K\ge2$
, we consider the following assumption for a function $f\colon[0,\infty[\to[0,\infty[$.
\begin{itemize}	
\item [(P)] There exists $q\le N/(N-2)$ such that $\liminf_{t\to0^+}f(t)/t^q>0$.
\end{itemize}

We choose to tackle problem \eqref{e-main} by means of relative compactness in the $L^p(\rn)$ norm, $2<p<2^*$, of bounded sequences in $\hrn$, while a different strategy will be used in Chapter \ref{K:super}. When $N\ge2$, let $\fH$ be a subspace of $\hrn$ that embeds compactly into $L^p(\rn)$ for every $p\in]2,2^*[$ and such that every critical point of $J|_{\fH^K}$ is a critical point of $J$. In particular, the role of $\fH$ will be played by either $\fH_\textup{r}$ or, if $F$ is even, $\fH_\textup{n}$ (recall their definitions from Subsection \ref{SS:cpt}).
When $N=1$, in the spirit of Theorem \ref{T:Cazenave} we will always take $\fH=\fH_\textup{r}$. Since we deal with minimizing sequences, if $F$ is even, then we can replace each element of such a sequence with its Schwarz rearrangement and still obtain a minimizing sequence.

Recalling the definitions of $\cS$ and $\cD$, given in \eqref{e-Sset} and \eqref{e-Dset} respectively, we define
\[
\fS:=\cS\cap\fH \quad \text{and} \quad \fD:=\cD\cap\fH
\]
and we will write $\fS_\textup{r}$, $\fD_\textup{r}$ (resp. $\fS_\textup{n}$, $\fD_\textup{n}$) when $\fH=\fH_\textup{r}$ (resp. $\fH=\fH_\textup{n}$). Our main results read as follows.

\begin{Th}\label{T:main1}
	Assume that $K=1$, (F0)--(F3) hold, and
	\begin{eqnarray}
		\label{e-etas} 2\eta_\infty\rho^{4/N}C_{N,2_\#}^{2_\#}<1,\\
		\label{e-etal} 2\eta_0\rho^{4/N}C_{N,2_\#}^{2_\#}>1.
	\end{eqnarray}
	
	(a) If $N\ge2$, then there exists a solution $(\lambda,u)\in]0,\infty[\times\fS_\textup{r}$ to \eqref{e-main} such that $J(u)=\inf_{\fD_\textup{r}}J<0$. If, moreover, $\eta_0=\infty$, then the same holds replacing $\fS_\textup{r}$ and $\fD_\textup{r}$ with $\fS$ and $\fD$ respectively.
	
	(b) If $F$ is even, then there exists a ground state solution $(\lambda,u)\in]0,\infty[\times\fS_\textup{r}$ to \eqref{e-main} such that $J(u)<0$ and $u$ is nonnegative, radially nonincreasing, and of class $\cC^2$.
\end{Th}

Since, as we will see later (cf. Lemma \ref{L:cbb}), $J|_{\cD}$ is bounded from below, by \textit{ground state solution} to \eqref{e-main} we mean a solution $(\lambda,u)$ such that $J(u)=\min_\cD J$. Note that this is more than just requiring $J(u)=\min_{\cS}J$, which, on the other hand, appears as a more ``natural'' condition. As for nonradial solutions, we have as follows.

\begin{Prop}\label{P:main}
	If $N=4$ or $N\ge6$, $K=1$, $F$ is even, (F0)--(F3) and \eqref{e-etas} hold, and $\eta_0=\infty$, then there exists a solution $(\mu,v)\in]0,\infty[\times\fS_\textup{n}$ to \eqref{e-main} such that $J(v)=\inf_{\fD_\textup{n}}J<0$.
\end{Prop}

Note that under the assumptions of Proposition \ref{P:main} there exist two distinct solutions to \eqref{e-main}: $(\lambda,u)\in]0,\infty[\times\fS_\textup{r}$, which is also a ground state solution, and $(\mu,v)\in]0,\infty[\times\fS_\textup{n}$. 

When $K\ge2$, $|\rho|$ replaces $\rho$ in \eqref{e-etas} and \eqref{e-etal}. In this case the following holds.

\begin{Th}\label{T:main2}
	Assume that $K\ge2$, (F0)--(F3) and \eqref{e-etas}--\eqref{e-etal} hold, $L\ge1$ is an integer, and for every $\ell\in\{1,\dots,L\}$ and every $j\in\{1,\dots,K\}$ there exist $F_j,\widetilde{F}_{j,\ell}\in\cC^1(\R)$ even, nonnegative, and nondecreasing on $[0,\infty[$ such that $\widetilde{F}_{j,\ell}(0)=0$ and
	\begin{equation*}
		F(u)=\sum_{j=1}^KF_j(u_j)+\sum_{\ell=1}^L\prod_{j=1}^K\widetilde{F}_{j,\ell}(u_j).
	\end{equation*}
	If $1\le N\le4$ or if $N\ge5$ and each $F_j'|_{[0,\infty[}$ satisfies (P) (not necessarily with the same $q$), then there exists a ground state solution $(\lambda,u)\in]0,\infty[^K\times\fS_\textup{r}$ to \eqref{e-main} such that $J(u)<0$ and each component of $u$ is positive, radially nonincreasing, and of class $\cC^2$.
\end{Th}

It is clear that a necessary condition for \eqref{e-etas} and \eqref{e-etal} to hold simultaneously is that $\eta_0>\eta_\infty$. This is what holds in the $L^2$-subcritical case, where $\eta_0=\infty$ and $\eta_\infty=0$. At the same time it rules out the $L^2$-critical case when $F$ is of power type. On the one hand, since \eqref{e-main} does have a solution with $F(u)=|u|^{2_\#}/2_\#$ for one specific value of $\rho$, we know that the conditions in Theorems \ref{T:main1} or \ref{T:main2} are not sharp, at least to ensure an existence result for some $\rho$; on the other hand, if \eqref{e-etas} and \eqref{e-etal} both hold, then we can find a ground state solution to \eqref{e-main} for uncountably many values of $\rho$ even when the behaviour of $F$ is $L^2$-critical both at zero and at infinity. Observe also that $\eta_0=\infty$ is a necessary condition for Theorem \ref{T:main2} to hold if $N\ge5$.


We provide some examples for $F$, beginning with the case $K=1$. A first model for the nonlinearity is
\[
F(u)=\frac{\nu}{2_\#}|u|^{2_\#}+\frac{\bar\nu}{p}|u|^p
\]
for some $\nu\ge0$, $\bar\nu>0$ and $2<p<2_\#$, in which case one has $\eta_0=\infty$ and $\eta_\infty=\nu/2_\#$. A second model is a sort of counterpart of the first one, i.e.,
\begin{equation}\label{e-xmpl1}
	F(u)=\int_0^{|u|}\min\{t^{2_\#},t^p\}\,dt
\end{equation}
with $2<p<2_\#$, in which case one has $\eta_0=1/2_\#$ and $\eta_\infty=0$.

Now define $F^*\colon[0,\infty[\to\R$ by $F^*(0)=0$ and
\[
F^*(t)=
\begin{cases}
	-\frac{t^2}{\ln t} \quad & \text{if } 0<t<\frac12\\
	\frac{b}2\bigl((b+2)t-\frac12-\frac{b}2\bigr) \quad & \text{if } \frac12\le t\le1\\
	-t^2+2ct-1-\frac{b}4(b+1) \quad & \text{if } 1<t\le c\\
	F^*(2c-t) \quad & \text{if } c<t\le2c\\
	0 \quad & \text{if } t>2c
\end{cases}
\]
with $b=1/\ln2$ and $c=b(b+2)/4+1$ and let $F(u):=F^*(|u|)$. Then (F0)--(F3) hold with $\eta_0=\infty$ and $\eta_\infty=0$. One can also modify the previous example in order to have $\eta_0<\infty$ by defining
\[
F^*(t)=
\begin{cases}
	\frac{t^{2_\#}}{2_\#} \quad & \text{if } 0\le t\le1\\
	t-1+\frac1{2_\#} \quad & \text{if } 1<t<2\\
	-t^2+5t-5+\frac1{2_\#} \quad & \text{if } 2\le t\le\frac52\\
	F^*(5-t) \quad & \text{if } \frac52<t<5\\
	0 \quad & \text{if } t\ge5.
\end{cases}
\]
Notice that, in both examples, $F^*$ is \textit{not} monotone, therefore such examples do not suit the case $K\ge2$; however, if we modify $F^*$ such that it is constant after it reaches its maximum, then we can consider it also for that case. Finally, an example of when both $\eta_\infty$ and $\eta_0$ are finite and positive is
\begin{equation}\label{e-xmpl2}
	F(u)=
	\begin{cases}
		\frac{|u|^{2_\#}}{2_\#} \quad & \text{if } 0\le|u|\le1\\
		|u|-1+\frac1{2_\#} \quad & \text{if } 1<|u|<2\\
		\frac{|u|^{2_\#}}{2^{2_\#-1}2_\#}+1-\frac1{2_\#} \quad & \text{if } |u|\ge2.
	\end{cases}
\end{equation}

Concerning the case $K\ge2$, similarly as before a model for the nonlinearity is
\begin{equation}\label{e-exK}
F(u) = \sum_{j=1}^K \left( \frac{\nu_j}{2_\#}|u_j|^{2_\#} + \frac{\bar\nu_j}{p_j}|u_j|^{p_j} \right) + \alpha\prod_{j=1}^K |u_j|^{r_j} + \beta\prod_{j=1}^K |u_j|^{\bar r_j}
\end{equation}
for some $\nu_j,\alpha,\beta\ge0$, $\bar\nu_j>0$, $r_j,\bar r_j>1$, and $2<p_j<2_\#$ such that $\alpha+\beta>0$\footnote{Of course, the case $\alpha=\beta=0$ is still allowed, but then the system \eqref{e-main} is uncoupled.}, $\sum_{j=1}^Kr_j=2_\#$, and $\sum_{j=1}^K\bar r_j<2_\#$. When $N\ge5$ (which implies $K=2$, c.f. Remark \ref{R:K}), we need to add the term
\[
\frac{\tilde\nu_1}{q_1}|u_1|^{q_1}+\frac{\tilde\nu_2}{q_2}|u_2|^{q_2}, \quad 2<q_1,q_2\le\frac{2N-2}{N-2}, \, \tilde\nu_1,\tilde\nu_2>0,
\] (then we can allow $\bar\nu_j=0$). In this case, again one has $\eta_0=\infty$. If $\alpha=0$, then $\eta_\infty=\max_{j=1,\dots,K}\nu_j/2_\#$; if $K=2$, $\alpha>0$, and $\nu_j=0$ for every $j\in\{1,2\}$, then $\eta_\infty=\alpha\sqrt{r_1^{r_1}r_2^{r_2}/2_\#^{2_\#}}$ (see Subsection \ref{SS:expl} for more details on such computations).

Finally, one can take $F_j$ and $\widetilde{F}_{j,k}$ as in \eqref{e-xmpl1} or \eqref{e-xmpl2}, with additional restrictions on $F_j$ in a similar way as before if $N\ge5$.

\begin{Rem}\label{R:K}
Although there are no explicit restrictions on $K$ in Theorem \ref{T:main2}, the example in \eqref{e-exK} shows that we could need $K$ not to be too large. As a matter of fact, since $r_j>1$ for every $j\in\{1,\dots,K\}$, there holds
\[
K < \sum_{j=1}^K r_j = 2_\#
\]
(or likewise with $\bar{r}_j$ if $\alpha=0$).
\end{Rem}

\subsection{Some explicit computations}\label{SS:expl}

\begin{Prop}\label{P:ex}
	Let $K\ge2$ and $\nu_j\ge0$, $j\in\{1,\dots,K\}$, and define
	\[
	F(u)=\sum_{j=1}^K\nu_j|u_j|^{2_\#}.
	\]
	Then $\displaystyle\limsup_{|u|\to\infty}\frac{F(u)}{|u|^{2_\#}}=\max_{j=1,\dots,K}\nu_j$.
\end{Prop}
\begin{proof}
	By taking $u_j=0$ for every $j\in\{1,\dots,K\}$ but one we easily obtain $\displaystyle\limsup_{|u|\to\infty}\frac{F(u)}{|u|^{2_\#}}\ge\max_{j=1,\dots,K}\nu_j$. Since $\displaystyle F(u)\le\max_{j=1,\dots,K}\nu_j\sum_{j=1}^K|u_j|^{2_\#}$, it suffices to prove that for every $u\in\R^K$
	\[
	\left(|u_1|^{2_\#}+\dots+|u_K|^{2_\#}\right)^{1/{2_\#}}\le\left(u_1^2+\dots+u_K^2\right)^{1/2}.
	\]
	To this aim, we will prove that the function $\varphi\colon]0,\infty[\to]0,\infty[$ is decreasing, where $\varphi(t):=(b_1^t+\dots+b_K^t)^{1/t}$ for some $b_1,\dots,b_K>0$. From
	\[
	\varphi'(t)=\varphi(t)\left(\frac{b_1^t\ln b_1+\dots+b_K^t\ln b_K}{t(b_1^t+\dots+b_K^t)}-\frac1{t^2}\ln(b_1^t+\dots+b_K^t)\right)
	\]
	we have that $\varphi'(t)<0$ is equivalent to
	\[
	b_1^t\ln(b_1^t+\dots+b_K^t)+\dots+b_K^t\ln(b_1^t+\dots+b_K^t)>b_1^t\ln b_1^t+\dots+b_K^t\ln b_K^t,
	\]
	which is true because $\ln$ is increasing.
\end{proof}

\begin{Prop}
	Let $r_1,r_2>1$ such that $r_1+r_2=2_\#$ and define
	\[
	F(u)=|u_1|^{r_1}|u_2|^{r_2}.
	\]
	Then $\displaystyle\limsup_{|u|\to\infty}\frac{F(u)}{|u|^{2_\#}}=\sqrt{r_1^{r_1}r_2^{r_2}/2_\#^{2_\#}}$.
\end{Prop}
\begin{proof}
	Observe that, for $u_2\ne0$, $\displaystyle\frac{F(u)}{|u|^{2_\#}}=\varphi\left(\frac{|u_1|}{|u_2|}\right)$, with $\displaystyle\varphi(t):=\frac{t^{r_1}}{(t^2+1)^{2_\#}}$. Since $\displaystyle\max\varphi=\varphi\left(\!\sqrt{\frac{r_1}{r_2}}\right)=\sqrt{\frac{r_1^{r_1}r_2^{r_2}}{2_\#^{2_\#}}}$, there holds
	\[
	\sqrt{\frac{r_1^{r_1}r_2^{r_2}}{2_\#^{2_\#}}}\ge\limsup_{|u|\to\infty}\frac{F(u)}{|u|^{2_\#}}\ge\limsup_{\substack{u_1=\sqrt{r_1/r_2}u_2\\ |u_2|\to\infty}}\frac{F(u)}{|u|^{2_\#}}=\sqrt{\frac{r_1^{r_1}r_2^{r_2}}{2_\#^{2_\#}}}.\qedhere
	\]
\end{proof}

\section{Proof of the results}

Henceforth, we will always assume that (F0) holds and we will make use of it without explicit mention.

\subsection{The scalar case $K=1$}

\begin{Lem}\label{L:cbb}
	If (F0)--(F2) and \eqref{e-etas} hold, then $J|_{\cD}$ is coercive and bounded from below.
\end{Lem}
\begin{proof}
	From (F1) and (F2), for every $\varepsilon>0$ there exists $c_\varepsilon>0$ such that $F(u)\le c_\varepsilon u^2+(\varepsilon+\eta_\infty)|u|^{2_\#}$ for every $u\in\R$. In view of Theorem \ref{T:GN}, for every $u\in\cD$ we have
	\[\begin{split}
		J(u)&\ge\frac12|\nabla u|_2^2-c_\varepsilon |u|_2^2-(\varepsilon+\eta_\infty)|u|_{2_\#}^{2_\#}\\
		&\ge\frac12|\nabla u|_2^2-c_\varepsilon \rho^2-(\varepsilon+\eta_\infty)\rho^{4/N}C_{N,2_\#}^{2_\#}|\nabla u|_2^2\\
		&=\left(\frac12-(\varepsilon+\eta_\infty)\rho^{4/N}C_{N,2_\#}^{2_\#}\right)|\nabla u|_2^2-c_\varepsilon\rho^2,
	\end{split}\]
	hence the statement holds true for sufficiently small $\varepsilon$.
\end{proof}

\begin{Rem}\label{R:cbb}
	Lemma \ref{L:cbb} still holds for $K\ge2$ because $\big||u|\big|_r=\left|u\right|_r$, $1\le r\le\infty$, and $\big|\nabla|u|\big|_2\le\left|\nabla u\right|_2$, hence one can use Theorem \ref{T:GN} with $|u|$.
\end{Rem}

For $u\in\hrn\setminus\{0\}$ and $s>0$ let $s\star u(x):=s^{N/2}u(sx)$. Note that $|u|_2=|s\star u|_2$.

\begin{Lem}\label{L:neg}
	If (F0), (F3), and \eqref{e-etal} hold, then $\inf_{\fD_\textup{r}}J<0$. If, moreover, $\eta_0=\infty$, then $\inf_{\fD}J<0$.
\end{Lem}
\begin{proof}
	Fix $u\in\fD\setminus\{0\}$ and note that
	\[
	J(s\star u)=s^2\int_{\rn}\frac12|\nabla u|^2-\frac{F(s^{N/2}u)}{(s^{N/2})^{2_\#}}\,dx.
	\]
	If $\eta_0=\infty$, then $\displaystyle\lim_{s\to0^+}\int_{\rn}\frac{F(s^{N/2}u)}{(s^{N/2})^{2_\#}}\,dx=\infty$. If $\eta_0<\infty$, then
	\[
	\limsup_{s\to0^+}\int_{\rn}\frac12|\nabla u|^2-\frac{F(s^{N/2}u)}{(s^{N/2})^{2_\#}}\,dx\le\int_{\rn}\frac12|\nabla u|^2-\eta_0|u|^{2_\#}\,dx,
	\]
	hence the statement holds true if $u\in\fD_\textup{r}$ and
	\begin{equation}\label{e-ineqeta0}
		\frac12\int_{\rn}|\nabla u|^2\,dx<\eta_0\int_{\rn}|u|^{2_\#}\,dx.
	\end{equation}
	Let $w\in\fH_\textup{r}$ be the unique positive radial solution to $-\Delta v+\frac2Nv=v^{2_\#-1}$ in $\rn$ \cite{Kwong}. Then (cf. \cite{Weinstein}) $w$ is such that $$|w|_2^{4/N}=\frac{2_\#}{2C_{N,2_\#}^{2_\#}}$$ and equality holds in the Gagliardo--Nirenberg interpolation inequality (Theorem \ref{T:GN}). If we define $u(x):=w(tx)$ for some $t>0$, then $u\in\fD_\textup{r}$ and \eqref{e-ineqeta0} become, respectively,
	\begin{equation}\label{e-ineqs}
		|w|_2^2\le\rho^2t^N \quad \text{and} \quad t^2|\nabla w|_2^2<2\eta_0|w|_{2_\#}^{2_\#}.
	\end{equation}
	Now, with the help of the properties of $w$, it is easy to check by direct computations that \eqref{e-ineqs} holds if and only if $\displaystyle t\in\left[\frac{\sqrt{1+2/N}}{\rho^{2/N}C_{N,2_\#}^{1+2/N}},\sqrt{2_\#\eta_0}\right[$.
\end{proof}

\begin{Rem}\label{R:neg}
	Lemma \ref{L:neg} still holds for $K\ge2$ because, if we set $W\in\hrn^K$ as $W=(w/\sqrt{K},\dots,w/\sqrt{K})$, then $|W|_p=|w|_p$, $2\le p\le2^*$, and $|\nabla W|_2=|\nabla w|_2$, therefore we can define $u(x):=W(tx)$ and conclude likewise.
\end{Rem}

\begin{Lem}\label{L:min}
	Assume that (F0)--(F2) and \eqref{e-etas} hold.
	
	(a) If $N\ge2$, then $\inf_{\fD}J$ is achieved.
	
	(b) If $F$ is even, then $\inf_{\fD_\textup{r}}J$ is achieved and $\min_{\fD_\textup{r}}J=\min_{\cD}J$.
\end{Lem}
\begin{proof}
	Let $u_n\in\fD$ such that $\lim_nJ(u_n)=\inf_{\fD}J$. In view of Lemma \ref{L:cbb} there exists $u\in\fD$ such that, up to a subsequence, $u_n\rightharpoonup u$ in $\hrn$ and $u_n\to u$ a.e. in $\rn$ as $n\to\infty$.
	
	\textit{(a)} Since $N\ge2$, then $u_n\to u$ in $L^{2_\#}(\rn)$ and $\lim_n\int_{\rn}F(u_n)\,dx=\int_{\rn}F(u)\,dx$ due to (F1), (F2) and \cite[Theorem 1]{BrezisLieb}, whence
	\[
	\inf_{\fD}J=\lim_nJ(u_n)\ge J(u)\ge \inf_{\fD}J.
	\]
	
	\textit{(b)} Since $F$ is even, choosing $\fH=\fH_\textup{r}$ we can replace $u_n$ with its Schwarz rearrangement $u_n^*$ and we obtain another minimizing sequence for $J|_{\fD_\textup{r}}$ with the additional property that each $u_n^*$ is radially nonincreasing. Then from Theorem \ref{T:Cazenave} we infer again that $u_n\to u$ in $L^{2_\#}(\rn)$ and conclude as in point \textit{(a)}. As for the last part, let $v_n\in\cD$ such that $\lim_nJ(v_n)=\inf_{\cD}J$. Then, similarly as before,
	\[
	\inf_{\cD}J\gets J(v_n)\ge J(v_n^*)\ge\min_{\fD_\textup{r}}J\ge\inf_{\cD}J.\qedhere
	\]
\end{proof}

\begin{proof}[Proof of Theorem \ref{T:main1}]
	\textit{(a)} In view of Lemmas \ref{L:neg} and \ref{L:min} there exists $u\in\fD$ (if $\eta_0<\infty$, then necessarily $\fD=\fD_\textup{r}$) such that $J(u)=\inf_{\fD}J<0$. In Proposition \ref{P:Clarke} (the version without the functions $\psi_k$'s) we put $m=K=1$, $\cH=\fH$, $f=J$, and $\phi=\phi_1=|\cdot|_2^2-\rho^2$. Then there exists $\lambda\ge0$ such that
	\begin{equation*}
		-\Delta u+\lambda u=F'(u).
	\end{equation*}
	Note that $\lambda=0$ if $u\in\mathring{\fD}$. Assume by contradiction that $\lambda=0$. Since 
	$u$ satisfies the Poho\v{z}aev identity 
	\begin{equation*}
		(N-2)\int_{\rn}|\nabla u|^2\,dx=2N\int_{\rn}F(u)\,dx
	\end{equation*}
	we have $0>J(u)=|\nabla u|_2^2/N$, which is impossible, hence $(\lambda,u)$ solves \eqref{e-main}.
	
	\textit{(b)} Since $F$ is even, the argument of the proof of Lemma \ref{L:min} yields that the minimizer $u$ of $J|_{\fD_\textup{r}}$ obtained in point \textit{(a)} is nonnegative and radially nonincreasing and that $(\lambda,u)$ is a ground state solution. Since $u\in W_\textup{loc}^{2,p}(\rn)$ for every $p<\infty$ (cf. Section \ref{S:NP}), we can use the argument of \cite[Lemma 1]{BerLions} and obtain $u\in\cC^2(\rn)$.
\end{proof}

\begin{Rem}
	In the assumptions of Theorem \ref{T:main1}, every minimizer of $J|_{\fD}$ lies in $\fS$.
\end{Rem}

\begin{Prop}
	Let the assumptions of Theorem \ref{T:main1} \textit{(b)} hold and let $(\lambda,u)$ be given therein. If $F$ is nondecreasing on $[0,\infty[$, then $u>0$. If, moreover, there exists $t_0>0$ such that $F'(t)\le\lambda t$ for every $t\in[0,t_0]$ and $F'(t)>\lambda t$ for every $t\in]t_0,\infty[$, then $u$ is radially decreasing.
\end{Prop}
\begin{proof}
	The first part follows from the maximum principle \cite[Lemma IX.V.1]{Evans}. Assume by contradiction that $u$ is constant in the annulus $A:=\{r_1<|x|<r_2\}$ for some $r_2>r_1>0$. Then $0=-\Delta u=F'(u)-\lambda u$ in $A$ and so $-\Delta u\le0$ in $\Omega:=\{|x|>r_1\}$ because $u$ is radially nonincreasing. At the same time $u$ attains its maximum over $\overline\Omega$ at every point of $A$. This is impossible because $u|_\Omega$ is not constant.
\end{proof}

\begin{proof}[Proof of Proposition \ref{P:main}]
	The argument is the same as in the proof of Theorem \ref{T:main1} \textit{(a)} with $\fH=\fH_\textup{n}$.
\end{proof}

\subsection{The vectorial case $K\ge2$}

We begin this section with the equivalent of Theorem \ref{T:main1} \textit{(a)} for $K\ge2$. Note that we can only ensure that the minimizer $u$ belongs to $\partial\cD$, while we need more assumptions in order that $u\in\cS$; in particular, we do not know if we can make use of the Schwarz rearrangements, which is also what prevents us from dealing with the case $N=1$.

\begin{Prop}\label{P:main1}
	If $N\ge2$ and (F0)--(F3) and \eqref{e-etas}--\eqref{e-etal} hold, then there exists $(\lambda,u)\in[0,\infty[^K\times\partial\fD_\textup{r}$ such that $\max_{j=1,\dots,K}\lambda_j>0$, $J(u)=\inf_{\fD_\textup{r}}J<0$, and for every $j\in\{1,\dots,K\}$
	\[
	-\Delta u_j+\lambda_ju_j=\partial_jF(u).
	\]
	If, moreover, $\eta_0=\infty$, then the same holds replacing $\fD_\textup{r}$ with $\fD$.
\end{Prop}
\begin{proof}
	Owing to Remarks \ref{R:cbb} and \ref{R:neg}, Lemmas \ref{L:cbb}, \ref{L:neg}, and \ref{L:min} \textit{(a)} still hold for $K\ge2$. 
	Then we proceed as in the proof of Theorem \ref{T:main1} \textit{(a)}.
\end{proof}

Note that Proposition \ref{P:main1} is valid for every $K\ge2$. The next result is inspired from \cite{Gidas}.

\begin{Lem}\label{L:N5}
	Let $N\ge5$ and $f\colon[0,\infty[\to[0,\infty[$ be a continuous function that satisfies (P) and such that $f(t)>0$ if $t>0$. Then the problem
	\begin{equation}\label{e-Gidas}
		\begin{cases}
			-\Delta u\ge f(u)\\
			u\ge0\\
			u\in\cC^2(\rn)\cap L^\infty(\rn)
		\end{cases}
	\end{equation}
	does not admit positive solutions.
\end{Lem}
\begin{proof}
	We can assume $q>1$. If $u$ is a solution to \eqref{e-Gidas}, then there exists $C=C(u)>0$ such that $f\bigl(u(x)\bigr)\ge Cu^q(x)$ for every $x\in\rn$. Then we argue as in \cite[Proof of Theorem 8.4]{QS} and obtain $u=0$.
\end{proof}

\begin{Lem}\label{L:Schwarz}
	Let $L$ be a positive integer and, for every $\ell\in\{1,\dots,L\}$, let $k(\ell)\ge2$ be an integer as well. Define $\bar{k}:=\max_{\ell=1,\dots,L}k(\ell)$. For every $\ell\in\{1,\dots,L\}$ and every $j\in\{1,\dots,k(\ell)\}$ let $\widetilde{F}_{j,\ell}\colon[0,\infty[\to[0,\infty[$ satisfy (at least) one of the following:
	\begin{itemize}
		\item $\widetilde{F}_{j,\ell}$ is increasing;
		\item $\widetilde{F}_{j,\ell}$ is nondecreasing, continuous, and $\widetilde{F}_{j,\ell}(0)=0$.
	\end{itemize}
	Finally, define $\widetilde{F}\colon\R^{\bar{k}}\to\R$ as
	\[
	\widetilde{F}(u)=\sum_{\ell=1}^L\prod_{j=1}^{k(\ell)}\widetilde{F}_{j,\ell}(|u_{i_j}|), \quad 1\le i_1<\dots<i_{k(\ell)}\le\bar{k}.
	\]
	Then for every measurable $u\colon\rn\to\R^{\bar{k}}$ that vanishes at infinity there holds
	\[
	\int_{\rn}\widetilde{F}(u)\,dx\le\int_{\rn}\widetilde{F}(u^*)\,dx,
	\]
	where $u^*$ is the Schwarz rearrangement of $u$.
\end{Lem}
\begin{proof}
	By linearity, we can suppose $L=1$ and relabel $k(1)=\bar{k}=k$ and $\widetilde{F}_{j,1}=\widetilde{F}_j$. Assume preliminarily that each $\widetilde{F}_j$ is increasing. From \cite[Theorem 1.13]{LiebLoss} we have
	\[\begin{split}
		\int_{\rn}\prod_{j=1}^k\widetilde{F}_j(|u_j|)\,dx & =\int_{\rn}\prod_{j=1}^k\int_0^\infty\chi_{\{\widetilde{F}_j(|u_j|)>t_j\}}(x)\,dt_j\,dx\\
		& =\int_0^\infty\cdots\int_0^\infty\int_{\rn}\prod_{j=1}^k\chi_{\{\widetilde{F}_j(|u_j|)>t_j\}}(x)\,dx\,dt_1\cdots \,dt_k
	\end{split}\]
	and similarly
	\[\begin{split}
		\int_{\rn}\prod_{j=1}^k\widetilde{F}_j(u_j^*)\,dx & =\int_0^\infty\cdots\int_0^\infty\int_{\rn}\prod_{j=1}^k\chi_{\{\widetilde{F}_j(u_j^*)>t_j\}}(x)\,dx\,dt_1\cdots \,dt_k\\
		& =\int_0^\infty\cdots\int_0^\infty\int_{\rn}\prod_{j=1}^k\chi_{\{\widetilde{F}_j(|u_j|)>t_j\}^*}(x)\,dx\,dt_1\cdots \,dt_k
	\end{split}\]
	because
	\[
	\{\widetilde{F}_j(u_j^*)>t_j\}=\{u_j^*>\widetilde{F}_j^{-1}(t_j)\}=\{|u_j|>\widetilde{F}_j^{-1}(t_j)\}^*=\{\widetilde{F}_j(|u_j|)>t_j\}^*,
	\]
	therefore it suffices to prove that, for every $A_1,\dots,A_k\subset\rn$ measurable with finite measure,
	\[
	\int_{\rn}\prod_{j=1}^k\chi_{A_j}(x)\,dx\le\int_{\rn}\prod_{j=1}^k\chi_{A_j^*}(x)\,dx,
	\]
	i.e., $|\cap_{j=1}^kA_j|\le|\cap_{j=1}^kA_j^*|$.
	
	Up to relabelling the sets, we can assume that $|A_1|\le\cdots\le|A_k|$, whence $A_1^*\subset\cdots\subset A_k^*$ and so $|\cap_{j=1}^kA_j|\le|A_1|=|A_1^*|=|\cap_{j=1}^kA_j^*|$.
	
	Now assume that some $\widetilde{F}_j$'s are continuous with $\widetilde{F}_j(0)=0$. We want to prove that, as in the previous case, $\{\widetilde{F}_j(u_j^*)>t_j\}=\{\widetilde{F}_j(|u_j|)>t_j\}^*$.
	Since $\widetilde{F}_j\colon[0,\infty[\to[0,\infty[$ is nondecreasing and $\widetilde{F}_j(0)=0$, following, e.g., \cite[p. 10]{RaoRen} we can define the \textit{generalized inverse function} $\widetilde{F}_j^{-1}\colon[0,\infty[\to[0,\infty[$ as
	\[
	\widetilde{F}_j^{-1}(t):=\inf\{s>0:\widetilde{F}_j(s)>t\}.
	\]
	Then it suffices to prove that $\{\widetilde{F}_j(v)>t\}=\{v>\widetilde{F}_j^{-1}(t)\}$ for every $t>0$ and every measurable $v\colon\rn\to[0,\infty[$ that vanishes at infinity. Observe that $\widetilde{F}_j^{-1}(t)=\infty$ if $t\ge\widetilde{F}_j\bigl(v(x)\bigr)$ for every $x\in\rn$ and, in that case, both sets above equal the empty set, hence we can assume that $t<\esssup\widetilde{F}_j\circ v$.
	
	Let $x\in\rn$ such that $\widetilde{F}_j\bigl(v(x)\bigr)>t$. Since $\widetilde{F}_j$ is continuous and $\widetilde{F}_j(0)=0$, there exists $s>0$ such that $t<\widetilde{F}_j(s)<\widetilde{F}_j\bigl(v(x)\bigr)$. Then $\widetilde{F}_j^{-1}(t)\le s$ and, since $\widetilde{F}_j$ is nondecreasing, $s<v(x)$. Now let $x\in\rn$ such that $v(x)>\widetilde{F}_j^{-1}(t)$. From the properties of infima, there exists $s>0$ such that $\widetilde{F}_j(s)>t$ and $v(x)\ge s$. Since $\widetilde{F}_j$ is nondecreasing, $\widetilde{F}_j\bigl(v(x)\bigr)\ge\widetilde{F}_j(s)$.
\end{proof}

\begin{Prop}\label{P:main2}
	Assume that (F0)--(F3) and \eqref{e-etas}--\eqref{e-etal} hold and let $L\ge1$ and $2\le k(\ell)\le K$ be integers, $\ell\in\{1,\dots,L\}$. If, for every $\ell\in\{1,\dots,L\}$ and every $j\in\{1,\dots,k(\ell)\}$, there exist $F_j,\widetilde{F}_{j,\ell}\in\cC^1(\R)$ even, nonnegative, and nondecreasing on $[0,\infty[$ such that $\widetilde{F}_{j,\ell}(0)=0$, $F_j'|_{[0,\infty[}$ satisfies (P) (not necessarily with the same $q$) if $N\ge5$, and
	\[
	F(u)=\sum_{j=1}^KF_j(u_j)+\sum_{\ell=1}^L\prod_{j=1}^{k(\ell)}\widetilde{F}_{j,\ell}(u_{i_j}), \quad 1\le i_1<\dots<i_{k(\ell)}\le K,
	\]
	then there exists $(\lambda,u)\in[0,\infty[^K\times\partial\fD_\textup{r}$ such that $\max_{j=1,\dots,K}\lambda_j>0$, $J(u)=\inf_{\fD_\textup{r}}J=\inf_{\cD}J<0$, and for every $j\in\{1,\dots,K\}$
	\[
	-\Delta u_j+\lambda_ju_j=\partial_jF(u).
	\]
	Moreover, for every $j\in\{1,\dots,K\}$, either $u_j=0$ or $|u_j|_2=\rho_j$. In the latter case, $\lambda_j>0$ and $u_j$ is positive, radially nonincreasing, and of class $\cC^2$.
\end{Prop}
\begin{proof}
	Owing to Remarks \ref{R:cbb} and \ref{R:neg} and Lemma \ref{L:Schwarz}, Lemmas \ref{L:cbb}, \ref{L:neg}, and \ref{L:min} \textit{(b)} still hold for $K\ge2$. Moreover, $u\in W^{2,p}_\textup{loc}(\rn)^K$ for all $p < \infty$ (cf. Section \ref{S:NP}), therefore we can argue as in the proof of Theorem \ref{T:main1} \textit{(b)} and obtain that there exists $(\lambda,u)\in[0,\infty[^K\times\partial\fD_\textup{r}$ such that $\max_{j=1,\dots,K}\lambda_j>0$, $u\in\cC^2(\rn)^K$, $J(u)=\inf_{\fD_\textup{r}}J=\inf_{\cD}J<0$, and for every $j\in\{1,\dots,K\}$ there holds $u_j\ge0$, $u_j$ is radially nonincreasing, and
	\[
	-\Delta u_j+\lambda_ju_j=\partial_jF(u).
	\]
	Let $j\in\{1,\dots,K\}$ such that $\lambda_j=0$ (which, in particular, is the case if $|u_j|_2<\rho_j$), which yields $-\Delta u_j\ge F_j'(u_j)\ge0$. If $N\in\{3,4\}$, then $u_j\in L^\frac{N}{N-2}(\rn)$ and so $u_j=0$ from \cite[Theorem A.2]{Ikoma}. If $N\ge5$, then $u_j=0$ from Lemma \ref{L:N5}. That $u_j>0$ follows from the maximum principle \cite[Lemma IX.V.1]{Evans}.
\end{proof}

\begin{proof}[Proof of Theorem \ref{T:main2}]
	In view of Proposition \ref{P:main2} we only need to prove that $u_j\ne0$ for every $j\in\{1,\dots,K\}$ (recall that $u_j=0$ if $\lambda_j=0$). Assume by contradiction that it does not hold. Up to changing the order, we can suppose that $u_K=0$. For every $j\in\{1,\dots,K\}$ define
	\[
	J_j(w)=\int_{\rn}\frac12|\nabla w|^2-F_j(w)\,dx, \quad w\in\hrn
	\]
	and
	\[
	\fD_\textup{r}(j)=\Set{w\in\fH_\textup{r}|\int_{\rn}w^2\,dx\le\rho_j^2}.
	\]
	Of course, $-\Delta u_j+\lambda_ju_j=F_j'(u_j)$ and, from Lemmas \ref{L:cbb} and \ref{L:neg}, $-\infty<c_j:=\inf_{\fD_\textup{r}(j)}J_j<0$ for every $j\in\{1,\dots,K\}$. Moreover,
	\[
	J(u)=\sum_{j=1}^{K-1}J_j(u_j)\ge\sum_{j=1}^{K-1}c_j.
	\]
	Since $\widetilde{F}_{j,\ell}\ge0$ for every $j\in\{1,\dots,K\}$ and $\ell\in\{1,\dots,L\}$, we have $J(w)\le\sum_{j=1}^KJ_j(w_j)$ for every $w=(w_1,\dots,w_K)\in\hrn^K$, thus
	\[
	\sum_{j=1}^Kc_j\ge\inf_{\fD_\textup{r}}J=J(u)\ge\sum_{j=1}^{K-1}c_j,
	\]
	whence $c_K\ge0$, a contradiction.
\end{proof}

\begin{Rem}
	In Theorem \ref{T:main2} we cannot allow more $k(\ell)$'s as in Proposition \ref{P:main2} because, in that case, we no longer know whether $J(u)=\sum_{j=1}^{K-1}J_j(u_j)$ when $u_K=0$ and the proof above does not apply. Nevertheless, it applies to rule out the case $u_j=0$ for all $j\in\{1,\dots,k(\ell)\}$ but one, for every $\ell\in\{1,\dots,L\}$.
\end{Rem}

\begin{Rem}
	In the assumptions of Theorem \ref{T:main2}, every minimizer of $J|_{\fD_\textup{r}}$ with nonnegative components belongs to $\fS_\textup{r}$.
\end{Rem}

\section{On the ground state energy map}\label{S:gsemsub}

In this section, for $\rho=(\rho_1,\dots,\rho_K)\in]0,\infty[^K$ we denote explicitly
\[\begin{split}
	\cS(\rho)&:=\Set{u\in\hrn^K|\left|u_j\right|_2=\rho_j \text{ for every } j\in\{1,\dots,K\}}\\
	\cD(\rho)&:=\Set{u\in\hrn^K|\left|u_j\right|_2\le \rho_j \text{ for every } j\in\{1,\dots,K\}}\\
	m(\rho)&:=\inf\Set{J(u)|u\in\cD(\rho)}.
\end{split}\]

\begin{Prop}\label{P:gsem}
	Assume that (F0) is satisfied and let $\cA$ be the subset of $]0,\infty[^K$ where \eqref{e-etas} and \eqref{e-etal} both hold.	We have as follows.
	\begin{itemize}
		\item [(i)] $m\colon]0,\infty[^K\to\R\cup\{-\infty\}$ is nonincreasing, i.e., if $0<\rho_j\le\bar{\rho}_j$ for every $j\in\{1,\dots,K\}$, then $m(\rho)\ge m(\bar{\rho})$.
		\item [(ii)] Assume (F1)--(F3) are satisfied. If $F$ is even (when $K=1$) or as in Proposition \ref{P:main2} (when $K\ge2$), then $m|_\cA$ is continuous.
		\item [(iii)] If (F1)--(F3) are satisfied and $\eta_0=\infty$, then $\displaystyle\lim_{|\rho|\to0^+}m(\rho)=0$.
		\item [(iv)] If (F3) is satisfied and $\eta_\infty=0$, then $\displaystyle\lim_{\substack{\rho_j\to\infty \\ j=1,\dots,K}}m(\rho)=-\infty$; in particular, $\displaystyle\lim_{\rho\to\infty}m(\rho)=-\infty$ if $K=1$.
		\item [(v)] If, for all $\rho\in\cA$, there exists a minimizer of $J|_{\cD(\rho)}$ and every minimizer of $J|_{\cD(\rho)}$ belongs to $\cS(\rho)$, then $m|_\cA$ is decreasing, i.e., if $0<\rho_j\le\bar{\rho}_j$ for every $j\in\{1,\dots,K\}$ and $\rho_k<\bar{\rho}_k$ for some $k\in\{1,\dots,K\}$, then $m(\rho)>m(\bar{\rho})$. The same holds true if, for all $\rho\in\cA$, there exists a minimizer of $J|_{\cD(\rho)}$ with nonnegative components and every minimizer of $J|_{\cD(\rho)}$ with nonnegative components belongs to $\cS(\rho)$.
	\end{itemize}
\end{Prop}
\begin{proof}
	\textit{(i)} It is obvious from the definition of $m$.
	
	\textit{(ii)} Note preliminarily that, under these assumptions, for every $\rho\in\cA$ there exists a minimizer of $J|_{\cD(\rho)}$. Let $\rho^n,\rho\in\cA$ such that $\rho^n\to \rho$ (note that $\cA$ is open). Let $u^n\in\cD(\rho^n)$ such that $J(u^n)=m(\rho^n)\le m(\rho/2)$. Up to replacing $u^n$ with $(u^n)^*$, we can assume $u^n\in\fH_\textup{r}$. Since $u^n\in\cD(2\rho)$, from Lemma \ref{L:cbb} and Remark \ref{R:cbb} there exists $u\in\cD(\rho)$ such that $u^n\rightharpoonup u$ up to a subsequence, whence arguing as in Lemma \ref{L:min}
	\begin{equation}\label{e-cont1}
		m(\rho)\le J(u)\le \liminf_nJ(u^n)=\liminf_nm(\rho^n).
	\end{equation}
	Next, let $v\in\cD(\rho)$ such that $J(v)=m(\rho)$ and define $v^n:=(\rho^n_jv_j/\rho_j)_{j=1}^K\in\cD(\rho^n)$. Then $v^n\to v$ and so
	\begin{equation}\label{e-cont2}
		m(\rho)=J(v)=\lim_nJ(v^n)\ge\limsup_nm(\rho^n).
	\end{equation}
	The continuity of $m$ follows from \eqref{e-cont1} and \eqref{e-cont2}.
	
	\textit{(iii)} Note that \eqref{e-etas} holds for every sufficiently small $\rho$. Fix $\varepsilon>0$. Let $\cA\ni \rho^n\to0$ and $u^n\in\cD(\rho^n)$ such that $J(u^n)\le m(\rho^n)+\varepsilon$. In particular, $u^n\to0$ in $L^2(\rn)^K$ and $u^n\in\cD(\tilde{\rho})$ for some $\tilde{\rho}\in]0,\infty[^K$, therefore $u^n$ is bounded in $\hrn^K$ due to Lemma \ref{L:cbb} and Remark \ref{R:cbb}. This, together with (F1), (F2), and Theorem \ref{T:GN}, yields $\int_{\rn}F(u^n)\,dx\to0$, which implies $\liminf_nJ(u^n)=\liminf_n|\nabla u^n|_2^2/2\ge0$,
	whence, in view also of Lemma \ref{L:neg} and Remark \ref{R:neg}, $0\ge\limsup_nm(\rho^n)\ge\liminf_nm(\rho^n)\ge-\varepsilon$. Letting $\varepsilon\to0^+$ we conclude.
	
	\textit{(iv)} Note that \eqref{e-etal} holds for every sufficiently large $\rho$. Fix $\rho\in\cA$, $u\in\cD(\rho)\setminus\{0\}$ and note that $\int_{\rn}F(u)\,dx>0$ from (F3). For every $j\in\{1,\dots,K\}$ let $\rho_j^n\to\infty$ and denote $a_n:=\max_{j=1,\dots,K}\rho_j/\rho_j^n$ and $u^n(x):=u(a_n^{2/N}x)$. Then $\lim_na_n=0$, $u^n\in\cD(\rho^n)$, and
	\[
	m(\rho^n)\le J(u^n)=\frac1{a_n^2}\left(a_n^{4/N}\int_{\rn}\frac12|\nabla u|^2\,dx-\int_{\rn}F(u)\,dx\right)\to-\infty.
	\]
	
	\textit{(v)} Let $\rho,\bar{\rho}\in\cA$ as in the statement. Clearly $m(\rho)\ge m(\bar{\rho})$ from item \textit{(i)}. If $m(\rho)=m(\bar{\rho})$, then there exists $u\in\cS(\rho)\subset\cD(\bar{\rho})\setminus\cS(\bar{\rho})$ such that $J(u)=m(\rho)=m(\bar{\rho})$, which is impossible. The same argument applies to the version of item \textit{(v)} that involves minimizers with nonnegative components.
\end{proof}

\begin{Rem}
	(a) If we assume $N\ge2$ instead of $F$ even or as in Proposition \ref{P:main2}, then Proposition \ref{P:gsem} \textit{(ii)} remains valid if $m(\rho)$ is replaced with $\inf\Set{J(u)|u\in\fH_\textup{r} \text{ and } |u_j|_2\le\rho_j \text{ for all } j\in\{1,\dots,K\}}$.
	
	(b) The assumptions of Proposition \ref{P:gsem} \textit{(v)} are satisfied by those of Theorem \ref{T:main1} \textit{(b)} or Theorem \ref{T:main2}. Similarly as in point (a), if $m(\rho)$ is replaced with $\inf\Set{J(u)|u\in\fH_\textup{r} \text{ and } |u_j|_2\le\rho_j \text{ for all } j\in\{1,\dots,K\}}$, then the assumptions of Proposition \ref{P:gsem} \textit{(v)} are satisfied by those of Theorem \ref{T:main1} \textit{(a)}.
\end{Rem}

\chapter{Autonomous Schr\"odinger equations in the mass-supercritical case}\label{K:super}

\section{Preliminaries and statement of the results}\label{S:prel}

In this chapter, based on \cite{MeSc}, we study problem \ref{e-main} together with the auxiliary problem
\begin{equation}\label{e-mainineq}
\begin{cases}
	-\Delta u_j+\lambda_ju_j=\partial_jF(u)\\
	\int_{\rn}u_j^2\,dx\le\rho_j^2\\
	(\lambda_j,u_j)\in\R\times\hrn
\end{cases}
j\in\{1,\dots,K\}
\end{equation}
with $N\ge3$ and $K\ge1$. As in Chapter \ref{K:sub}, $\rho=(\rho_1,\dots,\rho_K)\in]0,\infty[^K$ is prescribed and $(\lambda,u)=(\lambda_1,\dots,\lambda_K,u_1,\dots,u_K)$ is the unknown. Denoting $f:=\nabla F$, $H(u):=f(u)\cdot u-2F(u)$ for $u\in\rk$, and $h:=\nabla H$, the assumptions about the nonlinearity correspond to the mass-supercritical case and are as follows.
\begin{itemize}
	\item [(F0)] $f$ and $h$ are continuous and there exists $S>0$ such that $|f(u)|+|h(u)|\le S(|u|+|u|^{2^*-1})$.
	\item [(F1)] $\displaystyle\eta:=\limsup_{u\to0}\frac{F(u)}{|u|^{2_\#}}<\infty$.
	\item [(F2)] $\displaystyle\lim_{|u|\to\infty}\frac{F(u)}{|u|^{2_\#}}=\infty$.
	\item [(F3)] $\displaystyle\lim_{|u|\to\infty}\frac{F(u)}{|u|^{2^*}}=0$.
	\item [(F4)] $\displaystyle2_\#H(u)\le h(u)\cdot u$.
	\item [(F5)] $\displaystyle(2_\#-2)F\le H\le(2^*-2)F$.
	\item [(F6)] There exists $\zeta\in\rn$ such that $H(\zeta)>0$.
\end{itemize}
Note that (F5) implies $F,H\ge0$ and that, if (F2) and (F5) hold, then so does (F6). Let us recall the definitions of $J$, $\cS$, $\cD$, and $\cM$ given by, respectively, \eqref{e-energy}, \eqref{e-Sset}, \eqref{e-Dset}, and \eqref{e-Mset}; in particular, let us provide a motive for the set $\cM$. Clearly, if $(\lambda,u)\in\rk\times\hrn^K$ is such that
\[
-\Delta u_j+\lambda_ju_j=\partial_jF \quad \text{for every }j\in\{1,\dots,K\},
\]
then it satisfies the Nehari identity
\begin{equation}\label{e-NehaL}
\int_{\rn}|\nabla u|^2+\sum_{j=1}^K\lambda_ju_j^2-\nabla F(u)\cdot u\,dx=0.
\end{equation}
Moreover, since, in that case, $u\in W_\textup{loc}^{2,p}(\rn)^K$ for every $p<\infty$ (cf. Section \ref{S:NP}), $(\lambda,u)$ satisfies also the Poho\v{z}aev identity
\begin{equation}\label{e-PohoL}
\int_{\rn}|\nabla u|^2+\frac{2^*}{2}\sum_{j=1}^K\lambda_j^2u_j^2-2^*F(u)\,dx=0.
\end{equation}
By a suitable linear combination of \eqref{e-NehaL} and \eqref{e-PohoL} we can get rid of the unknown quantity $\lambda$ and have that every solution satisfies
\[
M(u):=\int_{\rn}|\nabla u|^2-\frac{N}2H(u)\,dx=0,
\]
so that $\cM=\Set{u\in\hrn^K\setminus\{0\}|M(u)=0}$. Recall also that $J$ and $M$ are well defined and of class $\cC^1$ from (F0). Moreover, it is easily checked that $\cM\ne\emptyset$; as a matter of fact, for every $u\in\hrn^K$ such that $\int_{\rn}H(u)\,dx>0$ let us define
\[
R:=R_u:=\sqrt{\frac{N\int_{\rn}H(u)\,dx}{2\int_{\rn}|\nabla u|^2\,dx}}>0
\]
so that $u(R\cdot)\in\cM$. Now, in view of (F6) and arguing as in \cite[page 325]{BerLions}, for every $r>0$ there exists $w\in H^1_0(B_r)\cap L^\infty(B_r)$ such that $\int_{\rn}H(u)\,dx>0$. In fact, we have the following.

\begin{Prop}
If (F0) and (F6) hold, then $\cM\subset\hrn^K$ is a submanifold of class $\cC^1$ and codimension $1$.
\end{Prop}
\begin{proof}
Let $u\in\hrn^K$ such that $M'(u)=0$, i.e., $u$ solves $-\Delta u=\frac{N}4h(u)$. Then $u$ satisfies the Poho\v{z}aev identity $\int_{\rn}|\nabla u|^2-\frac{N^2}{2(N-2)}H(u)\,dx=0$. If $M(u)=0$, then we infer $|\nabla u|_2=0$.
\end{proof}

We introduce the following relation:
\begin{itemize}
	\item[] Let $f_1,f_2\colon\R^K\to\R$. Then $f_1\preceq f_2$ if and only if $f_1\le f_2$ and for every $\varepsilon>0$ there exists $u\in\R^K$, $|u|<\varepsilon$, such that $f_1(u)<f_2(u)$.
\end{itemize}
For better outcomes we will need a stronger variant of (F4), denoted (F4,$\preceq$), where the inequality $\le$ is replaced with $\preceq$.

The first result, which concerns the auxiliary problem \eqref{e-mainineq}, reads as follows.

\begin{Th}\label{T:Main1}
	Suppose (F0)--(F5) hold and
	\begin{equation}\label{e-eta2}
	2^*C_{N,2_\#}^{2_\#}\eta|\rho|^{4/N}<1.
	\end{equation}

	(a) There exists $u\in\cM\cap\cD$ such that
	$J(u)=\inf_{\cM\cap\cD}J$. Moreover, $u$ is a $K$-tuple of radial, nonnegative, and radially nonincreasing functions provided that $F$ is of the form
	\begin{equation}\label{e-Gsp}
	F(u)=\sum_{j=1}^KF_j(u_j)+\sum_{\ell=1}^L\beta_\ell\prod_{j=1}^K|u_j|^{r_{j,\ell}},
	\end{equation}
	where $L\geq 1$, $F_j\colon\R\to[0,\infty)$ is even, $r_{j,\ell}>1$ or $r_{j,\ell}=0$, $\beta_\ell\geq 0$, $2_\#\le\sum_{j=1}^Kr_{j,\ell}<2^*$, and for every $\ell$ there exists $j_1\neq j_2$ such that $r_{j_1,\ell},r_{j_2,\ell}>1$.
	
	(b) If, moreover, (F4,$\preceq$) holds, then $u$ is of class $\cC^2$ and there exists $\lambda\in[0,\infty[^K$ such that $(\lambda,u)$ is a ground state solution to \eqref{e-mainineq}.
\end{Th}

Since $J|_{\cS}$ is unbounded from below, by \textit{ground state solution} to \eqref{e-mainineq} we mean a solution $(\lambda,u)$ such that $J(u)=\min_{\cM\cap\cD}J$. Likewise, by ground state solution to \eqref{e-main} we mean a solution $(\lambda,u)$ such that $J(u)=\min_{\cM\cap\cD}J$. Note that, similarly to Chapter \ref{K:sub}, this is more than just requiring the ``more natural'' condition $J(u)=\min_{\cM\cap\cS}J$.

From on, when we say that $F$ is of the form \eqref{e-Gsp}, we also mean the additional conditions on $F_j$, $\beta_\ell$, and $r_{j,\ell}$ listed in Theorem \ref{T:Main1} \textit{(a)}; in addition, when $K=1$, we agree that $\beta_\ell=0$ for all $\ell\in\{1,\dots,L\}$. Observe that $F$ of the form \eqref{e-Gsp} satisfies (F4) if and only if each $F_j$ satisfies the scalar variant of (F4) for all $j\in \{1,\dots,K\}$. If, in addition, $F_j$ satisfies (F4,$\preceq$) for some $j$, then $F$ satisfies (F4,$\preceq$) as well.

More can be said if  $N\in\{3,4\}$.

\begin{Th}\label{T:Main2}
	Assume that (F0)--(F3), (F4,$\preceq$), (F5), and \eqref{e-eta2} are satisfied, $F$ is of the form \eqref{e-Gsp}, and $N\in\{3,4\}$. Then there exist $u\in\cM\cap\partial\cD$ of class $\cC^2$ and $\lambda\in[0,\infty[^K$ such that $(u,\lambda)$ is a ground state solution to \eqref{e-mainineq} and each $u_j$ is radial, nonnegative, and radially nonincreasing.
	Moreover, for every $j\in\{1,\dots,K\}$ either $u_j=0$ or $|u_j|_2=\rho_j$ and, if $u_j\ne0$, then $\lambda_j>0$ and $u_j>0$. In particular, if $u\in\cS$, then $\lambda\in]0,\infty[^K$ and $(\lambda,u)$ is a ground state solution to \eqref{e-main}.
\end{Th}

If $K=2$, $L=1$, and the coefficient of the coupling term is large, then we find ground state solutions to \eqref{e-main}.

\begin{Th}\label{T:Main3}
	Assume that (F0)--(F3), (F4,$\preceq$), (F5), and \eqref{e-eta2} are satisfied, $N\in\{3,4\}$, $K=2$, and $L=1$. If $F$ is of the form \eqref{e-Gsp} and $r_{1,1}+r_{2,1}>2_\#$, then for every sufficiently large $\beta_1>0$ there exist $u\in\cM\cap\cS$ of class $\cC^2$ and $\lambda\in]0,\infty[^2$ such that ($\lambda,u)$ is a ground state solution to \eqref{e-main} and $u_1,u_2$ are positive, radial, and radially nonincreasing.
\end{Th}

Regarding possible examples of scalar functions $F_1,F_2$ we refer to (E1)--(E4) in \cite{BiegMed}; in particular, we can deal with
\[
F_j(u_j)=\frac{\mu_j}{p_j}|u_j|^{p_j}+\frac{\nu_j}{2_\#}|u_j|^{2_\#}, \quad \mu_j,\nu_j>0, \, j\in\{1,2\},
\]
where $\eta=\max\{\nu_1,\nu_2\}/2_\#>0$ due to Proposition \ref{P:ex}.

\section{Proof of the results}

We begin this section with two preliminary lemmas, whose proofs are omitted because the reasoning is the same as for the scalar case $K=1$, provided in \cite[Lemma 2.1]{BiegMed} and \cite[Theorem 1.4]{NonradMed} respectively.

\begin{Lem}\label{L:ineq}
	Let $F_1,F_2\in\cC(\R^K)$ and assume there exists $C>0$ such that $|F_1(u)|+|F_2(u)|\le C(|u|^2+|u|^{2^*})$ for every $u\in\R^K$. Then $F_1\preceq F_2$ if and only if $F_1\le F_2$ and for every $u\in H^1(\rn)^K\setminus\{0\}$
	\[
	\int_{\rn}F_1(u)-F_2(u)\,dx<0.
	\]
\end{Lem}

\begin{Lem}\label{L:split}
	Let $u^n\in H^1(\rn)^K$ be bounded. Then there exist sequences $(\tilde{u}^i)_{i=0}^\infty\subset H^1(\rn)^K$ and $(y^{i,n})_{i=0}^\infty\subset\rn$ such that $y^{0,n}=0$, $\lim_n|y^{i,n}-y^{j,n}|=0$ if $i\ne j$, and for every $i\ge0$ and every $G\colon\rn\to[0,\infty[$ of class $\cC^1$ such that
	\[
	\lim_{u\to0}\frac{G(u)}{|u|^2}=\lim_{|u|\to\infty}\frac{G(u)}{|u|^{2^*}}=0
	\]
	there holds
	\begin{eqnarray}
		u^n(\cdot+y^{i,n})\rightharpoonup\tilde{u}^i\text{ in }H^1(\rn)^K\text{ as }n\to\infty \label{eq:Me1}\\
		\lim_n\int_{\rn}|\nabla u^n|^2\,dx=\sum_{j=0}^i\int_{\rn}|\nabla\tilde{u}^j|^2\,dx+\lim_n\int_{\rn}|\nabla v^{i,n}|^2\,dx \label{eq:Me2}\\
		\limsup_n\int_{\rn}G(u^n)\,dx=\sum_{i=0}^\infty\int_{\rn}G(\tilde{u}^i)\,dx, \label{eq:Me3}
	\end{eqnarray}
	where $v^{i,n}(x):=u^n(x)-\sum_{j=0}^{i}\tilde{u}^j(x-y^{j,n})$.
\end{Lem}

Henceforth, we will always assume that (F0) holds and we will make use of it without explicit mention. Recall also that (F6) holds if both (F2) and (F5) do.

\begin{Lem}\label{L:bdaw1}
	Assume that (F0), (F1), (F5), (F6), and \eqref{e-eta2} hold. Then
	$\inf\Set{\left|\nabla u\right|_2^2|u\in\cM\cap\cD}>0$.
\end{Lem}
\begin{proof}
	From Theorem \ref{T:GN}, for every $\varepsilon>0$ there exists $c_\varepsilon$ such that for every $u\in\cM\cap\cD$
	\[\begin{split}
		|\nabla u|_2^2&=\frac{N}{2}\int_{\rn} H(u)\,dx\le 2^*\bigl(c_\varepsilon|u|_{2^*}^{2^*}+(\varepsilon+\eta)|u|_{2_\#}^{2_\#}\bigr)\\
		&=2^*\left(c_\varepsilon\big||u|\big|_{2^*}^{2^*}+(\varepsilon+\eta)\big||u|\big|_{2_\#}^{2_\#}\right)\\
		&\le2^*\left(c_\varepsilon C_{N,2^*}^{2^*}\big|\nabla|u|\big|_2^{2^*}+(\varepsilon+\eta)C_{N,2_\#}^{2_\#}|\rho|^{4/N}\big|\nabla|u|\big|_2^2\right)\\
		&\le2^*\bigl(c_\varepsilon C_{N,2^*}^{2^*}|\nabla u|_2^{2^*}+(\varepsilon+\eta)C_{N,2_\#}^{2_\#}|\rho|^{4/N}|\nabla u|_2^2\bigr)
	\end{split}\]
	i.e.
	\[
	0\le2^*c_\varepsilon C_{N,2^*}^{2^*}|\nabla u|_2^{2^*}+\bigl(2^*(\varepsilon+\eta)C_{N,2_\#}^{2_\#}|\rho|^{4/N}-1\bigr)|\nabla u|_2^2
	\]
	Taking $\varepsilon$ sufficiently small so that
	\[
	2^*(\varepsilon+\eta)C_{N,2_\#}^{2_\#}|\rho|^{4/N}<1
	\]
	we conclude.
\end{proof}

Similarly to Chapter \ref{K:sub}, for $u\in H^1(\rn)^K\setminus\{0\}$ and $s>0$ define $s\star u(x):=s^{N/2}u(sx)$ and $\varphi(s):=J(s\star u)$. Recall that $|s\star u|_2=|u|_2$.

\begin{Lem}\label{L:phi}
	Assume that (F0)--(F2), (F4), and (F5) hold and let $u\in H^1(\rn)^K\setminus\{0\}$ such that
	\begin{equation}\label{e-eta}
		\eta<\frac{|\nabla u|_2^2}{2|u|_{2_\#}^{2_\#}}.
	\end{equation}
	Then there exist $a=a(u)>0$ and $b=b(u)\ge a$ such that each $s\in[a,b]$ is a global maximizer for $\varphi$ and $\varphi$ is increasing on $]0,a[$ and decreasing on $]b,\infty[$. Moreover $s\star u\in\cM$ if and only if $s\in[a,b]$, $M(s\star u)>0$ if and only if $s\in]0,a[$
	and $M(s\star u)<0$ is and only if $s>b$. If (F4,$\preceq$) holds, then $a=b$.
\end{Lem}
Note that \eqref{e-eta2} implies \eqref{e-eta} provided that $u\in\cD$.
\begin{proof}
	Notice that from (F1)
	\[
	\varphi(s)=\int_{\rn}\frac{s^2}{2}|\nabla u|^2-\frac{F(s^{N/2}u)}{s^N}\,dx\to0
	\]
	as $s\to0^+$, while from (F2) $\lim_{s\to\infty}\varphi(s)=-\infty$.
	From (F1) for every $\varepsilon>0$ there exists $c_\varepsilon>0$ such that
	\[
	F(u)\le(\varepsilon+\eta)|u|^{2_\#}+c_\varepsilon|u|^{2^*},
	\]
	therefore,
	\[
	\varphi(s)\geq s^2\Big(\int_{\rn}\frac{1}{2}|\nabla u|^2-(\eta+\varepsilon)|u|^{2_\#}\,dx\Big)-c_\varepsilon s^{2^*}\int_{\R^N}|u|^{2^*}\,dx>0
	\]
	for sufficiently small $\varepsilon$ and $s$.	
	It follows that there exists an interval $[a,b]\subset]0,\infty[$ such that $\varphi|_{[a,b]}=\max\varphi$. Moreover
	\[
	\varphi'(s)=s\int_{\rn}|\nabla u|^2-\frac{N}{2}\frac{H(s^{N/2}u)}{s^{N+2}}\,dx
	\]
	and the function
	\[
	s\in]0,\infty[\mapsto\int_{\rn}\frac{H(s^{N/2}u)}{s^{N+2}}\,dx
	\]
	is nondecreasing (resp. increasing) due to (F4) (resp. (F4,$\preceq$) and Lemma \ref{L:ineq}) and tends to $\infty$ as $s\to\infty$ due to (F2) and (F5). There follows that $\varphi'(s)>0$ if $s\in(0,a)$ and $\varphi'(s)<0$ if $s>b$ and that $a=b$ if (F4,$\preceq$) holds. Finally, observe that
	\[
	s\varphi'(s)=\int_{\rn}s^2|\nabla u|^2-\frac{N}{2}\frac{H(s^{N/2}u)}{s^N}\,dx=M(s\star u).\qedhere
	\]
\end{proof}

\begin{Lem}\label{L:coerc}
	If (F0)--(F2), (F4), (F5), and \eqref{e-eta2} are verified, then $J$ is coercive on $\cM\cap\cD$.
\end{Lem}
\begin{proof}
	First of all, note that, if $u\in\cM$, then due to (F5)
	\[
	J(u)=J(u)-\frac12M(u)=\int_{\rn}\frac{N}{4}H(u)-F(u)\,dx\ge0
	\]
	and so, a fortiori, $J$ is nonegative on $\cM\cap\cD$.
	Let $u^n\in\cM\cap\cD$ such that $\|u^n\|\to\infty$, i.e., $\lim_n|\nabla u^n|_2=\infty$, and define
	\[
	s_n:=|\nabla u^n|_2^{-1}>0\quad\text{and}\quad w^n:=s_n\star u^n.
	\]
	Note that $s_n\to0$, $|w^n_j|_2=|u^n_j|_2\le\rho_j$ for $j\in\{1,\dots,K\}$, and $|\nabla w^n|_2=1$, in particular $w^n$ is bounded in $H^1(\rn)^K$.
	Suppose by contradiction that
	\[
	\limsup_n\max_{y\in\rn}\int_{B(y,1)}|w^n|^2\,dx>0.
	\]
	Then there exist $y^n\in\rn$ and $w\in H^1(\rn)^K$ such that, up to a subsequence, $w^n(\cdot+y^n)\rightharpoonup w\ne0$ in $H^1(\rn)^K$ and $w^n(\cdot+y^n)\to w$ a.e. in $\rn$. Thus, owing to (F2),
	\[\begin{split}
		0&\le\frac{J(u^n)}{|\nabla u^n|_2^2}=\frac12-\int_{\rn}\frac{F(u^n)}{|\nabla u^n|_2^2}\,dx=\frac12-s_n^{N+2}\int_{\rn}F\bigl(u^n(s_nx)\bigr)\,dx\\
		&=\frac12-s_n^{N+2}\int_{\rn}F(s_n^{-N/2}w^n)=\frac12-\int_{\rn}\frac{F(s_n^{-N/2}w^n)}{|s_n^{-N/2}w^n|^{2_\#}}|w^n|^{2_\#}\,dx\\
		&=\frac12-\int_{\rn}\frac{F\bigl(s_n^{-N/2}w^n(x+y^n)\bigr)}{|s_n^{-N/2}w^n(x+y^n)|^{2_\#}}|w^n(x+y^n)|^{2_\#}\,dx\to-\infty.
	\end{split}\]
	It follows that
	\[
	\lim_n\max_{y\in\rn}\int_{B(y,1)}|w^n|^2\,dx=0
	\]
	and so, from \cite[Lemma I.1]{Lions84_2}, $w^n\to0$ in $L^{2_\#}(\rn)^K$. Since
	\[
	s_n^{-1}\star w^n=u^n\in\cM,
	\]
	Lemma \ref{L:phi} yields
	\[
	J(u^n)=J(s_n^{-1}\star w^n)\ge J(s\star w^n)=\frac{s^2}{2}-s^N\int_{\rn}F(s^{N/2}w^n)\,dx
	\]
	for every $s>0$. Taking into account that
	\[
	\lim_n\int_{\rn}F(s^{N/2}w^n)\,dx=0,
	\]
	we have that $\liminf_nJ(u^n)\ge s^2/2$ for every $s>0$, i.e., $\lim_nJ(u^n)=\infty$.
\end{proof}

\begin{Lem}\label{L:bdaw2}
	If (F0)--(F2), (F4), (F5), and \eqref{e-eta2} are verified, then $\inf_{\cM\cap\cD}J>0$.
\end{Lem}
\begin{proof}
	We prove that there exists $\alpha>0$ such that
	\begin{equation}\label{e-alpha}
		|\nabla u|_2\le\alpha\Rightarrow J(u)\ge\frac{|\nabla u|_2^2}{2N}.
	\end{equation}
	From Theorem \ref{T:GN} and \eqref{e-eta2}, for every $\varepsilon>0$ there exists $c_\varepsilon>0$ such that
	\[\begin{split}
		\int_{\rn}F(u)\,dx&\le c_\varepsilon C_{N,2^*}^{2^*}|\nabla u|_2^{2^*}+(\varepsilon+\eta)C_{N,2_\#}^{2_\#}|\rho|^{4/N}|\nabla u|_2^2\\
		&\le\left(c_\varepsilon C_{N,2^*}^{2^*}|\nabla u|_2^{2^*-2}+\varepsilon C_{N,2_\#}^{2_\#}|\rho|^{4/N}+\frac12-\frac1N\right)|\nabla u|_2^2.
	\end{split}\]
	Choosing
	\[
	\varepsilon=\frac{1}{4NC_{N,2_\#}^{2_\#}|\rho|^{4/N}} \quad \text{and} \quad \alpha=\frac{1}{(4Nc_\varepsilon C_{N,2^*}^{2^*})^\frac{1}{2^*-2}}
	\]
	we obtain, provided $|\nabla u|_2\le\alpha$,
	\[
	\int_{\rn}F(u)\,dx\le\biggl(\frac12-\frac{1}{2N}\biggr)|\nabla u|_2^2
	\]
	and so $\displaystyle J(u)\ge\frac{|\nabla u|_2^2}{2N}$.
	Now take $u\in\cM\cap\cD$ and $\alpha>0$ such that \eqref{e-alpha} holds and define
	\[
	s:=\frac{\alpha}{|\nabla u|_2}\quad\text{and}\quad w:=s\star u.
	\]
	Clearly $|w_j|_2=|u_j|_2\le\rho_j$ for $j\in\{1,\dots,K\}$ and $|\nabla w|_2=\alpha$, whence in view of Lemma \ref{L:phi}
	\[
	J(u)\ge J(w)\ge\frac{|\nabla w|_2^2}{2N}=\frac{\alpha^2}{2N}>0.\qedhere
	\]
\end{proof}


\begin{Lem}\label{L:infismin}
	If (F0)--(F5) and \eqref{e-eta2} hold, then $\inf_{\cM\cap\cD}J$ is attained.
\end{Lem}
\begin{proof}
	Let $u^n\in\cM\cap\cD$ such that $\lim_nJ(u^n)=\inf_{\cM\cap\cD}J$. Then $u^n$ is bounded due to Lemma \ref{L:coerc} and, in view of Lemma \ref{L:split}, we find $(\tilde{u}^i)_{i=0}^\infty\subset H^1(\rn)^K$ and $(y_n^{i,n})_{i=0}^\infty\subset\rn$ such that \eqref{eq:Me1}--\eqref{eq:Me3} hold. Let
	\[
	I:=\{i\ge0:\tilde{u}^i\ne0\}
	\]
	and suppose by contradiction that $I=\emptyset$. Then, since $u^n\in\cM\cap\cD$, there holds
	\[
	\lim_n\int_{\rn}|\nabla u^n|\,dx=\lim_n\frac{N}{2}\int_{\rn}H(u^n)\,dx=0
	\]
	owing to \eqref{eq:Me3}, which contradicts Lemma \ref{L:bdaw1}. Now we prove that
	\begin{equation}\label{eq:claim1}
		\frac{N}{2}\int_{\rn}H(\tilde{u}^i)\,dx\ge\int_{\rn}|\nabla\tilde{u}^i|^2\,dx
	\end{equation}
	for some $i\in I$. Assume by contradiction that
	\[
	\frac{N}{2}\int_{\rn}H(\tilde{u}^i)\,dx<\int_{\rn}|\nabla\tilde{u}^i|^2\,dx
	\]
	for every $i\in I$. Then from \eqref{eq:Me2} and \eqref{eq:Me3} we have
	\[\begin{split}
		\limsup_n\,\frac{N}{2}\int_{\rn}H(u^n)\,dx=\limsup_n\int_{\rn}|\nabla u^n|^2\ge\sum_{i=0}^\infty\int_{\rn}|\nabla\tilde{u}^i|^2\\
		>\sum_{i=0}^\infty\frac{N}{2}\int_{\rn}H(\tilde{u}^i)\,dx=\limsup_n\frac{N}{2}\int_{\rn}H(u^n)\,dx,
	\end{split}\]
	a contradiction. Let $\tilde u=\tilde{u}^i$ satisfy \eqref{eq:claim1} for some $i\in I$. Then there exists $R>0$ such that $\tilde u(R\cdot)\in\cM$ and again from \eqref{eq:claim1} we indeed know that $R\geq 1$, whence $\tilde{u}(R\cdot)\in\cD$. Hence Fatou's Lemma yields
	\[\begin{split}
		\inf_{\cM\cap\cD}J&\le J\bigl(\tilde u(R\cdot)\bigr)=J\bigl(\tilde u(R\cdot)\bigr)-\frac12M\bigl(\tilde u(R\cdot)\bigr)\,dx\\
		&=\frac{1}{R^N}\int_{\rn}\frac{N}{4}H(\tilde u)-F(\tilde{u})\,dx\le\liminf_n\int_{\rn}\frac{N}{4}H(u^n)-F(u^n)\,dx\\
		&=\liminf_nJ(u^n)-\frac12M(u^n)=\liminf_nJ(u^n)=\inf_{\cM\cap\cD}J,
	\end{split}\]
	i.e., $R=1$ and $J(\tilde{u})=\inf_{\cM\cap\cD}J$.	
\end{proof}

\begin{Lem}\label{L:radmin}
	Assume that (F0)--(F5) and \eqref{e-eta2} are verified and $F$ is of the form \eqref{e-Gsp}.
	Then $\inf_{\cM\cap\cD}J$ is attained by a $K$-tuple of radial, nonnegative and radially nonincreasing functions.
\end{Lem}
\begin{proof}
	Let $\tilde{u}\in\cM\cap\cD$ such that $J(\tilde{u})=\inf_{\cM\cap\cD}J$ be given by Lemma \ref{L:infismin}.
	For simplicity, let us denote, for every $j=1,\dots,K$, $u_j:=\tilde{u}_j^*$ and $u:=(u_1,\dots,u_K)$. Let $a=a(u)$ be determined by Lemma \ref{L:phi}. Since
	\[
	M(1\star u)=M(u)\le M(\tilde{u})=0,
	\]
	in view of Lemma \ref{L:phi} we have that $a\le1$ and, consequently, $M(a\star\tilde{u})\ge0$. Let
	\[
	d:=\frac{N}{2}\max_{\ell=1,\dots, L}\left(\sum_{j=1}^Kr_{j,\ell}-2\right)\ge2.
	\]
	Then, from Lemma \ref{L:Schwarz},
	\[\begin{split}
		\inf_{\cM\cap\cD}J\le \, & J(a\star u)=J(a\star u)-\frac{1}{d}M(a\star u)\\
		= \, &\int_{\rn}\sum_{j=1}^Ka^2\biggl(\frac12-\frac{1}{d}\biggr)|\nabla u_i|^2+\frac{1}{a^N}\biggl(\frac{N}{2d}H_j(a^{N/2}u_j)-F_j(a^{N/2}u_j)\biggr)\,dx\\
		&-\frac{1}{a^N}
		\sum_{\ell=1}^L\beta_\ell\Bigg(1-\frac1{d}\bigg(\sum_{j=1}^Kr_{j,\ell}-2\bigg)\Bigg)\prod_{j=1}^K|a^{N/2}u_j|^{r_{j,\ell}}\\
		\le \, &\int_{\rn}\sum_{j=1}^Ka^2\biggl(\frac12-\frac{1}{d}\biggr)|\nabla\tilde{u}_j|^2+\frac{1}{a^N}\biggl(\frac{N}{2d}H_j(a^{N/2}\tilde{u}_j)-F_j(a^{N/2}\tilde{u}_j)\biggr)\,dx\\
		&-\frac{1}{a^N}
		\sum_{\ell=1}^L\beta_\ell\Bigg(1-\frac1{d}\bigg(\sum_{j=1}^Kr_{j,\ell}-2\bigg)\Bigg)\prod_{j=1}^K|a^{N/2}\tilde{u}_j|^{r_{j,\ell}}\\
		= \, &J(a\star\tilde{u})-\frac{1}{d}M(a\star\tilde{u})\le J(a\star\tilde{u})\le J(\tilde{u})=\inf_{\cM\cap\cD}J,
	\end{split}\]
	i.e., $J(a\star u)=\inf_{\cM\cap\cD}J$.
\end{proof}

\begin{Lem}\label{L:minons}
	(a) Assume that (F0)--(F3), (F4,$\preceq$), (F5), and \eqref{e-eta2} hold and let $u\in\cM\cap\cD$ such that $J(u)=\inf_{\cM\cap\cD}J$ and $u_j$ is radial, nonnegative, and radially nonincreasing for every $j\in\{1,\dots,K\}$. Then $u$ is of class $\cC^2$.
	
	(b) If, in addition, $N\in\{3,4\}$ and $F$ is of the form \eqref{e-Gsp}, then $u\in\partial\cD$. Moreover, for every $j\in\{1,\dots,K\}$ either $u_j=0$ or $|u_j|_2=\rho_j$.
\end{Lem}
\begin{proof}
	\textit{(a)} 
	In Proposition \ref{P:Clarke} we put $m=K$, $n=1$, $\cH=\hrn^K$, $f=J$,  $\phi_j=|\cdot|_2^2-\rho_j^2$ for all $1\leq j\leq K$, and $\psi_1=M$. Then there exist $(\lambda_1,\dots,\lambda_K)\in[0,\infty[^K$ and $\sigma\in\R$ such that
	\begin{equation}\label{e-solution}
	-(1-2\sigma)\Delta u_j+\lambda_ju_j=\partial_jF(u)-\sigma\frac{N}{2}\partial_jH(u)
	\end{equation}
	for every $i\in\{1,\dots,K\}$ and $u$ satisfies the Nehari identity
	\begin{equation}\label{e-sigN}
	\int_{\rn}(1-2\sigma)|\nabla u|^2+\sum_{j=1}^K\lambda_ju_j^2+\sigma\frac{N}{2}h(u)\cdot u-f(u)\cdot u\,dx=0.
	\end{equation}
	If $\sigma=1/2$, then (F4,$\preceq$), (F5), and \eqref{e-sigN} yield
	\[\begin{split}
	0&\ge\int_{\rn}\frac{N}{4}h(u)\cdot u-f(u)\cdot u\,dx=\int_{\rn}\frac{N}{4}h(u)\cdot u-H(u)-2F(u)\,dx\\
	&>\int_{\rn}\frac{N}{2}H(u)-2F(u)\,dx\ge0,
	\end{split}\]
	a contradiction. Hence $\sigma\ne1/2$ and $u$ satisfies also the Poho\v{z}aev identity
	\begin{equation}\label{e-sigP}
	\int_{\rn}(1-2\sigma)|\nabla u|^2+\frac{2^*}{2}\sum_{j=1}^K\lambda_ju_j^2+2^*\left(\sigma\frac{N}{2}H(u)-F(u)\right)\,dx=0.
	\end{equation}
	Combining \eqref{e-sigN} and  \eqref{e-sigP} we obtain
	\[
	(1-2\sigma)\int_{\rn}|\nabla u|^2\,dx+\frac{N}{2}\int_{\rn}\sigma N\left(\frac12h(u)\cdot u-H(u)\right)-H(u)\,dx=0
	\]
	and, using the fact that $u\in\cM$,
	\[
	(1-2\sigma)\int_{\rn}H(u)\,dx+\int_{\rn}\sigma N\left(\frac12h(u)\cdot u-H(u)\right)-H(u)\,dx=0,
	\]
	that is,
	\[
	\sigma\int_{\rn}h(u)\cdot u-2_\#H(u)\,dx=0,
	\]
	which together with (F4,$\preceq$) yields $\sigma=0$. Since $u\in W^{2,p}_\textup{loc}(\R^N)^K$ for all $p < \infty$, we can argue as in the proof of \cite[Lemma 1]{BerLions} and have that $u$ is of class $\cC^2$.
	
	\textit{(b)} Suppose by contradiction that $\lambda_1=\dots=\lambda_K=0$, which is the case when $|u_j|<\rho_j$ for every $j$. From \eqref{e-sigN} and \eqref{e-sigP} (with $\sigma=0$) there follows
	\begin{equation}\label{e-Gg2^*}
	\int_{\rn}f(u)\cdot u-2^*F(u)\,dx=0.
	\end{equation}
	In view of (F5) 
	\begin{equation}\label{e-pointSob}
	2^*F\bigl(u(x)\bigr)=f\bigl(u(x)\bigr)\cdot u(x)
	\end{equation}
	for all $x\in\R^N$. Since $F_j$ satisfies the scalar variant of (F5), $2^*F_j\bigl(u_j(x)\bigr)\geq f_j\bigl(u_j(x)\bigr)u_j(x)$ for every $j\in\{1,\dots,K\}$ and note that
	\[
	2^*\sum_{\ell=1}^L\beta_\ell\prod_{j=1}^K|u_j(x)|^{r_{j,\ell}}\geq \sum_{\ell=1}^L\beta_\ell\sum_{k=1}^K r_{k,\ell}\prod_{j=1}^K|u_j(x)|^{r_{j,\ell}},
	\]
	hence, from \eqref{e-pointSob}, the equalities above are actually equalities. On the other hand, for every $\ell\in\{1,\dots,L\}$, $\sum_{j=1}^Kr_{j,\ell}<2^*$, which yields $\beta_\ell=0$ or $\prod_{j=1}^K|u_j(x)|^{r_{j,\ell}}=0$ for every $x\in\rn$, thus
	$$2^* F_j\bigl(u_j(x)\bigr)=f_j\bigl(u_j(x)\bigr)u_j(x)$$
	for every $j\in\{1,\dots,K\}$ and every $x\in\rn$.
	
	Now fix $j\in\{1,\dots,K\}$ such that $u_j\neq 0$. Since $u_j\in\hrn$, there exists an open interval $I\subset \R$ such that $0 \in\overline{I}$ and $2^*F_j(s)=f_j(s)s$ for $s\in\overline{I}$. Then $F_j(s)=c|s|^{2^*}$ for some $c>0$, $s\in \overline{I}$, and $u_j\ge0$ solves $-\Delta u_j=2^*c|u_j|^{2^*-2}u_j$. Hence $u_j$ is an Aubin--Talenti instanton, up to scaling and translations, which is not $L^2$-integrable because $N\in\{3,4\}$, see \cite{Aubin_fr,Talenti}. 
	
	Suppose that there exists $\nu\in\{1,\dots,K-1\}$ such that, up to changing the order, $|u_j|_2<\rho_j$ for every $j\in\{1,\dots,\nu\}$ and $|u_j|_2=\rho_j$ for every $j\in\{\nu+1,\dots,K\}$.
	Arguing as before, there exist $0=\lambda_1=\dots=\lambda_{\nu}\leq \lambda_{\nu+1},\dots,\lambda_K$ such that
	\begin{equation*}
	\begin{cases}
	-\Delta u_j=\partial_jF(u), \quad j\in\{1,\dots,\nu\}\\
	-\Delta u_j+\lambda_ju_j=\partial_jF(u), \quad j\in\{\nu+1,\dots,K\}.
	\end{cases}
	\end{equation*}
	Since $F_j$ satisfies the scalar variant of (F5), $s\in]0,\infty[\mapsto F_j(s)/s^{2_\#}\in\R$ is nondecreasing, hence
	$F_j$ is nondecreasing as well for all $j$. Then, the first $\nu$ equations in the system above yield $-\Delta u_j\ge0$ for $j\in\{1,\dots,\nu\}$. Since  $u\in L^{\frac{N}{N-2}}(\rn)^K$ as $N\in\{3,4\}$, \cite[Lemma A.2]{Ikoma} implies $u_j=0$ for every $j\in\{1,\dots,\nu\}$. Notice we have proved that $\lambda_j=0$ implies $u_j=0$, hence, in particular, $\lambda_j>0$ for every $j\in\{\nu+1,\dots,K\}$.
\end{proof}

\begin{Rem}\label{R:alt}
We point out that, if (F0)--(F3), (F4,$\preceq$), (F5), and \eqref{e-eta2} hold, $u\in\cM\cap\cD$, and $J(u)=\inf_{\cM\cap\cD}J$, then we can show that $u\in\partial\cD$ for any dimension $N\geq 3$ provided that $H\preceq(2^*-2)F$ holds. As a matter of fact, observe that \eqref{e-Gg2^*} contradicts $H\preceq(2^*-2)F$  and Lemma \ref{L:ineq}. This gives a somewhat alternative proof of Lemma \ref{L:minons} \textit{(b)}.
\end{Rem}

\begin{proof}[Proof of Theorem \ref{T:Main1}]
	Point \textit{(a)} follows from Lemmas \ref{L:infismin} and \ref{L:radmin}. Now we prove point \textit{(b)}. From Lemma \ref{L:minons} \textit{(a)}, $u$ is of class $\cC^2$, while from Proposition \ref{P:Clarke} there exist $(\lambda_1,\dots,\lambda_K)\in[0,\infty[^K$ and $\sigma\in\R$ such that \eqref{e-solution} holds and $\sigma=0$ as in the proof of Lemma \ref{L:minons} \textit{(a)}.
\end{proof}

\begin{proof}[Proof of Theorem \ref{T:Main2}]
	It follows from Lemma \ref{L:minons} \textit{(b)}, Theorem \ref{T:Main1} (b), and the maximum principle \cite[Lemma IX.V.1]{Evans} (the implication $u_j\ne0\Rightarrow\lambda_j>0$ is proved as in the proof of Lemma \ref{L:minons} \textit{(b)}).
\end{proof}

\begin{Lem}\label{L:minons2}
	Suppose that $K=2$, $L=1$ and the assumptions in Lemma \ref{L:minons} \textit{(b)} hold. If $r_{1,1}+r_{2,1}>2_\#$ and $\beta_1$ is sufficiently large, then $u\in\cS$.
\end{Lem}
\begin{proof}
	Since $L=1$, we denote $\beta_1$, $r_{1,1}$, $r_{2,1}$ by  $\beta$, $r_{1}$, $r_{2}$ respectively.
	Suppose by contradiction that $u_1=0$ or $u_2=0$, say $u_1=0$, which implies that $|u_2|_2=\rho_2$. We want to find a suitable $w\in\cS$ such that
	\begin{equation}\label{e-sup}
		J(a\star w)<\inf_{\cM\cap\cD}J,
	\end{equation}
	where $a=a(w)$ is defined in Lemma \ref{L:phi} (note that $a(w)=b(w)$ because (F4,$\preceq$) holds), which is impossible.
	First we show that $\inf_{\cM\cap\cD}J$ does not depend on $\beta$. Consider the functional
	\[
	J_*\colon v\in H^1(\rn)\mapsto\int_{\rn}\frac12|\nabla v|^2-F_2(v)\,dx\in\R
	\]
	and the sets
	\[\begin{split}
		\cD_* & := \Set{ v \in H^1(\R^N) | \int_{\R^N} v^2 \, dx \leq \rho_2^2 },\\
		\cM_* & := \Set{ v \in H^1(\R^N) \setminus \{0\} | \int_{\R^N} |\nabla v|^2 \, dx = \frac{N}2 \int_{\rn} H_2(v) }.
	\end{split}\]
	Observe that $J(0,v)=J_*(v)$ for $v\in H^1(\R^N)$.  Moreover, $(0,v)\in\cD$ if and only if $v\in\cD_*$, and $(0,v)\in\cM$ if and only if $v\in\cM_*$. In particular,
	\[\begin{split}
	\inf_{\cM\cap\cD}J&=J(0,u_2)=J_*(u_2)\ge\inf_{\cM_*\cap\cD_*}J_*\\
	&=\inf\Set{J(0,v)|(0,v)\in\cM\cap\cD}\ge\inf_{\cM\cap\cD}J,
	\end{split}\]
	i.e., $\inf_{\cM\cap\cD}J=\inf_{\cM_*\cap\cD_*}J_*$, and the claim follows because $J_*$, $\cD_*$, and $\cM_*$ do not depend on $\beta$.
	
	In view of Theorem \ref{T:Main2} for $K=1$, there exists $\bar{v}\in\cM_*\cap\partial\cD_*$ such that $$J_*(\bar{v})=\inf_{\cM_*\cap\cD_*}J_*=\inf_{\cM\cap\cD}J=\inf_{\cM_*\cap\partial\cD_*}J_*.$$ Note that $\bar{v}$ does not depend on $\beta$. Define $w=(w_1,w_2):=\bigl(\frac{\rho_1}{\rho_2}\bar{v},\bar{v}\bigr)$.
	From Lemma \ref{L:phi}, $a=a_\beta$ is implicitly defined by
	\[\begin{split}
		\int_{\rn}|\nabla w|^2\,dx=&\,\frac{N}{2}\int_{\rn}\frac{F_1'(a_\beta^{N/2}w_1)a_\beta^{N/2}w_1-2F_1(a_\beta^{N/2}w_1)}{a_\beta^{N+2}}\\
		&+\frac{F_2'(a_\beta^{N/2}w_2)a_\beta^{N/2}w_2-2F_2(a_\beta^{N/2}w_2)}{a_\beta^{N+2}}\\
		&+\beta(r_1+r_2-2)a_\beta^{N(r_1+r_2-2)/2-2}w_1^{r_1}w_2^{r_2}\,dx\\
		\ge&\,\beta(r_1+r_2-2)a_\beta^{N(r_1+r_2-2)/2-2}\frac{N}{2}\int_{\rn} w_1^{r_1}w_2^{r_2}\,dx,
	\end{split}\]
	hence there exist $C>0$ not depending on $\beta$ such that
	\begin{equation}\label{e-beta1}
		0<\beta a_\beta^{N(r_1+r_2-2)/2-2}\le C,
	\end{equation}
	whence
	\begin{equation}\label{e-beta2}
		\lim_{\beta\to\infty}a_\beta=0.
	\end{equation}
	Since $a_\beta\star w\in\cM$, we have from (F5)
	\[\begin{split}
		J(a_\beta\star w)=&\,\int_{\rn}\frac{N}{4}H(a_\beta\star w)-F(a_\beta\star w)\,dx\le\frac{2}{N-2}\int_{\rn} F(a_\beta\star w)\,dx\\
		=&\,\frac{2}{N-2}\int_{\rn}\frac{F_1(a_\beta^{N/2}w_1)+F_2(a_\beta^{N/2}w_2)}{a_\beta^N}\,dx\\
		&+\frac{2\beta a_\beta^{N(r_1+r_2-2)/2}}{N-2}\int_{\rn} w_1^{r_1}w_2^{r_2}\,dx,
	\end{split}\]
	therefore \eqref{e-sup} holds true for sufficiently large $\beta$ owing to (F1), \eqref{e-beta1}, and \eqref{e-beta2}.
\end{proof}

\begin{Rem}
	The proof of Lemma \ref{L:minons2} shows that, under the assumptions of Theorem \ref{T:Main3}, every ground state solution $(\lambda,u)$ to \eqref{e-mainineq} is such that $u\in\cS$, hence a ground state solution to \eqref{e-main}.
\end{Rem}

\begin{proof}[Proof of Theorem \ref{T:Main3}]
	It follows directly from Lemma \ref{L:minons2} and Theorem \ref{T:Main2}.
\end{proof}

\section{On the ground state energy map}\label{S:gsemsuper}

In this section, for $\rho=(\rho_1,\dots,\rho_K)\in]0,\infty[^K$ we denote explicitly
\[\begin{split}
	\cS(\rho)&:=\Set{u\in\hrn^K|\left|u_j\right|_2=\rho_j \text{ for every } j\in\{1,\dots,K\}}\\
	\cD(\rho)&:=\Set{u\in\hrn^K|\left|u_j\right|_2\le \rho_j \text{ for every } j\in\{1,\dots,K\}}\\
	m(\rho)&:=\inf\Set{J(u)|u\in\cM\cap\cD(\rho)}.
\end{split}\]

\begin{Prop}\label{P:GSE}
	Assume that (F0) is satisfied and let $\cA$ be the subset of $]0,\infty[^K$ where \eqref{e-eta2} holds. Then $m\colon]0,\infty[^K\to\R\cup\{-\infty\}$ is nonincreasing and. If (F0)--(F5) are satisfied, then $m|_\cA$ is continuous and
	\[
	\lim_{|\rho|\to0^+}m(\rho)=\infty.
	\]
	If (F0)--(F5) are satisfied and every ground state solution to \eqref{e-mainineq} belongs to $\cS(\rho)$ (e.g., if the assumptions of Theorem \ref{T:Main3} are satisfied), then $m|_\cA$ is decreasing.
\end{Prop}
Here, `nonincreasing' and `decreasing' have the same meaning as in Proposition \ref{P:gsem}.
\begin{proof}
	The monotonicity of $m$ is obvious. Fix $\rho\in\cA$ and let $\cA\ni\rho^n\to\rho$ (note that $\cA$ is open). Let $u^n\in\cM\cap\cD(\rho^n)\subset\cM\cap\cD(2\rho)$ such that $J(u^n)=m(\rho^n)\le m(\rho/2)$. In view of Lemma \ref{L:coerc}, $u^n$ is bounded and so, arguing as in Lemma \ref{L:infismin}, there exists $u\in\cD(\rho)\setminus\{0\}$ such that, up to subsequences and translations, $u^n\rightharpoonup u$ in $H^1(\rn)^K$, $u^n\to u$ a.e. in $\rn$, and $R\ge1$, where $R=R_u>0$ is such that $u(R\cdot)\in\cM$. Fatou's Lemma and (F4) yield
	\[\begin{split}
	m(\rho)&\le J\bigl(u(R\cdot)\bigr)=\frac{1}{R^N}\int_{\rn}\frac{N}{4}H(u)-F(u)\,dx\le\int_{\rn} \frac{N}{4}H(u)-F(u)\,dx\\
	&\le\liminf_n\int_{\rn}\frac{N}{4}H(u^n)-F(u^n)\,dx=\liminf_nJ(u^n)=\liminf_nm(\rho^n).
	\end{split}\]
	Now let $w\in\cM\cap\cD(\rho)$ such that $J(w)=m(\rho)$. Denote $w_i^n:=\rho_i^nw_i/\rho_i$ and consider $w^n=(w_1^n,\dots,w_K^n)\in\cD(\rho^n)$. Due to Lemma \ref{L:phi}, for every $n$ there exists $s_n>0$ such that $s_n\star w^n\in\cM$. Note that
	\begin{equation}\label{e-conts}
	\begin{split}
	\frac{N}{2}\int_{\rn}\frac{H\bigl(s_n^{N/2}(\rho_1^nw_1/\rho_1,\dots,\rho_K^nw_K/\rho_K)\bigr)}{s_n^{N+2}}\,dx\\
	=\int_{\rn}|\nabla w^n|^2\,dx\to \int_{\rn}|\nabla w|^2\,dx.
	\end{split}
	\end{equation}
	If $\limsup_ns_n=\infty$, then from (F2) and (F5) the left-hand side of \eqref{e-conts} tends to $\infty$ up to a subsequence, which is a contradiction. If $\liminf_ns_n=0$, then from (F1), (F3), (F5) and \eqref{e-eta2} and arguing as in Lemma \ref{L:bdaw1} we obtain that the limit superior of the left-hand side of \eqref{e-conts} is less than $|\nabla w|_2^2$, which is again a contradiction. There follows that, up to a subsequence, $s_n\to s$ for some $s>0$ and $s\star w\in\cM$. In view of Lemma \ref{L:phi},
	\[
	\limsup_nm(\rho^n)\le\lim_nJ(s_n\star w_n)=J(s\star w)=J(w)=m(\rho)
	\]
	and the continuity of $m|_\cA$ is proved.
	
	Let $\rho^n\to0^+$ and $u^n\in\cM\cap\cD(\rho^n)$ such that $J(u^n)=m(\rho^n)$. Denote $s_n:=|\nabla u^n|_2^{-1}$ and $w^n:=s_n\star u^n$ and note that $s_n^{-1}\star w^n=u^n\in\cM$, $|\nabla w^n|_2=1$ and
	\[
	|w^n|_2^2=|u^n|_2^2=|\rho^n|^2\to0
	\]
	as $n\to\infty$. In particular $w^n$ is bounded in $L^{2^*}(\rn)^K$ and so
	\[
	|w^n|_{2_\#}\le|w^n|_2^\frac{2}{N+2}|w^n|_{2^*}^\frac{N}{N+2}\to0
	\]
	as $n\to\infty$. Then, in view of (F1) and (F3), for every $s>0$
	\[
	\lim_n\int_{\rn}\frac{F(s^{N/2}w^n)}{s^{-N}}\,dx=0
	\]
	and, consequently,
	\[
	J(u^n)=J(s_n^{-1}\star w^n)\ge J(s\star w^n)=\frac{s^2}{2}-\int_{\rn}\frac{F(s^{N/2}w^n)}{s^{-N}}\,dx=\frac{s^2}{2}+o(1),
	\]
	whence $\lim_nJ(u^n)=\infty$.
	
	The last part is proved similarly to Proposition \ref{P:gsem} \textit{(v)}.
\end{proof}

\chapter{Maxwell's and nonautonomous Schr\"odinger equations}\label{K:join}

\section{Introduction and statement of the results}

In this chapter, we provide a new outcome that somewhat joins Part \ref{1} and the previous chapters of Part \ref{2} together, i.e., we prove the existence of solutions to the problem
\begin{equation}\label{e-normcurl}
\begin{cases}
\nabla\times\nabla\times\UU+\lambda\UU=g(\UU)\\
\int_{\rn}|\UU|^2\,dx=\rho^2\\
(\lambda,\UU)\in\R\times H^1(\rn,\rn)
\end{cases}
\end{equation}
with $N\ge3$, where, as usual, $\rho>0$ is prescribed and $(\lambda,\UU)$ is the unknown.

The proof is carried out in two steps: first, we use the same machinery as in Chapter \ref{K:cylsym} to reduce the differential operator in \eqref{e-normcurl} to $-\Delta$ under suitable symmetry assumptions about $g$, then we utilize the results in Chapters \ref{K:sub} or \ref{K:super} in the case $K=1$, which is possible because the function $\UU\colon\rn\to\rn$ in \eqref{e-normcurl} is treated as a single vector-valued function rather than an $N$-tuple of scalar-valued ones (i.e., we have a single $L^2$-constraint) and so \eqref{e-normcurl} is formally equivalent to its scalar counterpart. Moreover, by an equivalence result in the spirit of Theorem \ref{T:ScalVec}, we also obtain solutions to the problem
\begin{equation}\label{e-normsing}
\begin{cases}
-\Delta u+\frac{u}{|y|^2}+\lambda u=f(u)\\
\int_{\rn}u^2\,dx=\rho^2\\
(\lambda,u)\in\R\times\bigl(\hrn\cap X\bigr),
\end{cases}
\end{equation}
where $x=(y,z)\in\R^2\times\R^{N-2}$ and $X$ is the same as in Section \ref{S:statecs}.

Recall, again from Section \ref{S:statecs}, the condition \eqref{e-fgh} (where $h$ is replaced with $g$), the definitions of $\SO$ and $\cF$, and define $F(u):=\int_0^uf(t)\,dt$, $G(\UU):=\int_0^1g(t\UU)\cdot\UU\,dt$, $H_\cF:=H^1(\rn,\rn)\cap\cF$, and
\[\begin{split}
D_\cF & := \Set{\UU\in H_\cF | \int_{\rn}|\UU|^2\,dx\le\rho^2},\\
S_\cF & := \Set{\UU\in H_\cF | \int_{\rn}|\UU|^2\,dx=\rho^2}=\partial D_\cF.
\end{split}\]
We recall as well the definitions of $\fH_\textup{r}$ and $\fH_\textup{s}$ from Subsection \ref{SS:cpt} and introduce
\[\begin{split}
\fH_\textup{r}^* & = \Set{\UU\in H^1(\rn,\rn)|e\UU=\UU(e\cdot) \text{ for all } e\in\SO(N)},\\
\fH_\textup{s}^* & = \Set{\UU\in H^1(\rn,\rn)|e\UU=\UU(e\cdot) \text{ for all } e\in\SO(2)\times\SO(N-2)}.
\end{split}\]
Finally, adapting the definitions from previous chapters to this context, we define
\begin{equation*}\begin{split}
E & \colon\UU\in H_\cF\mapsto\int_{\rn}\frac12|\nabla\times\UU|^2-G(\UU)\,dx\in\R,\\
\cM & :=\Set{\UU\in H_\cF\setminus\{0\}|\int_{\rn}|\nabla\times\UU|^2\,dx=\frac{N}2\int_{\rn}H(\UU)\,dx},
\end{split}\end{equation*}
where $H(w)=g(w)\cdot w-2G(w)$, $w\in\rn$. Our results about \eqref{e-normcurl} read as follows. We begin with the mass-(sub)critical case.

\begin{Th}\label{T:vec}
If $N\ge4$, $g$ satisfies \eqref{e-fgh}, $F$ satisfies (F0)--(F3) from Chapter \ref{K:sub}, \eqref{e-etas} holds, and $\eta_0=\infty$, then there exist $\lambda>0$ and $\UU\in S_\cF\cap\fH_\textup{s}^*$ such that $(\lambda,\UU)$ is a solution to \eqref{e-normcurl} and $E(\UU)=$$\inf_{D_\cF\cap\fH_\textup{s}^*}E<0$.
\end{Th}

We have no results about $\fH_\textup{r}^*$ because $\fH_\textup{r}^*\cap\cF=\{0\}$. It follows from the fact that each function in $\fH_\textup{r}^*\cap\cF$ is both divergence-free, from Lemma \ref{L:DivFree}, and curl-free, from \cite[Theorem 1.1]{BDPR}\footnote{In this reference, $\cO(3)$ is used, but the argument perfectly applies to $\SO(N)$, $N\ge3$.}, therefore trivial as it belongs to $\cD^{1,2}(\rn,\rn)$. As a consequence, Theorem \ref{T:vec} does not work for $N=3$ because we cannot use $\SO(3)$ to recover compactness, which makes it physically irrelevant. As for the mass-supercritical case, the following holds.

\begin{Th}\label{T:Vec}
If $g$ satisfies \eqref{e-fgh}, $F$ satisfies (F0)--(F3), (F4,$\preceq$), (F5) from Chapter \ref{K:super}, $H\preceq(2^*-2)F$, and \eqref{e-eta2} hold, then there exist $\lambda>0$ and $\UU\in\cM\cap S_\cF$ such that $(\lambda,\UU)$ is a solution to \eqref{e-normcurl} and $E(\UU)=\inf_{\cM\cap D_\cF}E>0$.
\end{Th}

Concerning \eqref{e-normsing}, recall from Section \ref{S:statecs} the definition of $X_{\SO}$ and define $H_{\SO} := X_{\SO}\cap\hrn$ and
\[\begin{split}
D_{\SO} & := \Set{u\in H_{\SO} | \int_{\rn}u^2\,dx\le\rho^2},\\
S_{\SO} & := \Set{u\in H_{\SO} | \int_{\rn}u^2\,dx=\rho^2}=\partial D_{\SO},\\
\cQ & := \Set{u\in H_{\SO} | \int_{\rn}|\nabla u|^2+\frac{u^2}{|y|^2}-\frac{N}2f(u)u+NF(u)\,dx=0}.
\end{split}\]

\begin{Th}\label{T:Scal}
(a) Assume $N=4$ or $N\ge6$. If $F$ is even and satisfies (F0)--(F3) from Chapter \ref{K:sub}, $\eta_0=\infty$, and \eqref{e-etas} holds, then there exist $\lambda>0$ and $u\in S_{\SO}\cap\fH_\textup{s}$ such that $(\lambda,u)$ is a solution to \eqref{e-normsing} and $J(u)=\inf_{D_{\SO}\cap\fH_\textup{s}}J<0$.

(b) If $F$ is even and satisfies (F0)--(F3), (F4,$\preceq$), (F5) from Chapter \ref{K:super}, $H\preceq(2^*-2)F$, and \eqref{e-eta2} hold, then there exist $\lambda>0$ and $u\in\cQ\cap S_{\SO}$ such that $(\lambda,u)$ is a solution to \eqref{e-normsing} and $J(u)=\inf_{\cQ\cap D_{\SO}}J>0$.
\end{Th}


\section{Proof of the results}

We begin with some results analogous to those from Section \ref{S:equiv}. Let $\UU\colon\rn\to\rn$ and $u\colon\rn\to\R$ satisfy \eqref{e-uU}, $\lambda\in\R$, and $\rho>0$. It is obvious that $\UU\in\fH_\textup{s}^*$ if and only if $u\in\fH_\textup{s}$. Although the nonlinearities in \eqref{e-normcurl} and \eqref{e-normsing} are autonomous, for the sake of completeness we state the following theorem for functions $f$ and $g$ which depend also on $x$.

\begin{Th}\label{T:equiv}
Assume $f\colon\rn\times\R\to\R$ is a Carath\'eodory function and there exists $C>0$ such that $f(ex,u)=f(x,u)$ and $|f(x,u)|\leq C(|u|+|u|^{2^*-1})$ for every $e\in\SO$, a.e. $x\in\rn$, and every $u\in\R$; let also $g$ satisfy \eqref{e-fgh}. Then
$\UU\in H_\cF$ if and only if $u\in H_{\SO}$ and, in such a case, $\nabla\cdot\UU=0$ and
$E(\UU)=J(u)$. Moreover,
$(\lambda,\UU)$ is a solution to \eqref{e-normcurl} 
if and only if $(\lambda,u)$ is a solution to \eqref{e-normsing}.
\end{Th}

When $N=3$, this follows from \cite[Theorem 1.1]{Bieganowski}, while the generalization to the case $N\ge3$ is trivial in view of Lemma \ref{L:dec}. In particular, one has $|\UU|=|u|$, $|\nabla\times\UU|_2=|\nabla\UU|_2=|\nabla u|_2$, $\int_{\rn}G(\UU)\,dx=\int_{\rn}F(u)\,dx$, and $\int_{\rn}g(\UU)\cdot\UU\,dx=\int_{\rn}f(u)u\,dx$, which immediately implies the following property.

\begin{Prop}\label{P:equiv}
Assume $f\colon\R\to\R$ is continuous and there exists $C>0$ such that $|f(u)|\le C(|u|+|u|^{2^*-1})$ for every $u\in\R$; let also $g$ satisfy \eqref{e-fgh}. Then $\UU\in\cM$ if and only if $u\in\cQ$.
\end{Prop}

\begin{Rem}
The set $\cQ$ is obtained as the counterpart of $\cM$ for the problem \eqref{e-normsing}. Nonetheless, one can obtain it directly from the Nehari and Poho\v{z}aev identities corresponding to \eqref{e-normsing}. As a matter of fact, arguing as in Section \ref{S:NP}, if $(\lambda,u)$ is such that
\[
-\Delta u+\frac{u}{|y|^2}+\lambda u=f(u) \quad \text{in } \rn,
\]
then it clearly satisfies the Nehari identity
\begin{equation}\label{e-NehSing}
\int_{\rn}|\nabla u|^2+\frac{u^2}{|y|^2}+\lambda u^2\,dx=\int_{\rn}f(u)u\,dx.
\end{equation}
In addition, if $u\in W_\textup{loc}^{2,p}(\rn)$ for every $p<\infty$ (or, in general, arguing heuristically), then $1$ is a critical point of the functional
\[
t\in]0,\infty[\mapsto J\bigl(u(t\cdot)\bigr)+\frac{\lambda}2\int_{\rn}u^2(tx)\,dx\in\R,
\]
i.e., after explicit computations,
\begin{equation}\label{e-PohSing}
\int_{\rn}(N-2)\left(|\nabla u|^2+\frac{u^2}{|y|^2}\right)+\lambda Nu^2-2NF(u)\,dx.
\end{equation}

Finally, as in Section \ref{S:prel}, combining linearly \eqref{e-NehSing} and \eqref{e-PohSing}, we obtain
\[
\int_{\rn}|\nabla u|^2+\frac{u^2}{|y|^2}-\frac{N}2f(u)u+NF(u)\,dx=0.
\]
\end{Rem}

Now we prove the main results of this chapter. Since the proofs are similar to those from Chapters \ref{K:sub} and \ref{K:super}, they are just sketched. Recall Palais's principle of symmetric criticality (Theorem \ref{T:Palais}) and that $\nabla\times\nabla\times\UU=-\Delta\UU$ for every $\UU\in H_\cF$ due to Theorem \ref{T:equiv}.

\begin{proof}[Proof of Theorem \ref{T:vec}]
We prove as in Lemma \ref{L:cbb} (see also Remark \ref{R:cbb}) that $E|_{D_\cF\cap\fH_\textup{s}^*}$ is coercive (and bounded from below) and consider a minimizing sequence $\UU_n\in D_\cF\cap\fH_\textup{s}^*$ for $E|_{D_\cF\cap\fH_\textup{s}^*}$, which is therefore bounded and, up to a subsequence, $\UU_n\rightharpoonup\UU$ in $H^1(\rn,\rn)$ for some $\UU\in D_\cF\cap\fH_\textup{s}^*$. Then, as in Lemma \ref{L:min}, we have that $\int_{\rn}G(\UU_n)\,dx\to\int_{\rn}G(\UU)\,dx$ and so $E(\UU)=\inf_{D_\cF\cap\fH_\textup{s}^*}E$. Moreover, as in Lemma \ref{L:neg}, $E(s\star\VV)<0$ for $\VV\in D_\cF\cap\fH_\textup{s}^*\setminus\{0\}$ and $0<s\ll1$, thus $E(\UU)<0$. Finally, as in the proof of Theorem \ref{T:main1}, there exists $\lambda\ge0$ such that $\nabla\times\nabla\times\UU+\lambda\UU=g(\UU)$ in $\rn$ and, if $\lambda=0$, then $E(\UU)\ge0$, hence $\lambda>0$ and, consequently, $\UU\in S_\cF\cap\fH_\textup{s}^*$, i.e., $(\lambda,\UU)$ is a solution to \eqref{e-normcurl}.
\end{proof}

\begin{proof}[Proof of Theorem \ref{T:Vec}]
We prove as in Lemma \ref{L:coerc} that $E|_{\cM\cap D_\cF}$ is coercive and consider a minimizing sequence $\UU_n\in\cM\cap D_\cF$ for $E|_{\cM\cap D_\cF}$, which is therefore bounded. Then, as in Lemma \ref{L:infismin}, we have that $\UU_n\rightharpoonup\UU$ up to subsequences and translations for some $\UU\in\cM\cap D_\cF$ such that $E(\UU)=\inf_{\cM\cap D_\cF}E$. Moreover, $E(\UU)>0$ arguing as in Lemma \ref{L:bdaw2}. Next, a similar argument to Lemma \ref{L:minons} and Remark \ref{R:alt} yields that $\UU\in\partial D_\cF=S_\cF$ and there exists $\lambda>0$ such that $\nabla\times\nabla\times\UU+\lambda\UU=g(\UU)$, i.e., $(\lambda,\UU)$ is a solution to \eqref{e-normcurl}.
\end{proof}

\begin{proof}[Proof of Theorem \ref{T:Scal}]
It follows from Theorem \ref{T:equiv} and either Theorem \ref{T:vec} (item \textit{(a)}) or Theorem \ref{T:Vec} and Proposition \ref{P:equiv} (item \textit{(b)}).
\end{proof}

\begin{Rem}
Assume $f$ and $g$ are as in Proposition \ref{P:equiv} and let $(\lambda,u)\in\R\times H_{\SO}$ be a solution to \eqref{e-normsing}. If we define $\UU$ from $u$ using \eqref{e-uU}, then, in virtue of Theorem \ref{T:equiv} and Lemma \ref{L:DivFree} respectively, $(\lambda,\UU)\in\R\times H_\cF$ is a solution to \eqref{e-normcurl} and $\nabla\times\nabla\times\UU=-\Delta\UU$. Consequently, $(\lambda,\UU)$ satisfies the Poho\v{z}aev identity (corresponding to \eqref{e-normcurl})
\[
\int_{\rn}|\nabla\times\UU|^2\,dx=\frac{N}{N-2}\int_{\rn}2G(\UU)-\lambda|\UU|^2\,dx
\]
and so, again from the same arguments as in Theorem \ref{T:equiv}, $(\lambda,u)$ satisfies \eqref{e-PohSing}. This last property is definitely not trivial because, due to the singular term $u/|y|^2$, a solution $\tilde{u}\in\hrn\cap X$ to the differential equation in \eqref{e-normsing} (for some $\lambda\in\R)$ need not belong to $W_\textup{loc}^{2,p}(\rn)$ for every $p<\infty$.
\end{Rem}


\begin{thebibliography}{999}

\bibitem{AckWeth} N. Ackermann, T. Weth: {\em Unstable normalized standing waves for the space periodic NLS}, Anal. PDE {\bf 12} (2019), no. 5, 1177-1213.

\bibitem{Adams} R. A. Adams: {\em Sobolev spaces}, Pure and Applied Mathematics, Vol. 65, Academic Press, New York-London, 1975.

\bibitem{AAS-C} N. Akhmediev, A. Ankiewicz, J. M. Soto-Crespo: {\em Does the nonlinear Schr\"odinger equation correctly describe beam propagation?}, Opt. Lett {\bf 18} (1993), no. 6, 411--413.

\bibitem{Akozbek} N. Ak\"ozbek, S. John: {\em Optical solitary waves in two- and three-dimensional nonlinear photonic band-gap structures}, Phys. Rev. E {\bf 57} (1998), no. 2, 2287--2319.

\bibitem{AmbRab} A. Ambrosetti, P. H. Rabinowitz: {\em Dual variational methods in critical point theory and applications}, J. Funct. Anal. {\bf 14} (1973), 349--381.

\bibitem{Aubin} T. Aubin: {\em Nonlinear analysis on manifolds}, Fundamental Principles of Mathematical Sciences, 252, Springer-Verlag, New York, 1982.

\bibitem{Aubin_fr} T. Aubin: {\em Probl\`emes isop\'erim\'etriques et espaces de Sobolev}, J. Differential Geom. {\bf 11} (1976), no. 4, 573--598.

\bibitem{AzzBenDApFor} A. Azzollini, V. Benci, T. D'Aprile, D. Fortunato: {\em Existence of Static Solutions of the Semilinear Maxwell Equations}, Ric. Mat. {\bf 55} (2006), no. 2, 283--297.

\bibitem{BadGuiRol} M. Badiale, M. Guida, S. Rolando: {\em Elliptic Equations With Decaying Cylindrical Potentials and power-type nonlinearities}, Adv. Differential Equations {\bf 12} (2007), no. 12, 1321--1362.

\bibitem{BadBenRol} M. Badiale, V. Benci, S. Rolando: {\em A nonlinear elliptic equation with singular potential and applications to nonlinear field equations}, J. Eur. Math. Soc. (JEMS) {\bf 9} (2007), no. 3, 355--381.

\bibitem{BadTar} M. Badiale, G. Tarantello: {\em A Sobolev-Hardy inequality with applications to a nonlinear elliptic equation arising in astrophysics}, Arch. Ration. Mech. Anal. {\bf 163} (2002), no. 4, 259--293.

\bibitem{bbf} P. Bartolo, V. Benci, D. Fortunato, \emph{Abstract critical point theorems and applications to some nonlinear problems with ``strong'' resonance at infinity},  Nonlinear Anal. {\bf 7} (1983), no. 9, 981--1012.

\bibitem{Bartsch} T. Bartsch, {\em Bifurcation in a multicomponent system of nonlinear Schr\"odinger equations}, J. Fixed Point Theory Appl., {\bf 13} (2013), no.1, 37--50.

\bibitem{BarDanWan} T. Bartsch, N. Dancer, Z.-Q. Wang, {\em A Liouville theorem, a-priori bounds, and bifurcating branches of positive solutions for a nonlinear elliptic system}, Calc. Var. Partial Differential Equations {\bf 37} (2010), no. 3-4, 345--361.

\bibitem{BDPR} T. Bartsch, T. Dohnal, M. Plum, W. Reichel: {\em Ground states of a nonlinear curl-curl problem in cylindrically symmetric media}, NoDEA Nonlininear Differential Equations Appl. {\bf 23} (2016), no. 5, Art. 52, 34 pp.

\bibitem{BarJean} T. Bartsch, L. Jeanjean: {\em Normalized solutions for nonlinear Schr\"odinger systems}, Proc. Roy. Soc. Edinbourgh Sect. A {\bf 148} (2018), no. 2, 225--242.

\bibitem{BarJeanSo} T. Bartsch, L. Jeanjean, N. Soave: {\em Normalized solutions for a system of coupled cubic Schr\"odinger equations on $\rr$}, J. Math. Pures Appl. (9) {\bf 106} (2016), no. 4, 583--614.

\bibitem{BartschMederski1} T. Bartsch, J. Mederski: {\em Ground and bound state solutions of semilinear time-harmonic Maxwell equations in a bounded domain}, Arch. Ration. Mech. Anal. {\bf 215} (2015), no.1, 283–306.

\bibitem{BartschMederski2} T. Bartsch, J. Mederski: {\em Nonlinear time-harmonic Maxwell equations in an anisotropic bounded medium}, J. Funct. Anal. {\bf 272} (2017), no. 10, 4304–4333.

\bibitem{BarMed} T. Bartsch, J. Mederski, {\em Nonlinear time-harmonic Maxwell equations in domains}, J. Fixed Point Theory Appl. {\bf 19} (2017), no. 1, 959--986.

\bibitem{BarSo0} T. Bartsch, N. Soave: {\em A natural constraint approach to normalized solutions of nonlinear Schr\"odinger equations and systems}, J. Funct. Anal. {\bf 272} (2017), no. 12, 4998--5037.

\bibitem{BarSo1} T. Bartsch, N. Soave: {\em Correction to `A natural constraint approach to normalized solutions of nonlinear Schr\"odinger equations and systems'}, J. Funct. Anal. {\bf 275} (2018), no. 2, 516--521.

\bibitem{BarSoa-m} T. Bartsch, N. Soave: {\em Multiple normalized solutions for a competing system of Schr\"odinger equations} Calc. Var. Partial Differential Equations {\bf 58} (2019), no. 1, Paper No. 22, 24 pp.

\bibitem{BartschVale} T. Bartsch, S. de Valeriola: {\em Normalized solutions of nonlinear Schr\"odinger equations}, Arch. Math. (Basel) {\bf 100} (2013), no. 1, 75--83.

\bibitem{BartschWillem} T. Bartsch, M. Willem: {\em Infinitely many nonradial solutions of a Euclidean scalar field equation}, J. Funct. Anal. {\bf 117} (1993), no. 2, 447--460.

\bibitem{BellJean} J. Bellazzini, L. Jeanjean: {\em On dipolar quantum gases in the unstable regime}, SIAM J. Math. Anal. {\bf 48} (2016), no. 3, 2028--2058.

\bibitem{BellJeanLuo} J. Bellazzini, L. Jeanjean, T. Luo: {\em Existence and instability of standing waves with prescribed norm for a class of Schr\"odinger-Poisson equations}, Proc. Lond. Math. Soc. (3) {\bf 107} (2013), no. 2, 303--339.

\bibitem{be} V. Benci, \emph{On critical point theory for indefinite functionals in the presence of symmetries}, Trans. Amer. Math. Soc. {\bf 274} (1982), no. 2, 533--572.

\bibitem{BenFor} V. Benci, D. Fortunato: 
{\em Towards a unified field theory for classical electrodynamics}, 
Arch. Ration. Mech. Anal. {\bf 173} (2004), no. 3, 379--414.

\bibitem{BerLions} H. Berestycki, P.L. Lions: {\em Nonlinear scalar field equations, I - existence of a ground state}, Arch. Ration. Mech. Anal. {\bf 82} (1983), no. 4, 313--345.

\bibitem{Bieganowski} B. Bieganowski: {\em Solutions to a nonlinear Maxwell equation with two competing nonlinearities in $\rr$}, Bull. Pol. Acad. Sci. Math. {\bf 69} (2021), no. 1, 37--60.

\bibitem{BiegMed} B. Bieganowski, J. Mederski: {\em Normalized ground states of the nonlinear Schr\"odinger equation with at least mass critical growth}, J. Funct. Anal. {\bf 280} (2021), no. 11, Paper No. 108989, 26 pp.

\bibitem{BCGJ} D. Bonheure, J.-B. Casteras, T. Guo, L. Jeanjean: {\em Normalized solutions to the mixed dispersion nonlinear Schr\"odinger equation in the mass critical and supercritical regime}, Trans. Amer. Math. Soc. {\bf 372} (2019), no. 3, 2167--2212.

\bibitem{BrezisLieb} H. Brezis, E. Lieb: {\em A relation between pointwise convergence of functions and convergence of functionals}, Proc. Amer. Math. Soc. {\bf 88} (1983), no. 3, 486--490.

\bibitem{BrezisLiebV} H. Brezis, E. Lieb: {\em Minimum action solutions of some vector field equations}, Comm. Math. Phys. {\bf 96} (1984), no. 1, 97--113.

\bibitem{BrNi} H. Brezis, L. Nirenberg: {\em Positive solutions of nonlinear elliptic equations involving critical Sobolev exponents}, Comm. Pure Appl. Math., {\bf 36} (1983), no. 4, 437--477.

\bibitem{BuffEstSere} B. Buffoni, M. J. Esteban, E. S\'er\'e: {\em Normalized solutions to strongly indefinite semilinear equations}, Adv. Nonlinear Stud. {\bf 6} (2006), no. 2, 323--347.

\bibitem{Cazenave:book} T. Cazenave: {\em Semilinear Schr\"odinger equations}, Courant Lecture Notes in Mathematics, 10, New York University, Courant Institute of Mathematical Sciences, New York, American Mathematical Society, Providence, RI, 2003.

\bibitem{CazeLions} T. Cazenave, P.-L. Lions: {\em Orbital stablity of standing waves for some nonlinear Schr\"odinger equations}, Comm. Math. Phys. {\bf 85} (1982), no. 4, 549--561.

\bibitem{Cerami} G. Cerami: {\em An existence criterion for the critical points on unbounded manifolds} (Italian), Ist.
Lombardo Accad. Sci. Lett. Rend. A {\bf 112} (1978), no. 2, 332--336 (1979).

\bibitem{ChenZou} Z. Chen, W. Zou, {\em An optimal constant for the existence of least energy solutions of a coupled Schr\"odinger system}, Calc. Var. Partial Differential Equations {\bf 48} (2013), no. 3-4, 695--711, 2013.

\bibitem{CCDY} A. Ciattoni, B. Crosignani, P. Di Porto, A. Yariv: {\em Perfect optical solitons: spatial Kerr solitons as exact solutions of Maxwell’s equations}, J. Opt. Soc. Am. B {\bf 22} (2005), no. 7, 1384--1394.

\bibitem{CinJean} S. Cingolani, L. Jeanjean: {\em Stationary waves with prescribed $L^2$-norm for the planar Schr\"odinger-Poisson system}, SIAM J. Math. Anal. {\bf 51} (2019), no. 4, 3533--3568.

\bibitem{ClPis} M. Clapp, A. Pistoia: {\em Existence and phase separation of entire solutions to a pure critical competitive elliptic system}, Calc. Var. Partial Differential Equations {\bf 57} (2018), no. 1, Art. 23, 20 pp.

\bibitem{Clarke} F. H. Clarke: {\em A New Approach to Lagrange Multipliers}, Math. Oper. Res. {\bf 1} (1976), no. 2, 165--174.


\bibitem{DApSic} T. D'Aprile, G. Siciliano:
{\em Magnetostatic solutions for a semilinear perturbation of the Maxwell equations},
Adv. Differential Equations {\bf 16} (2011), no. 5-6, 435--466.

\bibitem{Ding} W. Ding: {\em On a Conformally Invariant Elliptic Equation on $\R^n$}, Comm. Math. Phys. {\bf 107} (1986), no. 2, 331--335.

\bibitem{EstLions} M. J. Esteban, P.-L. Lions: {\em Stationary solutions of nonlinear Schr\"odinger equations with an external magnetic field}, Partial differential equations and the calculus of variations, Vol. I, 401–449, Progr. Nonlinear Differential Equations Appl., 1, Birkhäuser Boston, Boston, MA (1989).

\bibitem{Evans} L. C. Evans: {\em Partial Differential Equations}, second ed., American Mathematical Society, Providence, RI, 2010.

\bibitem{FibichMerle} G. Fibich, F. Merle: {\em Self-focusing on bounded domains}, Phys D. {\bf 155} (2001), no. 1-2, 132--158.

\bibitem{Gacz} M. Gaczkowski, J. Mederski, J. Schino, {\em Multiple solutions to cylindrically symmetric curl-curl problems and related Schr\"odinger equations with singular potentials}, arXiv:2006.03565.

\bibitem{Ghoussoub} N. Ghoussoub: {\em Duality and Perturbation Methods in Critical Point Theory}, Cambridge Tracts in Mathematics, 107, Cambridge University Press, Cambridge, 1993.

\bibitem{Gidas} B. Gidas: {\em Symmetry properties and isolated singularities of positive solutions of nonlinear elliptic equations}, Nonlinear partial differential equations in engineering and applied science (Proc. Conf., Univ. Rhode Island, Kingston, R.I., 1979), pp. 255–273, Lecture Notes in Pure and Appl. Math., 54, Dekker, New York, 1980.

\bibitem{GouJean1} T. Gou, L. Jeanjean: {\em Existence and orbital stability of standing waves for nonlinear Schr\"odinger Systems}, Nonlinear Anal. {\bf 144} (2016), 10--22.

\bibitem{GouJean2} T. Gou, L. Jeanjean: {\em Multiple positive normalized solutions for nonlinear Schr\"odinger systems}, Nonlinearity {\bf 31} (2018), no. 5, 2319--2345.

\bibitem{HirRei} A. Hirsch, W. Reichel: {\em Existence of Cylindrically Symmetric Ground States to a Nonlinear Curl-Curl	Equation with Non-Constant Coeffcients}, Z. Anal. Awend. {\bf 36} (2017), no. 4, 419--435.

\bibitem{Ikoma} N. Ikoma: {\em Compactness of Minimizing Sequences in Nonlinear Schr\"odinger Systems Under Multicostraint Conditions}, Adv. Nonlinear Stud. {\bf 14} (2014), no. 1, 115--136.

\bibitem{Iwaniec} T. Iwaniec: {\em Projections onto gradient fields and $L^p$-estimates for degenerated elliptic operators}, Studia Math. {\bf 75} (1983), no. 3, 293--312.

\bibitem{Jeanjean97} L. Jeanjean: {\em Existence of solutions with prescribed norm for semilinear elliptic equations}, Nonlinear Anal. {\bf 28} (1997), no. 10, 1633--1659.

\bibitem{JeanLu1} L. Jeanjean, S.-S. Lu: {\em A mass supercritical problem revisited}, Calc. Var. Partial Differential Equations {\bf 59} (2020), no. 5, Paper No. 174, 43 pp.

\bibitem{JeanLu2} L. Jeanjean, S.-S. Lu: {\em Nonradial normalized solutions for nonlinear scalar field equations}, Nonlinearity {\bf 32} (2019), no. 12, 4942--4966.


\bibitem{Kwong} M. K. Kwong: {\em Uniqueness of positive solutions of $\Delta u-u+u^p=0$ in $\R^N$}, Arch. Ration. Mech. Anal. {\bf 105} (1989), no. 3, 243--266

\bibitem{Leinfelder} H. Leinfelder: {\em Gauge invariance of Schr\"odinger operators and related spectral properties}, J. Operator Theory {\bf 9} (1983), no. 1, 163--179.

\bibitem{LiZou} H. Li, W. Zou: {\em Normalized ground states for semilinear elliptic systems with critical and subcritical nonlinearities}, J. Fixed Point Theory Appl. {\bf 23} (2021), no. 3, Paper No. 43, 30 pp.

\bibitem{Lieb} E. H. Lieb: {\em Existence and uniqueness of the minimizing solution of Choquard's nonlinear equation}, Stud. Appl. Math. {\bf 57} (1976/77), no. 2, 93--105.

\bibitem{LiebLoss} E. H. Lieb, M. Loss: {\em Analysis}, second ed., Graduate Studies in Mathematics, 14, American Mathematical Society, Providence, RI, 2001.

\bibitem{LSSY} E. H. Lieb, R. Seiringer, J. P. Solovej, J. Yngvason, {\em The Mathematics of the Bose Gas and its Condensation}, Oberwolfach Seminars, 34, Birk\"auser Verlag, Basel, 2005.

\bibitem{Lions_F} P.-L. Lions: {\em Sym\'etrie et compacit\'e dans les espaces de Sobolev}, J. Functional Analysis {\bf 49} (1982), no. 3, 315--334.

\bibitem{Lions84_1} P.-L. Lions: {\em The concentration-compactness principle in the calculus of variations. The locally compact case. Part I}, Ann. Inst. H. Poincar\'e Anal. Non Lin\'eaire {\bf 1} (1984), no. 2, 109--145.

\bibitem{Lions84_2} P.-L. Lions: {\em The concentration-compactness principle in the calculus of variations. The locally compact case. Part II}, Ann. Inst. H. Poincar\'e Anal. Non Lin\'eaire {\bf 1} (1984), no. 4 223--283.

\bibitem{LinWei1} T.-C. Lin, J. Wei, {\em Ground state of $N$ coupled nonlinear Schr\"odinger equations in $\R^n$, $n\le3$}, Comm. Math. Phys. {\bf 255} (2005), no. 3, 629--653.

\bibitem{LinWei2} T.-C. Lin, J. Wei, {\em Spikes in two coupled nonlinear Schr\"odinger equations}, Ann. Inst. Henri Poincaré, Anal. Non Linéaire, {\bf 22} (2005), no. 4, 403--439.

\bibitem{MaiMonPel} L.A. Maia, E. Montefusco, B. Pellacci, {\em Positive solutions for a weakly coupled nonlinear Schr\"odinger system}, J. Differential Equations, {\bf 229} (2006), no. 2, 743--767.

\bibitem{Mandel} R. Mandel, {\em Minimal energy solutions for cooperative nonlinear Schr\"odinger systems}, NoDEA Nonlinear Differ. Equ. Appl., {\bf 22} (2015), no. 2, 239--262.

\bibitem{McStTr} J. B. McLeod, C. A. Stuart, W. C. Troy: {\em An exact reduction of Maxwell's equations}, Nonlinear diffusion equations and their equilibrium states, 3 (Gregynog 1989), 391--405, Progr. Nonlinear Differential Equations Appl., vol. 7, Birkh\"auser Boston, Boston, MA (1992).

\bibitem{MederskiZeroMass} J. Mederski: {\em General class of optimal Sobolev inequalities and nonlinear scalar field equations}, J. Differential Equations {\bf 281} (2021), 411--441.

\bibitem{MederskiCPDE} J. Mederski: {\em Ground states of a system of nonlinear Schrödinger equations with periodic potentials}, Comm. Partial Differential Equations {\bf 41} (2016), no. 9, 1426--1440.

\bibitem{Mederski} J. Mederski: {\em Ground states of time-harmonic semilinear Maxwell equations in $\R^3$ with vanishing permittivity}, Arch. Ration. Mech. Anal. {\bf 218} (2015), no. 2, 825--861.

\bibitem{MedSURV} J. Mederski: {\em Nonlinear time-harmonic Maxwell equations in a bounded domain: lack of compactness}, Sci. China Math. {\bf 61} (2018), no. 11, 1936--1970.

\bibitem{NonradMed} J. Mederski: {\em Nonradial solutions for nonlinear scalar field equations}, Nonlinearity {\bf 33} (2020), no. 12, 6349--6380.

\bibitem{MederskiJFA} J. Mederski: {\em The Brezis-Nirenberg problem for the curl-curl operator}, J. Funct. Anal. {\bf 274} (2018), no. 5, 1345--1380.

\bibitem{MeSc} J. Mederski, J. Schino: {\em Least energy solutions to a cooperative system of Schr\"odinger equations with prescribed $L^2$-bounds: at least $L^2$-critical growth}, Calc. Var. Partial Differential Equations {\bf 61} (2022), no. 1, Paper No. 10, 31 pp.

\bibitem{MeScSz} J. Mederski, J. Schino, A. Szulkin: {\em Multiple solutions to a semilinear curl-curl problem in $\R^3$}, Arch. Ration. Mech. Anal. {\bf 236} (2020), no. 1, 253--288.

\bibitem{MeSzu} J. Mederski, A. Szulkin: {\em Sharp constant in the curl inequality and ground states for curl-curl problem with critical exponent}, Arch. Ration. Mech. Anal. {\bf 241} (2021), no. 3, 1815–1842.

\bibitem{NguyenWang} N. V. Nguyen, Z.-Q. Wang: {\em Existence and stability of a two-parameter family of solitary waves for a 2-coupled nonlinear Schrödinger system}, Discrete Contin. Dyn. Syst. {\bf 36} (2016), no. 2, 1005--1021.

\bibitem{NoTaVe1} B. Noris, H. Tavares, G. Verzini: {\em Existence and orbital stability of the ground states with prescribed mass for the $L^2$-critical and supercritical NLS on bounded domains}, Anal. PDE {\bf 7} (2014), no. 8, 1807--1838.

\bibitem{NoTaVe2} B. Noris, H. Tavares, G. Verzini: {\em Stable solitary waves with prescribed $L^2$-mass for the cubic Schrödinger system with trapping potentials}, Discrete Contin. Dyn. Sys. {\bf 35} (2915), no. 12, 6085--6112.

\bibitem{NoTaVe3} B. Noris, H. Tavares, G. Verzini: {\em Normalized solutions for nonlinear Schrödinger systems on bounded domains}, Nonlinearity {\bf 32} (2019), no. 3, 1044-1072.

\bibitem{ONeill} B. O'Neill: {\em Semi-Riemannian geometry}, Pure and Applied Mathematics, 103, Academic Press, Inc., New York, 1983.

\bibitem{Palais} R. S. Palais: {\em The Principle of Symmetric criticality}, Comm. Math. Phys. {\bf 69} (1979), no. 1, 19--30.

\bibitem{PierVer} D. Pierotti, G. Verzini: {\em Normalized bound states for the nonlinear Schrödinger equation in bounded domains}, Calc. Var. Partial Differential Equations {\bf 56} (2017), no. 5, Paper No. 133, 27 pp.

\bibitem{PiSt} L. Pitaevskii, S. Stringari, {\em Bose-Einstein condensation}, International Series of Monographs on Physics, 116, The Clarendon Press, Oxford University Press, Oxford, 2003.


\bibitem{QinTangLIN} D. Qin, X. Tang, {\em Time-harmonic Maxwell equations with asymptotically linear polarization}, Z. Angew. Math. Phys. {\bf 67} (2016), no. 3, Art. 39, 22 pp.

\bibitem{QS} P. Quittner, P. Souplet: {\em Superlinear Parabolic Problems}, second ed. Birkh\"auser/Springer, Cham, 2019.

\bibitem{Rabin} P. Rabinowitz: {\em Minimax Methods in Critical Point Theory with Applications to Differential Equations}, CBMS Reg. Conf. Ser. Math., 65, Amer. Math. Soc., Providence, RI, 1986.

\bibitem{RaoRen} M. M. Rao, Z. D. Ren, \emph{Theory of Orlicz spaces}, Monographs and Textbooks in Pure and Applied Mathemathics, 146, Marcel Dekker, Inc., New York, 1991.

\bibitem{Rudin} W. Rudin: {\em Real and complex analysis}, third ed., McGraw-Hill Book Co., New York, 1987.

\bibitem{Ruf} B. Ruf: {\em A sharp Trudinger--Moser type inequality for unbounded domains in $\R^2$}, J. Funct. Anal. {\bf 219} (2005), no. 2, 340--367.

\bibitem{SatWan} Y. Sato, Z.-Q. Wang, {\em Least energy solutions for nonlinear Schr\"odinger systems with mixed attractive and repulsive couplings}, Adv. Nonlinear Stud. {\bf 15} (2015), no. 1, 1--22.

\bibitem{Schino} J. Schino: {\em Normalized ground states to a cooperative system of Schr\"odinger equations with generic $L^2$-subcritical or $L^2$-critical nonlinearity}, Adv. Differential Equations {\bf 27} (2022), no. 7-8, 467–496.

\bibitem{Shibata} M. Shibata: {\em A new rearrangement inequality and its application for $L^2$-constraint minimizing problems}, Math. Z. {\bf 287} (2017), no. 1-2, 341--359.

\bibitem{Sirakov} B. Sirakov, {\em Least energy solitary waves for a system of nonlinear Schr\"odinger equations in $\rn$}, Comm. Math. Phys. {\bf 271} (2017), no. 1, 199--201, 2007.%

\bibitem{SluEgg} R. E. Slusher, B. J. Eggleton: {\em Nonlinear Photonic Crystals}, Springer Verlag, Berlin, Heidenberg, 2003.

\bibitem{SquassinaSzulkin} M.  Squassina, A. Szulkin: {\em Multiple solutions to logarithmic Schr\"odinger equations with periodic potential}, Calc. Var. Partial Differential Equations {\bf 54} (2015), no. 1, 585--597.

\bibitem{Soave} N. Soave, {\em Normalized ground states for the NLS equation with combined nonlinearities}, J. Differential Equations {\bf 269} (2020), no. 9, 6941--6987.

\bibitem{SoaveC} N. Soave, {\em Normalized ground states for the NLS equation with combined nonlinearities: the Sobolev critical case}, J. Funct. Anal. {\bf 279} (2020), no. 6, 108610, 43 pp.

\bibitem{Soave0} N. Soave, {\em On existence and phase separation of solitary waves for nonlinear Schr\"odinger systems modelling symultaneous cooperation and competition}, Calc. Var. Partial Differential Equations {\bf 53} (2015), no. 3-4, 689--718.

\bibitem{SoaTav} N. Soave, H. Tavares, {\em New existence and symmetry results for least energy positive solutions of Schr\"odinger systems with mixed competition and cooperation terms}, J. Differential Equations {\bf 261} (2016), no. 1, 505--537.

\bibitem{Strauss} W. A. Strauss: {\em Existence of solitary waves in higher dimensions}, Comm. Math. Phys. {\bf 55} (1977), no. 2, 149--162.

\bibitem{Struwe} M. Struwe: {\em Variational methods}, fourth ed., Results in Mathematics and Related Areas. third Series, A Series of Modern Surveys in Mathematics, 34, Springer-Verlag, Berlin, 2008.

\bibitem{Stuart79} C. A. Stuart: {\em A variational method for bifurcation problems when the linearization has no eigenvalues}, Confer. Sem. Mat. Univ. Bari No. 158--162 (1979), 157--180.

\bibitem{Stuart80} C. A. Stuart: {\em A variational approach to bifurcation without eigenvalues}, J. Funct. Anal. {\bf 38} (1980), 169--187.

\bibitem{Stuart82} C. A. Stuart: {\em Bifurcation for Dirichlet problems without eigenvalues}, Proc. London Math. Soc. (3) {\bf 45} (1982), no. 1, 169--192.

\bibitem{Stuart81} C. A. Stuart: {\em Bifurcation from the continuous spectrum in the $L^2$-theory of elliptic equations on $\R^n$}, Recent methods in nonlinear analysis and applications (Naples, 1980), 231–300, Liguori, Naples, 1981.

\bibitem{Stuart93} C. A. Stuart: {\em Guidance properties of nonlinear planar waveguides}, Arch. Ration. Mech. Anal. {\bf 125} (1993), no. 2, 145--200.

\bibitem{Stuart04} C. A. Stuart: {\em Modelling axi-symmetric travelling waves in a dielectric with nonlinear refractive index}, Milan J. Math. {\bf 72} (2004), 107--128.

\bibitem{Stuart90} C. A. Stuart: {\em Self-trapping of an electromagnetic field and bifurcation from the essential spectrum}, Arch. Ration. Mech. Anal. {\bf 113} (1990), no. 1, 65--96.

\bibitem{StuartZhou96} C. A. Stuart, H.-S. Zhou: {\em A variational problem related to self-trapping of an electromagnetic field}, Math. Methods Appl. Sci. {\bf 19} (1996), no. 17, 1397--1407.

\bibitem{StuartZhou03} C. A. Stuart, H.-S. Zhou: {\em A constrained minimization problem and its application to guided cylindrical TM-modes in an anisotropic self-focusing dielectric}, Calc. Var. Partial Differential Equations {\bf 16} (2003), no. 4, 335--373.

\bibitem{StuartZhou05} C. A. Stuart, H.-S. Zhou: {\em Axisymmetric TE-modes in a self-focusing dielectric}, SIAM J. Math. Anal. {\bf 37} (2005), no. 1, 218--237.

\bibitem{StuartZhou01} C. A. Stuart, H.-S. Zhou: {\em Existence of guided cylindrical TM-modes in a homogeneous self-focusing dielectric}, Ann. Inst. H. Poincar\'e Anal. Non Lin\'eaire {\bf 18} (2001), no. 1, 69--96.

\bibitem{StuartZhou10} C. A. Stuart, H.-S. Zhou: {\em Existence of guided cylindrical TM-modes in an inhomogeneous self-focusing dielectric}, Math. Models Methods Appl. Sci. {\bf 20} (2010), no. 9, 1681--1719.

\bibitem{SzWe} A. Szulkin, T. Weth: {\em Ground state solutions for some indefinite variational problems}, J. Funct. Anal. {\bf 257} (2009), no. 12, 3802--3822. 

\bibitem{SzWeHandbook} 
A. Szulkin, T. Weth: {\em The method of Nehari manifold}, Handbook of nonconvex analysis and applications, 597--632, Int. Press, Somerville, MA, 2010.

\bibitem{Talenti} G. Talenti: \textit{Best constants in Sobolev inequality}, Ann. Mat. Pura Appl. (4) \textbf{110} (1976), 353--372.

\bibitem{TavTer} H. Tavares, S. Terracini: {\em Sign-changing solutions of competition-diffusion elliptic systems and optimal partition problems}, Ann. Inst. H. Poincar\'e Anal. Non Lin\'eaire {\bf 29} (2012), no. 2, 279--300.

\bibitem{TerVer} S. Terracini, G. Verzini, {\em Multipulse phases in $k$-mixtures of Bose-Einstein condensates}, Arch. Ration. Mech. Anal., {\bf 194} (2009), no. 3, 717--741.

\bibitem{WeiWeth} J. Wei, T. Weth, {\em Radial solutions and phase separation in a system of two coupled Schr\"odinger equations}, Arch. Ration. Mech. Anal., {\bf 190} (2008), no. 1, 83--106.

\bibitem{Weinstein} M. I. Weinstein: {\em Nonlinear Schr\"odinger equations and sharp interpolation estimates}, Comm. Math. Phys, {\bf 87} (1982/83), no. 4, 567--576.

\bibitem{Willem} M. Willem: {\em Minimax Theorems}, Progress in Nonlinear Differential Equations and their Applications, 24, Birkhäuser Boston, Inc., Boston, MA, 1996.

\end{thebibliography}
\end{document}